\documentclass[
11pt,  reqno]{amsart}
\usepackage[margin=1.2in,marginparwidth=1.5cm, marginparsep=0.5cm]{geometry}
\pdfoutput=1

\usepackage{booktabs} 
\usepackage{microtype}
\usepackage{amssymb}
\usepackage{mathrsfs}
\usepackage{stmaryrd}

\usepackage{upgreek}

\usepackage{color}
\usepackage[implicit=true]{hyperref}

\usepackage{xsavebox}

\usepackage{cases}

\allowdisplaybreaks[3]

\newcommand{\deff}{\stackrel{\textup{def}}{=}}

\usepackage{tikz}

\usepackage{marginnote}
\usepackage{scalerel} 

\usetikzlibrary{shapes.misc}
\usetikzlibrary{shapes.symbols}
\usetikzlibrary{decorations}
\usetikzlibrary{decorations.markings}

\tikzset{
	dot/.style={circle,fill=black,draw=black,inner sep=0pt,minimum size=0.5mm},
	>=stealth,
	}

\tikzset{
	dot2/.style={circle,fill=black,draw=black,inner sep=0pt,minimum size=0.2mm},
	>=stealth,
	}

\tikzset{
	ddot/.style={circle,fill=black,draw=black,inner sep=0pt,minimum size=0.8mm},
	>=stealth,
	}

\tikzset{decision/.style={ 
        draw,
        diamond,
        aspect=1.5
    }}

\tikzset{dia2/.style
={diamond,fill=white,draw=black,inner sep=0pt,minimum size=1mm},
	>=stealth,
	}

\tikzset{dia/.style
={star,fill=black,draw=black,inner sep=0pt,minimum size=1mm},
	>=stealth,
	}

\tikzset{dia/.style
={diamond,fill=black,draw=black,inner sep=0pt,minimum size=1.3mm},
	>=stealth,
	}

\makeatletter
\def\DeclareSymbol#1#2#3{\xsavebox{#1}{\tikz[baseline=#2,scale=0.15]{#3}}}
\def\<#1>{\xusebox{#1}}
\makeatother

\newsavebox{\peA}
\newsavebox{\pneA}
\newsavebox{\plA}
\newsavebox{\pgA}
\newsavebox{\pleA}
\newsavebox{\pgeA}
\newsavebox{\pezA}

\savebox{\peA}{\tikz \draw (0,0) node[shape=circle,draw,inner sep=0pt,minimum size=8.5pt] {\scriptsize  $=$};}
\savebox{\pneA}{\tikz \draw (0,0) node[shape=circle,draw,inner sep=0pt,minimum size=8.5pt] {\footnotesize $\neq$};}
\savebox{\plA}{\tikz \draw (0,0) node[shape=circle,draw,inner sep=0pt,minimum size=8.5pt] {\scriptsize $<$};}
\savebox{\pgA}{\tikz \draw (0,0) node[shape=circle,draw,inner sep=0pt,minimum size=8.5pt] {\scriptsize $>$};}
\savebox{\pleA}{\tikz \draw (0,0) node[shape=circle,draw,inner sep=0pt,minimum size=8.5pt] {\scriptsize $\leqslant$};}
\savebox{\pgeA}{\tikz \draw (0,0) node[shape=circle,draw,inner sep=0pt,minimum size=8.5pt] {\scriptsize $\geqslant$};}
\savebox{\pezA}{\tikz \draw (0,0) node[shape=circle,draw,
fill=white, 
inner sep=0pt,minimum size=8.5pt]{} ;}

\def \peB{\mathchoice
{\scalebox{.7}{{\usebox{\peA}}}}
{\scalebox{.7}{{\usebox{\peA}}}}
{\scalebox{.7}{{\usebox{\peA}}}}
{}
}

\def \pezB{\mathchoice
{\scalebox{.7}{{\usebox{\pezA}}}}
{\scalebox{.7}{{\usebox{\pezA}}}}
{\scalebox{.7}{{\usebox{\pezA}}}}
{}
}

\newcommand{\pe}{\mathbin{{\peB}}}

\newcommand{\pez}{\mathbin{{\pezB}}}

\usetikzlibrary{decorations.pathmorphing, patterns,shapes}

\DeclareSymbol{dot}{-2.7}{\draw (0,0) node[dot]{};}

 \DeclareSymbol{1}{0}{\clip (-.7,0) rectangle (.7,1.6); \draw (0,0)  -- (0,1.2) node[dot] {};}

\DeclareSymbol{2}{0}{\draw (-0.5,1.2) node[dot] {} -- (0,0) -- (0.5,1.2) node[dot] {};}
\DeclareSymbol{3}{0}{\draw (0,0) -- (0,1.2) node[dot] {}; \draw (-.7,1) node[dot] {} -- (0,0) -- (.7,1) node[dot] {};}
\DeclareSymbol{30}{-3}{\draw (0,0) -- (0,-1); \draw (0,0) -- (0,1.2) node[dot] {}; \draw (-.7,1) node[dot] {} -- (0,0) -- (.7,1) node[dot] {};}
\DeclareSymbol{31}{-3}{\draw (0,0) -- (0,-1) -- (1,0) node[dot] {}; \draw (0,0) -- (0,1.2) node[dot] {}; \draw (-.7,1) node[dot] {} -- (0,0) -- (.7,1) node[dot] {};}
\DeclareSymbol{32}{-3}{\draw (0,0) -- (0,-1) -- (1,0) node[dot] {}; \draw (0,0) -- (0,-1) -- (-1,0) node[dot] {}; \draw (0,0) -- (0,1.2) node[dot] {}; \draw (-.7,1) node[dot] {} -- (0,0) -- (.7,1) node[dot] {};}
\DeclareSymbol{20}{-3}{\draw (0,0) -- (0,-1);\draw (-.7,1) node[dot] {} -- (0,0) -- (.7,1) node[dot] {};}
\DeclareSymbol{22}{-3}{\draw (0,0.3) -- (0,-1) -- (1,0) node[dot] {}; \draw (0,0.3) -- (0,-1) -- (-1,0) node[dot] {};\draw (-.7,1) node[dot] {} -- (0,0.3) -- (.7,1) node[dot] {};}
\DeclareSymbol{31p}{-3}{\draw (0,0) -- (0,-1) -- (1,0) node[dot] {}; \draw (0,0) -- (0,1.2) node[dot] {}; \draw (-.7,1) node[dot] {} -- (0,0) -- (.7,1) node[dot] {};
\draw (0,-1) node{\scaleobj{0.5}{\pez}};
\draw (0,-1) node{\scaleobj{0.5}{\pe}}; }
\DeclareSymbol{32p}{-3}{\draw (0,0) -- (0,-1) -- (1,0) node[dot] {}; \draw (0,0) -- (0,-1) -- (-1,0) node[dot] {}; \draw (0,0) -- (0,1.2) node[dot] {}; \draw (-.7,1) node[dot] {} -- (0,0) -- (.7,1) node[dot] {};
\draw (0,-1) node{\scaleobj{0.5}{\pez}};
\draw (0,-1) node{\scaleobj{0.5}{\pe}};}
\DeclareSymbol{22p}{-3}{\draw (0,0.3) -- (0,-1) -- (1,0) node[dot] {}; \draw (0,0.3) -- (0,-1) -- (-1,0) node[dot] {};\draw (-.7,1) node[dot] {} -- (0,0.3) -- (.7,1) node[dot] {};
\draw (0,-1) node{\scaleobj{0.5}{\pez}};
\draw (0,-1) node{\scaleobj{0.5}{\pe}};}
\DeclareSymbol{21p}{-3}{\draw (0,0.3) -- (0,-1) -- (1,0) node[dot] {};
\draw (-.7,1) node[dot] {} -- (0,0.3) -- (.7,1) node[dot] {};
\draw (0,-1) node{\scaleobj{0.5}{\pez}};
\draw (0,-1) node{\scaleobj{0.5}{\pe}};}

\DeclareSymbol{21}{-3}{\draw (0,0.3) -- (0,-1) -- (1,0) node[dot] {};
\draw (-.7,1) node[dot] {} -- (0,0.3) -- (.7,1) node[dot] {};
}

\DeclareSymbol{70}{-3}{
\draw (0,-0.4) -- (0,-1.2);
\draw (0,0.8) node[dot] {} -- (0,-0.4) -- (0.8,0.6);
\draw (0.8, .6) -- (1.4, 1.35)  node[dot] {};
\draw (0.8, .6) -- (1.8, .8)  node[dot] {};
\draw (0.8, .6) -- (0.85, 1.6)  node[dot] {};
\draw (0,.6) -- (0,-0.4) -- (-0.8,0.6) ;
\draw (-0.8, 0.6) -- (-1.4, 1.35)  node[dot] {};
\draw (-0.8, .6) -- (-1.8, .8)  node[dot] {};
\draw (-0.8, .6) -- (-0.85, 1.6)  node[dot] {};
}

\DeclareSymbol{7}{-3}{
\draw (0,0.2) node[dot] {} -- (0,-1) -- (0.8,0);
\draw (0.8, 0) -- (1.4, 0.75)  node[dot] {};
\draw (0.8, 0) -- (1.8, 0.2)  node[dot] {};
\draw (0.8, 0) -- (0.85, 1)  node[dot] {};
\draw (0,0) -- (0,-1) -- (-0.8,0) ;
\draw (-0.8, 0) -- (-1.4, 0.75)  node[dot] {};
\draw (-0.8, 0) -- (-1.8, 0.2)  node[dot] {};
\draw (-0.8, 0) -- (-0.85, 1)  node[dot] {};
}

\DeclareSymbol{11p}{-3}{\draw (0,0) node[dot2] {} -- (0,-1) -- (1,0) node[dot] {};
\draw (0,0) -- (0,1.2) node[dot] {};  
\draw (0,-1) node{\scaleobj{0.5}{\pez}};
\draw (0,-1) node{\scaleobj{0.5}{\pe}}; }
\DeclareSymbol{100}{-3}
{\draw[white] (-.4,0) -- (.4,0);
\draw (0,0) node[dot2] {} -- (0,-1)  ;
\draw (0,0) -- (0,1.2) node[dot] {};}  

\DeclareSymbol{320}{-3}{
\draw (0,0.4) -- (0,-1.2);
\draw (0,0.6) -- (0,-0.4) -- (1,0.6) node[dot] {};
\draw (0,0.6) -- (0,-0.4) -- (-1,0.6) node[dot] {};
\draw (0,0.6) -- (0,1.8) node[dot] {};
 \draw (-.7, 1.6) node[dot] {} -- (0,0.6) -- (.7,1.6) node[dot] {};}

\usetikzlibrary{mindmap,backgrounds,trees,arrows,decorations.pathmorphing,decorations.pathreplacing,positioning,shapes.geometric,shapes.misc,decorations.markings,decorations.fractals,calc,patterns}

\tikzset{>=stealth',
         cvertex/.style={circle,draw=black,inner sep=1pt,outer sep=3pt},
         vertex/.style={circle,fill=black,inner sep=1pt,outer sep=3pt},
         star/.style={circle,fill=yellow,inner sep=0.75pt,outer sep=0.75pt},
         tvertex/.style={inner sep=1pt,font=\scriptsize},
         gap/.style={inner sep=0.5pt,fill=white}}
\tikzstyle{mybox} = [draw=black, fill=blue!10, very thick,
    rectangle, rounded corners, inner sep=10pt, inner ysep=20pt]
\tikzstyle{boxtitle} =[fill=blue!50, text=white,rectangle,rounded corners]

\tikzstyle{decision} = [diamond, draw, fill=blue!20,
    text width=4.5em, text badly centered, node distance=2.5cm, inner sep=0pt]
\tikzstyle{block} = [rectangle, draw, fill=blue!20,
    text width=2.7cm, text centered, rounded corners, minimum height=3em]

\tikzstyle{line} = [draw, very thick, color=black!50, -latex']

\tikzstyle{cloud} = [draw, ellipse,fill=red!40,
node distance=2.5cm,
    minimum height=1em]

\tikzstyle{cloud2} = [draw, ellipse,fill=red!30, text=white,text width=10em, node distance=2.5cm, text centered, minimum height=4em]

\tikzstyle{cloud3} = [draw, ellipse, fill=cyan!30,
node distance=2.5cm,
    minimum height=3em,
    text width=1.7cm]

\tikzstyle{cloud4} = [draw, ellipse,fill=orange!70, node distance=2.5cm,
   text width=2.1cm,
           minimum height=3em]

\tikzstyle{cloud5} = [draw, ellipse,fill=red!20, node distance=2.5cm,
    text centered,
    text width=3.0cm,
    minimum height=3em]

\tikzstyle{cloud6} = [draw, ellipse,fill=red!20, node distance=2.5cm,
    text width=1.5cm,
    text centered,
    minimum height=3em]

\tikzset{
    position/.style args={#1:#2 from #3}{
        at=(#3.#1), anchor=#1+180, shift=(#1:#2)
    }
}

\newtheorem{theorem}{Theorem} [section]

\newtheorem{lemma}[theorem]{Lemma}
\newtheorem{proposition}[theorem]{Proposition}
\newtheorem{remark}[theorem]{Remark}

\DeclareMathOperator{\Id}{Id}

\newcommand{\noi}{\noindent}
\newcommand{\Z}{\mathbb{Z}}
\newcommand{\R}{\mathbb{R}}
\newcommand{\T}{\mathbb{T}}

\let\P= \undefined
\newcommand{\P}{\mathbf{P}}

\newcommand{\E}{\mathbb{E}}

\renewcommand{\L}{\mathcal{L}}

\newcommand{\F}{\mathcal{F}}

\newcommand{\al}{\alpha}
\newcommand{\be}{\beta}
\newcommand{\dl}{\delta}

\newcommand{\nb}{\nabla}

\newcommand{\Dl}{\Delta}
\newcommand{\eps}{\varepsilon}

\newcommand{\g}{\gamma}
\newcommand{\G}{\Gamma}
\newcommand{\ld}{\lambda}
\newcommand{\Ld}{\Lambda}
\newcommand{\s}{\sigma}
\newcommand{\Si}{\Sigma}
\newcommand{\ft}{\widehat}

\newcommand{\wt}{\widetilde}
\newcommand{\cj}{\overline}

\newcommand{\dt}{\partial_t}

\newcommand{\ta}{\theta}

\renewcommand{\l}{\ell}
\renewcommand{\o}{\omega}
\renewcommand{\O}{\Omega}

\newcommand{\les}{\lesssim}
\newcommand{\ges}{\gtrsim}

\newcommand{\jb}[1]
{\langle #1 \rangle}

\newcommand{\jbb}[1]
{[\hspace{-0.6mm}[ #1 ]\hspace{-0.6mm}]}

\newcommand{\ind}{\mathbf 1}

\newcommand{\N}{\mathbb{N}}

\newtheorem*{ackno}{Acknowledgements}

\newcommand{\I}{\mathcal{I}}
\newcommand{\J}{\mathcal{J}}
\newcommand{\If}{\mathfrak{I}}

\numberwithin{equation}{section}
\numberwithin{theorem}{section}

\newcommand{\PP}{\mathbb{P}}

\DeclareMathOperator{\Law}{Law}

\newcommand{\ZZ}{\mathfrak{Z}}

\newcommand{\rhoo}{\vec{\rho}}

\newcommand{\Dr}{\Theta}

\newcommand{\NN}{\mathcal{N}}
\newcommand{\D}{\mathcal{D}}

\newcommand{\dia}{\diamond}

\newcommand{\too}{\longrightarrow}

\newcommand{\Ups}{\Upsilon}

\usepackage{comment}

\makeatletter
\@namedef{subjclassname@2020}{%
  \textup{2020} Mathematics Subject Classification}
\makeatother

\begin{document}
\baselineskip = 14pt

\title[Global dynamics of the fractional hyperbolic $\Phi^4_3$-model]{Existence, uniqueness, and universality of global dynamics for the fractional hyperbolic $\Phi^4_3$-model}

\author[R.~Liu, N.~Tzvetkov, and Y.~Wang]
{Ruoyuan Liu, Nikolay Tzvetkov, and Yuzhao Wang}

\address{
Ruoyuan Liu, School of Mathematics,
The University of Edinburgh,
and The Maxwell Institute for the Mathematical Sciences,
James Clerk Maxwell Building,
The King's Buildings,
Peter Guthrie Tait Road,
Edinburgh,
EH9 3FD,
United Kingdom\\
and Mathematical Institute\\
University of Bonn\\
Endenicher Allee 60\\
53115\\
Bonn\\
Germany}
\email{ruoyuanl@math.uni-bonn.de}

\address{Nikolay Tzvetkov, Institut Universitaire de France (IUF) and Ecole Normale Sup\'erieure de Lyon, UMPA, UMR CNRS-ENSL 5669, 46, all\'ee d'Italie, 69364-Lyon Cedex 07, France}
\email{nikolay.tzvetkov@ens-lyon.fr}

\address{Yuzhao Wang, School of Mathematical Sciences, Dalian University of Technology, Dalian, 116024,
China, and School of Mathematics, Watson Building, University of Birmingham, Edgbaston, Birmingham, B15 2TT, United Kingdom}
\email{wangyuzhao2008@gmail.com}



\begin{abstract}
We study the fractional $\Phi^4_3$-measure (with order $\al > 1$) and the dynamical problem of its canonical stochastic quantization: the three-dimensional stochastic damped fractional nonlinear wave equation with a cubic nonlinearity, also called the fractional hyperbolic $\Phi^4_3$-model. We first construct the fractional $\Phi^4_3$-measure via the variational approach by Barashkov-Gubinelli (2020). When $\al \leq \frac 98$, this fractional $\Phi^4_3$-measure turns out to be singular with respect to the base Gaussian measure. We then prove almost sure global well-posedness of the fractional hyperbolic $\Phi^4_3$-model and invariance of the fractional $\Phi^4_3$-measure for all $\al > 1$ by further developing the globalization framework due to Oh-Okamoto-Tolomeo (2024) on the hyperbolic $\Phi^3_3$-model. Furthermore, when $\al > \frac 98$, we prove weak universality of the fractional hyperbolic $\Phi^4_3$-model by utilizing the convergence of Gibbs measures.
\end{abstract}

%
\maketitle
\tableofcontents

\section{Introduction}
\label{SEC:1}

\subsection{Overview}
\label{SUB:1.1}
Since Bourgain's seminal work in \cite{BO94, BO96}, there have been extensive studies in invariant Gibbs measures for nonlinear Hamiltonian dispersive PDEs. A typical example is the following stochastic damped cubic nonlinear wave equation:
\begin{align*}
\dt^2 u + \dt u + (1 - \Dl) u + u^3 = \xi
\end{align*}

\noi
with $\xi$ denoting a space-time white noise forcing, which formally leaves the following Gibbs measure (also called the $\Phi^4$-measure from the constructive quantum field theory) invariant:
\begin{align*}
\text{``} \, d \vec \rho (\vec u) = Z^{-1} \exp \bigg( - \frac 14 \int u^4 dx - \frac 12 \int |\jb{\nb} u|^2 dx - \frac 12 \int (\dt u)^2 dx \bigg) d \vec u \, \text{''},
\end{align*}

\noi
where $\vec u = (u, \dt u)$, $\jb{\cdot} = (1 + |\cdot|^2)^{\frac 12}$, and $Z > 0$ is a normalizing factor. In particular, there has been a huge progress in the study of global well-posedness and invariant Gibbs measures for deterministic or stochastic (damped) nonlinear wave equations (NLW) with a power type nonlinearity on various domains. We refer the readers to \cite{BB12, BB14, Bring1, Bring2, BDNY, BTz07, BTz08, GKOT, OOTol1, OOT2, ORTz, OTh2, OTWZ} and the references therein. Many of these works addressed the case where the solution is not a function but merely a distribution, in which case we say that the equation is singular. We would also like to mention that the work by Oh, Robert, and the second author in \cite{ORTz} on stochastic NLW on compact surfaces is the first time when Bourgain's invariant measure argument \cite{BO94} was applied to a singular dispersive PDE with a stochastic forcing. We also refer the readers to \cite{BOP2, Deya1, Deya2, GKO, GKO2, OOcomp, OOR, OPTz, OWZ, Tolomeo2} for more well-posedness results for stochastic NLW or NLW with rough random initial data.

A natural and interesting question is how the strength of dispersion can affect well-posedness or invariance results of wave equations. This leads to the consideration of the NLW with a fractional Laplacian $(1 - \Dl)^\al$, which we refer to as the fractional NLW. When $\al = 1$, one recovers the usual NLW. By varying the value of $\al > 0$, one can find the limitations for certain well-established methods in a quantitative manner. In addition, one can see interesting phase transitions of the equation as $\al > 0$ meets some certain threshold.
Fractional NLW was first studied by Sun, Xu, and the second author in \cite{STzX} and also by Forcella-Pocovnicu in \cite{FP} on the two-dimensional torus $\T^2$ and with dispersion $(1 - \Dl)^\al$ for some range of $\al < 1$. 
In the special case when $\al = 2$, the fractional NLW becomes the nonlinear beam equation, which has been studied in \cite{BMS, BOS, CLL}. We also mention that nonlinear PDEs with a fractional Laplacian have been studied in elliptic models \cite{Duch22}, parabolic models \cite{Duch21, DGR, EW}, and other dispersive models such as the Schr\"odinger equations \cite{LW1, LW2, STz20, STz21}.

Another interesting topic regarding the aforementioned models is weak universality. The goal of weak universality is to show that the models can be identified as a universal limit of microscopic models with general nonlinear interactions satisfying some structural assumptions. Weak universality was initiated by Hairer-Quastel \cite{HQ} in deriving the KPZ equation from a large class of microscopic growth models. It has then been studied in various settings of singular stochastic parabolic equations. We refer the readers to \cite{EX, FG, GP16, GP17, HX18, HX19, KWX, SX}. Recently, there has been some important development for weak universality of dispersive equations; see \cite{GKO, GKO2, OTh2, STzX}. In particular in \cite{STzX}, in the context of the two-dimensional fractional nonlinear wave equation, Sun, Xu, and the second author established for the first time a weak universality result for a dispersive equation in the presence of probabilistically super-critical nonlinearity. Due to the lack of $L^\infty$ based estimates in the dispersive setting, they exploited the invariant Gibbs measure for the equation and obtained a key probabilistic a priori bound on the solution.

Our work in this paper contributes to the above lines of research. Specifically, we study the following stochastic damped fractional nonlinear wave equation (SdfNLW)
with a cubic nonlinearity on
the three-dimensional torus $\T^3 = (\R/2\pi\Z)^3$, formally given by
\begin{align}
\dt^2 u + \dt u + (1 -  \Dl)^\al  u  +  u^3  = \sqrt{2} \xi,
\qquad (x, t) \in \T^3\times \R_+,
\label{SNLW0}
\end{align}
where $\al \geq 1$, $u$ is real-valued, and
$\xi$ denotes a Gaussian space-time white noise on $\T^3\times \R_+$.
The SdfNLW \eqref{SNLW0} is equipped with an invariant Gibbs measure, formally given by
\begin{align}
\text{``} \, d \vec \rho (\vec u) = Z_\al^{-1} \exp \bigg( - \frac 14 \int_{\T^3} u^4 dx - \frac 12 \int_{\T^3} |\jb{\nb}^\al u|^2 dx - \frac 12 \int_{\T^3} (\dt u)^2 dx \bigg) d \vec u \, \text{''}.
\label{rho0}
\end{align}

\noi
From the constructive quantum field theory's point of view, we also refer to $\vec \rho$ in \eqref{rho0} as the {\it fractional $\Phi^4_3$-measure}, where the ``4'' refers to the quartic interaction and the ``3'' refers to the dimension. We also point out that from the stochastic quantization point of view, SdfNLW \eqref{SNLW0} is the canonical stochastic quantization equation or the Hamiltonian stochastic quantization for the fractional $\Phi^4_3$-measure $\vec \rho$ (see \cite{PW, Simon}). Thus, we refer to SdfNLW \eqref{SNLW0} as the {\it fractional hyperbolic $\Phi^4_3$-model}.

In fact, there is a phase transition happening for the fractional $\Phi^4_3$-measure \eqref{rho0}. When $\al > \frac 98$, the constructed Gibbs measure $\vec \rho$ is equivalent to the base Gaussian measure $\vec \mu_\al$ formally defined as
\begin{align}
\text{``} \, d \vec \mu_\al (\vec u) = Z_\al^{-1} \exp \bigg( - \frac 12 \int_{\T^3} |\jb{\nb}^\al u|^2 dx - \frac 12 \int_{\T^3} (\dt u)^2 dx \bigg) d \vec u \, \text{''},
\label{mu0}
\end{align}

\noi
which is given by \eqref{rho0} with only the quadratic terms.
Nevertheless, when $1 < \al \leq \frac 98$, the Gibbs measure $\vec \rho$ turns out to be mutually singular with respect to the base Gaussian measure $\vec \mu_\al$. This singularity poses additional challenges in several places of the analysis, since Bourgain's invariant measure argument in \cite{BO94, BO96} cannot be directly applied. See Subsection~\ref{SUBSEC:Gibbs} and Subsection~\ref{SUBSEC:1.3} for more details.

Invariance of the Gibbs measure then lays the foundation for the weak universality property.
In this paper, we also show weak universality for the fractional hyperbolic $\Phi^4_3$-model \eqref{SNLW0} in the range $\al > \frac 98$. More precisely, we show that the fractional hyperbolic $\Phi^4_3$-model \eqref{SNLW0} can be identified as a global-in-time universal limit of corresponding microscopic models with a general polynomial interaction. 
In particular, we extend the weak universality result in \cite{STzX} from the two-dimensional setting to the three-dimensional setting and also to the stochastic setting. 
See Subsection~\ref{SUBSEC:wu} for details.

\subsection{Construction of the fractional $\Phi^4_3$-measure}
\label{SUBSEC:Gibbs}

In this subsection, we describe the construction of the fractional $\Phi^4_3$-measure \eqref{rho0} in a rigorous way and mention the required renormalization procedure.

Let us start with a rigorous definition of the base Gaussian measure $\vec \mu_\al$ in \eqref{mu0}. For $s \in \R$, we define the measure $\mu_s$ as
\begin{align}
d \mu_s (u) \deff \prod_{n \in \Z^3} Z_{s, n}^{-1} e^{- \frac 12 \jb{n}^{2 s} |\ft u (n)|^2} d \ft u (n),
\label{gauss2}
\end{align}

\noi
which can be formally written as $d \mu_s (u) = Z_s^{-1} e^{- \frac 12 \| u \|_{H^s}^2} du$ with $H^s (\T^3)$ being the Cameron-Martin space of $\mu_s$. Here, the infinite product makes sense as an infinite product of probability measures and $Z_{s, n}$ are normalizing factors ($u$ is identified as its Fourier coefficients $\ft u (n)$). Equivalently, we can define
\begin{align}
\mu_s = \Law \Big( \frac{1}{(2 \pi)^{\frac 32}} \sum_{n \in \Z^3} \frac{g_n (\o)}{\jb{n}^s} e^{i n \cdot x} \Big),
\label{gauss3}
\end{align}

\noi
where $\Law (X)$ denotes the law of a random variable $X$ and $\{g_n\}_{n \in \Z^3}$ is a sequence of standard i.i.d complex-valued Gaussian random variables conditioned such that $g_n = \cj{g_{-n}}$ for each $n \in \Z^3$ (in particular, $g_0$ is real-valued).
In this paper, we mainly care about the cases $s = \al$ and $s = 0$, where in the latter case $\mu_0$ corresponds to the white noise measure on $L^2 (\T^3)$. From \cite[Lemma~B.1]{BTz08-1}, we see that the measure $\mu_\al$ is supported on $H^{\al - \frac 32 - \eps} (\T^3) \setminus H^{\al - \frac 32} (\T^3)$ for any $\eps > 0$.

We now take $\vec \mu_\al = \mu_\al \otimes \mu_0$ and try to define the Gibbs measure $\vec \rho$ in \eqref{rho0}. When $\al \leq \frac 32$, the support of $\mu_\al$ defined in \eqref{gauss2} excludes $L^2 (\T^3)$, so that quartic power $u^4$ in \eqref{rho0} needs a proper renormalization. Given $N \in \N$, we define $\pi_N$ as the sharp frequency cutoff onto spatial frequencies $\{n \in \Z^3: |n| \leq N\}$:
\begin{align}
\ft{\pi_N f} (n) = \ind_{\{ |n|_\infty \leq N \}} \ft f (n).
\label{defpiN}
\end{align} 

\noi
Here, $|n|_\infty = |(n_1, n_2, n_3)|_\infty = \max (|n_1|, |n_2|, |n_3|)$.
We also define the variance parameter $\s_N$ as
\begin{align}
\s_N = \E_{\mu_\al} \big[ | \pi_N u (x) |^2  \big] = \frac{1}{(2 \pi)^3} \sum_{\substack{n \in \Z^3 \\ |n|_\infty \leq N }} \frac{1}{\jb{n}^{2 \al}},
\label{sigmaN}
\end{align}

\noi
where $\E_{\mu_\al}$ denotes the expectation with respect to the Gaussian measure $\mu_\al$ in \eqref{gauss2}. We observe that $\s_N$ is independent of $x \in \T^3$. Also, $\s_N$ diverges as $N \to \infty$ when $\al \leq \frac 32$ (from now on, we focus on this range of $\alpha$).
In order to tame the divergence of $u^4$, by writing $u_N = \pi_N u$, we introduce the Wick power
\begin{align}
:\! (\pi_N u)^4 \!: \, \deff (\pi_N u)^4 -  6 \s_N (\pi_N u)^2 + 3 \s_N^2.
\label{uN4}
\end{align}

\noi
It turns out that in \eqref{uN4}, there is a compensation among divergent terms; see Subsection~\ref{SUBSEC:mul} for more details. We set
\begin{align}
R_N (u) \deff \frac 14 \int_{\T^3} :\! (\pi_N u)^4 \!: dx
\label{RNu}
\end{align}

\noi
and the frequency truncated fractional $\Phi^4_3$-measure
\begin{align}
d \rho_N (u) \deff Z_N^{-1} e^{-R_N (u)} d \mu_\al (u),
\label{GibbsN1}
\end{align}

\noi
where 
\begin{align*}
Z_N = \int e^{-R_N (u)} d \mu_\al (u)
\end{align*} 

\noi
is a well-defined constant.

We now state our theorem regarding the construction of the fractional $\Phi^4_3$-measure.

\begin{theorem}
\label{THM:Gibbs}
Let $1 < \al \leq \frac 32$.

\smallskip \noi
\textup{(i)} Let $\al > \frac 98$. Then, $\{ Z_N \}_{N \in \N}$ converges as $N \to \infty$, and the sequence of truncated fractional $\Phi^4_3$-measures $\{ \rho_N \}_{N \in \N}$ converges in total variation to a limiting fractional $\Phi^4_3$-measure $\rho$, which is mutually absolutely continuous with respect to the base Gaussian measure $\mu_\al$.

\smallskip \noi
\textup{(ii)} Let $1 < \al \leq \frac 98$. Then, there exists a sequence of divergent constants $\{ \al_N \}_{N \in \N}$ such that if we define the truncated fractional $\Phi^4_3$-measure
\begin{align}
d \wt \rho_N (u) \deff \wt{Z}_N^{-1} e^{- R_N (u) - \al_N} d \mu_\al (u)
\label{GibbsN2}
\end{align}

\noi
with 
\begin{align*}
\wt{Z}_N \deff \int e^{- R_N (u) - \al_N} d \mu_\al (u) ,
\end{align*}

\noi
then $\{ \wt{Z}_N \}_{N \in \N}$ converges as $N \to \infty$, and $\{ \wt{\rho}_N \}_{N \in \N}$ converges weakly to a limiting fractional $\Phi^4_3$-measure $\wt \rho$, which is singular with respect to the base Gaussian measure $\mu_\al$.
\end{theorem}

In the sequel, we write $\wt \rho_N = \rho_N$, $\wt Z_N = Z_N$, and $\wt \rho = \rho$ in the case $1 < \al \leq \frac 98$ to simplify our notations.

Let us explain the intuition behind the phase transition for the fractional $\Phi^4_3$-measure at $\al = \frac 98$ stated in Theorem~\ref{THM:Gibbs}. For simplicity, we ignore the normalizing factor $Z_N$ in our heuristic explanation below. We know from Jensen's inequality that
\begin{align*}
\int e^{- R_N (u)} d \mu_\al (u) \geq  \exp \Big( - \int R_N (u) d \mu_\al (u) \Big),
\end{align*}

\noi
which can be used to obtain a lower bound on the partition function $Z_N$. In order to obtain convergence of $\rho_N$ as $N \to \infty$, we need an upper bound for its partition function uniformly in $N \in \N$. This upper bound can be obtained from the Bou\'e-Dupuis variational formula (see Lemma~\ref{LEM:BD}):
\begin{align}
\begin{split}
\int &e^{- R_N (u)} d \mu_\al (u) \\
&\leq \exp \bigg( \E_{\mu_\al} \bigg[ \sup_{\Dr \in H^\al (\T^3)} \bigg\{ - \frac 14 \int_{\T^3} :\! (\pi_N u + \pi_N \Dr)^4 \!: dx - \frac 12 \| \Dr \|_{H^\al}^2 \bigg\} \bigg] \bigg),
\end{split}
\label{BD0}
\end{align}

\noi
which can be viewed as a reversed version of Jensen's inequality with the cost of a drift term $\Dr$ belonging to the Cameron-Martin space $H^\al (\T^3)$ of the base Gaussian measure $\mu_\al$. We recall that the support of $\mu_\al$ is on $H^{\al - \frac 32 - \eps} (\T^3)$ for any $\eps > 0$, so that on the right-hand-side of \eqref{BD0}, $u$ is a rougher object than $\Dr$. After expanding the quartic power  (see \eqref{HkX} and \eqref{herm_decomp1}) on the right-hand-side of \eqref{BD0}, we see that the worst term $:\! (\pi_N u)^4 \!:$ has mean zero (see \eqref{mean0X}) and so vanishes thanks to the appearance of the expectation $\E_{\mu_\al}$. The second worst term is 
\begin{align*}
\int_{\T^3} :\! (\pi_N u)^3 \!: \pi_N \Dr dx \quad \text{with } :\! (\pi_N u)^3 \!: \, \deff (\pi_N u)^3 - 3 \s_N \pi_N u.
\end{align*}

\noi
We know from the support of $\mu_\al$ that $:\! (\pi_N u)^3 \!:$ has regularity $3 (\al - \frac 32 - \eps)$ (see Lemma~\ref{LEM:YNk}) and the drift term $\pi_N \Dr$ has regularity $\al$, both uniformly in $N \in \N$. Thus, by standard analysis in Besov spaces, the product $:\! (\pi_N u)^3 \!: \pi_N \Dr$ as $N \to \infty$ converges only when
\begin{align*}
3 \Big( \al - \frac 32 - \eps \Big) + \al > 0,
\end{align*}

\noi
which implies that $\al > \frac 98$. This is exactly where the phase transition occurs. 

In the case $1 < \al \leq \frac 98$, we need to introduce the divergent constants $\{ \al_N \}_{N \in \N}$ as stated in Theorem~\ref{THM:Gibbs}~(ii) as a further renormalization that allows us to deal with the problematic product. Note that this additional renormalization in \eqref{GibbsN2} appears only at the level of the measure, and the associated dynamical problem in Subsection~\ref{SUBSEC:1.3} does not see it. Due to this additional renormalization, the limiting fractional $\Phi^4_3$-measure $\wt \rho$ turns out to be singular with respect to the base Gaussian measure $\mu_\al$. In fact, one can show that $\wt \rho$ is absolutely continuous with respect to a shifted measure of $\mu_\al$, where the shifting element almost surely lies outside the Cameron-Martin space $H^\al (\T^3)$ (see Remark~\ref{RMK:shift}). The phase transition at $\al = \frac 98$ for the fractional $\Phi^4_3$-measure also matches the result in Corollary~4.4 of a recent preprint \cite{HKN}.

As mentioned above, 
our proof of Theorem~\ref{THM:Gibbs} is based on the Bou\'e-Dupuis variational formula, which was first applied by Barashkov-Gubinelli \cite{BG} to the construction of Gibbs measures. In this paper, we follow the approach in \cite{OOTol1, OOT2} and also use the simplified version of the Bou\'e-Dupuis variational formula as in \cite{FT}.

We would like to remark that the $\al = 1$ case of Theorem~\ref{THM:Gibbs} corresponds to the hyperbolic $\Phi^4_3$-model, whose corresponding $\Phi^4_3$-measure was constructed in \cite{GJ, BG}. Such construction was one of the early achievements in constructive quantum field theory. We also refer the readers to the introductions in \cite{AK, GH18b}. The $\Phi^4_3$-measure is also known to be mutually singular with respect to the base Gaussian measure (i.e.~the Gaussian free field).

\begin{remark} \rm
\label{RMK:quad}
It would be of interest to see whether the measure can be constructed in the $\al < 1$ range. In fact, it is reasonable to conjecture that the measure construction can be done for all $\al > \frac 34$, since this is the range where one can renormalize the square of a random variable distributed by the Gaussian measure $\mu_\al$ in \eqref{gauss2}. Namely, by letting $u_\al^\o$ be a random variable with $\Law (u_\al^\o) = \mu_\al$ and $\s_N$ be as defined in \eqref{sigmaN}, we can compute that
\begin{align*}
\E \Big[ \big\| |\pi_N u_\al^\o|^2 - \s_N \big\|_{H^s}^2 \Big] \sim \sum_{|n|_\infty \leq N} \frac{1}{\jb{n}^{4 \al}} + \sum_{\substack{|n_1|_\infty, |n_2|_\infty \leq N \\ n_1 \neq n_2}} \frac{\jb{n_1 - n_2}^{2s}}{\jb{n_1}^{2 \al} \jb{n_2}^{2 \al}},
\end{align*}

\noi
which is convergent as $N \to \infty$ if $\al > \frac 34$ and say $s \leq -3$ (see Lemma~\ref{LEM:SUM} below). If we are able to renormalize the square, then it is likely that we can deal with higher powers for the measure construction. On the other hand, if the renormalization of the square does not work, there might be some triviality or non-construction results. See also Remark~\ref{RMK:ZZY} for another piece of supporting evidence for this conjecture.
\end{remark}

\subsection{Fractional hyperbolic $\Phi^4_3$-model}
\label{SUBSEC:1.3}

In this subsection, we consider the dynamical problem for the (formal) fractional hyperbolic $\Phi^4_3$-model \eqref{SNLW0}. More precisely, 
we study the process $u$ in \eqref{SNLW0} as a limit of $u_N$ satisfying the following frequency truncated cubic SdfNLW for $N \in \N$:
\begin{align}
\begin{split}
\dt^2 & u_N + \dt u_N  + (1 -  \Dl)^\al  u_N +  \pi_N  :\!  (\pi_N u_N)^3 \!: \, = \sqrt{2} \xi,
\end{split}
\label{SNLW2}
\end{align}

\noi
where $\pi_N$ is the frequency truncation defined in \eqref{defpiN} and the renormalized cubic nonlinearity is defined as
\begin{align*}
:\!  (\pi_N u_N)^3 \!:  \, \deff
(\pi_N u_N)^3 - 3 \s_N \pi_N u_N.
\end{align*}

\noi
Here, we note that $\pi_N :\!  (\pi_N u_N)^3 \!:$ is the G\^{a}teaux derivative of $R_N (u_N)$ in \eqref{RNu} with respect to $u_N$. Note that the equation \eqref{SNLW2} can be solved essentially as a finite-dimensional problem due to the frequency cut-off. See Lemma~\ref{LEM:GWP4} below, which guarantees that the solution $u_N$ to \eqref{SNLW2} exists globally-in-time.

We now provide the statement of almost sure global well-posedness
of the fractional hyperbolic $\Phi^4_3$-model and also invariance of the Gibbs measure under the resulting dynamics.
In the following, we denote $\vec H^s (\T^3) = H^s (\T^3) \times H^{s - \al} (\T^3)$ given $s \in \R$.

\begin{theorem}
\label{THM:GWP}
Let $1 < \al \leq \frac 32$, $T > 0$, and $\eps > 0$. Then, the sequence of solutions $\{ \vec u_N \}_{N \in \N}$ with $\vec u_N = (u_N, \dt u_N)$ to the truncated cubic SdfNLW \eqref{SNLW2}, with initial data distributed by the fractional $\Phi^4_3$-measure $\vec \rho = \rho \otimes \mu_0$ constructed in Theorem~\ref{THM:Gibbs}, converges almost surely in $C ([0, T]; \vec H^{\al - \frac 32 - \eps} (\T^3))$. Furthermore, the limit at any time $t \in [0, T]$ is distributed by the fractional $\Phi^4_3$-measure $\vec \rho$.
\end{theorem}

The proof of Theorem~\ref{THM:GWP} is based on  \cite{OPTz, OWZ} for the solution ansatz and \cite{OOT2, BDNY} for the globalization procedure. In \cite{OOT2} (see also \cite{Bring2}), the Gibbs measure is represented via the variational approach due to Barashkov-Gubinelli \cite{BG}, while in \cite{BDNY}, it is represented via a heat flow. In this paper, we choose to mainly follow the variational approach, more particularly the work \cite{OOT2} by Oh-Okamoto-Tolomeo, by incorporating a nonlinear smoothing argument from \cite{BDNY}. Similar to \cite{OOT2}, our globalization argument mainly consists of the following three components:
\begin{enumerate}
\item[1.] Uniform (in $N$) exponential integrability of the truncated enhanced data set (see \eqref{data3x} below) with respect to the truncated Gibbs measure $\rhoo_N$.

\item[2.] A stability result established by adding an exponentially decaying weight in time and modifying the local well-posedness argument.

\item[3.] A uniform (in $N$) control of the solution with large probability by exploiting the invariance of the truncated Gibbs measure $\rhoo_N$.
\end{enumerate}

\noi
Compared to \cite{OOT2}, the three components above are highly non-trivial in our setting. The main difference between \cite{OOT2} and this paper is that the authors in \cite{OOT2} consider a quadratic nonlinearity, but in this paper we study a cubic nonlinearity. As a result, we need to deal with more involved stochastic objects (for component 1) and, more importantly, incorporate space-time analysis via the $X^{s,b}$-spaces defined in Subsection \ref{SUBSEC:Xsb} (for components 2 and 3). Our solution to the above issues turns out to contain some new and interesting ingredients: 
\begin{enumerate}
\item[1.]
For the uniform exponential integrability in component 1, we design a sum space for a cubic stochastic object in order to control the drift terms coming from the application of the variational formula in Lemma \ref{LEM:BD}. In addition, we group together various terms of similar regularity properties, which allows us to significantly reduce the number of estimates originally required for dealing with all drift terms.

\item[2.]
For the stability result in component 2, we exploit the decay property from the exponentially decaying temporal weight by  losing a small amount of time regularity (see Lemma~\ref{LEM:nhomo2}). 

\item[3.]
For the uniform bound with large probability in component 3, we use the idea from \cite{BDNY} to exploit global nonlinear smoothing from invariance of the truncated Gibbs measure $\rhoo_N$. In addition, we also establish an improved trilinear estimate (see \eqref{tri2} in Lemma \ref{LEM:str4}), which is important for establishing a tame estimate in component 3.
\end{enumerate}

We remark that our analysis works well for the cubic nonlinearity. However, for more complicated nonlinearities such as a quintic nonlinearity, one would need to use more sophisticated tools and analysis.

We also would like to remark that, instead of the approximation procedure in \eqref{SNLW2}, one can proceed with the following cubic fractional SdfNLW with a truncated noise:
\begin{align}
\dt^2 u_N + \dt u_N + (1 - \Dl)^\al u_N + :\! u_N^3 \!: \, = \sqrt{2} \pi_N \xi.
\label{SNLW3}
\end{align}

\noi
While the approximation in \eqref{SNLW3} appears more natural than \eqref{SNLW2}, the advantage of using \eqref{SNLW2} is that it is invariant under the frequency truncated fractional $\Phi^4_3$-measure $\rho_N$ in \eqref{GibbsN1}, which is needed for our globalization argument. Nevertheless, we point out that our local well-posedness argument also works for \eqref{SNLW3}.

\begin{remark} \rm
\label{RMK:deter}
We point out that with minor modifications, the proof of Theorem~\ref{THM:GWP} also works for the following (deterministic) cubic fractional NLW:
\begin{align*}
\dt^2 u + (1 - \Dl)^\al u \, + :\! u^3 \!: \, = 0.
\end{align*}

\noi
More precisely, given $1 < \al \leq \frac 32$, $T > 0$, and $\eps > 0$, the sequence of solutions to the truncated cubic fractional NLW
\begin{align*}
\dt^2 u_N + (1 - \Dl)^\al u_N + \pi_N :\! (\pi_N u_N)^3 \!: \, = 0
\end{align*}

\noi
with initial data distributed by the fractional $\Phi^4_3$-measure $\vec \rho = \rho \otimes \mu_0$ converges almost surely in $C ([0, T]; \vec H^{\al - \frac 32 - \eps} (\T^3))$, and the limit at any time $t \in [0, T]$ is distributed by the fractional $\Phi^4_3$-measure $\vec \rho$.
\end{remark}

\begin{remark} \rm
One can compare the cubic fractional NLW in this paper with NLW with a Hartree nonlinearity studied in \cite{Bring1, Bring2}:
\begin{align*}
\dt^2 u + (1 - \Dl) u \, + :\! (V * u^2) u \!: \, = 0 ,
\end{align*}

\noi
where $V$ is an interaction potential. The major difference lies in the smoothing effect on the cubic nonlinearity: given a trilinear term $u_1 u_2 u_3$, the Hartree nonlinearity considered in \cite{Bring1, Bring2} imposes slight smoothing effect only on the product $u_1 u_2$, whereas the fractional Laplacian considered in this paper imposes slight smoothing effect on $u_1$, $u_2$, $u_3$, and the entire product $u_1 u_2 u_3$. As a result, a para-controlled ansatz (see also \cite{GIP, GKO2}) is required for studying the dynamical problem of the Hartree NLW when $V$ has small smoothing effect, which is due to the lack of regularity gain from a ``low-low-high'' interaction with the high frequency falling on the random linear solution (i.e. the term $\<1>$ defined in \eqref{loli}). We point out that for the fractional NLW considered in this paper, a para-controlled ansatz is not needed thanks to the slightly better regularity of the random linear solution.
\end{remark}

\begin{remark} \rm
In \cite{DNYprob}, Deng-Nahmod-Yue proposed the probabilistic scaling paradigm for nonlinear heat, wave, and Schr\"odinger equations. Their scaling heuristic is based on the observation that, to obtain probabilistic well-posedness, the second iterate should not be rougher than the linear evolution, in terms of both ``high-high to high'' interaction and ``high-high to low'' interaction. By using the same heuristics, we can also obtain a critical value of $\al$ for our fractional hyperbolic $\Phi^4_3$-model \eqref{SNLW0}: $\al = \frac 56$ (for both interactions). It would be of interest to see if probabilistic well-posedness and invariance of the Gibbs measure of the fractional hyperbolic $\Phi^4_3$-model \eqref{SNLW0} hold in the range $\frac 56 < \al \leq 1$.
\end{remark}

Before moving on to the next subsection, let us also point out that
one should be able to prove Theorem~\ref{THM:GWP} by the method developed in 
\cite{BDNY} in the case $\al = 1$ (in particular the 1533-cancellation used in \cite{BDNY} seems to work in our setting), but we believe 
that our proof is still of interest for two reasons. The first reason is that
it is much simpler than \cite{BDNY}. Indeed, the main difference is that the second Picard iterate (i.e.~the cubic stochastic object $\<30>$ defined in \eqref{so4b} below) has positive regularity when $\al > 1$ and has negative regularity when $\al = 1$; see Lemma~\ref{LEM:sto_reg}~(iii) below. As a result, most of the (deterministic) analysis performed in this paper fails in the case $\al = 1$.
To deal with such issue, a more refined resolution 
ansatz and some cancellations among stochastic objects are used in \cite{BDNY}. 
The second and more important reason is that our proof provides a basis for the weak universality result stated in the next section, while it is not clear whether the proof in \cite{BDNY} is a 
good base for a weak universality result.

\subsection{Weak universality}
\label{SUBSEC:wu}
In this subsection, we discuss a weak universality result for the fractional hyperbolic $\Phi_3^4$-model \eqref{SNLW0} as mentioned at the end of Subsection~\ref{SUB:1.1}.

Following the setup in \cite{STzX}, let us start with the microscopic process $U_N$ on the dilated torus $\T_N^3 = (\R / 2 \pi N \Z)^3$ that satisfies
\begin{align}
\dt^2 U_N + N^{- \al} \dt U_N + (1 - \Delta_N)^\al U_N + N^{- \ta} \Pi_N V' (\Pi_N U_N) = \sqrt{2} N^{- \frac{\al}{2}} \xi_N.
\label{micro}
\end{align}

\noi
Here, we restrict our attention to the range $1 < \al < \frac 32$ (see Remark \ref{RMK:wu} below for a discussion on weak universality for $\al = \frac 32$). Also, $\ta > 0$ is a parameter to be fixed later, $V$ is an even polynomial satisfying some conditions to be specified below, $\xi_N$ is a space-time white noise on $\T_N^3 \times \R_+$, and $\Pi_N$ is the Fourier projection on $\T_N^3$ given by
\begin{align*}
\Pi_N f (x) = \sum_{\substack{ n \in \Z^3 \\ |n| \leq N }} (\F_N f) (n) e^{i n \cdot \frac{x}{N}}
\end{align*}

\noi
where
\begin{align*}
(\F_N f) (n) \overset{\text{def}}{=} \frac{1}{(2 \pi N)^3} \int_{\T_N^3} f(x) e^{- i n \cdot \frac{x}{N}} dx.
\end{align*}

\noi
Also, the differential operator $(1 - \Delta_N)^\al$ is given by\footnote{Here, we choose to work with this form of the operator $(1 - \Dl_N)^\al$ in order to be consistent with well-posedness results from previous sections. One can also start with a homogeneous operator $(- \Dl_N)^\al$ (Fourier multiplier with $\frac{|n|^{2 \al}}{N^{2 \al}}$) as in \cite{STzX} and insert (later in the macroscopic model) a linear correction term coming from the nonlinearity. The proof of weak universality of this setting follows in essentially the same way as the proof we present in this paper.}
\begin{align*}
\F_N (( 1 - \Dl_N )^\al f) (n) \overset{\text{def}}{=} \Big( \frac{1 + |n|^2}{N^2} \Big)^\al (\F_N f) (n).
\end{align*}

\noi
In this microscopic model, we choose $(U_N, \dt U_N)|_{t = 0} = (U_0^\o, U_1^\o)$, where $(U_0^\o, U_1^\o)$ are the following random initial data 
\begin{align*}
U_0^\o (x) = \frac{1}{(2 \pi)^{\frac 32}} N^{- (\frac 32 - \al)} \sum_{n \in \Z^3} \frac{g_n (\o)}{\jb{n}^\al} e^{in \cdot \frac{x}{N}}, \quad U_1^\o (x) = \frac{1}{(2 \pi)^{\frac 32}} N^{- \frac 32} \sum_{n \in \Z^3} h_n (\o) e^{in \cdot \frac{x}{N}},
\end{align*}

\noi
where $\{g_n\}_{n \in \Z^3}$ and $\{h_n\}_{n \in \Z^3}$ are standard i.i.d complex-valued Gaussian random variables conditioned such that $g_{-n} = \cj{g_n}$ and $h_{-n} = \cj{h_n}$ for all $n \in \Z^3$. This choice of initial data is natural since the Gaussian measure induced by the initial data is invariant under the linear dynamics of \eqref{micro}.

By letting 
\begin{align*}
\s \overset{\text{def}}{=} \frac{1}{(2 \pi)^3} \int_{\R^3} \frac{\ind_{\{|\xi|_\infty \leq 1\}}}{|\xi|^{2 \al}} d\xi,
\end{align*}

\noi
we define the averaged potential $\jb V$ as
\begin{align*}
\jb V (z) \overset{\text{def}}{=} \frac{1}{\sqrt{2 \pi \s}} \int_\R V(z + y) e^{- \frac{y^2}{2 \s}} dy.
\end{align*}

\noi
As in \cite{STzX}, we assume that $V$ is an even polynomial of degree $2m \geq 4$ with the form
\begin{align*}
V (z) = \sum_{j = 0}^m a_j z^{2 j},
\end{align*}

\noi
so that the averaged potential $\jb V$ is of the form
\begin{align*}
\jb V (z) = \sum_{j = 0}^m \cj{a}_j z^{2 j},
\end{align*}

\noi
where for each $0 \leq j \leq m$, we have
\begin{align}
\cj{a}_j = \frac{1}{(2j)!} \E \big[ V^{(2j)} (\NN (0, \s)) \big] = \frac{1}{(2j)!} \sum_{\l = j}^m \frac{(2 \l)! }{(2 \l - 2 j)!!}  a_\l \s^{(\l - j)}.
\label{defaj}
\end{align}

\noi
We assume the criticality condition $\cj{a}_1 = 0$ and the positivity condition
\begin{align}
\sum_{j = 2}^m \cj{a}_j z^{2 (j - 2)} > 0 \quad \text{for all } z \in \R,
\label{positive}
\end{align}

\noi
which are equivalent to $\jb{V}'' (0) = 0$, $\cj{a}_2 > 0$, and 
\begin{align*}
\jb{V} (z) - \jb{V} (0) > 0 \quad \text{for all } z \neq 0.
\end{align*}

\medskip

We now switch from the microscopic model \eqref{micro} to a macroscopic model. For $(x, t) \in \T^3 \times \R_+$, we define the macroscopic process 
\begin{align*}
u_N (x, t) = N^{\frac 32 - \al} U_N (N x, N^\al t),
\end{align*}
which satisfies the equation
\begin{align*}
\dt^2 u_N + \dt u_N + (1 - \Dl)^\al u_N + N^{\frac 32 + \al - \ta} \pi_N V' \Big( \frac{\pi_N u_N}{N^{\frac 32 - \al}} \Big) = \sqrt{2} \xi
\end{align*}

\noi
with initial data $(u_N, \dt u_N)|_{t = 0} = (u_0^\o, u_1^\o)$ being defined by 
\begin{align}
u_0^\o (x) = \frac{1}{(2 \pi)^{\frac 32}} \sum_{n \in \Z^3} \frac{g_n (\o)}{\jb{n}^\al} e^{in \cdot x}, \quad u_1^\o (x) = \frac{1}{(2 \pi)^{\frac 32}} \sum_{n \in \Z^3} h_n (\o) e^{in \cdot x}.
\label{rand_init}
\end{align}

\noi
Here, $\xi$ is now a space-time white noise on $\T^3 \times \R_+$. We also note that $\Law (u_0^\o, u_1^\o) = \mu_\al \otimes \mu_0$ in \eqref{gauss2}. In order for $u_N$ to converge to an equation with a cubic nonlinearity, we choose $\ta = 4 \al - 3 $ so that $\frac 32 + \al - \ta = 3 (\frac 32 - \al)$. We remark that the threshold $\al = \frac 34$ corresponds to the case where the scaling factor $N^{- \ta}$ in \eqref{micro} vanishes, and it somehow coincides with the threshold we identified in Remark~\ref{RMK:quad} on the prediction of the Gibbs measure constructions.

For each $N \in \N$, we let
\begin{align*}
V_N (z) = N^{4 (\frac 32 - \al)} V \Big( \frac{z}{N^{\frac 32 - \al}} \Big),
\end{align*}

\noi
so that
\begin{align}
V_N (\pi_N u) = \sum_{j = 0}^m \overline{a}_{j, N} N^{- (2j - 4) (\frac 32 - \al)} : \! (\pi_N u)^{2 j} \! :,
\label{defVN}
\end{align}

\noi
where $: \! (\pi_N u)^{2 j} \! :$ denotes a renormalization for $(\pi_N u)^{2 j}$ (see \eqref{HkX} and \eqref{herm_decomp1}). The coefficients $\cj{a}_{j, N}$ can be explicitly computed as
\begin{align}
\cj{a}_{j, N} = \frac{1}{(2 j)!} \E \big[ V^{(2j)} (\NN (0, \wt{\s}_N)) \big] = \frac{1}{(2j)!} \sum_{\l = j}^m \frac{(2 \l)! }{(2 \l - 2 j)!!}  a_\l \wt \s_N^{(\l - j)} ,
\label{defaNj}
\end{align}

\noi
where
\begin{align}
\wt \s_N \overset{\text{def}}{=} \E [(\pi_N U_0^\o)^2 (x)] = \frac{N^{- (3 - 2 \al)}}{(2 \pi)^3} \sum_{\substack{n \in \Z^3 \\ |n|_\infty \leq N}} \frac{1}{\jb{n}^{2 \al}} = \s + O (N^{-2 (\frac 32 - \al)}).
\label{sNs}
\end{align}

\noi
We can disregard the 0-degree term (i.e.~$j = 0$) in $V_N (z)$ in our later discussion, since it does not appear in our equation. 
From \eqref{defaj}, \eqref{defaNj}, and \eqref{sNs}, we have for each $j = 2, \dots, m$ that 
\begin{align}
| \cj{a}_{j, N} - \cj{a}_j | \leq C(m) N^{- 2 (\frac 32 - \al)} ,
\label{ajconv}
\end{align}

\noi
where the constant $C (m) > 0$ also depends on $a_j, \dots, a_m$ but we omit such dependence for simplicity.
For $j = 1$, we have the following more precise approximation, which can be shown is the same way as \cite[Proposition 1.4]{STzX}:
\begin{align}
\cj{a}_{1, N} = \cj{a}_1 + \frac{\kappa}{2} N^{-2 (\frac 32 - \al)} + O(N^{-1}) + O(N^{-4 (\frac 32 - \al)}),
\label{a1conv}
\end{align}

\noi
where $\kappa \in \R$. Here, we see that the criticality condition $\cj{a}_1 = 0$ exactly cancels out the divergence coming from the $j = 1$ term.

The macroscopic model now becomes
\begin{align}
\dt^2 u_N + \dt u_N + (1 - \Dl)^\al u_N + \pi_N V_N' (\pi_N u_N) = \sqrt{2} \xi,
\label{fNLWw}
\end{align}

\noi
Our goal is to show that, as $N \to \infty$, $u_N$ converges in some function space to the same limit as $u_N^{\dagger}$, which satisfies the following cubic fractional stochastic damped wave equation:
\begin{align}
\dt^2 u_N^\dagger + \dt u_N^\dagger + (1 - \Dl)^\al u_N^\dagger + \kappa \pi_N u_N^\dagger + 4 \cj{a}_2  \pi_N \big( :\! (\pi_N u_N^\dagger)^3 \!: \big) = \sqrt{2} \xi.
\label{fNLWw_v}
\end{align}

\noi
The limit of $u_N^\dagger$ as $N \to \infty$ can be established as in Theorem~\ref{THM:GWP} with minor modifications.

\medskip
We now state our weak universality result. 
\begin{theorem}
\label{THM:wu}
Let $\frac 98 < \al < \frac 32$, $T \geq 1$, and $\eps > 0$ be sufficiently small. Let $\vec u_N = (u_N, \dt u_N)$ be the solution to \eqref{fNLWw} with the random initial data $\vec u_N |_{t = 0} = (u_0^\o, u_1^\o)$ as defined in \eqref{rand_init}. 
Then, there exists $\kappa \in \R$ such that, if $\vec u_N^\dagger = (u_N^\dagger, \dt u_N^\dagger)$ satisfies \eqref{fNLWw_v} with the same random initial data, then $\vec u_N$ converges almost surely in $C ([0, T]; \vec H^{\al - \frac 32 - \eps} (\T^3))$ to the same limit as $\vec u_N^\dagger$ converges to.
\end{theorem}

As in \cite{STzX}, our proof of Theorem~\ref{THM:wu} uses a first order expansion as in \cite{McK, BO96, DPD2} and the invariance and convergence of Gibbs measures for the macroscopic model \eqref{fNLWw}.
The convergence of Gibbs measures requires the positivity condition \eqref{positive}, which is almost necessary in the sense of \cite[Proposition~1.12]{STzX}.
We point out that the range $\frac 98 < \al < \frac 32$ stated in Theorem~\ref{THM:wu} exactly corresponds to the situation where we have strong convergence of the Gibbs measure for the macroscopic model \eqref{fNLWw} as $N \to \infty$ (see Proposition~\ref{PROP:Gibbs2}).

 Nevertheless, compared to \cite{STzX}, our proof of Theorem \ref{THM:wu} has two novel features. Firstly, in \cite{STzX}, the authors allowed a logarithmic growth in $N$ in their estimates, which required them to establish an additional probabilistic a priori bound on high frequencies of the solution in the spirit of Bourgain-Bulut \cite{BB14}. In our case, we allow a growth in $N$ in our estimate which is slightly more than a logarithmic growth (see Proposition~\ref{PROP:uN_bdd} below). As a result, our approach does not require this additional a priori bound on high frequencies of the solution. Secondly, we exploit extra multilinear dispersive smoothing from all second-order stochastic objects (see Lemma \ref{LEM:sto_k2}), which allows us to cover the whole range $\frac 98 < \al < \frac 32$ by using only the first order expansion. Specifically, a standard scaling analysis shows that the scaling critical regularity for the deterministic cubic fractional NLW is $\frac 32 - \al$, and well-posedness for all subcritical regularity can be achieved thanks to the Strichartz estimate established by Bourgain-Demeter's $\l^2$-decoupling theorem in \cite{BD15}. Moreover, the second-order stochastic terms can achieve regularity $3 \al - 3 - \eps$ for any $\eps > 0$, which lies in the scaling subcritical regime if
\begin{align*}
3 \al - 3 > \frac 32 - \al,
\end{align*}

\noi
which is equivalent to $\al > \frac 98$.

\begin{remark} \rm
As in Remark~\ref{RMK:deter}, our proof of the weak universality result in Theorem~\ref{THM:wu} also works for the (deterministic) fractional NLW at a microscopic level:
\begin{align*}
\dt^2 U_N + (1 - \Dl_N)^\al U_N + N^{- 4 \al + 3} \Pi_N V' (\Pi_N U_N) = 0 .
\end{align*}

\noi
Namely, after rescaling, $U_N$ converges almost surely to the same limit as that of $u_N^\dagger$ satisfying the following (deterministic) cubic fractional NLW:
\begin{align*}
\dt^2 u_N^\dagger + (1 - \Dl)^\al u_N^\dagger + \kappa \pi_N u_N^\dagger + 4 \cj{a_2} \pi_N \big( :\! (\pi_N u_N^\dagger)^3 \!: \big) = 0 .
\end{align*}
\end{remark}

\begin{remark} \rm
\label{RMK:wu}
When $\al = \frac 32$, the fractional hyperbolic $\Phi^4_3$-model just becomes a singular stochastic PDE (i.e.~the stochastic convolution has regularity $-\eps$ for $\eps > 0$ arbitrarily small), and so the setting is similar to that in \cite{GKO} by Gubinelli-Koch-Oh. Thus, we expect that for the fractional hyperbolic $\Phi^4_3$-model with $\al = \frac 32$, a weak universality result similar to that in \cite{GKO} holds.
\end{remark}

\begin{remark} \rm
For the fractional hyperbolic $\Phi^4_3$-model, weak universality in the range $1 < \al \leq \frac 98$ has additional difficulties. Firstly, as in the case of the fractional $\Phi^4_3$-measure in Theorem~\ref{THM:Gibbs}, the Gibbs measure for the macroscopic model \eqref{fNLWw} requires an additional renormalization procedure, and we expect only weak convergence of the truncated Gibbs measure. Secondly, our method based on the first order expansion does not seem to be enough due to the worse regularity of the second-order stochastic objects. As such, a higher order expansion is needed. In a recent preprint \cite{Liu}, the first author addressed this problem and established weak universality for the fractional hyperbolic $\Phi^4_3$-model in the range $1 < \al \leq \frac 98$. The weak universality problem for $\al = 1$, corresponding to the usual hyperbolic $\Phi^4_3$-model, remains open.
\end{remark}

\subsection{Organization of the paper}

This paper is organized as follows. In Section~\ref{SEC:2}, we introduce notations and recall and prove some lemmas that will be used in the proof of our theorems.
In Section~\ref{SEC:count}, we establish lattice point counting estimates for regularity estimates of stochastic objects and random operators, both for our global well-posedness result and for our weak universality result for the fractional hyperbolic $\Phi^4_3$-model.
In Section~\ref{SEC:meas}, we construct the fractional $\Phi^4_3$-measure as stated in Theorem~\ref{THM:Gibbs} and prove the convergence of Gibbs measures for our weak universality result.
In Section~\ref{SEC:LWP}, we present a local well-posedness result for the fractional hyperbolic $\Phi^4_3$-model \eqref{SNLW0} with Gaussian initial data. 
In Section~\ref{SEC:GWP}, we prove Theorem~\ref{THM:GWP}, global well-posedness for the fractional hyperbolic $\Phi^4_3$-model \eqref{SNLW0} and invariance of the fractional $\Phi^4_3$-measure.
Lastly, in Section~\ref{SEC:wu}, we prove Theorem~\ref{THM:wu}, weak universality of the fractional hyperbolic $\Phi^4_3$-model.

\section{Notations and basic lemmas}
\label{SEC:2}

In this section, we introduce notations and preliminary facts.

Throughout the paper, for two quantities $a, b > 0$, we use the notation $a \les b$ if $a \leq C b$ for some constant $C > 0$ uniform with respect to the set where $a$ and $b$ are allowed to vary. We also write $a \sim b$ if we have both $a \les b$ and $b \les a$. 


Let $N \in 2^{\N}$ be a dyadic number. If $N \geq 2$ we write $|n| \sim N$ to denote $\frac{N}{2} < |n| \leq N$. If $N = 1$, we write $|n| \sim N$ to denote $n \in \{n \in \Z^3: |n| \leq N \}$. We also denote $\pi_N$ as the sharp frequency cutoff onto frequencies $\{n \in \Z^3: |n| \leq N \}$.

For any distribution $f$, we denote by $\ft f$ the Fourier transform of $f$ and $f^\vee$ the inverse Fourier transform of $f$. For a space-time function $u$, we denote by $\ft u$ the spatial Fourier transform or the space-time Fourier transform of $u$, depending on the context. To avoid ambiguity, we use $x$ and $t$ as the spatial and the temporal physical variables, respectively; and we use $n$ and $\tau$ (or $\mu$) as the spatial and the temporal frequency variables, respectively.
We also use the notation $\vec u = (u, \dt u)$ unless it is specified explicitly.

Given Banach spaces $B_1$ and $B_2$,
we use $\L(B_1; B_2)$ to denote the space
of bounded linear operators from $B_1$ to $B_2$.

\subsection{Sobolev spaces and convolution inequalities}
\label{SUBSEC:21}

Let us first introduce Sobolev spaces in this subsection.

For any $s \in \R$ and $1 \leq p \leq \infty$, we define the $L^p$-based Sobolev space $W^{s, p} (\T^3)$ via the norm
\begin{align*}
\| f \|_{W^{s, p}} \deff \| \jb{\nb}^s f \|_{L^p} = \big\| \big( \jb{n}^s \ft f (n) \big)^\vee \big\|_{L^p}.
\end{align*}

\noi
When $p = 2$, we write $H^s (\T^3) = W^{s, 2} (\T^3)$. We also denote $\vec W^{s, p} (\T^3) = W^{s, p} (\T^3) \times W^{s - \al, p} (\T^3)$ and $\vec H^s (\T^3) = H^s (\T^3) \times H^{s - \al} (\T^3)$. 
By Plancherel's theorem, we have 
\begin{align*}
\| f \|_{H^s} = \big\| \jb{n}^s \ft f (n) \big\|_{\ell^2_n}.
\end{align*}

Let us now mention some useful properties of Sobolev spaces. We first record the following interpolation result.

\begin{lemma}
\label{LEM:interp}
Let $s, s_1, s_2 \in \R$ and $1 < p, p_1, p_2 < \infty$ be such that
\begin{align*}
s = \ta s_1 + (1 - \ta) s_2 \quad \textup{and} \quad \frac{1}{p} = \frac{\ta}{p_1} + \frac{1 - \ta}{p_2}
\end{align*}

\noi
for some $0 < \ta < 1$. Then, we have
\begin{align*}
\| f \|_{W^{s, p}} \les \| f \|_{W^{s_1, p_1}}^\ta \| f \|_{W^{s_2, p_2}}^{1 - \ta}.
\end{align*}
\end{lemma}

The proof of Lemma~\ref{LEM:interp} follows immediately from the Littlewood-Paley characterization of Sobolev spaces; see \cite[Theorem~1.3.6]{Graf}.

\medskip
We now record the following product estimates. 
For a proof, see \cite[Lemma~3.4]{GKO} and also \cite{BOZ} for the endpoint case.

\begin{lemma}
\label{LEM:gko}
Let $s > 0$.

\smallskip \noi
\textup{(i)} Let $1 \leq r \leq \infty$ and $1 < p_1, p_2, q_1, q_2 \leq \infty$ be such that
\begin{align*}
\frac{1}{r} = \frac{1}{p_1} + \frac{1}{q_1} = \frac{1}{p_2} + \frac{1}{q_2}.
\end{align*}

\noi
Then, we have
\begin{align*}
\| fg \|_{W^{s, r}} \les \| f \|_{W^{s, p_1}} \| g \|_{L^{q_1}} + \| f \|_{L^{p_2}} \| g \|_{W^{s, q_2}}.
\end{align*}

\smallskip \noi
\textup{(ii)} Let $1 < p, q \leq \infty$ and $1 < r \leq p$ be such that
\begin{align*}
\frac{s}{3} + \frac{1}{r} = \frac{1}{p} + \frac{1}{q} \leq 1.
\end{align*}

\noi
Then, we have
\begin{align*}
\| fg \|_{W^{-s, r}} \les \| f \|_{W^{-s, p}} \| g \|_{W^{s, q}}.
\end{align*}
\end{lemma}

We now recall the following convolution inequality. For a proof, see \cite[Lemma 4.2]{GTV}.
\begin{lemma}
\label{LEM:conv}
Let $0 \leq \be \leq \gamma$ be such that $\gamma > 1$. Then, for any $a \in \R$, we have
\begin{align*}
\int_\R \frac{1}{\jb{x}^\be \jb{x - a}^\gamma} dx \les \frac{1}{\jb{a}^\be}.
\end{align*} 
\end{lemma}

Next, we present the following basic lemma on a discrete convolution.

\begin{lemma}\label{LEM:SUM}
\textup{(i)}
Let $\be, \gamma > 0$ be such that $\be + \g > 3$ and $\be < 3$.
Then, for any $n \in \Z^3$, we have
\begin{align}
 \sum_{n = n_1 + n_2} \frac{1}{\jb{n_1}^\be \jb{n_2}^\gamma}
\les 
\begin{cases}
\jb{n}^{3 - \be - \g} & \text{if } \g < 3 \\
\jb{n}^{- \be + \eps} & \text{if } \g = 3 \\
\jb{n}^{- \be} & \text{if } \g > 3
\end{cases}
\label{convsum1}
\end{align}

\noi
for some $\eps > 0$ arbitrarily small. Moreover, when $\g < 3$, we have
\begin{align}
\sum_{|n_1| \leq \frac{|n|}{2}} \frac{1}{\jb{n_1}^\be \jb{n - n_1}^\g} \sim \jb{n}^{3 - \be - \g}.
\label{convsum1-1}
\end{align}

\smallskip \noi
\textup{(ii)}
Let $k \in \N$ and $\be \in \R$ satisfy
\[ \frac{3 (k - 1)}{k} < \be < 3. \] 

\noi
Then, we have
\begin{align}
\sum_{n = n_1 + \cdots + n_k} \frac{1}{\jb{n_1}^\be \cdots \jb{n_k}^\be} \les \jb{n}^{3 (k - 1) - k \be}.
\label{convsum2}
\end{align}
\end{lemma}

\begin{proof}
The inequality \eqref{convsum1} and the upper bound in \eqref{convsum1-1} follows
from elementary  computations.
See, for example,
 \cite[Lemma 4.2]{GTV} and \cite[Lemma 4.1]{MWX}. The lower bound in \eqref{convsum1-1} follows from
\begin{align*}
\sum_{|n_1| \leq \frac{|n|}{2}} \frac{1}{\jb{n_1}^\be \jb{n - n_1}^\g} \geq \sum_{ \frac{|n|}{4} < |n_1| \leq \frac{|n|}{2} } \frac{1}{\jb{n_1}^\be \jb{n - n_1}^\g} \ges \frac{\jb{n}^3}{\jb{n}^{\be + \g}} = \jb{n}^{3 - \be - \g}.
\end{align*} 

\noi
The inequality \eqref{convsum2} follows from applying \eqref{convsum1} repetitively.
\end{proof}

\subsection{Fourier restriction norm method}
\label{SUBSEC:Xsb}
We consider the following nonhomogeneous linear damped fractional wave equation:
\begin{align}
\begin{cases}
\dt^2 u + \dt u + (1 - \Dl)^\al u = F \\
(u, \dt u)|_{t = 0} = (u_0, u_1).
\end{cases}
\label{wave_lin}
\end{align}

\noi
The solution to \eqref{wave_lin} is given by
\begin{align*}
u (t) = S(t) (u_0, u_1) + \int_0^t \D (t - t') F(t') dt'
\end{align*}

\noi
for $t \geq 0$, where the operator $\D (t)$ is defined by
\begin{align*}
\D(t) = e^{-\frac{t}{2} }\frac{\sin\Big(t\sqrt{(1-\Dl)^\al -\tfrac14}\Big)}{\sqrt{(1-\Dl)^\al -\tfrac14}}
\end{align*}

\noi
and the operator $S(t)$ is defined by
\begin{align*}
S(t) (f, g) = \dt\D(t)f +  \D(t) (f + g).
\end{align*}

\noi
By setting
\begin{align*}
\jbb{\nabla} = \sqrt{(1-\Dl)^\al -\frac14} \quad \text{and} \quad \jbb{n} = \sqrt{\jb{n}^{2\al}-\frac 14},
\end{align*}

\noi
we have
\begin{align}
\D (t) = e^{-\frac{t}{2}} \frac{\sin (t \jbb{\nb})}{\jbb{\nb}} \quad \text{or} \quad \D(t) f  =  \frac{1}{(2 \pi)^{\frac 32}} e^{-\frac{t}{2}} \sum_{n \in \Z^3}  \frac{\sin (t \jbb{n})}{\jbb{n}} \ft f (n) e^{in \cdot x} ,
\label{W3}
\end{align}

\noi
and so we can write
\begin{align}
S (t) (f, g) = e^{- \frac t2} \Big( \cos (t \jbb{\nb}) + \frac{\sin (t \jbb{\nb})}{2 \jbb{\nb}} \Big) f + e^{- \frac t2} \frac{\sin (t \jbb{\nb})}{\jbb{\nb}} g.
\label{St0}
\end{align}

\noi
For convenience, with $t_0 \in \R$ and $t \geq t_0$, we define the Duhamel integral operator
\begin{align}
\I_{t_0} (F) (t) \deff \int_{t_0}^t \D(t - t') F(t') dt'
= \int_{t_0}^t e^{-\frac{t-t'}{2} }\frac{\sin\big((t-t')\jbb{\nabla} \big)}{\jbb{\nabla}}
F(t') dt' ,
\label{lin1}
\end{align}

\noi
and we write $\I (F) \deff \I_0 (F)$.
Moreover, we denote
\begin{align}
\begin{split}
\vec \D (t) f &= (\D (t) f, \dt \D (t) f), \\ 
\vec S (t) (f, g) &= (S (t) (f, g), \dt S (t) (f, g)) \\
\vec \I_{t_0} (F) &= (\I_{t_0} (F), \dt \I_{t_0} (F)).
\end{split}
\label{vecSI}
\end{align}

We now recall the Fourier restriction norm method, which was first introduced by Klainerman-Machedon \cite{KM93} in the setting of the wave equation and by Bourgain \cite{Bour93} in the setting of the Schr\"odinger equation. Following their ideas, we can define the $X^{s, b}$-spaces for the fraction wave equation as follows:
\begin{align}
\| u \|_{X^{s, b} (\T^3 \times \R)} \overset{\textup{def}}{=} \big\| \jb{n}^s \jb{|\tau| - \jbb{n}}^b \ft u (n, \tau) \big\|_{\l_n^2 L_\tau^2 (\Z^3 \times \R)}.
\label{Xsb}
\end{align}

\noi
We also denote $\vec X^{s, b} = X^{s, b} \times X^{s - \al, b}$.
Note that the following equivalence holds (see \cite[Lemma 4.2]{Bring2}):
\begin{align}
\| u \|_{X^{s, b} (\T^3 \times \R)} \sim \min_{\substack{u_+, u_- \in X^{s, b} \\ u = u_+ + u_-}} \max_\pm \big\| \jb{\nb}^s e^{\mp it \jbb{\nb}} u_\pm \big\|_{L_x^2 H_t^b (\T^3 \times \R)}.
\label{equi_norm}
\end{align}

 Given an interval $I \subset \R$, we can define the local-in-time version of the $X^{s, b}$-space as follows:
\begin{align*}
\| u \|_{X_I^{s, b}} = \inf \big \{ \| v \|_{X^{s, b}}: v |_I = u \}.
\end{align*}

\noi
If $I = [0, T]$ for some $T > 0$, we write $X^{s, b}_T = X_{[0, T]}^{s, b}$.

We now consider relevant estimates of the $X^{s, b}$-norms. We first record the following embedding lemma. See, for example, in \cite[Corollary~2.10]{Tao}.
\begin{lemma}
\label{LEM:HsXsb}
Let $s \in \R$, $b > \frac 12$, and $I \subset \R$ be a closed interval. Then, we have the following embedding:
\begin{align*}
\| u \|_{C_I H^s} \les \| u \|_{X^{s, b}_I}.
\end{align*}
\end{lemma}

We now record the following gluing lemma. For a proof, see \cite[Lemma~4.5]{Bring2}.
\begin{lemma}
\label{LEM:glue}
Let $s \in \R$, $-\frac 12 < b_1 < \frac 12$, and $\frac 12 < b_2 < 1$. Let $I_1, I_2 \subset \R$ be two closed intervals such that $|I_1 \cap I_2| > 0$. Then, we have
\begin{align*}
\| u \|_{X^{s, b_1}_{I_1 \cup I_2}} \les \| u \|_{X^{s, b_1}_{I_1}} + \| u \|_{X^{s, b_1}_{I_2}}.
\end{align*}

\noi
Furthermore, we have
\begin{align*}
\| u \|_{X^{s, b_2}_{I_1 \cup I_2}} \les |I_1 \cap I_2|^{\frac 12 - b_2} \big( \| u \|_{X^{s, b_2}_{I_1}} + \| u \|_{X^{s, b_2}_{I_2}} \big).
\end{align*}
\end{lemma}

Next, recall that $\vec S(t)$ defined in \eqref{vecSI} and \eqref{St0} are only defined for $t \geq 0$. Thus, to make use of the $X^{s,b}$-space, we need to extend these two linear operators to the whole real line. For this purpose, for $t \in \R$, we define
\begin{align*}
S_0(t) (f, g) &\overset{\textup{def}}{=} e^{- \frac{|t|}{2}} \Big( \cos (t \jbb{\nb}) + \frac{\sin (t \jbb{\nb})}{2 \jbb{\nb}} \Big) f + e^{- \frac{|t|}{2}} \frac{\sin (t \jbb{\nb})}{\jbb{\nb}} g , \\
S_1(t) (f, g) &\overset{\textup{def}}{=} e^{- \frac{|t|}{2}} \Big( - \sin (t \jbb{\nb}) \jbb{\nb} - \frac{\sin (t \jbb{\nb})}{4 \jbb{\nb}} \Big) f + e^{- \frac{|t|}{2}} \Big( \cos (t \jbb{\nb}) - \frac{\sin (t \jbb{\nb})}{2 \jbb{\nb}} \Big) g
\end{align*}

\noi
as substitutes for $S(t)$ and $\dt S (t)$, respectively. The idea of this extension come from \cite{MR}.

We now show the following homogeneous linear estimate.
\begin{lemma}
\label{LEM:homo}
Let $s \in \R$, $\frac 12 < b < 1$, and $I \subset \R$ be a closed interval. Then, we have
\begin{align*}
\big\| \big( S_0(t), S_1(t) \big) (f, g) \big\|_{\vec X_I^{s, b}} \les (1 + |I|) \| (f, g) \|_{\vec H^{s}}.
\end{align*}
\end{lemma}
\begin{proof}
From Lemma \ref{LEM:glue}, it suffices to show 
\begin{align*}
\big\| \eta (t) \big( S_0(t), S_1(t) \big) (f, g) \big\|_{\vec X^{s, b}} \les \| (f, g) \|_{\vec H^{s}},
\end{align*}

\noi
where $\eta : \R \to [0,1]$ is a smooth cut-off function on an interval with unit length. We can assume that $s = 0$, since otherwise we can run the argument for $(f', g') = (\jb{\nb}^s f, \jb{\nb}^{s - \al} g)$. From \eqref{St0} and \eqref{W3}, we only need to show
\begin{align*}
\big\| \eta(t) e^{-\frac{|t|}{2}} e^{\pm i t \jbb{\nb}} f \big\|_{X^{0, b}} \les \| f \|_{L^2}.
\end{align*}

\noi
By \eqref{Xsb}, the triangle inequality, the fact that $(e^{-\frac{|t|}{2}})^\wedge (\tau) \sim \frac{1}{1 + |\tau|^2}$, Young's convolution inequality, and the fact that $b < 1$, we have
\begin{align*}
\big\| \eta(t) e^{-\frac{|t|}{2}} e^{\pm i t \jbb{\nb}} f \big\|_{X^{0, b}} &= C \Big\| \Big( \ft{\eta} * \frac{1}{1 + |\cdot|^2} \Big) (\tau) \jb{|\tau  \pm \jbb{n}| - \jbb{n}}^b \ft f (n) \Big\|_{\l^2_n L^2_\tau} \\
&\les \Big\| \Big( \ft{\eta}  * \frac{1}{1 + |\cdot|^2} \Big) (\tau) \jb{\tau}^b \ft f (n) \Big\|_{\l^2_n L^2_\tau} \\
&\les \big\| \ft \eta (\tau) \jb{\tau}^{|b|} \big\|_{L_\tau^1} \Big\| \frac{\jb{\tau}^b}{1 + |\tau|^2} \Big\|_{L_\tau^2} \| f \|_{L^2} \\
&\les \| f \|_{L^2},
\end{align*}
so the desired estimate follows.
\end{proof}

For simplicity, we will abuse notation and write $\vec S(t) = (S (t), \dt S(t)) = (S_0 (t), S_1 (t))$ later in this paper.

\medskip
We then consider the Duhamel integral operator $\vec \I_{t_0}$ in \eqref{vecSI} and \eqref{lin1}. Again, we need to extend $\vec \I_{t_0}$ to the full real line. For this purpose, we consider the following substitute for the operator $\I_{t_0}$ given $t_0 \in \R$:
\begin{align}
\begin{split}
\I_{0, t_0} (F) (t) &\overset{\textup{def}}{=} \frac{1}{2i \jbb{\nb}} \big( \I_{t_0}^+ (F) (t) - \I_{t_0}^{-} (F) (t) \big), \\
\I_{1, t_0} (F) (t) &\overset{\textup{def}}{=} \frac 12 \big( \I_{t_0}^+ (F) (t) + \I_{t_0}^{-} (F) (t) \big) - \frac{1}{4i \jbb{\nb}} \big( \I_{t_0}^+ (F) (t) - \I_{t_0}^{-} (F) (t) \big) ,
\end{split}
\label{defI}
\end{align}

\noi
where
\begin{align*}
\I_{t_0}^\pm (F) (t) = \sum_{n \in \Z^3} e^{i n \cdot x} \int_\R \frac{e^{i (t - t_0) \mu} - e^{- \frac{|t - t_0|}{2} \pm i (t - t_0) \jbb{n}}}{\frac 12 + i \mu \mp i \jbb{n}} e^{i t_0 \mu} \ft F (n, \mu) d \mu .
\end{align*}

\noi
Note that for $t \geq t_0$, $\I_{0, t_0} (F) (t) = \I_{t_0} (F) (t)$ and $\I_{1, t_0} = \dt \I_{t_0} (F) (t)$. Indeed, by taking the Fourier transform in space and a change of variable in $t'$, we have
\begin{align*}
&\ft{\I_{t_0} (F)} (n, t) \\
&= \int_{t_0}^t e^{-\frac{t - t'}{2}} \frac{e^{i (t - t') \jbb{n}} - e^{- i (t - t') \jbb{n}}}{2i \jbb{n}} \ft F (n, t') dt' \\
&= \int_{0}^{t - t_0} e^{-\frac{t - t_0 - t'}{2}} \frac{e^{i (t - t_0 - t') \jbb{n}} - e^{- i (t - t_0 - t') \jbb{n}}}{2i \jbb{n}} \int_\R e^{i (t' + t_0) \mu} \ft F (n, \mu) d\mu dt' \\
&= \frac{1}{2i \jbb{n}} \int_\R \bigg( \frac{e^{i (t - t_0) \mu} - e^{-\frac{t - t_0}{2} + i (t - t_0) \jbb{n}}}{\frac 12 + i \mu - i \jbb{n}} - \frac{e^{i (t - t_0) \mu} - e^{-\frac{t - t_0}{2} - i (t - t_0) \jbb{n}}}{\frac 12 + i \mu + i\jbb{n}} \bigg) e^{i t_0 \mu} \ft F (n, \mu) d\mu \\
&= \ft{\I_{0, t_0} (F)} (n, t),
\end{align*}

\noi
and a similar identity holds for $\dt \I_{t_0} (F)$.
We now show the following inhomogeneous linear estimate.
\begin{lemma}
\label{LEM:nhomo}
Let $\al \geq 0$, $s \in \R$, $\frac 12 < b < 1$, and $I = [t_0, t_1] \subset \R$ be a closed interval with $t_0 \geq 0$. Then, we have
\begin{align*}
\big\| \big( \I_{0, t_0} (F), \I_{1, t_0} (F) \big) \big\|_{\vec X_I^{s, b}} \les \| F \|_{X_I^{s - \al, b - 1}}.
\end{align*}
\end{lemma}
\begin{proof}
From \eqref{defI} and a change of variable in time, it suffices to show
\begin{align*}
\big\| \eta(t) \I_0^\pm (F) \big\|_{X^{s, b}} \les \| F \|_{X^{s, b - 1}} 
\end{align*}

\noi
for any $s \in \R$, where $\eta : \R \to [0,1]$ is a smooth cut-off function such that $\eta \equiv 1$ on $[0, |I|]$. We can assume without loss of generality that $s = 0$, since otherwise we can run the argument for $G = \jb{\nb}^s F$. By taking the Fourier transform in time, we have
\begin{align*}
\ft{\I_{0}^\pm (F)} (n, \tau) =  \int_\R \eta (t) e^{- i t \tau} \int_\R \frac{e^{i t \mu} - e^{-\frac{|t|}{2} \pm i t \jbb{n}}}{\frac 12 + i \mu \mp i \jbb{n}} \ft{F} (n, \mu) d\mu dt
\end{align*}
Our goal is to show that
\begin{align*}
\Big\| \jb{|\tau| - \jbb{n}}^b \ft{\I_0^\pm (F)} (n, \tau) \Big\|_{\l_n^2 L_\tau^2} \les \big\| \jb{|\mu| - \jbb{n}}^{b-1} \ft F (n, \mu) \big\|_{\l_n^2 L_\mu^2},
\end{align*}

\noi
and so it suffices to show
\begin{align}
\bigg\| \int_\R \mathcal{K} (n, \tau, \mu) G(\mu) d \mu \bigg\|_{L_\tau^2} \les \| G (\mu) \|_{L_\mu^2}
\label{goal}
\end{align}

\noi
uniformly in $n \in \Z^3$, where
\begin{align*}
\mathcal{K} (n, \tau, \mu) = \jb{|\tau| - \jbb{n}}^b \jb{|\mu| - \jbb{n}}^{1 - b} \int_\R \eta (t) \frac{e^{i t (\mu - \tau)} - e^{-\frac{|t|}{2} - i \tau t \pm i t \jbb{n}}}{\frac 12 + i \mu \mp i \jbb{n}}  dt 
\end{align*}

We first consider the situation when $|\tau - \mu| \leq 1$. Note that we have
\begin{align*}
|\mathcal{K} (n, \tau, \mu)| \ind_{\{|\tau - \mu| \leq 1\}}  \les \frac{\jb{|\tau| - \jbb{n}}^b \jb{|\mu| - \jbb{n}}^{1 - b}}{\jb{\mu \mp \jbb{n}}} \ind_{\{|\tau - \mu| \leq 1\}} \les 1.
\end{align*}

\noi
Thus, by Schur's test, we have
\begin{align*}
\big\| \mathcal{K} (n, \tau, \mu) \ind_{\{|\tau - \mu| \leq 1\}} \big\|_{L_\mu^2 \to L_\tau^2}^2 &\leq \sup_{\tau \in \R} \int_{|\tau - \mu| \leq 1} |\mathcal{K} (n, \tau, \mu)| d\mu \cdot \sup_{\mu \in \R} \int_{|\tau - \mu| \leq 1} |\mathcal{K} (n, \tau, \mu)| d\tau \\
&\les 1,
\end{align*}

\noi
so that \eqref{goal} holds when $|\tau - \mu| \leq 1$.

To deal with the situation when $| \tau - \mu | > 1$, we use integration by parts to obtain
\begin{align*}
\mathcal{K}(&n, \tau, \mu) \ind_{\{|\tau - \mu| > 1\}} \\
&= \ind_{\{|\tau - \mu| > 1\}} \jb{|\tau| - \jbb{n}}^b \jb{|\mu| - \jbb{n}}^{1 - b} \bigg( \int_\R \eta'' (t) \frac{e^{i t (\mu - \tau)}}{-(\mu - \tau)^2 (\frac 12 + i \mu \mp i \jbb{n})} dt \\
&\quad -\frac{\eta(0)}{\big( \frac 14 + (\tau \mp \jbb{n})^2 \big) (\frac 12 + i \mu \mp i \jbb{n})}   - \int_0^\infty \eta' (t) \frac{e^{-\frac{t}{2} - i\tau t \pm i t \jbb{n}}}{(\frac 12 + i \tau \mp i \jbb{n}) (\frac 12 + i \mu \mp i \jbb{n})} dt \\
&\quad + \int_{-\infty}^0 \eta' (t) \frac{e^{\frac{t}{2} - i \tau t \pm i t \jbb{n}}}{(\frac 12 - i \tau \pm i \jbb{n}) (\frac 12 + i \mu \mp i \jbb{n})} dt \bigg).
\end{align*}

\noi
Using integration by parts again and triangle inequalities, we can deduce that
\begin{align*}
|\mathcal{K} (n, \tau, \mu)| \ind_{\{|\tau - \mu| > 1\}} &\les \frac{\jb{|\tau| - \jbb{n}}^b \jb{|\mu| - \jbb{n}}^{1 - b}}{\jb{\mu - \tau}^2 \jb{\mu \mp \jbb{n}}} + \frac{\jb{|\tau| - \jbb{n}}^b \jb{|\mu| - \jbb{n}}^{1 - b}}{\jb{\tau \mp \jbb{n}}^2 \jb{\mu \mp \jbb{n}}} \\
&\leq \frac{\jb{|\tau| - \jbb{n}}^b}{\jb{\mu - \tau}^2 \jb{\mu \mp \jbb{n}}^b} + \frac{1}{\jb{\tau \mp \jbb{n}}^{2 - b} \jb{\mu \mp \jbb{n}}^b} \\
&= \mathcal{K}_1 (n, \tau, \mu) + \mathcal{K}_2 (n, \tau, \mu).
\end{align*}

\noi
For $\mathcal{K}_2$, by H\"older's inequality, we have
\begin{align*}
\bigg\| \int_\R \mathcal{K}_2 (n, \tau, \mu) G(\mu) d\mu \bigg\|_{L_\tau^2} &\les \| \mathcal{K}_2 (n, \tau, \mu) \|_{L_{\tau, \mu}^2} \| G(\mu) \|_{L_\mu^2} \\
&\les  \bigg( \int_\R \int_\R \frac{1}{\jb{\tau \mp \jbb{n}}^{4 - 2b} \jb{\mu \mp \jbb{n}}^{2b}} d \tau d \mu \bigg)^{1/2} \| G(\mu) \|_{L_\mu^2} \\
&\les \| G(\mu) \|_{L_\mu^2},
\end{align*}

\noi
as desired. For $\mathcal{K}_1$, we note that
\begin{align*}
\mu - \tau = (\mu \mp \jbb{n}) - (\tau \mp \jbb{n}).
\end{align*}

\noi
If $|\tau \mp \jbb{n}| \gg |\mu \mp \jbb{n}|$, we can show \eqref{goal} for $\mathcal{K}_1$ in the same way as $\mathcal{K}_2$. If $|\tau \mp \jbb{n}| \les |\mu \mp \jbb{n}|$, by Schur's test, Lemma \ref{LEM:conv}, and triangle inequalities, we have
\begin{align*}
\| \mathcal{K}_1 (n, \tau, \mu) \|_{L_\mu^2 \to L_\tau^2}^2 &\leq \sup_{\tau \in \R} \int_\R \frac{\jb{|\tau| - \jbb{n}}^b}{\jb{\mu - \tau}^2 \jb{\mu \mp \jbb{n}}^b} d \mu \cdot \sup_{\mu \in \R} \int_\R \frac{\jb{|\tau| - \jbb{n}}^b}{\jb{\mu - \tau}^2 \jb{\mu \mp \jbb{n}}^b} d \tau \\
&\les \sup_{\tau \in \R} \frac{\jb{|\tau| - \jbb{n}}^b}{\jb{\tau \mp \jbb{n}}^b} \cdot \sup_{\mu \in \R} \int_\R \frac{1}{\jb{\mu - \tau}^2} d\tau \\
&\les 1,
\end{align*}

\noi
as desired.
\end{proof}

Again, we will abuse notation and write $\vec \I_{t_0} = (\I_{t_0}, \dt \I_{t_0}) = (\I_{0, t_0}, \I_{1, t_0})$ later in this paper.

\medskip
We also record the following time localization estimate. 
The proof follows from \cite[Lemma~3.11]{ErTz} along with \eqref{equi_norm}.
\begin{lemma}
\label{LEM:time}
Let $s \in \R$ and $-\frac 12 < b_1 \leq b_2 < \frac 12$. Let $\varphi$ be a continuous function that satisfies $|\ft \varphi (\tau)| \les \jb{\tau}^{-2}$. Define $\varphi_T (t) = \varphi (t / T)$ for $0 < T \leq 1$. Then, we have
\begin{align}
\| \varphi_T u \|_{X^{s, b_1}} \les T^{b_2 - b_1} \| u \|_{X^{s, b_2}}.
\label{time1}
\end{align}

\noi
Furthermore, if we let $I \subset \R$ be a closed interval with $0 < |I| \leq 1$, then we have
\begin{align}
\| u \|_{X_I^{s, b_1}} \les |I|^{b_2 - b_1} \| u \|_{X_I^{s, b_2}} .
\label{time2}
\end{align}
\end{lemma}

\subsection{Strichartz estimates}
\label{SUBSEC:str}
Having established basic linear estimates of the $X^{s,b}$-space, we now discuss the Strichartz estimate and its consequences.
We first record the following Strichartz estimate. See \cite[Proposition 1.1]{Sch}, whose proof is based on the $\l^2$-decoupling theorem in \cite{BD15}.
\begin{lemma}
\label{LEM:str1}
Let $\al > 1$, $p \geq \frac{10}{3}$, $\eps > 0$, and $\phi \in L^2 (\T^3)$ with $\ft \phi$ supported on $\{ \frac{N}{2} \leq |n| \leq N \}$. Then, for each closed interval $I \subset \R$, we have
\begin{align*}
\big\| e^{\pm it \jbb{\nb}} \phi \big\|_{L_I^p L_x^p} \les N^{\frac{3}{2} - \frac{3 + \min (\al, 2)}{p} + \eps} (1 + |I|)^{\frac{1}{p}} \| \phi \|_{L^2}.
\end{align*}
\end{lemma}

Using Lemma \ref{LEM:str1}, \eqref{equi_norm}, and the transference principle \cite[Lemma 2.9]{Tao}, we obtain the following $X^{s,b}$-version of the Strichartz estimate.
\begin{lemma}
\label{LEM:str2}
Let $1 < \al \leq 2$, $b > \frac 12$, $p \geq \frac{10}{3}$, and $s \geq 0$ satisfying
\begin{align*}
s > \frac 32 - \frac{3 + \al}{p}.
\end{align*}
\noi
Then, for each closed interval $I \subset \R$, we have
\begin{align*}
\| u \|_{L_I^p L_x^p} \les (1 + |I|)^{\frac{1}{p}} \| u \|_{X_I^{s, b}}.
\end{align*}
\end{lemma}

Note that by Lemma \ref{LEM:str2} and interpolation with the following $L^2$-bound:
\begin{align*}
\| u \|_{L_I^2 L_x^2} \leq \| u \|_{X_I^{0, 0}},
\end{align*}

\noi
we obtain that 
\begin{align}
\| u \|_{L_I^p L_x^p} \les (1 + |I|)^{\frac 1p} \| u \|_{X_I^{s, \frac 12 - \dl}},
\label{Lp_str}
\end{align}

\noi
where $s$ and $p$ satisfy the same conditions as in Lemma \ref{LEM:str2} and $\dl = \dl (s, \al, p) > 0$ is sufficiently small.

\medskip
We now present some useful multilinear estimates using the Strichartz estimate \eqref{Lp_str}. The following lemma is a (slightly stronger) counterpart of \cite[Lemma~2.6]{OWZ}.
\begin{lemma}
\label{LEM:str_var}
Let $1 < \al \leq 2$ and $\eps, \dl > 0$ be sufficiently small. Let $s \in \R$ be such that
\begin{align*}
\max \Big( - \al + \frac 32, - \frac{3}{10} \al + \frac 35 \Big) < s < \frac{13}{10} \al - \frac 35.
\end{align*}

\noi
Then, for each closed interval $I \subset \R$, we have
\begin{align}
\| u_1 u_2 u_3 v \|_{X_I^{s - \al + \eps, -\frac 12 + 4 \dl}} \les (1 + |I|) \Big( \prod_{j = 1}^3 \| u_j \|_{X_I^{s, \frac 12 + \dl}} \Big) \| v \|_{L_I^\infty W_x^{-\eps, \infty}}.
\label{str123v}
\end{align}
\end{lemma}

\begin{proof}
We let $p \geq \frac{10}{3}$ be such that
\begin{align*}
s - \eps > \frac 32 - \frac{3 + \al}{p} \quad \Leftrightarrow \quad \frac 1p > \frac{3 - 2s + 2 \eps}{6 + 2\al}.
\end{align*}

\noi
We also let $q \geq \frac{10}{3}$ be such that
\begin{align*}
\frac{3}{p} + \frac{1}{q} = \frac{1}{1 + \dl'} \quad \text{and} \quad -s + \al - 2 \eps > \frac{3}{2} - \frac{3 + \al}{q}
\end{align*}

\noi
for some $\dl' > 0$ arbitrarily small. These conditions can be satisfied given $\max (-\al + \frac 32, - \frac{3}{10} \al + \frac 35) < s < \frac{13}{10} \al - \frac 35$ and $\eps > 0$ sufficiently small. 
By duality, Lemma \ref{LEM:gko}, H\"older's inequality, and the Strichartz estimate \eqref{Lp_str}, we have
\begin{align*}
\| u_1 u_2 u_3 v \|_{X_I^{s - \al + \eps, -\frac 12 + 4 \dl}} 
&\leq \sup_{\| w \|_{X^{-s + \al - \eps, \frac 12 - 4 \dl}} \leq 1} \bigg| \int_I \int_{\T^3} u_1 u_2 u_3 v w dx dt \bigg| \\
&\leq \sup_{\| w \|_{X^{-s + \al - \eps, \frac 12 - 4 \dl}} \leq 1}  \int_I \int_{\T^3} \big| \jb{\nabla}^\eps (u_1 u_2 u_3 w) \big| dx dt \| v \|_{L_I^\infty W_x^{- \eps, \infty}}  \\
&\leq \sup_{\| w \|_{X^{-s + \al - \eps, \frac 12 - 4 \dl}} \leq 1} (1 + |I|)^{\frac{\dl'}{1 + \dl'}} \bigg( \prod_{j = 1}^3 \| \jb{\nabla}^\eps u_j \|_{L_I^p L_x^p} \bigg) \\
&\quad \times \| \jb{\nabla}^\eps w \|_{L_I^q L_x^q} \| v \|_{L_I^\infty W_x^{- \eps, \infty}} \\
&\les (1 + |I|) \Big( \prod_{j = 1}^3 \| u_j \|_{X_I^{s, \frac 12 + \dl}} \Big) \| v \|_{L_I^\infty W_x^{-\eps, \infty}},
\end{align*}

\noi
which proves the desired estimate \eqref{str123v}.
\end{proof}

We also show the following multilinear estimate, which can be viewed as an extended version of Lemma~\ref{LEM:str_var}. This extended multilinear estimate is important in proving our weak universality result.
\begin{lemma}
\label{LEM:str_ext}
Let $1 < \al \leq 2$, $\l \geq 4$, and $\dl > 0$ be sufficiently small. Let $s \in \R$ be such that 
$$\max \Big( -\al + \frac 32, - \frac{3}{10} \al + \frac 35 \Big) < s < \frac{13}{10} \al - \frac 35.$$ 
Then, for each closed interval $I \subset \R$, we have
\begin{align}
\| u^\l v \|_{X_I^{s - \al, - \frac 12 + 4 \dl}} \les (1 + |I|) \| u \|_{X_I^{s, \frac 12 + \dl}}^{3 + \eta} \| u \|_{L_I^{\frac{1}{\eta}} L_x^{\frac{1}{\eta}}}^{\l - 3 - \eta} \| v \|_{L_I^{\frac{1}{\eta}} L_x^{\frac{1}{\eta}}}
\label{str_l}
\end{align}

\noi
for some $\eta > 0$ sufficiently small.
\end{lemma}

\begin{proof}
We let $p \geq \frac{10}{3}$ be such that
\begin{align*}
s > \frac 32 - \frac{3 + \al}{p + \eta} \quad \Leftrightarrow \quad \frac{1}{p + \eta} < \frac{3 - 2s}{6 + 2\al}.
\end{align*}

\noi
We also let $q \geq \frac{10}{3}$ be such that
\begin{align*}
\frac{3}{p + \eta} + \frac{\eta}{p + \eta} + (1 - \eta) \eta + \eta(\l - 3) + \frac{1}{q} = 1 \quad \text{and} \quad -s + \al > \frac{3}{2} - \frac{3 + \al}{q}.
\end{align*}

\noi
These conditions can be satisfied given $\max (-\al + \frac 32, - \frac{3}{10} \al + \frac 35) < s < \frac{13}{10} \al - \frac 35$ and $\eta > 0$ sufficiently small.
By duality, H\"older's inequality, interpolation, and the Strichartz estimate \eqref{Lp_str}, we have
\begin{align*}
\| u^\l v \|_{X_I^{s - \al, -\frac 12 + 4 \dl}} 
&\leq \sup_{\| w \|_{X^{-s + \al, \frac 12 - 4 \dl}} \leq 1} \bigg| \int_I \int_{\T^3} u^\l v w dx dt \bigg| \\
&\leq \sup_{\| w \|_{X^{-s + \al, \frac 12 - 4 \dl}} \leq 1} \| u \|_{L_I^{p + \eta} L_x^{p + \eta}}^3 \| u \|_{L_I^{\frac{1}{\frac{\eta}{p + \eta} + (1 - \eta)\eta}} L_x^{\frac{1}{\frac{\eta}{p + \eta} + (1 - \eta)\eta}}} \| u \|_{L_I^{\frac{1}{\eta}} L_x^{\frac{1}{\eta}}}^{\l - 4} \\
&\quad \times \| v \|_{L_I^{\frac{1}{\eta}} L_x^{\frac{1}{\eta}}} \| w \|_{L_I^q L_x^q} \\
&\les \| u \|_{L_I^{p + \eta} L_x^{p + \eta}}^{3 + \eta} \| u \|_{L_I^{\frac{1}{\eta}} L_x^{\frac{1}{\eta}}}^{\l - 3 - \eta} \| v \|_{L_I^{\frac{1}{\eta}} L_x^{\frac{1}{\eta}}} \\
&\les \| u \|_{X_I^{s, \frac 12 + \dl}}^{3 + \eta} \| u \|_{L_I^{\frac{1}{\eta}} L_x^{\frac{1}{\eta}}}^{\l - 3 - \eta} \| v \|_{L_I^{\frac{1}{\eta}} L_x^{\frac{1}{\eta}}},
\end{align*}

\noi
which proves the desired estimate \eqref{str_l}.
\end{proof}

We also show the following trilinear estimates, which extends \cite[Lemma~2.7]{OWZ} (see also \cite[Proposition~8.6]{Bring2}). The estimates will be useful when one of the term is the stochastic convolution. In particular, the estimate \eqref{tri2} below plays a crucial role in our globalization argument.
\begin{lemma}
\label{LEM:str4}
Let $1 < \al \leq \frac 32$ and $\eps, \dl > 0$ be sufficiently small. Let $s \in \R$ be such that
\begin{align*}
\max \Big( -\al + \frac 32, - \frac{3}{10} \al + \frac 35 \Big) < s < \min \Big( \frac{13}{10} \al - \frac 35, 2 \al - \frac 32 \Big).
\end{align*}

\noi
Then, there exists $r \geq \frac{10}{3}$ satisfying $s > \frac 32 - \frac{3 + \al}{r}$ such that for each closed interval $I \subset \R$, we have
\begin{align}
\big\| u v^2 \big\|_{X_I^{s - \al + \eps, -\frac 12 + 4 \dl }} \les (1 + |I|) \| u \|_{L_I^\infty W_x^{\al - \frac 32 - \eps, \infty}} \| v \|_{L^r_I L^r_x} \| v \|_{X_I^{s, \frac 12 + \dl}}.
\label{tri2}
\end{align}

\noi
Furthermore, we have
\begin{align}
\| u_1 u_2 u_3 \|_{X_I^{s - \al + \eps, -\frac 12 + 4 \dl}} \les (1 + |I|) \| u_1 \|_{L_I^\infty W_x^{\al - \frac 32 - \eps, \infty}} \| u_2 \|_{X_I^{s, \frac 12 + \dl}} \| u_3 \|_{X_I^{s, \frac 12 + \dl}}.
\label{tri3}
\end{align}
\end{lemma}

\begin{proof}
We mainly consider \eqref{tri2} and indicate the modifications needed for \eqref{tri3}. For a dyadic number $N \geq 1$, we denote by $\P_N$ the Littlewood-Paley projection onto frequencies $\{ |n| \sim N \}$. By using the Littlewood-Paley decompositions, we have
\begin{align*}
\big\| u v^2 \big\|_{X_I^{s - \al + \eps, -\frac 12 + 4\dl}} \leq \sum_{N_1, N_{23}, N_{123} \geq 1} \big\| \P_{N_{123}} \big( \P_{N_1} u \cdot \P_{N_{23}} (v^2) \big) \big\|_{X_I^{s - \al + \eps, - \frac 12 + 4 \dl}}.
\end{align*}

\noi
Since for $N_{123} \gg N_1 + N_{23}$ the terms in the above sum vanish, we can separately discuss the two cases: (i) $N_{123} \sim \max(N_1, N_{23})$ and (ii) $N_{123} \ll \max (N_1, N_{23})$.

\smallskip \noi
\textbf{Case 1:} $N_{123} \sim \max (N_1, N_{23})$.

We let $\frac 12 \leq \frac{1}{p} \leq \frac 35$ be such that
\begin{align*}
\frac{-s + 2\al - \frac 32 - 3 \eps}{3} \geq \frac 1p - \frac 12 \quad \text{and} \quad s > \frac 32 - \frac{3 + \al}{2p}.
\end{align*}

\noi
Such $p$ can be found given $\eps > 0$ sufficiently small, $-\frac{3}{10} \al + \frac 35 < s < 2 \al - \frac 32$, and $s > \frac{-2 \al^2 - 6 \al + 9}{3 - \al}$. The last condition is automatically satisfied given $s > -\al + \frac 32 > \frac{-2 \al^2 - 6 \al + 9}{3 - \al}$ since $\al > 1$.

By Sobolev's embedding in time and in space, H\"older's inequality, the fact that $\al - \frac 32 - \eps < 0$ given $1 < \al \leq \frac 32$, and the Strichartz estimate \eqref{Lp_str}, we have
\begin{align*}
\big\| &\P_{N_{123}} \big( \P_{N_1} u \cdot \P_{N_{23}} (v^2) \big) \big\|_{X_I^{s - \al + \eps, -\frac 12 + 4 \dl}} \\
&\les N_{123}^{\al - \frac 32 - 2 \eps} \big\| \P_{N_1} u \cdot \P_{N_{23}} (v^2) \big\|_{L_I^p H_x^{s - 2\al + \frac 32 + 3\eps}} \\
&\les N_{123}^{\al - \frac 32 - 2 \eps} \big\| \P_{N_1} u \cdot \P_{N_{23}} (v^2) \big\|_{L_I^p L_x^p} \\
&\les N_{123}^{-\eps} \| \P_{N_1} u \|_{L_I^\infty W_x^{\al - \frac 32 - \eps, \infty}} \| v \|_{L_I^{2p} L_x^{2p}}^2 \\
&\les N_{123}^{-\eps} (1 + |I|)^{\frac{1}{2p}} \| u \|_{L_I^\infty W_x^{\al - \frac 32 - \eps, \infty}} \| v \|_{L_I^{2p} L_x^{2p}} \| v \|_{X_I^{s, \frac 12 + \dl}}.
\end{align*}

\noi
By summing up the above inequality in dyadic $N_1, N_{23}, N_{123} \geq 1$, we obtain \eqref{tri2} in this case. Note that in this case, \eqref{tri3} follows directly from the above inequality and the Strichartz estimate \eqref{Lp_str}.

\smallskip \noi
\textbf{Case 2:} $N_{123} \ll \max (N_1, N_{23})$.

In this case, we further apply the Littlewood-Paley decompositions for the two $v$'s and write
\begin{align*}
v^2 = \sum_{N_2, N_3 \geq 1} (\P_{N_2} v) (\P_{N_3} v).
\end{align*}

\noi
By symmetry, we can assume that $N_2 \leq N_3$. Then, we have
\begin{align*}
N_{123} \les N_1 \sim N_{23} \les N_3.
\end{align*}

We let $q_1 \geq \frac{10}{3}$ be such that
\begin{align*}
-s + \al - 2 \eps > \frac 32 - \frac{3 + \al}{q_1}.
\end{align*}

\noi
We also let $q_2 \geq \frac{10}{3}$ be such that $\frac{1}{q_2} = \frac 12 - \frac{1}{q_1}$ and
\begin{align*}
s > \frac 32 - \frac{3 + \al}{q_2}.
\end{align*}

\noi
Thus, we require $\frac 15 \leq \frac{1}{q_1} \leq \frac{3}{10}$, and such $q_1$ can be found given $-\frac{3}{10} \al + \frac 35 < s < \frac{13}{10} \al - \frac 35$, $\eps > 0$ sufficiently small, and $\al > 1$.

With the understanding that $\P_N = 0$ for $N < 1$, by duality, H\"older's inequality twice, and the Strichartz estimate \eqref{Lp_str}, we have
\begin{align*}
\big\| &\P_{N_{123}} \big( \P_{N_1} u \cdot \P_{N_{23}} (\P_{N_2} v \P_{N_3} v) \big) \big\|_{X_I^{s - \al + \eps, -\frac 12 + 4 \dl}}  \\
&\les N_{123}^{s - \al + \eps} \sup_{\| w \|_{X^{0, \frac 12 - 4 \dl}} \leq 1} \bigg| \int_I \int_{\T^3} \P_{N_{123}} \big( \P_{N_1} u \cdot \P_{N_{23}} (\P_{N_2} v \P_{N_3} v) \big) \\
&\quad \times (\P_{N_{123} / 2} + \P_{N_{123}} + \P_{2 N_{123}}) w dx dt \bigg| \\
&\les N_{123}^{s - \al + \eps} \big\| \P_{N_1} u \cdot \P_{N_{23}} (\P_{N_2} v \P_{N_3} v) \big\|_{L_I^{q_1'} L_x^{q_1'}} \\
&\quad \times \sup_{\| w \|_{X^{0, \frac 12 - 4 \dl}} \leq 1} \| (\P_{N_{123} / 2} + \P_{N_{123}} + \P_{2 N_{123}}) w \|_{L_I^{q_1} L_x^{q_1}} \\
&\les N_{123}^{-\eps} (1 + |I|)^{\frac{1}{q_1}} \| \P_{N_1} u \|_{L_I^\infty L_x^\infty} \| \P_{N_2} v \|_{L_I^{q_2} L_x^{q_2}} \| \P_{N_3} v \|_{L_I^2 L_x^2} \\
&\les N_{123}^{-\eps} N_1^{-\al + \frac 32 + \eps} N_3^{- s + \eps} (1 + |I|) \| u \|_{L_I^\infty W_x^{\al - \frac 32 - \eps, \infty}} \| v \|_{L_I^{q_2} L_x^{q_2}} \| v \|_{X_I^{s, \frac 12 + \dl}}.
\end{align*}

\noi
Since $N_3$ is the largest frequency scale and $s > -\al + \frac 32$, we can sum up the above inequality in dyadic $N_1, N_2, N_3, N_{23}, N_{123} \geq 1$, we obtain \eqref{tri2} in this case.

For \eqref{tri3}, due to the lack of symmetry of $u_2$ and $u_3$, we need to put both of them in the $X_I^{s, \frac 12 + \dl}$-space, which can be done by similar steps as above along with the Strichartz estimate \eqref{Lp_str}.
\end{proof}

\subsection{Tools from stochastic analysis and multiple stochastic integrals}
\label{SUBSEC:mul}

Let us first define the Hermite polynomials $H_k (x; \s)$ with variance $\s > 0$ via the following generating function:
\begin{align}
e^{tx - \frac 12 \s t^2} = \sum_{k = 0}^\infty \frac{t^k}{k!} H_k (x; \s)
\label{herm_gen}
\end{align}

\noi
for $t, x \in \R$. From \eqref{herm_gen}, we obtain the following identities for any $k \in \N$ and $\s, \s_1, \s_2 > 0$: 
\begin{align}
H_k (x + y; \s) &= \sum_{\l = 0}^k \binom{k}{\l} H_{k - \l} (x; \s) y^\l, \label{herm_decomp1} \\
H_k (x + y; \s_1 + \s_2) &= \sum_{\l = 0}^k \binom{k}{\l} H_{k - \l} (x; \s_1) H_\l (y; \s_2).
\label{herm_decomp}
\end{align}



\medskip
We now recall basic definitions and properties of multiple stochastic integrals. We mainly follow the notation in \cite[Appendix B]{OWZ}. See \cite{Nua} and also \cite{Bring2, OWZ, BDNY} for further discussion.

Let $\{B_n\}_{n \in \Z^3}$ be a sequence of Gaussian processes on a probability space $(\O, \mathcal{E}, \PP)$ with the following properties:

\smallskip \noi
(i) $B_n$ is a complex-valued standard Brownian motion for $n \in \Z^3 \setminus \{0\}$ and $B_0$ is a real-valued standard Brownian motion. 
Specifically, for all $t \geq 0$ and $n \in \Z^3$, we have $\E [|B_n(t)|^2] = t$.

\smallskip \noi
(ii) For all $n_1, n_2 \in \Z^3$ with $n_1 \neq \pm n_2$, the processes $B_{n_1}$ and $B_{n_2}$ are independent.

\smallskip \noi
(iii) For all $n \in \Z^3$, $\cj{B_n} = B_{-n}$.

\smallskip
Given $f \in L^2 ((\Z^3 \times \R_+)^k)$, we define the multiple stochastic integral $I_k[f]$ by
\begin{align*}
I_k [f] \deff \sum_{n_1, \dots, n_k \in \Z^3} \int_{\R_+} \cdots \int_{\R_+} f(n_1,t_1, \dots, n_k, t_k) d B_{n_1} (t_1) \cdots d B_{n_k} (t_k).
\end{align*}

\noi
We recall that all multiple stochastic integrals have mean zero: for any $k \in \N$ and $f \in L^2 ( (\Z^3 \times \R_+)^k )$, we have
\begin{align}
\E \big[ I_k [f] \big] = 0.
\label{mean0}
\end{align}

\noi
When we work on standard Gaussian random variables, say $\{ g_n \}_{n \in \Z^3}$, we write them as a stochastic integral against a sequence of standard Brownian motions $\{ B_n \}_{n \in \Z^3}$ as
\begin{align}
g_{n'} = \int_0^1 1 \, d B_{n'} (t) = I_1 [\ind_{\{ n = n' \}} \ind_{\{ 0 \leq t \leq 1 \}}]
\label{gn_mul}
\end{align}

\noi
for any $n' \in \Z^3$. This notation allows us to treat Gaussian processes and Gaussian random variables uniformly and is convenient for us to discuss Wick orderings in \eqref{HkIk} below for Gaussian random variables.

By using a rescaled version of \cite[Proposition~1.1.4]{Nua}, we have the following relation between Hermite polynomials and multiple stochastic integrals: given $k \in \N$, $f \in L^2 (\Z^3 \times \R_+)$ real-valued, and $\s_f = \| f \|_{L^2 (\Z^3 \times \R_+)}^2$, we have the Wick ordering
\begin{align}
:\! (I_1 [f])^k \!: \, \deff H_k (I_1 [f]; \s_f) = I_k [f^{\otimes k}],
\label{HkIk}
\end{align}

\noi
where $f^{\otimes k}$ denotes the $k$-fold tensor product of $f$ with itself and $H_k$ is the $k$th Hermite polynomial as in \eqref{herm_gen}. If $X$ is a real-valued Gaussian random variable with variance $\s_X$, we define the Wick ordering
\begin{align}
:\! X^k \!: \, \deff H_k (X; \s_X),
\label{HkX}
\end{align}

\noi
which aligns with the definition in \eqref{HkIk} in view of the notation \eqref{gn_mul}. Also, from \eqref{mean0}, we have
\begin{align}
\E [ :\! X^k \!: ] = 0
\label{mean0X}
\end{align}

\noi
for any $k \in \N$.

Given $f \in L^2 ((\Z^3 \times \R_+)^k)$, we  define the symmetrization $\text{Sym} (f)$ by
\begin{align*}
\text{Sym} (f) (z_1, \dots, z_k) = \frac{1}{k!} \sum_{\phi \in S_k} f (z_{\phi(1)}, \dots, z_{\phi (k)}),
\end{align*}

\noi
where $z_j = (n_j, t_j)$ and $S_k$ is the group of permutations on $\{ 1, \dots, k \}$.  Let us recall the following Ito isometry from \cite[Property~(iii)~in~page~9]{Nua}: given $k, \l \in \N$, $f \in L^2 ((\Z^3 \times \R_+)^k)$, and $g \in L^2 ((\Z^3 \times \R_+)^\l)$, we have
\begin{align}
\begin{split}
\E \big[ I_k [f] \cj{I_\l [g]} \big] = \ind_{\{ k = \l \}} \cdot k! \sum_{n_1, \dots, n_k \in \Z^3} &\int_{\R_+} \cdots \int_{\R_+} \text{Sym} (f) (z_1, \dots, z_k) \\
&\quad \times \cj{\text{Sym} (g)} (z_1, \dots, z_k) d t_1 \cdots d t_k.
\end{split}
\label{ito}
\end{align}

\noi
From the Cauchy-Schwarz inequality, we have
\begin{align}
| \text{Sym} (f) (z_1, \dots, z_k) |^2 \leq \frac{1}{k!} \sum_{\phi \in S_k} | f (z_{\phi(1)}, \dots, z_{\phi (k)}) |^2,
\label{jensen}
\end{align}

\noi
which will be used to simplify estimates after the use of the Ito isometry \eqref{ito}. If $X$ and $Y$ are jointly real-valued Gaussian random variables with variances $\s_X$ and $\s_Y$, respectively, we obtain the following version of the Ito isometry from the notation \eqref{gn_mul}, \eqref{HkX}, \eqref{HkIk}, and \eqref{ito}:
\begin{align}
\E [ :\! X^k \!: \, :\! Y^\l \!: ] = \ind_{\{ k = \l \}} \cdot k! (\E [X Y])^k
\label{itoX}
\end{align}

\noi
for any $k, \l \in \N$.

Lastly, we recall the notion of contraction and the product formula for multiple stochastic integrals. For $k, \l, r \in \N$ with $0 \leq r \leq \min(k, \l)$ and $f \in L^2 ((\Z^3 \times \R_+)^k)$ and $g \in L^2 ((\Z^3 \times \R_+)^\l)$, the contraction $f \otimes_r g$ of $r$ indices is defined by
\begin{align*}
(f \otimes_r g) (z_1, \dots, z_{k + \l - 2r}) &\deff \sum_{m_1, \dots, m_r \in \Z^3} \int_{\R_+^r} f(z_1, \dots, z_{k-r}, \zeta_1, \dots, \zeta_r) \\
&\quad \times g (z_{k + 1 - r}, \dots, z_{k + \l - 2r}, \wt{\zeta}_1, \dots, \wt{\zeta}_r) ds_1 \cdots ds_r,
\end{align*}

\noi
where $\zeta_j = (m_j, s_j)$ and $\wt{\zeta}_j = (- m_j, s_j)$. When two spatial frequencies $m_1$ and $m_2$ are contracted, we say that $(m_1, m_2)$ is a \textit{pairing}, for which we must have $m_1 + m_2 = 0$.
From \cite[Proposition 1.1.3]{Nua}, we have the following product formula given that $f$ and $g$ are symmetric in the sense that $\textup{Sym} (f) = f$ and $\textup{Sym} (g) = g$:
\begin{align}
I_k [f] \cdot I_\l [g] = \sum_{r = 0}^{\min (k, \l)} r! \binom{k}{r} \binom{\l}{r} I_{k + \l - 2r} [f \otimes_r g].
\label{prod}
\end{align}

\section{Integer lattice counting estimates}
\label{SEC:count}
In this section, we prove some counting estimates
related to the fractional wave equations,
which are important in controlling the stochastic terms, both for our global well-posedness result and for our weak universality result. 
Most of the counting estimates below have prototypes in \cite{Bring2}. 


\subsection{Basic counting estimates}
We start with the following basic counting lemma for the fractional wave equation. In the following, we denote
\begin{align*}
\jbb{n} = \jbb{n}_\al \deff \sqrt{\jb{n}^{2\al}-\frac 14},
\end{align*}

\noi
where $\al \geq 0$.

We first present the following basic counting lemma. The lemma corresponds to \cite[Lemma~4.15]{Bring2}, but in our setting, we need different computations.

\begin{lemma}
\label{LEM:BC_fwave}
Let $1 < \al \leq 2$, $N, A \in 2^\N$ be dyadic numbers,
and $a \in \Z^3$ be such that $|a| \sim A$. Then,
we have 
\begin{align} 
\begin{split}
\sup_{\zeta \in  \R}  \# \big\{ n \in \Z^3\colon |n|\sim N, \big| \jbb{a + n} \pm \jbb{n} - \zeta \big| \les  1 \big\}   \les  
\min(A,N)^{-1} N^{3}. 
\end{split}
\label{BC_fwave}
\end{align}
\end{lemma}

\begin{proof}
If $a = 0$, we obtain the bound \eqref{BC_fwave} immediately. Thus, we can assume that $a \neq 0$, so that $|a| \geq 1$. Since $\xi \mapsto \jbb{ a + \xi } \pm \jbb{ \xi } $ is Lipschitz with constant $L \sim \max(A,N)^{\al-1}$ for $|\xi| \sim N$, we see that, for any fixed $\zeta \in \R$, the $1$-neighborhood of the set on the left-hand side of \eqref{BC_fwave} 
is contained in 
\begin{align*}
\big\{ \xi \in \R^3\colon |\xi|\sim N, \big| \jbb{ a + \xi } \pm \jbb{ \xi }   -\zeta' \big| \les  1 \text{ for some } \zeta' \in \R \text{ with } |\zeta' - \zeta| \les L \big\}.
\end{align*}

\noi
Then, we have 
\begin{align}
\begin{split}
 \sup_{\zeta \in  \R}  \# \big\{ & n \in \Z^3\colon |n|\sim N, \big|  \jbb{ a + n } \pm \jbb{ n } - \zeta \big|\les  1 \big\} \\
& \les  \max(A,N)^{\al-1} \sup_{\zeta' \in \R}
\big| \big\{ \xi \in \R^3\colon |\xi|\sim N, \big|  \jbb{ a + \xi } \pm \jbb{ \xi }  - \zeta' \big|\les  1 \big\} \big|,
\end{split}
\label{BC0}
\end{align}

\noi
where $|S|$ denotes the Lebesgue measure of a measurable set $S \subset \R^3$. Let us decompose 
\begin{align}
\begin{split}
& \big|\big\{ \xi \in \R^3\colon |\xi|\sim N, \big|  \jbb{ a + \xi } \pm \jbb{ \xi } - \zeta' \big|\les  1 \big\}\big|\\
&\les  \sum_{\substack{m_1,m_2\in \Z \\ | m_1 \pm m_2 - \zeta' |\les  1 }} 
\big|\big\{ \xi \in \R^3\colon |\xi| \sim N,   \jbb{a + \xi } = m_1 + O(1) , \jbb{\xi} = m_2 + O(1) \big\}\big| \\
&\les  N^\al \sup_{\substack{m_1,m_2 \in \Z \\ m_1, m_2 \geq 1 }} \big|\big\{ \xi \in \R^3\colon |\xi|\sim N,   \jbb{a + \xi } = m_1 + O(1) , \jbb{\xi} = m_2 + O(1) \big\}\big| ,
\end{split}
\label{BC1}
\end{align}

\noi
where in the last step,  
there are at most $\sim N^\al$ non-trivial choices of $m_2$, and then for a fixed $m_2 \in \Z$, the condition $ | m_1 \pm m_2 - \zeta' |\les  1$ implies that there are at most $\sim 1 $ choices for $m_1$. 
Since the Lebesgue measure is rotational invariant, we can assume that $a = |a| e_3$ with $e_3 = (0, 0, 1) \in \R^3$. By using polar coordinates, we note that 
\begin{align}
\begin{split}
& \big|\big\{ \xi \in \R^3\colon |\xi|\sim N,   \jbb{a + \xi } = m_1 + O(1) , \jbb{\xi} = m_2 + O(1) \big\}\big| \\
&\les   N^2 \int_0^\infty \int_0^\pi \ind_{\{ \jb{r}^\al=m_2 + O(1) \}} \ind_{ \{ (1+ |a|^2 + 2 r |a| \cos(\theta) + r^2)^{\frac \al2} = m_1 + O(1) \}} \sin(\theta) d\theta dr,
\end{split}
\label{BC1-1}
\end{align}

\noi
where $\theta$ is the angle between the vectors $\xi$ and $a$. Note that given $m_2 \geq 1$ and $1 < \al \leq 2$, we have
\begin{align}
\big\{r > 0: \jb{r}^\al = m_2 + O(1) \big \} 
\subset \Big \{r > 0: r = (m_2^{\frac{2}{\al}} - 1)^{\frac12} + O \big(m_2^{\frac{1}{\al} - 1} \big) \Big \}.
\label{BC1-2}
\end{align}

\noi
Thus, from \eqref{BC1-1} and \eqref{BC1-2}, we obtain
\begin{align}
\begin{split}
\big|\big\{ &\xi \in \R^3\colon |\xi|\sim N,   \jbb{a + \xi } = m_1 + O(1) , \jbb{\xi} = m_2 + O(1) \big\}\big| \\
  &\les N^2 \int_0^\infty \int_0^\pi 
\ind_{ \big\{ r =(m_2^{\frac2\al} - 1)^{\frac12}  + O ( m_2^{\frac 1\al-1} ) \big\}}  \\
&\quad \times \ind_{ \big\{ 1+ |a|^2 + 2 r |a| \cos(\theta) + r^2 = m_1^{\frac 2\al} + O (m_1^{\frac 2\al -1} ) \big\}} \sin(\theta) d\theta dr.
\end{split}
\label{BC2}
\end{align}

\noi
Since $m_1 \les  \max(A,N)^\al$,
we have 
\begin{align*}
\cos(\theta) = \frac{m_1^{\frac2{\al}}}{2|a|r}  - \frac{1+ |a|^2+r^2}{2|a|r}  + O \big( \max(A,N)^{ 2 -\al} A^{-1} N^{-1}\big).
\end{align*}

\noi
This shows that for a fixed $r$, $\cos(\theta)$ is contained in an interval of size $\sim \max(A,N)^{2 -\al} A^{-1} N^{-1}$.
Thus, we obtain
\begin{align}
\begin{split}
\eqref{BC2} &\les   \max(A,N)^{2 -\al} A^{-1} N^{-1} 
N^2  \int_0^\infty 
\ind_{ \big\{ r = (m_2^{\frac2\al} - 1)^{\frac12}  + O \big(m_2^{\frac 1\al-1} \big) \big\}}dr \\
&\les   
\max(A,N)^{2 -\al} A^{-1} N^{-1} 
N^2  N^{1 - \al}. 
\end{split}
\label{BC3}
\end{align}

\noi
By combining \eqref{BC0}, \eqref{BC1}, \eqref{BC2}, and \eqref{BC3}, 
we have
\begin{align*}
\sup_{\zeta \in \R} \#  \big\{ & n \in \Z^3  \colon |n|\sim N, \big|  \jbb{ n + a } \pm \jbb{ n } - \zeta \big| \les  1 \big\}\\
&\les  N^\al \max(A,N) A^{-1} N^{-1} 
N^2  N^{1 - \al}\\
& = \min(A,N)^{-1} 
N^{3},
\end{align*}

\noi
which finishes the proof.
\end{proof}

We also record a ``two-ball" counting, which is a direct consequence of Lemma~\ref{LEM:BC_fwave}. For a proof, see \cite[Lemma~4.17]{Bring2}.
\begin{lemma}
\label{LEM:two_balls}
Let $1 < \al \leq 2$, $N, A, B \geq 2^\N$ be dyadic numbers, and $a\in \Z^3$ be such that $|a|\sim A$. Then, we have
\begin{align*}
 \sup_{\zeta \in  \R} & \# \big\{  n \in \Z^3\colon |n|\sim N, |n+a|\sim B, \big| \jbb{ a + n } \pm \jbb{ n } - \zeta \big| \les  1 \big\}\\
&  \les   
\min(A,B,N)^{-1} \min(B,N)^{3}.
\end{align*}
\end{lemma}

\subsection{General counting estimates}

In this subsection, we prove some counting estimates that are in more general forms compared to those in \cite{Bring2}. These general counting estimates play a crucial role in the analysis of general Wick ordered stochastic terms for our weak universality result.

We start with the following estimate, which can be viewed as a generalized version of the cubic counting estimate as in \cite[Proposition~4.18~(i)]{Bring2}. This will be used as a lemma for later counting estimates.

\begin{lemma}
\label{LEM:cubic_counting}
Let $1 < \al \leq 2$ and $k \geq 3$ be an integer. Let $\eps_{\textup{sum}},\eps_1,\dots,\eps_k \in \{ \pm1\}$ and define the phase 
\begin{align}
\kappa (\bar n) = \kappa (n_1,\dots,n_k)\overset{\textup{def}}{=} 
 \eps_{\textup{sum}} \bigg\llbracket \sum_{j = 1}^k n_j \bigg\rrbracket + \sum_{j = 1}^k \eps_j \jbb{ n_{j} }.
\label{kappa}
\end{align}

\noi
Let $N_1,\dots,N_k \in 2^\N$ be dyadic numbers and let $\zeta \in \R$. Then, we have
\begin{align*}
\begin{split}
\sup_{\zeta \in \R} \# \big\{ & (n_1,\dots,n_k) \colon |n_1|\sim N_1,\dots, |n_k|\sim N_k, |\kappa (\bar n) - \zeta|\leq 1 \big\} \\
&\les  \big( \max_{(2)} ( N_1, \dots, N_k ) \big)^{-1} \prod_{j = 1}^k N_j^3,
\end{split}
\end{align*}

\noi
where $\max_{(2)}$ denotes the second largest element.
\end{lemma}

\begin{proof}
By symmetry, we can assume that $N_1 \leq N_2 \leq \cdots \leq N_k$. We first consider the case when $N_{k - 1} < \frac{1}{10 k} N_k$ or $N_{k - 2} < \frac{1}{10 k} N_k$. In this case, we have 
\begin{align*}
\bigg\langle \sum_{\substack{ j = 1 \\ j \neq k - 1}}^k n_j \bigg\rangle \geq |n_k| - \sum_{j = 1}^{k - 2} |n_j| > \frac 12 N_k - \frac{1}{10} N_k \ges N_k \geq N_{k - 1}.
\end{align*}

\noi
Thus, applying Lemma \ref{LEM:BC_fwave} to the sum in $n_{k-1} \in \Z^3$, we obtain
\begin{align*}
\# \big\{ & (n_1,\dots,n_k) \colon |n_1|\sim N_1,\dots, |n_k|\sim N_k, |\kappa (\bar n) - \zeta|\leq 1 \big\} \\
&\les \sum_{n_1, \dots, n_{k - 2}, n_k \in \Z^3} \bigg( \prod_{\substack{ j = 1 \\ j \neq k - 1}}^k \ind_{\{ |n_j|\sim N_j\}} \bigg) \min \bigg( \bigg\langle \sum_{\substack{ j = 1 \\ j \neq k - 1}}^k n_j \bigg\rangle, N_{k - 1} \bigg)^{-1} N_{k - 1}^3 \\
&\les N_{k-1}^{-1} \prod_{j = 1}^k N_j^3,
\end{align*}

\noi
which is acceptable.

We now consider the case when $N_\l < \frac{1}{10 k} N_k$ for some $1 \leq \l \leq k - 3$ and $N_j \geq \frac{1}{10 k} N_k$ for all $\l + 1 \leq j \leq k - 1$. For convenience, we denote $\overline{n}_\l = \sum_{j = 1}^\l n_j$, so that
\begin{align}
|\overline{n}_\l| \leq \sum_{j = 1}^\l |n_j| \leq \sum_{j = 1}^\l N_j < \frac{1}{10} N_k.
\label{nl_bdd}
\end{align}

\noi
We claim that
\begin{align}
\max_{\l + 1 \leq k' \leq k} \bigg\{ \bigg\langle \overline{n}_\l + \sum_{\substack{j = \l + 1 \\ j \neq k'}}^k n_j \bigg\rangle \bigg\} \geq \frac{1}{10 k} N_k.
\label{nl_max_bdd}
\end{align}

\noi
Indeed, assuming that for all $\l + 1 \leq k' \leq k - 1$ we have
\begin{align*}
\bigg\langle \overline{n}_\l + \sum_{\substack{j = \l + 1 \\ j \neq k'}}^k n_j \bigg\rangle < \frac{1}{10 k} N_k,
\end{align*}

\noi
we can deduce from \eqref{nl_bdd} that
\begin{align*}
\bigg\langle \overline{n}_\l + \sum_{j = \l + 1}^{k - 1} n_j \bigg\rangle &= \bigg\langle \overline{n}_l + \sum_{j = \l + 1}^{k - 1} n_j + \frac{k - \l - 1}{k - \l - 2} n_k - \frac{k - \l - 1}{k - \l - 2} n_k \bigg\rangle \\
&= \bigg\langle - \frac{k - \l - 1}{k - \l - 2} n_k + \overline{n}_\l + \frac{1}{k - \l - 2} \sum_{k' = \l + 1}^{k - 1} \sum_{\substack{j = \l + 1 \\ j \neq k'}}^k n_j \bigg\rangle \\
&\geq |n_k| - \frac{1}{k - \l - 2} |\overline{n}_\l| - \frac{1}{k - \l - 2} \sum_{k' = \l + 1}^{k - 1} \bigg\langle \overline{n}_\l +  \sum_{\substack{j = \l + 1 \\ j \neq k'}}^k n_j \bigg\rangle \\
&\geq \frac 12 N_k - \frac{1}{10} N_k - \frac{1}{10} N_k \geq \frac{1}{10k} N_k.
\end{align*}

\noi
This shows \eqref{nl_max_bdd}. We can assume without loss of generality that
\begin{align*}
\max_{\l + 1 \leq k' \leq k} \bigg\{ \bigg\langle \overline{n}_\l + \sum_{\substack{j = \l + 1 \\ j \neq k'}}^k n_j \bigg\rangle \bigg\} = \bigg\langle \overline{n}_\l + \sum_{j = \l + 1}^{k - 1} n_j \bigg\rangle = \bigg\langle \sum_{j = 1}^{k - 1} n_j \bigg\rangle.
\end{align*}

\noi
Applying Lemma \ref{LEM:BC_fwave} to the sum in $n_k \in \Z^3$ and using \eqref{nl_max_bdd}, we obtain
\begin{align*}
\# \big\{ & (n_1,\dots,n_k) \colon |n_1|\sim N_1,\dots, |n_k|\sim N_k, |\kappa (\bar n) - \zeta|\leq 1 \big\} \\
&\les \sum_{n_1, \dots, n_{k - 1} \in \Z^3} \bigg( \prod_{j = 1}^{k - 1} \ind_{\{ |n_j| \sim N_j \}} \bigg) \min \bigg( \bigg\langle \sum_{j = 1}^{k - 1} n_j \bigg\rangle, N_k \bigg)^{-1} N_k^3 \\
&\les N_k^{-1} \prod_{j = 1}^k N_j^3,
\end{align*}

\noi
which is sufficient for our purpose. The case when $N_j \geq \frac{1}{10k} N_k$ for all $1 \leq j \leq k - 1$ is similar by letting $\l = 0$ and $n_0 = 0$.
\end{proof}

\begin{remark}\rm
\label{RMK:CC2}
Let $n_{\textup{sum}} = \sum_{j = 1}^k n_j$. In the variables $(n_{\textup{sum}},n_1, \dots, n_{k - 1})$, the phase takes the form 
\begin{align*}
\kappa = \eps_{\textup{sum}} \jbb{ n_{\textup{sum}} } + \sum_{j = 1}^{k - 1} \eps_j \jbb{ n_{j} } + \eps_k \bigg\llbracket n_{\textup{sum}} - \sum_{j = 1}^{k - 1}n_j  \bigg\rrbracket.
\end{align*}

\noi
Here, we view $n_{\textup{sum}}$ as a free variable.
Let $N_{\textup{sum}}, N_1, \dots, N_{k-1}$ be dyadic numbers. After changing $(n_1, \dots, n_{k - 1}) \mapsto (-n_1, \dots, -n_{k - 1})$ and 
applying Lemma \ref{LEM:cubic_counting}, we obtain 
\begin{align}
\begin{split}
\sup_{\zeta \in \R} \# \big\{ & (n_{\textup{sum}},n_1, \dots,n_{k - 1}) \colon |n_{\textup{sum}}|\sim N_{\textup{sum}},|n_1|\sim N_1, \dots, |n_{k - 1}|\sim N_{k - 1},  \\
&|\kappa - \zeta|\leq 1 \big\} 
\les  \big( \max_{(2)} ( N_{\textup{sum}},N_1,\dots, N_{k - 1} ) \big)^{-1} N_{\textup{sum}}^3 \prod_{j = 1}^{k - 1} N_j^{3},
\end{split}
\label{CC2}
\end{align}

\noi
where $\max_{(2)}$ denotes the second largest element.
\end{remark}

We now show the following summation counting estimate, which can be viewed as a generalization of the cubic sum estimate in \cite[Proposition~4.20]{Bring2}.

\begin{lemma}
\label{LEM:CS}
Let $k \geq 3$ be an integer, $1 < \al \leq \frac 32$, $s \geq \al - 1$, and $N_1,\dots,N_k \in 2^\N$ be dyadic numbers. 
Let the phase function $\kappa$ be as in \eqref{kappa}.
Then, we have
\begin{align*}
\begin{split}
\sup_{\zeta \in \R}  &\sum_{n_1,\dots,n_k\in \Z^3} \bigg[ \bigg( \prod_{j=1}^{k} \ind_{\{|n_j| \sim N_j\}}  \jbb{ n_j }^{-2} \bigg) \bigg\langle \sum_{j = 1}^k n_j \bigg\rangle^{2(s-\al)}   \ind_{\{ |\kappa - \zeta|\leq 1\}}\bigg] \\
&\les  N_{\max}^{(k - 3) (3 - 2\al) + 2s - 6 \al + 6}, 
\end{split}
\end{align*}

\noi
where $N_{\max} = \max(N_1,\dots,N_k)$.
\end{lemma}

Before proving this lemma, let us mention that Lemma~\ref{LEM:CS} with $k = 3$ is important for establishing the regularity of the cubic stochastic object defined in \eqref{so4b} below. We observe that the upper bound in Lemma~\ref{LEM:CS} does not converge in the case $\al = 1$, and the assumption $\al > 1$ provides the essential smoothing that allows us to use much simpler solution ansatz compared to that in \cite{BDNY}.

\begin{proof}[Proof of Lemma~\ref{LEM:CS}]
By symmetry, we may assume that $N_1 \le \dots \le N_k$. Denote $n_{\textup{sum}} = \sum_{j = 1}^k n_j$.
We apply 
\eqref{CC2} in Remark \ref{RMK:CC2}
to get
\begin{align}
\begin{aligned}
 & \sum_{n_1,\dots,n_k\in \Z^3}  \bigg[ \bigg( \prod_{j=1}^{k} \ind_{\{|n_j| \sim N_j\}} \jbb{ n_j }^{-2}  \bigg) \bigg\langle \sum_{j = 1}^k n_j \bigg\rangle^{2(s-\al)}  \ind_{\{ |\kappa - \zeta|\leq 1\}}\bigg] \\
&\les  \sum_{N_{\textup{sum}} \les N_k} N_{\textup{sum}}^{2(s - \al)}  \bigg( \prod_{j=1}^{k} N_j^{-2\al} \bigg)\\ 
& \hphantom{XXX} \times \# \big\{ (n_1,\dots,n_k) \colon |n_{\textup{max}}| \sim N_{\textup{max}}, |n_1| \sim N_1, \dots, |n_k| \sim N_k, |\kappa - \zeta| \leq 1 \big\}\\
& \les  \prod_{j = 1}^k N_j^{-2 \al} \sum_{ N_{\textup{sum}} \les  N_k }  N_{\textup{sum}}^{2s-2 \al} 
  \big( \max_{(2)} ( N_{\textup{sum}},N_1,\dots,N_{k - 1} ) \big)^{-1} N_{\textup{sum}}^3 \prod_{j = 1}^{k - 1} N_j^3,
\end{aligned}
\label{CS1}
\end{align}

\noi
where $\max_{(2)}$ denotes the second largest element.
We note that $\max_{(2)} ( N_{\textup{sum}},N_1,\dots,N_{k-1} ) \les N_{k - 1}$.

We then distinguish two cases.
If $\max_{(2)} ( N_{\textup{sum}},N_1,\dots,N_{k-1} ) \sim N_{k - 1}$,
then \eqref{CS1} is bounded by
\begin{align}
\begin{split}
\prod_{j = 1}^{k - 2} N_j^{3 - 2\al} &\sum_{N_{\textup{sum}} \les N_k}  N_{\textup{sum}}^{2s - 2 \al + 3} N_{k - 1}^{2 - 2\al} 
N_k^{- 2\al} \\
&\les N_{k - 1}^{(k - 3)(3 - 2 \al) + 5 - 4 \al} N_k^{2s - 4 \al + 3} \\
&\les 
\begin{cases}
N_k^{(k - 3)(3 - 2 \al) + 2s - 8 \al + 8} & \text{if } \al \leq \frac{3k - 4}{2k - 2} \\
N_k^{(k - 3)(3 - 2 \al) + 2s - 4 \al + 3} & \text{if } \al > \frac{3k - 4}{2k - 2},
\end{cases}
\end{split}
\label{CS2}
\end{align}


If $\max_{(2)} ( N_{\textup{sum}},N_1,\dots,N_{k-1} ) \ll N_{k - 1}$,
then we have $\max ( N_{\textup{sum}}, N_1, \dots, N_{k - 2} ) \ll N_{k-1}$,
which implies that $N_{k - 1} \sim N_k$.
If furthermore $N_{\textup{sum}} \les N_{k - 2} \ll N_{k - 1}$, 
then \eqref{CS1} is bounded by
\begin{align}
\begin{split}
\prod_{j = 1}^{k - 3} N_j^{3 - 2 \al} &\sum_{N_{\textup{sum}} \les N_k}  N_{\textup{sum}}^{2s - 2 \al + 3} N_{k - 2}^{2 - 2 \al} N_{k - 1}^{3 - 2 \al} 
N_k^{-2 \al} \\
&\les N_{k - 2}^{(k - 3) (3 - 2 \al) + 2 - 2\al} N_k^{2s - 6 \al + 6} \\
&\les 
\begin{cases}
N_k^{(k - 3)(3 - 2 \al) + 2s - 8 \al + 8} & \text{if } \al \leq \frac{3k - 7}{2k - 4} \\
N_k^{(k - 3)(3 - 2 \al) + 2s - 6 \al + 6} & \text{if } \al > \frac{3k - 7}{2k - 4}.
\end{cases}
\end{split}
\label{CS3}
\end{align}

\noi
If $N_{k - 2} \ll  N_{\textup{sum}} \ll N_{k - 1}$,
then \eqref{CS1} is bounded by
\begin{align}
\begin{split}
\prod_{j = 1}^{k - 1} N_j^{3 - 2\al} \sum_{N_{\textup{sum}} \les N_k}  N_{\textup{sum}}^{2s - 2\al + 2} 
N_k^{-2 \al} \les
N_k^{(k - 3) (3 - 2 \al) + 2s - 8 \al + 8} 
\end{split}
\label{CS4}
\end{align}

\noi
given that $s \geq \al - 1$. Combining \eqref{CS2}, \eqref{CS3}, and \eqref{CS4}, we obtain the desired estimate.

\end{proof}

Next, we consider a general counting estimate involving a supremum, 
which will be used in the proof of the deterministic random tensor estimate. The estimate generalizes \cite[Lemma~4.22~(i)]{Bring2}.
\begin{lemma}
\label{tools:lem_sup}
Let $k \geq 3$ be an integer, $1 < \al \leq 2$, $N_1,\dots,N_k \in 2^\N$ be dyadic numbers, and $\zeta \in \R$.  
Let the phase function $\kappa$ be as in \eqref{kappa}.
Then, we have
\begin{align*}
\sup_{n \in \Z^3} & \# \bigg\{ (n_1,\dots,n_k)\colon |n_1|\sim N_1, \dots, |n_k|\sim N_k, n = \sum_{j = 1}^k n_j, |\kappa - \zeta|\leq 1 \bigg\}\\
&\les \big( \max ( N_1,\dots,N_k ) \big)^{-3} \big( \max_{(3)} ( N_1,\dots,N_k ) \big)^{-1} \prod_{j = 1}^{k } N_j^3 ,
\end{align*}

\noi
where $\max_{(3)}$ denotes the third largest element.
\end{lemma}

\begin{proof}
By symmetry, we assume without loss of generality that $N_1 \leq \cdots \leq N_k$. 
Using Lemma~\ref{LEM:BC_fwave}
for the sum in $n_{k - 1} \in \Z^3$, we have that 
\begin{align*}
 \# \bigg\{ & (n_1,\dots,n_k)\colon |n_1|\sim N_1, \cdots, |n_k|\sim N_k, n= \sum_{j = 1}^k n_j, |\kappa - \zeta|\leq 1\bigg\}\\
& =  \# \bigg\{ (n_1,\dots,n_{k - 1})\colon  |n_1|\sim N_1, \cdots |n_{k - 1}|\sim N_{k - 1}, 
\\
& \hphantom{XXXXXX} \bigg| \eps_{\textup{sum}} \jbb{ n }  + \sum_{j = 1}^{k - 1} \eps_{j} \jbb{ n_j }  + \eps_{k} \bigg\llbracket n - \sum_{j = 1}^{k - 1} n_j \bigg\rrbracket - \zeta \bigg| \leq 1\bigg\} \\
&\les  \sum_{n_1, \dots, n_{k - 2} \in \Z^3} \min \bigg( \bigg\langle n - \sum_{j = 1}^{k - 2} n_j \bigg\rangle, N_{k - 1} \bigg)^{-1} N_{k - 1}^{3} \\
&\les  N_{k - 2}^2 N_{k - 1}^3 \prod_{j = 1}^{k - 3} N_j^3.
\end{align*}

\noi
In the last inequality above, if $\jb{n - \sum_{j = 1}^{k - 2} n_j}$ is the smaller one, then we can fix $n_1, \dots, n_{k - 3}$ and sum over $n_{k - 2}$ to obtain the desired bound. Thus, we finish the proof.
\end{proof}

\subsection{Quintic counting estimates}
For our global well-posedness result, we will also need to establish the regularity of a quintic stochastic term, which requires some more counting estimates. We first consider the following non-resonant quintic sum estimate similar to \cite[Lemma~4.27]{Bring2}. Below, we use notation $n_{12\cdots k} = n_1 + n_2 + \cdots + n_k$ from \cite{Bring2}.

\begin{lemma}
\label{LEM:QC}
Let $1 < \al \leq \frac 32$, $s\leq \al - \frac 12$, and $N_1,N_2,N_3,N_4,N_5 \in 2^\N$ be dyadic numbers. 
Furthermore, we define three phase functions by 
\begin{align}
\kappa_2 (n_1,n_2,n_3) &\overset{\textup{def}}{=} 
 \eps_{123} \jbb{ n_{123} } + \eps_1 \jbb{ n_{1} } + \eps_2 \jbb{ n_{2} } +  \eps_3 \jbb{ n_{3} }, \label{defk2} \\ 
\kappa_3(n_1,\hdots,n_5)
 &\overset{\textup{def}}{=} \eps_{12345} \jbb{ n_{12345} } + \eps_{123} \jbb{ n_{123} } + \eps_4 \jbb{ n_4 } + \eps_5 \jbb{ n_5 }, \notag \\
\kappa_4 (n_1,\hdots,n_5) &\overset{\textup{def}}{=}  \eps_{12345} \jbb{ n_{12345} } + \sum_{j=1}^5 \eps_j \jbb{ n_j } .  \notag
\end{align} 

\noi
Then, there exists $\theta >0$ such that 
\begin{align*}
\sup_{\zeta ,\zeta' \in \R } & \sum_{n_1,\hdots,n_5\in \Z ^3} \bigg[ \Big( \prod_{j=1}^{5} \ind_{\{|n_j|\sim N_j\}} \Big) \langle n_{12345} \rangle^{2(s-\al)} 
\jbb{ n_{123} }^{-2}  \Big( \prod_{j=1}^5 \jbb{ n_j }^{-2} \Big) \\
& \times \ind_{\{ |\kappa_2 - \zeta|\leq 1 \}} \cdot \Big( \ind_{\{ |\kappa_3 - \zeta'|\leq 1\}} +  \ind_{\{ |\kappa_4 - \zeta'|\leq 1\}} \Big) \bigg] \les  N_{\textup{max}}^{- \theta},
\end{align*}

\noi
where $N_{\textup{max}} = \max ( N_1, \dots, N_5 )$.
\end{lemma}

Before proving Lemma \ref{LEM:QC},
let us recall the following frequency-scale estimate from \cite[Lemma A.1]{Bring2}.
\begin{lemma}
\label{LEM:FS}
Let $N_4,N_5,N_{1235},N_{12345} \ge 1$ and 
\begin{equation*}
\ind_{\{ |n_4|\sim N_4 \}} \cdot \ind_{\{ |n_5|\sim N_5 \}} \cdot \ind_{\{ |n_{1235}|\sim N_{1235}\}}\cdot \ind_{\{ |n_{12345}|\sim N_{12345} \}} \not \equiv 0. 
\end{equation*}
Then, it holds that
\begin{equation*}
 \frac{\min(N_4,N_{12345})^2  \min(N_5,N_{1235})}{\min(N_{12345},N_{1235},N_4) } \les  N_4 \cdot N_{12345}. 
\end{equation*}
\end{lemma}

Now we are ready to prove 
 Lemma \ref{LEM:QC}.

\begin{proof}[Proof of Lemma \ref{LEM:QC}:]
Let $\zeta, \zeta' \in \R$ be arbitrary. 
Let $N_{12345}$ and $N_{1235}$ 
be the dyadic size of $n_{12345}$ and $n_{1235}$ respectively. 
Using Lemma \ref{LEM:two_balls} for the sum in $n_4 \in \Z^3$ 
and then summing in $n_5 \in \Z^3$, 
we obtain that 
\begin{align*}
& \sum_{n_1,\hdots,n_5\in \Z^3} \bigg[ \Big( \prod_{j=1}^{5} 
\ind_{\{|n_j|\sim N_j\}} \Big) \ind_{\{ |n_{12345} | \sim N_{12345} \}} \ind_{\{ |n_{1235}| \sim N_{1235} \}} \langle n_{12345} \rangle^{2(s-\al)}  \\
&\hphantom{XX} \times 
\jbb{ n_{123} }^{-2} \Big( \prod_{j=1}^5 \jbb{ n_j }^{-2} \Big)  \ind_{\{ |\kappa_2 - \zeta|\leq 1\}} \cdot \Big( \ind_{\{ |\kappa_3 - \zeta'|\leq 1\}} +  \ind_{\{ |\kappa_4 - \zeta'|\leq 1\}} \Big)  \bigg] \allowdisplaybreaks[3] \\
&\les  N_{12345}^{2(s-\al)}  \min(N_{12345},N_{1235},N_4)^{-1} \min(N_4,N_{12345})^{3} N_4^{-2 \al} \\
&\hphantom{XX}\times \sum_{n_1,n_2,n_3,n_5\in \Z^3}  \Big( \prod_{j=1,2,3,5} \ind_{\{|n_j|\sim N_j\}} \jbb{n_j}^{-2} \Big) \ind_{\{ |n_{1235}| \sim N_{1235} \}} 
\jb{ n_{123} }^{-2}  \ind_{\{ |\kappa_2 - \zeta|\leq 1\}}\allowdisplaybreaks[3] \\
&\les   N_{12345}^{2(s-\al)}   \min(N_{12345},N_{1235},N_4)^{-1} \min(N_4,N_{12345})^{3}  \min(N_5,N_{1235})^3 \\
 &\hphantom{XX} \times  (N_4 N_5)^{-2 \al} \sum_{n_1,n_2,n_3 \in \Z^3}  \Big( \prod_{j=1}^3 \ind_{\{|n_j|\sim N_j\}} \jbb{n_j}^{-2} \Big) 
\jb{ n_{123} }^{-2} \ind_{\{ |\kappa_2-\zeta|\leq 1\}}. 
\end{align*}

\noi
Using Lemma \ref{LEM:CS} with $k = 3$ and $s = \al - 1$ to bound the remaining sum in $n_1$, $n_2$, and $n_3$ and applying Lemma \ref{LEM:FS}, we obtain a bound of the total contribution by
\begin{align*}
 &N_{12345}^{2(s-\al)}(N_4 N_5)^{-2 \al}  \frac{\min(N_4,N_{12345})^{3}  \min(N_5,N_{1235})^3}{\min(N_{12345},N_{1235},N_4) } 
 \max(N_1,N_2,N_3)^{-\theta} \\
&\les  N_{12345}^{2s - 2\al + 1} \min(N_4,N_{12345}) \min(N_5,N_{1235})^2 N_4^{1 - 2 \al} N_5^{-2\al} \max(N_1,N_2,N_3)^{-\theta}\\
&\les  N_{12345}^{2s - 2\al + 1} N_4^{2-2\al} N_5^{2 - 2\al} \max(N_1,N_2,N_3)^{-\theta}
\end{align*}
for some $\ta > 0$. Due to our condition on $s$, this is acceptable
by choosing $\theta_1 = \min \{2\al-2,  \theta\} - 2\eps$. We then sum over the dyadic $N_{12345}$ and $N_{1235}$ to obtain the desired estimate.
\end{proof}

We will also need the following basic resonance estimate similar to \cite[Lemma~4.25]{Bring2}. 

\begin{lemma}
\label{LEM:BR}
Let $1 < \al \leq 2$, $n_1,n_2 \in \Z ^3$ be arbitrary, $N_3 \in 2^\N$ be a dyadic number, and $\kappa = \kappa_2$ be as defined in \eqref{defk2}. Then, there exists $\theta >0$ such that 
 \begin{align*} 
  \sum_{n' \in \Z } \sum_{n_3 \in \Z ^3} \langle n' \rangle^{-1} \ind_{\{|n_{3}|\sim N_{3} \}} 
 \jbb{ n_{123} }^{-1} \jbb{ n_{3} }^{-2} \ind_{\{ |\kappa - n'| \leq 1 \}} \les  N_3^{-\theta} \langle n_{12} \rangle^{-\al}. 
 \end{align*}
\end{lemma}

\begin{proof}
Since $n_1,n_2 \in \Z^3$ are fixed, there are at most $\sim N_3^3$ non-trivial choices of $n' \in \Z$. By losing a log factor, 
it suffices to prove 
\begin{equation*}
\sup_{n' \in \Z} \sum_{n_3 \in \Z^3} \ind_{\{|n_{3}|\sim N_{3} \}} 
 \jbb{ n_{123} }^{-1} \jbb{ n_{3} }^{-2} \ind_{\{ |\kappa - n'| \leq 1 \}} \les  N_3^{- \theta' }  \langle n_{12} \rangle^{-\al},
\end{equation*}

\noi
for some $\theta' > 0$.
By inserting an additional dyadic localization and applying 
Lemma \ref{LEM:two_balls}, we obtain that 
\begin{equation}
\begin{aligned}
\sum_{n_3 \in \Z^3} &\ind_{\{|n_{3}|\sim N_{3} \}} 
 \jbb{ n_{123} }^{- 1} \jbb{ n_{3} }^{-2} \ind_{\{ |\kappa - n'| \leq 1 \}} \\
 &\leq N_3^{-2\al} \sum_{N_{123}\geq 1}  N_{123}^{-\al}  \sum_{n_3 \in \Z^3}  \ind_{\{ |n_{123}|\sim N_{123} \}} \ind_{\{|n_{3}|\sim N_{3} \}} 
\ind_{\{ |\kappa - n'| \leq 1 \}}\\
 &\les \sum_{N_{123}\geq 1}   N_{123}^{-\al} N_3^{-2\al} \min(N_{123},\jb{n_{12}},N_3)^{-1} \min(N_{123},N_3)^{3}. 
\end{aligned}
 \label{BR1}
\end{equation}

\noi
We now distinguish three cases.
If $N_{123} \ll N_3$, then $\jb{n_{12}} \sim N_3$. Thus, we can bound the right-hand-side of \eqref{BR1} by
\begin{equation*}
\begin{split}
 \sum_{N_{123} \ll N_3}  N_{123}^{2 - \al} N_3^{-2\al} \les  \jb{n_{12}}^{-\al} N_{3}^{2 - 2 \al}. 
 \end{split}
\end{equation*}

\noi
If $N_{123} \sim N_3$, then $\jb{n_{12}} \les  N_{123} \sim N_3 $. Thus, we can bound the right-hand-side of \eqref{BR1} by
\begin{equation*}
N_3^{3 - 3 \al} \jb{n_{12}}^{-1} \les  \jb{n_{12}}^{-\al} N_{3}^{2 - 2 \al}. 
\end{equation*}

\noi
Finally, if $N_{123} \gg N_3$, then $N_{123}\sim \jb{n_{12}} \gg N_3$. Thus, we can bound the right-hand-side of \eqref{BR1} by
\begin{equation*}
\begin{split}
\sum_{N_{123} \sim \jb{n_{12}}}  N_{123}^{-\al} N_3^{-2\al} \min(N_{123},\jb{n_{12}},N_3)^{-1} \min(N_{123},N_3)^{3} \les  \jb{n_{12}}^{-\al} N_3^{2 - 2\al}. 
 \end{split}
\end{equation*}

\noi
Thus, we complete the proof by taking $\theta < 2 \al - 2$. 
\end{proof}

We now consider the following double-resonance quintic sum estimate similar to \cite[Lemma~4.29]{Bring2}.

\begin{lemma}
\label{LEM:DRC}
Let $1 < \al \leq 2$, $N_1,N_2,N_3 \in 2^\N$ be dyadic numbers, and $\kappa = \kappa_2$ be as defined in \eqref{defk2}. Then, it holds that  
\begin{align*}
\begin{split}
\sup_{\zeta \in \R} &\sup_{|n_1|\sim N_1} \sum_{n_2,n_3\in \Z ^3} \bigg[ \Big( \prod_{j=2}^3 \ind_{\{ |n_j|\sim N_j \}} \Big) \jbb{ n_{123} }^{- 1} \jbb{ n_2 }^{-2} \jbb{ n_3 }^{-2}   \ind_{\{|\kappa - \zeta| \le 1 \}} \bigg] \\
&\les  N_2^{2 - 2\al + \eps} N_3^{2 - 2\al + \eps},  
\end{split}
\end{align*}
for arbitrarily small $\eps > 0$.
\end{lemma}

\begin{proof}
By losing a log factor,
we may further dyadically localize $|n_{123}| \sim N_{123}$ and $|n_{12}|\sim N_{12}$. 
We first sum in $n_3\in \Z^3$ using Lemma \ref{LEM:two_balls} and then sum in $n_2 \in \Z^3$ using the dyadic constraint to obtain
\begin{align*}
N_2^{-2\al} & N_3^{-2\al}  N_{123}^{-\al} \sum_{n_2,n_3\in \Z^3} \bigg[ \Big( \prod_{j=1}^3 \ind_{\{ |n_j|\sim N_j \}}\Big) \ind_{\{ |n_{123}|\sim N_{123}\}}  \ind_{\{ |n_{12}|\sim N_{12}\}}  \ind_{\{|\kappa - \zeta| \le 1 \}} \bigg] \\
&\les  N_2^{-2\al} N_3^{-2\al}  \min(N_{123},N_{12},N_3)^{-1} \min(N_3,N_{123})^{3} N_{123}^{-\al}  \\
&\quad \times  \sum_{n_2 \in \Z^3} \ind_{\{ |n_2|\sim N_2 \}} \ind_{\{ |n_{12}|\sim N_{12}\}} \\
&\les   N_2^{-2\al} N_3^{-2\al}  \min(N_{123},N_{12},N_3)^{-1} \min(N_3,N_{123})^{3} \min(N_{2},N_{12})^3  N_{123}^{-\al}.
\end{align*}
Using a minor variant of Lemma \ref{LEM:FS}, the above contribution is bounded by
\begin{align*}
N_2^{-2\al} N_3^{1 - 2\al} \min(N_3,N_{123}) \min(N_2,N_{12})^2 N_{123}^{1 - \al} \les N_2^{2 - 2\al} N_3^{2 - 2\al} N_{123}^{1 - \al}. 
\end{align*}

\noi
Thus we finish the proof.
\end{proof}

\subsection{Septic counting estimates}
We now show a septic counting estimate similar to \cite[Lemma~4.31]{Bring2}, which will be helpful for us to analyze the regularity of a septic term. Let us recall the notion of a pairing at the end of Subsection~\ref{SUBSEC:mul}. If $(n_j, n_k)$ with $j \neq k$ is a pairing, we say that the tuple of indices $(j, k)$ is a {\it paired index}. If $\mathcal P$ is a set of disjoint paired indices, we write $j \in \mathcal P$ if $j$ belongs to some tuple of paired index in $\mathcal P$ and $j \notin \mathcal P$ otherwise. We will the notation $\sum_{\{n_j\}_{j \in \mathcal P}}$, where only one independent frequency is summed per pair.

\begin{lemma}
\label{LEM:SepC}
Let $1 < \al \leq \frac 32$ and $N_1, \dots, N_7 \in 2^\N$ be dyadic numbers. 
Let $\kappa = \kappa_2$ be as in \eqref{defk2}.
Furthermore, we define
\begin{equation*}
\Ld (n_1,n_2,n_3) \deff \sum_{\eps_1,\eps_2,\eps_3 \in \{ \pm 1 \}} \sum_{m\in \Z } \langle m \rangle^{-1} \jbb{n_{123}}^{-1} \Big( \prod_{j=1}^3 \jbb{ n_j }^{- 1} \Big) \ind_{\{ |\kappa - m|\leq 1 \}}. 
\end{equation*}

\noi
We also let $\mathcal P$ be a set of disjoint paired indices in $\{1,\hdots,7\}$ such that we cannot have $(j, k) \in \mathcal P$ if $j$ and $k$ both belong to any one of the three sets: $\{1, 2, 3\}$, $\{4, 5, 6\}$, and $\{7\}$. Moreover, we define
\begin{align*}
n_{\textup{sum}} \deff \sum_{j \notin \mathcal P } n_j.
\end{align*}

\noi
Then, there exists $\ta > 0$ such that 
\begin{align*}
&\sum_{\{n_j\}_{j \notin \mathcal P }} \hspace{-1ex} \langle n_{\textup{sum}} \rangle^{-1}
\bigg( \sum_{\{n_j\}_{j \in \mathcal P}}    
 \Big( \prod_{j = 1}^7 \ind_{\{|n_j| \sim N_j\}} \Big) \Ld (n_1,n_2,n_3) \Ld (n_4,n_5,n_6)  \jbb{ n_7 }^{-1} \bigg)^2 \les N_{\textup{max}}^{-\ta},
\end{align*}

\noi
where $N_{\textup{max}} = \max (N_1, \dots, N_7)$.
\end{lemma}

\begin{proof}
Thanks to the open condition $\al > 1$, we can take out an $N_j^{-\ta}$ factor from each $\jbb{n_j}^{-1}$. This loss of $\ta$ regularity will not affect the convergence property, and so for simplicity, we still work with $\jbb{n_j}^{-1}$ below and we aim to show 
\begin{align*}
\sum_{\{n_j\}_{j \notin \mathcal P}}  \langle n_{\text{sum}} \rangle^{-1}
 \bigg(  \sum_{\{n_j\}_{j \in \mathcal P}} \Big( \prod_{j = 1}^7 \ind_{\{|n_j| \sim N_j\}} \Big) \ind_{\{|n_{7}|\sim N_{7} \}}  \Ld (n_1,n_2,n_3)  \Ld (n_4,n_5,n_6)  \jbb{ n_7 }^{-1} \bigg)^2  \les  1.
\end{align*}

\noi
Using Lemma~\ref{LEM:CS} with $k = 3$ and $s = \al - 1$, we have 
\begin{equation}\label{p1}
\sum_{n_1,n_2,n_3 \in \Z^3}  \Ld(n_1,n_2,n_3)^2 \les  \max (N_{1}, N_2, N_3)^{- \theta'} 
\end{equation}

\noi
for some $\theta' > 0$.
Using the basic resonance estimate (Lemma~\ref{LEM:BR}), we have for all $N_3 \geq 2^\N$ that 
\begin{equation}\label{p2}
\sum_{n_3\in \Z^3} \ind_{\{|n_{3}|\sim N_{3} \}} \jbb{ n_3 }^{-1} \Ld (n_1,n_2,n_3) \les  N_3^{-\theta'} 
\langle n_{12} \rangle^{-\al} \langle n_1\rangle^{-\al} \langle n_2 \rangle^{-\al}. 
\end{equation}

\noi
for some $\theta' > 0$.
Due to the symmetry of $\Ld$, 
we only need to consider the following two cases.

\smallskip \noi
\textbf{Case 1:} $7 \notin \mathcal P$.  

We further dyadically localize $|n_{123456}| \sim N_{123456}$. By using the Cauchy-Schwarz inequality, summing in $n_7$, and using \eqref{p1}, we obtain
\begin{align*}
\sum_{\{n_j\}_{j \notin \mathcal P}} &\langle n_{\text{sum}} \rangle^{-1}
 \bigg(  \sum_{\{n_j\}_{j \in \mathcal P}} \Big( \prod_{j = 1}^7 \ind_{\{|n_j| \sim N_j\}} \Big) \ind_{\{|n_{123456}|\sim N_{123456} \}} \\
 &\quad \times \Ld (n_1,n_2,n_3) \Ld (n_4,n_5,n_6)  \jbb{ n_7 }^{-1} \bigg)^2 \\
 &\les  \sum_{\{n_j\}_{j \notin \mathcal P}} \bigg[ \ind_{\{ |n_{\text{sum}} - n_7|\sim N_{123456} \}} \langle n_{\text{sum}} \rangle^{-1}  \jbb{ n_7 }^{-2}
 \Big( \sum_{\{n_j\}_{j \in \mathcal P \wedge j \in \{1, 2, 3\}}}   \Ld (n_1,n_2,n_3)^2 \Big) \\
 &\quad \times  \Big( \sum_{\{n_j\}_{j \in \mathcal P \wedge j \in \{4, 5, 6\}}} \Ld (n_4,n_5,n_6)^2 \Big) \bigg]  \\
 &\les  N_{123456}^{2 - 2\al + \eps}  \sum_{(n_j)_{ \substack{j \notin \mathcal P \wedge j \neq 7}}} 
  \Big( \sum_{\{n_j\}_{j \in \mathcal P \wedge j \in \{1, 2, 3\}}}   \Ld (n_1,n_2,n_3)^2 \Big)  \\ 
  &\quad \times \Big( \sum_{\{n_j\}_{j \in \mathcal P \wedge j \in \{4, 5, 6\}}} \Ld (n_4,n_5,n_6)^2 \Big)  \\
  &= N_{123456}^{2 - 2\al + \eps}  \Big( \sum_{n_1,n_2,n_3\in \Z^3}   \Ld (n_1,n_2,n_3)^2 \Big)     
  \Big( \sum_{n_4,n_5,n_6\in \Z^3} \Ld (n_4,n_5,n_6)^2 \Big)  \\
  &\les   N_{123456}^{2 - 2\al + \eps},
 \end{align*}
for $\eps > 0$ arbitrarily small. This is acceptable given $\al > 1$.
 
\smallskip \noi
\textbf{Case 2:} $7 \in \mathcal P$.

We further dyadically localize $|n_{1234567}| \sim N_{1234567}$. By symmetry, 
 we may assume that $(3,7)\in \mathcal P$.  
 We let $\mathcal{P}'$ be the pairing on $\{1,2,4,5,6\}$ obtained by removing the pair $(3,7)$ from $\mathcal P$. We also view the condition $j \notin \mathcal{P}'$ as a subset of $\{1,2,4,5,6\}$. By first using \eqref{p2} and then the Cauchy-Schwarz inequality, we have
 \begin{align*}
\sum_{\{n_j\}_{j \notin \mathcal P}} &\langle n_{\text{sum}} \rangle^{-1}
 \bigg(  \sum_{(n_j)_{j \in \mathcal P}}  \Big( \prod_{j = 1}^7 \ind_{\{|n_j| \sim N_j\}} \Big) \ind_{\{|n_{1234567}|\sim N_{1234567} \}} \\
 &\quad \times  \Ld (n_1,n_2,n_3) \Ld (n_4,n_5,n_6)  \jbb{ n_7 }^{-1} \bigg)^2 \\
 &\les  N_7^{- 2 \theta} N_{1234567}^{-1} \sum_{\{n_j\}_{j \notin \mathcal{P}'}}   \Big( \sum_{\{n_j\}_{j \in \mathcal{P}'}}  \langle n_{12} \rangle^{-\al} \langle n_1 \rangle^{-\al} \langle n_2 \rangle^{- \al} \Ld (n_4,n_5,n_6) \Big)^2 \\
 &\les  N_7^{- 2 \theta} 
 N_{1234567}^{-1}  \sum_{\{n_j\}_{j\notin \mathcal{P}'}}  
 \Big(  \sum_{\{n_j\}_{j \in \mathcal{P}' \wedge j \in \{1, 2\}}}  \langle n_{12} \rangle^{-2\al} \langle n_1 \rangle^{-2\al} \langle n_2 \rangle^{-2\al} \Big) \\
 &\quad \times \Big(  \sum_{\{n_j\}_{j \in \mathcal{P}' \wedge j \in \{4, 5, 6\}}} \Ld (n_4,n_5,n_6)^2 \Big) ,
 \end{align*}
 
 \noi
which is sufficient thanks to \eqref{p1}
and the fact that
\[
 \sum_{n_1,n_2 \in \Z^3}  \langle n_{12} \rangle^{-2\al} \langle n_1 \rangle^{-2\al} \langle n_2 \rangle^{-2\al} 
 < \infty
\]

\noi
using the discrete convolution inequality in Lemma~\ref{LEM:SUM}~(i) along with $\al > 1$. Thus, we finish the proof. 
\end{proof}

\subsection{Deterministic tensor estimate}
In this subsection, we show the following deterministic tensor estimate, which will be useful in dealing with random operators both for our global well-posedness result and for our weak universality result.

We first recall some notations from \cite{DNY3}. For a finite index set $A$, we denote $n_A$ as the tuple $(n_j)_{j \in A}$. We also denote $\| \cdot \|_{n_A \to n_B}$ as the operator norm $\| \cdot \|_{\l^2_{n_A} \to \l^2_{n_B}}$ for any finite index sets $A$ and $B$.

The following lemma generalizes \cite[Lemma~4.33]{Bring2}.

\begin{lemma}
\label{LEM:ten1}
Let $k \geq 3$ be an integer, $\eps' > 0$ be sufficiently small, and $\zeta \in \R$. Let $1 < \al < \frac 32$ and $\eps > 0$ be arbitrarily small if $k > 3$; and let $1 < \al \leq \frac 32$ and $\eps = 0$ if $k = 3$. Let $N, N_1, \dots, N_k \geq 1$ be dyadic numbers such that $N_j \les N$ for all $j$ and let $\kappa$ be as defined in \eqref{kappa}. We also let $s \in \R$ be such that $\max(- \al + \frac 32, \al - 1) < s < 2\al - \frac 32$.
Then, the tensor $\mathfrak{h}^{\zeta}$ defined by
\begin{align*}
\mathfrak{h}^{\zeta}_{n n_1 \cdots n_k} \overset{\textup{def}}{=} \ind_{\{n = \sum_{j = 1}^k n_j\}} \bigg( \prod_{j = 1}^k \ind_{\{|n_j| \sim N_j\}} \prod_{j = 1}^{k - 1} \ind_{\{|n_j|_\infty \leq N\}} \bigg) \cdot \ind_{\{|\kappa - \zeta| \leq 1\}} \frac{\jb{n}^{s - \al + \eps'}}{\prod_{j = 1}^{k - 1} \jbb{n_j} \cdot \jb{n_k}^{s}}
\end{align*}

\noi
satisfies
\begin{align*}
\max_{(B_1, B_2)} \big\{ \| \mathfrak{h}^{\zeta} \|_{n_k n_{B_1} \to n n_{B_2}} \big\} 
\les N^{(k - 3) (\frac 32 - \al) - \eps} \big( \max ( N_1, \dots, N_k ) \big)^{- \ta}.
\end{align*}

\noi
for some $\ta > 0$, where $(B_1, B_2)$ runs through all partitions of $\{1, \dots, k - 1\}$.
\end{lemma}

\begin{proof}
By losing a factor of $\log (1 + \max(N_1, \dots, N_k))$, we further dyadically localize $|n| \sim N_{\textup{sum}}$.

We first consider the case when $B_2 = \varnothing$. By Schur's test, we obtain
\begin{align*}
\begin{split}
\| \mathfrak{h}^{\zeta} \|_{n_1 \cdots n_k \to n}^2 &\leq N_{\textup{sum}}^{2s - 2 \al + 2 \eps'} N_k^{-2s} \prod_{j = 1}^{k - 1} N_j^{-2 \al}  \\
&\quad \hspace{-20pt} \times \sup_{n \in \Z^3} \sum_{n_1, \dots, n_k \in \Z^3} \bigg( \prod_{j = 1}^k \ind_{\{ |n_j| \sim N_j \}} \bigg) \ind_{\{n = \sum_{j = 1}^k n_j\}} \ind_{\{|n| \sim N_{\textup{sum}}\}} \ind_{\{|\kappa - \zeta| \leq 1\}} \\
&\quad \hspace{-20pt} \times \sup_{n_1, \dots, n_k \in \Z^3} \sum_{n \in \Z^3} \bigg( \prod_{j = 1}^k \ind_{\{ |n_j| \sim N_j \}} \bigg) \ind_{\{n = \sum_{j = 1}^k n_j\}} \ind_{\{|n| \sim N_{\textup{sum}}\}} \ind_{\{|\kappa - \zeta| \leq 1\}}
\end{split}
\end{align*}

\noi
Note that the second supremum is bounded by 1, since $n$ is uniquely determined by $n_1, \dots, n_k$. To deal with the first supremum, we apply Lemma \ref{tools:lem_sup} to obtain
\begin{align*}
\| \mathfrak{h}^{\zeta} \|_{n_1 \cdots n_k \to n}^2 \les N_{\textup{sum}}^{2s - 2\al + 2 \eps'} N_k^{3 - 2s} \big( \max ( N_1, \dots, N_k ) \big)^{-3} \big( \max_{(3)} ( N_1, \dots, N_k ) \big)^{-1} \prod_{j = 1}^{k - 1} N_j^{3 - 2 \al}.
\end{align*}

\noi
When $N_k = \max_{(3)} ( N_1, \dots, N_k )$, we have (recalling that $s < 2 \al - \frac 32 \leq \al$)
\begin{align*}
\| \mathfrak{h}^{\zeta} \|_{n_1 \cdots n_k \to n}^2 &\les N_k^{2 - 2s} \big( \max ( N_1, \dots, N_k ) \big)^{-3} \prod_{j = 1}^{k - 1} N_j^{3 - 2 \al} \\
&\les N^{(k - 3) (3 - 2 \al) - 2 \eps} N_k^{2 - 2s} \big( \max ( N_1, \dots, N_k ) \big)^{3 - 4 \al + 2 \eps}.
\end{align*}

\noi
If $s \geq 1$, then the above bound is acceptable given $\al > 1 > \frac 34$. If $s < 1$, then the above bound is also acceptable since $(2 - 2s) + (3 - 4 \al + 2 \eps) < 0$ given $s > -\al + \frac 32 > -2 \al + \frac 52$. When $N_k \neq \max_{(3)} ( N_1, \dots, N_k )$, say without loss of generality that $N_{k - 1} = \max_{(3)} ( N_1, \dots, N_k )$, we have (recalling that $\al > 1$ and $s < 2 \al - \frac 32 \leq \al \leq \frac 32$)
\begin{align*}
\| \mathfrak{h}^{\zeta} \|_{n_1 \cdots n_k \to n}^2 &\les N_k^{3 - 2s} \big( \max ( N_1, \dots, N_k ) \big)^{-3} \prod_{j = 1}^{k - 2} N_j^{3 - 2 \al} \\
&\les N^{(k - 3) (3 - 2 \al) - 2 \eps} \big( \max ( N_1, \dots, N_k ) \big)^{3 - 2 s - 2 \al + 2 \eps},
\end{align*}

\noi
which is acceptable given $s > -\al + \frac 32$.

We now consider the case when $B_1 = \varnothing$. Using a similar method with Schur's test and Lemma \ref{tools:lem_sup}, we obtain
\begin{align*}
\| &\mathfrak{h}^{\zeta} \|_{n_k \to n n_1 \cdots n_{k - 1}}^2 \leq N_{\textup{sum}}^{2s - 2 \al + 2 \eps'} N_k^{-2s} \prod_{j = 1}^{k - 1} N_j^{- 2 \al} \\
& \quad \times \sup_{n_k \in \Z^3} \sum_{n, n_1, \dots, n_{k - 1} \in \Z^3} \bigg( \prod_{j = 1}^k \ind_{\{|n_j| \sim N_j\}} \bigg) \ind_{\{n = \sum_{j = 1}^k n_j\}} \ind_{\{|n| \sim N_{\textup{sum}}\}} \ind_{\{|\kappa - \zeta| \leq 1\}} \\
& \quad \times \sup_{n, n_1, \dots, n_{k - 1} \in \Z^3} \sum_{n_k \in \Z^3} \bigg( \prod_{j = 1}^k \ind_{\{|n_j| \sim N_j\}} \bigg) \ind_{\{n = \sum_{j = 1}^k n_j\}} \ind_{\{|n| \sim N_{\textup{sum}}\}} \ind_{\{|\kappa - \zeta| \leq 1\}} \\
&\les N_{\textup{sum}}^{3 + 2s - 2\al + 2 \eps'} N_k^{-2s} \big( \max ( N_{\textup{sum}}, N_1, \dots, N_{k - 1} ) \big)^{-3} \\
&\quad \times \big( \max_{(3)} ( N_{\textup{sum}}, N_1, \dots, N_{k - 1} ) \big)^{-1} \prod_{j = 1}^{k - 1} N_j^{3-2\al}
\end{align*}

\noi
When $N_{\text{sum}} = \max_{(3)} ( N_\text{sum}, N_1, \dots, N_{k-1} )$, we have (recalling that $s > \al - 1 > 0$)
\begin{align*}
\| \mathfrak{h}^{\zeta} \|_{n_k \to n n_1 \cdots n_{k - 1}}^2 &\les N_{\text{sum}}^{2 + 2s - 2 \al + 2 \eps'} \big( \max ( N_{\textup{sum}}, N_1, \dots, N_{k - 1} ) \big)^{-3} \prod_{j = 1}^{k - 1} N_j^{3-2\al} \\
&\les N^{(k - 3) (3 - 2 \al) - 2 \eps} \big( \max ( N_{\textup{sum}}, N_1, \dots, N_{k - 1} ) \big)^{5 + 2s - 6\al + 2 \eps' + 2 \eps},
\end{align*}
which is acceptable given $s < 2\al - \frac 32 < 3 \al - \frac 52$. When $N_{\text{sum}} \neq \max_{(3)} ( N_\text{sum}, N_1, \dots, N_k )$, say without loss of generality that $N_{k-1} = \max_{(3)} ( N_\text{sum}, N_1, \dots, N_{k-1} )$, we have (recalling that $\al > 1$)
\begin{align*}
\| \mathfrak{h}^{\zeta} \|_{n_k \to n n_1 \cdots n_{k - 1}}^2 &\les N_{\text{sum}}^{3 + 2s - 2 \al + 2 \eps'} \big( \max ( N_{\textup{sum}}, N_1, \dots, N_{k - 1} ) \big)^{-3} \prod_{j = 1}^{k - 2} N_j^{3-2\al} \\
&\les N^{(k - 3) (3 - 2 \al) - 2 \eps} \big( \max ( N_{\textup{sum}}, N_1, \dots, N_{k - 1} ) \big)^{3 + 2s - 4\al + 2 \eps' + 2 \eps},
\end{align*}

\noi
which is acceptable given $s < 2\al - \frac 32$.

For the case when $B_1 \neq \varnothing$ and $B_2 \neq \varnothing$, by Schur's test, we obtain
\begin{align*}
\| \mathfrak{h}^{\zeta} \|_{n_k n_{B_1} \to n n_{B_2}}^2 &\leq N_{\textup{sum}}^{2s - 2\al + 2 \eps'} N_k^{-2 s} \prod_{j = 1}^{k - 1} N_j^{- 2 \al} \\
&\quad \times \sup_{n_k, n_{B_1} } \sum_{n, n_{B_2}} \bigg( \prod_{j = 1}^k \ind_{\{|n_j| \sim N_j\}} \bigg) \ind_{\{n = \sum_{j = 1}^k n_j\}} \ind_{\{|n| \sim N_{\textup{sum}}\}} \ind_{\{|\kappa - \zeta| \leq 1\}} \\
&\quad \times \sup_{n, n_{B_2}} \sum_{n_k, n_{B_1}} \bigg( \prod_{j = 1}^k \ind_{\{|n_j| \sim N_j\}} \bigg) \ind_{\{n = \sum_{j = 1}^k n_j\}} \ind_{\{|n| \sim N_{\textup{sum}}\}} \ind_{\{|\kappa - \zeta| \leq 1\}}.
\end{align*}

\noi
If $k = 3$, by symmetry we can assume that $B_1 = \{1\}$ and $B_2 = \{2\}$, so that
\begin{align*}
\| \mathfrak{h}^{\zeta} \|_{n_k n_{B_1} \to n n_{B_2}}^2 &\les N_{\textup{sum}}^{2s - 2\al + 2 \eps'} N_1^{-2\al} N_2^{-2\al} N_3^{-2s} \min (N_{\textup{sum}}, N_2)^3 \min (N_1, N_3)^3 \\
&\les \big( \max ( N_1, N_2, N_3 ) \big)^{- 2 \ta}
\end{align*}

\noi
for some $\ta > 0$ as long as $2s - 4\al + 3 < 0$ and $-2s - 2\al + 3 < 0$, which is equivalent to $-\al + \frac 32 < s < 2\al - \frac 32$. If $k > 3$, we consider the following subcases.

\smallskip \noi
\textbf{Subcase 1:} $|B_1| \geq 2$ and $\max ( N_1, \dots, N_k ) = N_{k'}$ for some $k' \in B_2$.

In this subcase, we apply Lemma \ref{tools:lem_sup} to the summation $\sum_{n_k, n_{B_1}}$ (recalling that $s > -\al + \frac 32 \geq 0$ and $\al > 1$) and perform the summation $\sum_{n, n_{B_2}}$ directly (without summing $n_{k'}$) to obtain
\begin{align*}
\| \mathfrak{h}^{\zeta} \|_{n_k n_{B_1} \to n n_{B_2}}^2 &\les N_{\textup{sum}}^{3 + 2s - 2\al + 2 \eps'} N_k^{-2s} \big( \max ( N_1, \dots, N_k ) \big)^{-3} N_{j'}^{-1} \prod_{j = 1}^{k - 1} N_j^{3 - 2 \al} \\
&\les N^{(k - 3) (3 - 2 \al) - 2 \eps} \big( \max ( N_1, \dots, N_k ) \big)^{2s - 4 \al + 3 + 2 \eps' + 2 \eps}
\end{align*}

\noi
for some $1 \leq j' \leq k - 1$. This estimate is acceptable given $s < 2\al - \frac 32$.

\smallskip \noi
\textbf{Subcase 2:} $|B_1| \geq 2$ and $\max ( N_1, \dots, N_k ) = N_{k'}$ for some $k' \in B_1 \cup \{k\}$.

In this subcase, we again apply Lemma \ref{tools:lem_sup} to the summation $\sum_{n_k, n_{B_1}}$ to obtain
\begin{align*}
\| \mathfrak{h}^{\zeta} \|_{n_k n_{B_1} \to n n_{B_2}}^2 \les N_{\textup{sum}}^{2s - 2\al + 2 \eps'} N_k^{3 - 2s} \big( \max ( N_1, \dots, N_k ) \big)^{-3} N_{j'}^{-1} \prod_{j = 1}^{k - 1} N_j^{3 - 2\al}
\end{align*}

\noi
for some $1 \leq j' \leq k$. If $1 \leq j' \leq k-1$, we obtain (recalling that $s < 2\al - \frac 32 \leq \al \leq \frac 32$ and $\al > 1$)
\begin{align*}
\| \mathfrak{h}^{\zeta} \|_{n_k n_{B_1} \to n n_{B_2}}^2 \les N^{(k - 3) (3 - 2 \al) - 2 \eps} \big( \max ( N_1, \dots, N_k ) \big)^{- 2s - 2\al + 3 + 2 \eps},
\end{align*}
which is acceptable given $s > -\al + \frac 32$. If $j' = k$, we obtain 
\begin{align*}
\| \mathfrak{h}^{\zeta} \|_{n_k n_{B_1} \to n n_{B_2}}^2 \les N^{(k - 3) (3 - 2 \al) - 2 \eps} N_k^{2 - 2s} \big( \max ( N_1, \dots, N_k ) \big)^{3 - 4\al + 2 \eps}.
\end{align*}

\noi
If $s \geq 1$, the above bound is valid given $\al > 1 > \frac 34$. If $s < 1$, the above bound is also valid since $-2s - 4\al + 5 < 0$ given $s > -\al + \frac 32 > -2\al + \frac 52$.

\smallskip \noi
\textbf{Subcase 3:} $|B_2| \geq 2$ and $\max ( N_1, \dots, N_k ) = N_{k'}$ for some $k' \in B_1 \cup \{k\}$.

In this subcase, we apply Lemma \ref{tools:lem_sup} to the summation $\sum_{n, n_{B_2}}$ (recalling that $s < 2\al - \frac 32 < \al \leq \frac 32$ and $\al > 1$) and perform the summation $\sum_{n_k, n_{B_1}}$ directly (without summing $n_{k'}$) to obtain
\begin{align*}
\| \mathfrak{h}^{\zeta} \|_{n_k n_{B_1} \to n n_{B_2}}^2 &\les N_{\textup{sum}}^{2s - 2\al + 2 \eps'} N_k^{3 - 2s} \big( \max ( N_1, \dots, N_k ) \big)^{-3} N_{j'}^{-1} \prod_{j = 1}^{k - 1} N_j^{3 - 2\al} \\
&\les N^{(k - 3) (3 - 2 \al) - 2 \eps} \big( \max ( N_1, \dots, N_k ) \big)^{-2s - 2\al + 3 + 2 \eps},
\end{align*}

\noi
where $1 \leq j' \leq k - 1$. This is acceptable given $s > -\al + \frac 32$.

\smallskip \noi
\textbf{Subcase 4:} $|B_2| \geq 2$ and $\max ( N_1, \dots, N_k ) = N_{k'}$ for some $k' \in B_2$.

In this subcase, we again apply Lemma \ref{tools:lem_sup} to the summation $\sum_{n, n_{B_2}}$ to obtain
\begin{align*}
\| \mathfrak{h}^{\zeta} \|_{n_k n_{B_1} \to n n_{B_2}}^2 &\les N_{\textup{sum}}^{3 + 2s - 2\al + 2 \eps'} N_k^{- 2s} \big( \max ( N_1, \dots, N_k ) \big)^{-3} N_{j'}^{-1} \prod_{j = 1}^{k - 1} N_j^{3 - 2\al}
\end{align*}

\noi
for some $0 \leq j' \leq k - 1$, where we denote $N_0 = N_{\textup{sum}}$ for convenience. If $1 \leq j' \leq k - 1$, we obtain (recalling that $s > -\al + \frac 32 \geq 0$ and $\al > 1$)
\begin{align*}
\| \mathfrak{h}^{\zeta} \|_{n_k n_{B_1} \to n n_{B_2}}^2 &\les N^{(k - 3) (3 - 2 \al) - 2 \eps} \big( \max ( N_1, \dots, N_k ) \big)^{2s - 4\al + 3 + 2 \eps' + 2 \eps},
\end{align*}

\noi
which is acceptable given $s < 2\al - \frac 32$. If $j' = 0$, we obtain (recalling that $s > \al - 1$)
\begin{align*}
\| \mathfrak{h}^{\zeta} \|_{n_k n_{B_1} \to n n_{B_2}}^2 &\les N^{(k - 3) (3 - 2 \al) - 2 \eps} \big( \max ( N_1, \dots, N_k ) \big)^{2s - 6 \al + 5 + 2 \eps' + 2 \eps},
\end{align*}

\noi
which is acceptable given $s < 2\al - \frac 32 < 3\al - \frac 52$.

\end{proof}

\section{On Gibbs measures}
\label{SEC:meas}
In this section, we discuss convergence of Gibbs measures for our weak universality result and also construct the Gibbs measure stated in Theorem~\ref{THM:Gibbs}.

The main tool that we are going to use in this section is the Bou\'e-Dupuis variational formula in \cite{BD}; see also \cite[Theorem~7]{Ust} and \cite[Theorem~2]{BG}.  
For this purpose, we let $W(t)$ be the cylindrical Wiener process on $L^2 (\T^3)$ with respect to an underlying probability measure $\PP$. We can write $W(t)$ as
\begin{align*}
W(t) = \frac{1}{(2 \pi)^{\frac 32}} \sum_{n \in \Z^3} B_n (t) e^{in \cdot x} ,
\end{align*}

\noi
where $\{B_n\}_{n \in \Z^3}$ is a family of i.i.d complex-valued Brownian motions conditioned such that $B_n = \cj{B_{-n}}$ for each $n \in \Z^3$ (with $B_0$ being real-valued). 
We then define a centered Gaussian process $Y(t)$ by
\begin{align}
Y(t) = \jb{\nb}^{-\al} W (t) = \frac{1}{(2 \pi)^{\frac 32}} \sum_{n \in \Z^3} \frac{B_n (t)}{\jb{n}^\al} e^{in \cdot x} .
\label{defY}
\end{align}

\noi
We also write $Y = Y(1)$, which is a random variable distributed according to the base Gaussian measure $\mu_\al$ in \eqref{gauss2}.
Given $\s \in \R$, we define $\mathbb{H}_a^\s$ be the space of progressively measurable processes that $\PP$-almost surely belong to $L^2 ([0, 1]; H^\s (\T^3))$. 
\begin{lemma}
\label{LEM:BD0}
Let $\al \geq 1$ and $N \in \N$. Let $Y = Y(1)$ be defined in \eqref{defY}. Let $F: C^\infty (\T^3) \to \R$ be a measurable function such that $\E [|F (\pi_N Y)|^p] < \infty$ and $\E [|e^{F (\pi_N Y (1))}|^q] < \infty$ for some $1 < p, q < \infty$ with $\frac 1p + \frac 1q = 1$. Then, we have
\begin{align*}
\log \E \big[ e^{F (\pi_N Y)} \big] = \sup_{\ta \in \mathbb{H}_a^0} \E \bigg[ F (\pi_N Y + \pi_N I_\al (\ta)) - \frac 12 \int_0^1 \| \ta (t) \|_{L_x^2}^2 dt \bigg] ,
\end{align*}

\noi
where $I_\al (\ta) = I_\al (\ta) (1)$ is defined by
\begin{align}
I_\al (\ta) (t) = \int_0^t \jb{\nb}^{- \al} \ta (t') dt' .
\label{defIa}
\end{align}
\end{lemma}

There is a simplified version of the variational formula in Lemma~\ref{LEM:BD0}; see \cite[Lemma~2.6]{FT} for a proof.
\begin{lemma}
\label{LEM:BD}
Let $\al \geq 1$ and $N \in \N$. Let $Y = Y(1)$ be defined in \eqref{defY}. Let $F : C^\infty (\T^3) \to \R$ be a measurable function such that $\E [ |F_- (\pi_N Y)|^p ] < \infty$ for some $p > 1$, where $F_-$ denotes the negative part of $F$. Then, we have
\begin{align*}
\log \E \big[ e^{F (\pi_N Y)} \big] \leq \E \bigg[ \sup_{\Dr \in H^\al (\T^3)} \Big\{ F (\pi_N Y + \pi_N \Dr) - \frac 12 \| \Dr \|_{H^\al}^2 \Big\} \bigg].
\end{align*}
\end{lemma}

Let us also record a lemma on regularities of Wick orderings for $Y$ defined in \eqref{defY}. Given $N \in \N$, we define $Y_N = \pi_N Y$. We define Wick ordering $:\! (Y_N)^k \!:$ as in \eqref{HkX}. 

\begin{lemma}
\label{LEM:YNk}
Let $\al \in \R$, $s > 0$, $\be \geq 0$, and $k \in \N$ satisfying
\begin{align*}
s \wedge \frac 32 > k \Big( \frac 32 - \al \Big) - \be > 0.
\end{align*}

\noi
Let $Y$ be defined in \eqref{defY}. Then, for any $1 \leq p \leq \infty$ and $1 \leq q < \infty$, we have
\begin{align*}
\sup_{N \geq 1} \Big\{ N^{-q \be} \E \big[ \| :\! Y_N^k \!: \|_{W^{-s, p}}^q \big] \Big\} < \infty.
\end{align*}
\end{lemma}
\begin{proof}
The proof of the bound follows directly from the Sobolev embedding (for $p = \infty$), the Gaussian hypercontractivity, the Ito isometry \eqref{ito}, and the discrete convolution inequality in Lemma~\ref{LEM:SUM}~(ii). See, for example, \cite[Lemma~B.3]{STzX}.
\end{proof}

\subsection{Convergence of Gibbs measures for the macroscopic model}
\label{SUBSEC:Gibbs1}

In Subsection~\ref{SUBSEC:wu}, we mentioned that our proof for the weak universality result relies crucially on the convergence of Gibbs measures for the macroscopic model. Let us now present the details.

Recalling the potential $V_N$ in \eqref{defVN} (with the $j = 0$ term omitted), we define the potential energy
\begin{align}
\mathcal{R}_N (u) \deff \int_{\T^3} V_N (\pi_N u) dx.
\label{defRN2}
\end{align}

\noi
We define the frequency truncated Gibbs measure $\nu_N$ by
\begin{align}
d \nu_N (u) \deff \mathcal{Z}_N^{-1} e^{- \mathcal{R}_N (u)} d \mu_\al (u),
\label{Gibbs_nuN}
\end{align}

\noi
where $\mu_\al$ is the base Gaussian measure as in \eqref{gauss2} and the normalizing factor $\mathcal{Z}_N$ is given by 
\begin{align}
\mathcal{Z}_N = \int e^{- \mathcal{R}_N (u)} d \mu_\al (u).
\label{ZN}
\end{align}

\noi
Our goal in this subsection is to prove the convergence of $\{\nu_N\}_{N \in \N}$ as $N \to \infty$.

\begin{proposition}[Convergence of Gibbs measures for the macroscopic model]
\label{PROP:Gibbs2}

Let $\frac 98 < \al < \frac 32$. Assume the criticality condition $\cj{a}_1 = 0$ and the positivity condition \eqref{positive}. Then, for any $1 \leq p < \infty$,  we have the uniform bound
\begin{align}
\sup_{N \in \N} \big\| e^{- \mathcal{R}_N (u)} \big\|_{L^p (\mu_\al)} \leq C_p < \infty.
\label{Lp_bdd2}
\end{align}

\noi
Moreover, there exists a limiting functional $\mathcal{R} (u)$ such that
\begin{align}
\lim_{N \to \infty} e^{- \mathcal{R}_N (u)} = e^{- \mathcal{R}(u)} \qquad \text{in } L^p (\mu_\al),
\label{Lp_conv2}
\end{align}

\noi
where we denote
\begin{align}
\mathcal{R} (u) = \frac{\kappa}{2} \int_{\T^3}  :\! u^2 \!: dx +  \cj{a}_2 \int_{\T^3} :\! u^4 \!: dx.
\label{defR2}
\end{align}

\noi
Consequently, the sequence of truncated Gibbs measures $\{ \nu_N \}_{N \in \N}$ in \eqref{Gibbs_nuN} converges in total variation to a limiting Gibbs measure $\nu$ denoted by
\begin{align}
d \nu (u) = \mathcal{Z}^{-1} e^{-\mathcal{R}(u)} d\mu_\al (u),
\label{Gibbs_nu}
\end{align}

\noi
The Gibbs measure $\nu$ is mutually absolutely continuous with respect to the base Gaussian measure $\mu_\al$.
\end{proposition}

We first prove the following lemma.

\begin{lemma}
\label{LEM:RNconv}
Let $\frac 98 < \al < \frac 32$. Given $N \in \N$, let $\mathcal{R}_N$ be defined in \eqref{defRN2}. Then, for any $1 \leq p < \infty$ and $N \in \N$, there exists a limiting functional $\mathcal{R} = \mathcal{R} (u)$ denoted as in \eqref{defR2} such that
\begin{align}
\| \mathcal{R}_N (u) - \mathcal{R} (u) \|_{L^p (\mu_\al)} \les N^{- \ta}
\label{RNconv}
\end{align} 

\noi
for some $\ta > 0$.
\end{lemma}

\begin{proof}
We only consider the case when $p = 2$, which implies the convergence for $1 \leq p < 2$ due to H\"older's inequality and the convergence for $p > 2$ due to the Gaussian hypercontractivity. 

We recall from \eqref{defRN2} and \eqref{defVN} that
\begin{align*}
\mathcal{R}_N (u) = \sum_{j = 1}^m \cj{a}_{j, N} N^{- (2 j - 4) (\frac 32 - \al)} \int_{\T^3} :\! (\pi_N u)^{2 j} \!: dx.
\end{align*}

\noi
For the $j = 1$ and the $j = 2$ terms, we can compute that for any $N_1 \geq N_2 \geq 1$, by the expression \eqref{gauss3}, the Ito isometry \eqref{itoX}, and the discrete convolution inequality in Lemma~\ref{LEM:SUM},
\begin{align}
\begin{split}
\bigg\| &\int_{\T^3} :\! (\pi_{N_1} u)^{2} \!:  dx - \int_{\T^3} :\! (\pi_{N_2} u)^{2} \!:  dx  \bigg\|_{L^2 (\mu_\al)}^2 \\
&\sim \sum_{\substack{n_1 + n_2 = 0 \\ N_2 < \max (|n_1|_\infty, |n_2|_\infty) \leq N_1}} \frac{1}{\jb{n_1}^{2 \al} \jb{n_2}^{2 \al}} \les N_2^{- \ta}
\end{split}
\label{RL0}
\end{align}

\noi
and
\begin{align}
\begin{split}
\bigg\| &\int_{\T^3} :\! (\pi_{N_1} u)^{4} \!:  dx - \int_{\T^3} :\! (\pi_{N_2} u)^{4} \!:  dx  \bigg\|_{L^2 (\mu_\al)}^2 \\
&\sim \sum_{\substack{n_1 + n_2 + n_3 + n_4 = 0 \\ N_2 < \max (|n_1|_\infty, |n_2|_\infty, |n_3|_\infty, |n_4|_\infty) \leq N_1}} \frac{1}{\jb{n_1}^{2 \al} \jb{n_2}^{2 \al} \jb{n_3}^{2 \al} \jb{n_4}^{2 \al}} \les N_2^{- \ta}
\end{split}
\label{RL1}
\end{align}

\noi
for some $\ta > 0$, where we used the condition $\al > \frac 98$. 
For the $j \geq 3$ terms, by the Ito isometry \eqref{itoX} and the discrete convolution inequality in Lemma~\ref{LEM:SUM}~(ii), we have
\begin{align}
\begin{split}
\bigg\| &\sum_{j = 3}^m \cj{a}_{j, N} N^{- (2j - 4) (\frac 32 - \al)} \int_{\T^3} :\! u_N^{2j} \!: dx \bigg\|_{L^2 (\mu_\al)}^2 \\
&= \sum_{j = 3}^m (2j)! |\cj{a}_{j, N}|^2 N^{- 4 (j - 2) (\frac 32 - \al)}  \sum_{\substack{n_1 + \cdots + n_{2j} = 0 \\ |n_\l|_\infty \leq N}} \frac{1}{\jb{n_1}^{2 \al} \cdots \jb{n_{2j}}^{2 \al}} \\
&\leq \sum_{j = 3}^m (2j)! |\cj{a}_{j, N}|^2 N^{-2j \eps} \sum_{\substack{n_1 + \cdots + n_{2j} = 0 \\ |n_\l|_\infty \leq N}} \frac{1}{\jb{n_1}^{2 \al + \frac{2j - 4}{j} (\frac 32 - \al) - \eps} \cdots \jb{n_{2j}}^{2 \al + \frac{2j - 4}{j} (\frac 32 - \al) - \eps}} \\
&\les \sum_{j = 3}^m (2j)! |\cj{a}_{j, N}|^2 N^{-2j \eps}.
\end{split}
\label{RL2}
\end{align}

\noi
for some $\eps > 0$. Here, we used the valid condition $\frac{3(2j - 1)}{2j} < 2 \al + \frac{2j - 4}{j} (\frac 32 - \al) - \eps < 3$ with $j \geq 3$ for applying Lemma~\ref{LEM:SUM}~(ii), which is satisfied given $\frac 98 < \al < \frac 32$ and $\eps > 0$ sufficiently small. 
Thus, combining \eqref{RL0}, \eqref{RL1}, the convergences in \eqref{ajconv} and \eqref{a1conv}, and \eqref{RL2}, we obtain the desired estimate \eqref{RNconv}.
\end{proof}

We now prove Proposition~\ref{PROP:Gibbs2}. Note that to obtain the convergence \eqref{Lp_conv2}, it suffices to prove the uniform bound \eqref{Lp_bdd2}. Indeed, by Lemma \ref{LEM:RNconv} and the continuity of the exponential function, we have
\begin{align*}
e^{- \mathcal{R}_N (u)} \too e^{- \mathcal{R} (u)}
\end{align*}

\noi
as $N \to \infty$ in probability. Then, by using the uniform bound \eqref{Lp_bdd2}, we can obtain the desired convergence of the densities \eqref{Lp_conv2} using a standard argument. See \cite[Remark~3.8]{Tz08} and also the proof of \cite[Proposition~1.2]{OTh}.

To prove the uniform bound \eqref{Lp_bdd2}, since the assumption \eqref{positive} is still valid if we replace $\cj{a}_j$ by $p \cj{a}_j$, we can assume without loss of generality that $p = 1$. By the Bou\'e-Dupuis variational formula (Lemma \ref{LEM:BD}), we have
\begin{align}
\log \big\| e^{- \mathcal{R}_N (u)} \big\|_{L^1 (\mu_\al)} \leq  \E \bigg[ \sup_{\Dr \in H^\al (\T^3)} \bigg\{ - \int_{\T^3} V_N (Y_N + \Dr_N ) dx - \frac 12 \| \Dr \|_{H^\al}^2 \bigg\} \bigg],
\label{BD1-1}
\end{align}

\noi
where $Y_N = Y_N (1) = \pi_N Y (1)$ with $Y$ being defined in \eqref{defY} and $\Dr_N = \pi_N \Dr$. The main task is to obtain an upper bound for \eqref{BD1-1}.

The following two lemmas are 3-dimensional analogues of \cite[Section~3.3.4 and Section~3.3.5]{STzX}, and so we will be brief in the proofs.

\begin{lemma}
\label{LEM:l=123}
Let $\frac 98 < \al < \frac 32$ and $j$ be an integer satisfying $2 \leq j \leq m$. Let $\dl > 0$ be arbitrarily small and $0 < \be_1, \be_2 < \al$. Then, there exist $1 \leq p_1, p_2 < \infty$ and a constant $C(\dl) > 0$ such that
\begin{align}
&N^{- (2j - 4) (\frac 32 - \al)} \int_{\T^3} :\! Y_N^{2j - 1} \!: \Dr_N dx \leq  C(\dl)  N^{- (4j - 8) (\frac 32 - \al)} \| :\! Y_N^{2j - 1} \!: \|_{H^{- \al}}^2 + \dl \| \Dr_N \|_{H^\al}^2,  \label{l=1} \\
&N^{- (2j - 4) (\frac 32 - \al)} \int_{\T^3} :\! Y_N^{2j - 2} \!: \Dr_N^2 dx \notag \\
&\quad \leq C(\dl) N^{- 4 (2j - 4) (\frac 32 - \al)} \| :\! Y_N^{2j - 2} \!: \|_{W^{- \be_1, p_1}}^4 + \dl \big( \| \Dr_N \|_{H^\al}^2 + \| \Dr_N \|_{L^4}^4 \big),
\label{l=2} \\
&N^{- (2j - 4) (\frac 32 - \al)} \int_{\T^3} :\! Y_N^{2j - 3} \!: \Dr_N^3 dx \notag \\
&\quad \leq C(\delta) \big( N^{- (2j - 4) (\frac 32 - \al)} \| :\! Y_N^{2j - 3} \!: \|_{W^{- \be_2, p_2}} \big)^{\frac{4 \al}{\al - \be_2}} + \dl \big( \| \Dr_N \|_{H^\al}^2 + \| \Dr_N \|_{L^4}^4 \big), \label{l=3}
\end{align}
\end{lemma}

\begin{proof}
Note that \eqref{l=1} follows directly from duality and Cauchy's inequality.

For \eqref{l=2}, we let $1 < p_1 < \infty$, $1 < q_1 \leq 2$, and $1 < q_2 \leq 4$ be such that
$\frac{1}{p_1} + \frac{1}{q_1} + \frac{1}{q_2} = 1$.
Then, by duality, the product estimate in Lemma~\ref{LEM:gko}~(i), and Young's inequality, we have
\begin{align*}
&N^{- (2j - 4) (\frac 32 - \al)} \int_{\T^3} :\! Y_N^{2j - 2} \!: \Dr_N^2 dx \\
&\quad \leq C N^{- (2j - 4) (\frac 32 - \al)} \| :\! Y_N^{2j - 2} \!: \|_{W^{- \be_1, p_1}} \| \Dr_N \|_{W^{\be_1, q_1}} \| \Dr_N \|_{L^{q_2}} \\
&\quad \leq C(\dl) N^{- 4 (2j - 4) (\frac 32 - \al)} \| :\! Y_N^{2j - 2} \!: \|_{W^{- \be_1, p_1}}^4 + \dl \big( \| \Dr_N \|_{H^\al}^2 + \| \Dr_N \|_{L^4}^4 \big)
\end{align*}

\noi
for some constants $C, C(\dl) > 0$, as desired.

For \eqref{l=3}, we use duality and the product estimate in Lemma~\ref{LEM:gko}~(i) to obtain
\begin{align}
\int_{\T^3} :\! Y_N^{2j - 3} \!: \Dr_N^3 dx \les \| :\! Y_N^{2j - 3} \!: \|_{W^{- \be_2, \frac{1 + \eps}{\eps}}} \| \Dr_N \|_{W^{\be_2, \frac{2 (1 + \eps)}{1 - \eps}}} \| \Dr_N \|_{L^4}^2
\label{l3-1}
\end{align}

\noi
for some $\eps > 0$ arbitrarily small. By interpolation (Lemma \ref{LEM:interp}) with $\be_2 < \al$, we have
\begin{align}
\| \Dr_N \|_{W^{\be_2, \frac{2 (1 + \eps)}{1 - \eps}}} \les \| \Dr_N \|_{W^{\al, p}}^{\frac{\be_2}{\al}} \| \Dr_N \|_{L^4}^{1 - \frac{\be_2}{\al}} \quad \text{with } p = \frac{4 (1 + \eps) \be_2}{(1 - 3 \eps) \al + (1 + \eps) \be_2}.
\label{l3-2}
\end{align}

\noi
Since $\be_2 < \al$, by letting $\eps > 0$ be sufficiently small, we can have $p \leq 2$. Thus, combining \eqref{l3-1} and \eqref{l3-2} and using Young's inequality, we obtain
\begin{align*}
&N^{- (2j - 4) (\frac 32 - \al)} \int_{\T^3} :\! Y_N^{2j - 3} \!: \Dr_N^3 dx \notag \\
&\quad \leq C N^{- (2j - 4) (\frac 32 - \al)} \| :\! Y_N^{2j - 3} \!: \|_{W^{- \be_2, \frac{1 + \eps}{\eps}}} \| \Dr_N \|_{H^{\al}}^{\frac{\be_2}{\al}} \| \Dr_N \|_{L^4}^{3 - \frac{\be_2}{\al}} \\
&\quad \leq C(\delta) \Big( N^{- (2j - 4) (\frac 32 - \al)} \| :\! Y_N^{2j - 3} \!: \|_{W^{- \be_2, p_2}} \Big)^{\frac{4 \al}{\al - \be_2}} + \dl \big( \| \Dr_N \|_{H^\al}^2 + \| \Dr_N \|_{L^4}^4 \big)
\end{align*}

\noi
for some constants $C, C(\dl) > 0$, as desired.
\end{proof}

\begin{lemma}
\label{LEM:l>=4}
Let $\frac 98 < \al < \frac 32$, $4 \leq \l \leq 2m - 1$, $\frac{\l + 1}{2} \leq j \leq m$, and $\dl > 0$ be arbitrarily small. Let $\ell_0 = \lfloor \frac{\l}{2} \rfloor + 1 \leq m$ and $\be_3 > 0$ be such that
\begin{align}
\frac{2 \l_0 \be_3}{(2 \l_0 - \l) \al + \be_3} < 2.
\label{be_cond}
\end{align}

\noi
We also define $\gamma$ and $\eta$ to be
\begin{align}
\gamma = \frac{(\l_0 - 2) (\l \al - \be_3) (\frac 32 - \al)}{\l_0 \al} \quad \text{and} \quad \eta = \frac{2 \l_0 \al}{(2\l_0 - \l) \al - (\l_0 - 1) \be_3}. \label{eta}
\end{align}

\noi
Then, there exists a constant $C(\dl) > 0$ such that
\begin{align*}
\begin{split}
N^{- (2j - 4) (\frac 32 - \al)} &\int_{\T^3} :\! Y_N^{2j - \l} \!: \Dr_N^\l dx \leq C(\dl) \big( N^{- (2j - 4) (\frac 32 - \al) + \gamma} \| :\! Y_N^{2j - \l} \!: \|_{W^{- \be_3, \infty}} \big)^\eta \\
& + \dl \big( \| \Dr_N \|_{H^\al}^2 + N^{- (2 \l_0 - 4) (\frac 32 - \al)} \| \Dr_N \|_{L^{2 \l_0}}^{2 \l_0} \big).
\end{split}
\end{align*}
\end{lemma}

\begin{proof}
By duality and the product estimate in Lemma~\ref{LEM:gko}~(i), we have
\begin{align}
\int_{\T^3} :\! Y_N^{2j - \l} \!: \Dr_N^\l dx &\les \| :\! Y_N^{2j - \l} \!: \|_{W^{- \be_3, \frac{1 + \eps}{\eps}}} \| \Dr_N \|_{W^{\be_3, p_\eps}} \| \Dr_N \|_{L^{2 \l_0}}^{\l - 1},
\label{l4-1}
\end{align}

\noi
where
\begin{align*}
p_\eps = \frac{2 (1 + \eps) \l_0}{2 \l_0 - (1 + \eps) (\l - 1)} \searrow \frac{2 \l_0}{2 \l_0 - \l + 1}
\end{align*}

\noi
as $\eps \to 0$. Note that the condition \eqref{be_cond} guarantees that $\be_3 < \al$. By interpolation (Lemma~\ref{LEM:interp}) with $\be_3 < \al$, we have
\begin{align}
\| \Dr_N \|_{W^{\be_3, p_\eps}} \les \| \Dr_N \|_{W^{\al, q_\eps}}^{\frac{\be_3}{\al}} \| \Dr_N \|_{L^{2 \l_0}}^{1 - \frac{\be_3}{\al}},
\label{l4-2}
\end{align}

\noi
where
\begin{align*}
q_\eps = \frac{2 \l_0 \be_3}{ \frac{2 \l_0 \al}{p_\eps} - (\al - \be_3) } \searrow \frac{2 \l_0 \be_3}{(2 \l_0 - \l) \al + \be_3}
\end{align*}

\noi
as $\eps \to 0$. Since we have the condition \eqref{be_cond}, by taking $\eps > 0$ to be sufficiently small, we have $q_\eps \leq 2$. Thus, combining \eqref{l4-1} and \eqref{l4-2} and applying Young's inequality, we obtain
\begin{align*}
&N^{- (2j - 4) (\frac 32 - \al)} \int_{\T^3} :\! Y_N^{2j - \l} \!: \Dr_N^\l dx \\
&\leq C N^{- (2j - 4) (\frac 32 - \al) + \gamma} \| :\! Y_N^{2j - \l} \!: \|_{W^{- \be_3, \infty}} \| \Dr_N \|_{H^\al}^{\frac{\be_3}{\al}} \cdot \big( N^{- \frac{\g \al}{\l \al - \be_3}} \| \Dr_N \|_{L^{2 \l_0}} \big)^{\l - \frac{\be_3}{\al}} \\
&\leq C(\dl) \big( N^{- (2j - 4) (\frac 32 - \al) + \gamma} \| :\! Y_N^{2j - \l} \!: \|_{W^{- \be_3, \infty}} \big)^\eta + \dl \big( \| \Dr_N \|_{H^\al}^2 + N^{- (2 \l_0 - 4) (\frac 32 - \al)} \| \Dr_N \|_{L^{2 \l_0}}^{2 \l_0} \big)
\end{align*}

\noi
for some constants $C, C(\dl) > 0$, with $\g$ and $\eta$ being defined in \eqref{eta}. The use of Young's inequality is valid given $\eta > 1$, which is true thanks to the condition \eqref{be_cond}. Thus, we have obtained the desired estimate.
\end{proof}

We are now ready to prove Proposition~\ref{PROP:Gibbs2}.
\begin{proof}[Proof of Proposition~\ref{PROP:Gibbs2}] 
As mentioned after the proof of Lemma~\ref{LEM:RNconv}, we only need to obtain an upper bound for \eqref{BD1-1}. 

By \eqref{defVN}, \eqref{HkX}, and the binomial expansion in \eqref{herm_decomp1}, we have
\begin{align*}
V_N (Y_N + \Dr_N) &= \cj{a}_{1, N} N^{2 (\frac 32 - \al)} \big( :\! Y_N^2 \!: + 2 Y_N \Dr_N + \Dr_N^2 \big) \\ 
&\quad + \sum_{\l = 0}^{2m - 1} \mathcal{Y}_{N, \l} \Dr_N^\l + \sum_{j = 2}^m \cj{a}_{j, N} N^{- (2j - 4) (\frac 32 - \al)} \Dr_N^{2j},
\end{align*}

\noi
where 
\begin{align}
\mathcal{Y}_{N, \l} = \sum_{j = 2 \vee \lceil \frac{\l + 1}{2} \rceil}^m \cj{a}_{j, N} \binom{2j}{\l} N^{- (2j - 4) (\frac 32 - \al)} :\! Y_N^{2j - \l} \!:. 
\label{defYc}
\end{align}

\noi
As in \cite[Proposition 3.5]{STzX}, by the positivity condition \eqref{positive} and the convergences in \eqref{ajconv} and \eqref{a1conv}, we obtain
\begin{align}
\sum_{j = 1}^m \cj{a}_{j, N} N^{- (2j - 4) (\frac 32 - \al)} \Dr_N^{2j} \geq c \sum_{j = 2}^m N^{- (2j - 4) (\frac 32 - \al)} \Dr_N^{2j} - C
\label{BD2}
\end{align}

\noi
for some constants $c, C > 0$ and sufficiently large $N \in \N$. Consequently, by \eqref{BD2} and the Cauchy-Schwarz inequality, we have
\begin{align}
\begin{split}
- \int_{\T^3} &V_N (Y_N + \Dr_N) dx - \frac 12 \| \Dr \|_{H^\al}^2  \\
&\leq - \int_{\T^3} \bigg( \cj{a}_{1, N} N^{2 (\frac 32 - \al)} \big( :\! Y_N^2 \!: + 2 Y_N \Dr_N \big) + \sum_{\l = 0}^{2m - 1} \mathcal{Y}_{N, \l} \Dr_N^\l \bigg) dx  \\
&\quad - c \sum_{j = 2}^m N^{- (2j - 4) (\frac 32 - \al)} \| \Dr_N \|_{L^{2j}}^{2j} - \frac 12 \| \Dr_N \|_{H^\al}^2 + C.
\end{split}
\label{BD3}
\end{align}

\noi
Also, from \eqref{defYc} and \eqref{mean0X}, we have
\begin{align}
\E \bigg[ \int_{\T^3} \mathcal{Y}_{N, 0} dx \bigg] = \int_{\T^3} \sum_{j = 2}^m \cj a_{j, N} N^{- (2j - 4) (\frac 32 - \al)} \E [:\! Y_N^{2j} \!:] dx = 0.
\label{ell0}
\end{align}

\noi
Thus, with \eqref{BD3}, \eqref{ell0}, and \eqref{a1conv} in hand, in order to show an upper bound for \eqref{BD1-1}, we only need to show that for any small $\dl > 0$, there exists a random constant $K = K(\dl, N)$ whose expectation is uniformly bounded in $N$ such that
\begin{align}
\begin{split}
\bigg| &\int_{\T^3} :\! Y_N^2 \!: dx \bigg| + \bigg| \int_{\T^3} Y_N \Dr_N dx \bigg| + \sum_{\l = 1}^{2m - 1} \bigg| \int_{\T^3} \mathcal{Y}_{N, \l} \Dr_N^\l dx \bigg| \\
&\leq K + \dl \bigg( \sum_{j = 2}^m N^{- (2j - 4) (\frac 32 - \al)} \| \Dr_N \|_{L^{2j}}^{2j} + \| \Dr_N \|_{H^\al}^2 \bigg).
\end{split}
\label{BD_goal}
\end{align}

For the first term on the left-hand-side of \eqref{BD_goal}, we use Lemma~\ref{LEM:YNk} to obtain
\begin{align*}
\bigg| \int_{\T^3} :\! Y_N^2 \!: dx \bigg| \leq \| :\! Y_N^2 \!: \|_{H^{-2 (\frac 32 -  \al) - \eps}} \leq K.
\end{align*}

\noi
For the second term on the left-hand-side of \eqref{BD_goal}, we use duality, Young's inequality, and Lemma~\ref{LEM:YNk} to obtain
\begin{align*}
\bigg| \int_{\T^3} Y_N \Dr_N dx \bigg| \leq \| Y_N \|_{H^{- \al}} \| \Dr_N \|_{H^\al} \leq K + \dl \| \Dr_N \|_{H^\al}^2,
\end{align*}

\noi
where we need $- \al < \al - \frac 32$ which is valid given $\al > \frac 98 > \frac 34$.

We now consider the third term on the left-hand-side of \eqref{BD_goal}. By \eqref{defYc}, Lemma \ref{LEM:l=123}, and Lemma \ref{LEM:l>=4}, we have
\begin{align}
\begin{split}
\sum_{\l = 1}^{2m - 1} \bigg| \int_{\T^3} \mathcal{Y}_{N, \l} \Dr_N^\l dx \bigg| &\leq C(\dl) \sum_{j = 2}^m \big( N^{- (2j - 4) (\frac 32 - \al)} \| :\! Y_N^{2j - 1} \!: \|_{H^{- \al}} \big)^2 \\
&\quad + C(\dl) \sum_{j = 2}^m \big( N^{- (2j - 4) (\frac 32 - \al)} \| :\! Y_N^{2j - 2} \!: \|_{H^{- \be_1, p_1}} \big)^4 \\
&\quad + C(\dl) \sum_{j = 2}^m \big( N^{- (2j - 4) (\frac 32 - \al)} \| :\! Y_N^{2j - 3} \!: \|_{W^{- \be_2, p_2}} \big)^{\frac{4 \al}{\al - \be_2}} \\
&\quad + C(\dl) \sum_{\l = 4}^{2m - 1} \sum_{j = \lceil \frac{\l + 1}{2} \rceil}^m \big( N^{- (2j - 4) (\frac 32 - \al) + \gamma} \| :\! Y_N^{2j - \l} \!: \|_{W^{- \be_3, \infty}} \big)^\eta \\
&\quad + \dl \bigg( \sum_{j = 2}^m N^{- (2j - 4) (\frac 32 - \al)} \| \Dr_N \|_{L^{2j}}^{2j} + \| \Dr_N \|_{H^\al}^2 \bigg),
\end{split}
\label{BDl}
\end{align}

\noi
where $C(\dl) > 0$ is a constant, $0 < \be_1, \be_2 < \al$, $\be_3$ satisfies \eqref{be_cond}, $1 \leq p_1, p_2 < \infty$ are as given by Lemma \ref{LEM:l=123}, and $\g$ and $\eta$ are as defined in \eqref{eta}. It remains to show that the expectation of the first four terms on the right-hand-side of \eqref{BDl} is uniformly bounded in $N$. For the first term on the right-hand-side of \eqref{BDl}, we can apply Lemma \ref{LEM:YNk} given $\frac 98 < \al < \frac 32$. For the second term, since $\al > \frac 98$, we can pick $\be_1$ such that $3 - 2 \al < \be_1 < \al$, and so Lemma~\ref{LEM:YNk} can be applied. For the third term, since $\al > \frac 98$, we can pick $\be_2$ such that $\frac 32 - \al < \be_2 < \al$, and so Lemma \ref{LEM:YNk} can be applied. For the fourth term, since we already know from \eqref{be_cond} that $\be_3 < \al < \frac 32$, in order to apply Lemma \ref{LEM:YNk}, we require
\begin{align}
\be_3 > \g - (\l - 4) \Big( \frac 32 - \al \Big) > 0
\label{BD_cond1}
\end{align}

\noi
and
\begin{align}
(2j - 4) \Big( \frac 32 - \al \Big) - \g \geq 0.
\label{BD_cond2}
\end{align}

\noi
By plugging in the value of $\g$ as in \eqref{eta}, we see that the condition \eqref{BD_cond1} is equivalent to
\begin{align*}
\be_3 > \frac{2 \al (2 \l_0 - \l) (\frac 32 - \al)}{\frac 32 \l_0 - 2 (\frac 32 - \al)}  \quad \text{and} \quad \frac{\l_0 \be_3}{(2 \l_0 - \l) \al + \be_3} < 2,
\end{align*}

\noi
where $\l_0 = \lfloor \frac{\l}{2} \rfloor + 1$. Along with the condition \eqref{be_cond}, such $\be_3$ can be found as long as
\begin{align*}
\frac{2 \al (2 \l_0 - \l) (\frac 32 - \al)}{\frac 32 \l_0 - 2 (\frac 32 - \al)} < \frac{(2 \l_0 - \l) \al}{\l_0 - 1},
\end{align*}

\noi
which holds given $\al > \frac 34$. Then, by plugging in the value of $\g$ as in \eqref{eta}, the condition \eqref{BD_cond2} is equivalent to
\begin{align*}
\be_3 \geq \frac{(4 \l_0 - 2 \l - (2j - \l) \l_0) \al}{\l_0 - 2}.
\end{align*}

\noi
Along with \eqref{be_cond}, such $\be_3$ can be found as long as
\begin{align*}
\frac{(4 \l_0 - 2 \l - (2j - \l) \l_0) \al}{\l_0 - 2} < \frac{(2 \l_0 - \l) \al}{\l_0 - 1},
\end{align*}

\noi
which holds given $2j - \l \geq 1$, and $\l_0 = \lfloor \frac{\l}{2} \rfloor + 1 \leq \l - 1$ (note that the equalities cannot hold simultaneously). This finishes the proof of Proposition~\ref{PROP:Gibbs2}.
\end{proof}

\subsection{Construction of the fractional $\Phi^4_3$-measure}
\label{SUBSEC:Gibbs2}
In this subsection, we show the construction of the fractional $\Phi^4_3$-measure in Theorem~\ref{THM:Gibbs}.

In the case $\al > \frac 98$, we can view the truncated fractional $\Phi^4_3$-measure $\rho_N$ in \eqref{GibbsN1} as a special case of $\nu_N$ in \eqref{Gibbs_nuN} by setting $V_N (\pi_N u) = \frac 14 :\! (\pi_N u)^4 \!:$ and $\mathcal{R} (u) = \frac 14 \int_{\T^3} :\! u^4 \!: dx$. Thus, the proof of Theorem~\ref{THM:Gibbs}~(i) is covered by the proof of Proposition~\ref{PROP:Gibbs2} in Subsection~\ref{SUBSEC:Gibbs1} (the case $\al = \frac 32$ for Theorem~\ref{THM:Gibbs}~(i) also follows from the same steps).

We now focus on the case $1 < \al \leq \frac 98$. As mentioned in Subsection~\ref{SUBSEC:Gibbs}, we need to introduce a further renormalization for the potential energy $R_N$ as defined in \eqref{RNu}. 
Let us first apply the Bou\'e-Dupuis variational formula (Lemma~\ref{LEM:BD0}) to obtain
\begin{align*}
\log \E \big\| e^{- R_N (u)} \big\|_{L^1 (\mu_\al)} = \sup_{\ta \in \mathbb{H}_a^0} \E \bigg[ - R_N (Y_N + \Dr_N) - \frac 12 \int_0^1 \| \ta (t) \|_{L_x^2}^2 dt \bigg] ,
\end{align*}

\noi
where $Y_N = Y_N (1) = \pi_N Y (1)$ with $Y$ being defined in \eqref{defY} and $\Dr_N = \Dr_N (1) = \pi_N \Dr (1)$ with $\Dr (t) = I_\al (\ta) (t)$ and $I_\al$ being defined in \eqref{defIa}.
The problematic term in $R_N (Y_N + \Dr_N)$ is
\begin{align*}
\int_{\T^3} \!:\!Y_N^3\!: \Dr_N dx.
\end{align*}

\noi
Here, we follow \cite[Section 6]{OOTol1} by introducing a change of variable
\begin{align}
\dot{\Ups}^N (t) \deff \dot{\Dr} (t) + \dot{\ZZ}_N (t)
\label{defUps}
\end{align}

\noi
with $\ZZ_N = \pi_N \ZZ^N$ being defined by $\ZZ^N (0) = 0$ and
\begin{align}
\dot{\ZZ}^N (t) \deff \jb{\nb}^{-2 \al} :\! Y_N^3 (t) \!: .
\label{defZZ}
\end{align}

\noi
Then, from the Ito product formula, \eqref{defUps}, and \eqref{defZZ}, we have
\begin{align}
\begin{split}
&\E \bigg[ \int_{\T^3} :\!Y_N^3\!: \Dr_N dx + \frac 12 \int_0^1 \| \ta (t) \|_{L_x^2}^2 dt \bigg] \\
&= \E \bigg[ \int_0^1 \int_{\T^3} :\!Y_N^3 (t)\!: \dot{\Dr}_N (t) dx dt \bigg] + \frac 12 \E \bigg[ \int_0^1 \| \ta (t) \|_{L_x^2}^2 dt \bigg] \\
&= \frac 12 \E \bigg[ \int_0^1 \| \dot{\Ups}^N (t) \|_{H^\al}^2 dt \bigg]
 - \frac 12 \E \bigg[ \int_0^1 \| \dot{\ZZ}_N (t) \|_{H_x^\al}^2 dt \bigg] ,
\end{split}
\label{YZ2}
\end{align}

\noi
and we define the constant $\al_N$ by
\begin{align}
\al_N \deff \frac 12  \E \bigg[ \int_0^1 \| \dot{\ZZ}_N (t) \|_{H_x^\al}^2 dt \bigg] .
\label{defalN}
\end{align}

\noi
From Lemma~\ref{LEM:YNk} and \eqref{defZZ}, we see that $\ZZ_N$ has (uniform-in-$N$) spatial regularity $5 \al - \frac 92 -$. Also, by the Ito isometry \eqref{itoX} and \eqref{convsum1-1} in Lemma~\ref{LEM:SUM}~(i) twice, we have 
\begin{align}
\begin{split}
\E \big[ \| \dot{\ZZ}_{N} (t) \|_{H^{5 \al - \frac 92}}^2 \big] &\sim t^3 \sum_{|n|_\infty \leq N} \jb{n}^{6 \al - 9} \sum_{\substack{n_1 + n_2 + n_3 = n \\ |n_1|_\infty, |n_2|_\infty, |n_3|_\infty \leq N}} \frac{1}{\jb{n_1}^{2 \al} \jb{n_2}^{2 \al} \jb{n_3}^{2 \al}} \\
&\sim t^3 \sum_{|n|_\infty \leq N} \jb{n}^{-3} \sim t^3 \log N ,
\end{split}
\label{ZZdiv}
\end{align}

\noi
so that $\al_N \to \infty$ as $N \to \infty$ since $5 \al - \frac 92 \leq \al$ given $1 < \al \leq \frac 98$. Thus, we replace the potential energy $R_N$ in \eqref{RNu} by $R_N^\dia \deff R_N + \al_N$, so that the Bou\'e-Dupuis variational formula (Lemma~\ref{LEM:BD0}) and \eqref{YZ2} give
\begin{align}
\begin{split}
\log \big\| e^{- R_N^\dia (u)} \big\|_{L^1 (\mu_\al)} &= \sup_{\ta \in \mathbb{H}_a^0} \E \bigg[ - R_N (Y_N + \Dr_N) - \frac 12 \int_0^1 \| \ta (t) \|_{L_x^2}^2 dt \bigg] \\
&= \sup_{\dot{\Ups}^N \in \mathbb{H}_a^\al} \E \bigg[ - \wt R_N^\dia (Y + \Ups^N - \ZZ_N) - \frac 12 \int_0^1 \| \dot{\Ups}^N (t) \|_{H_x^\al}^2 dt \bigg],
\end{split}
\label{BDRNd}
\end{align}

\noi
where
\begin{align}
\begin{split}
\wt R_N^\dia (Y + \Ups^N - \ZZ_N) &= \frac 14 \int_{\T^3} : \! Y_N^4 \! : dx + \frac 32 \int_{\T^3} : \! Y_N^2 \! : \Dr_N^2 dx + \int_{\T^3} Y_N \Dr_N^3 dx \\
&\quad + \frac 14 \int_{\T^3} \Dr_N^4 dx
\end{split}
\label{RNt}
\end{align}

\noi
with $\Dr_N = \Dr_N (1) = \pi_N \Dr (1) = \pi_N ( \Ups^N (1) - \ZZ_N (1))$.

\medskip
We now establish some estimates to deal with terms on the right-hand-side of \eqref{RNt}. 
In Lemma~\ref{LEM:Dr1} and Lemma~\ref{LEM:Dr2} below, the frequency cutoff does not play any role, and so we drop the subscript $N$ for simplicity.

\begin{lemma}
\label{LEM:Dr1}
Let $1 < \al \leq \frac 98$ and $\eps > 0$. Then, there exists a constant $C > 0$ such that
\begin{align*}
&\int_{\T^3} \Dr^4 dx \geq \frac 12 \int_{\T^3} \Ups^4 dx - C \|\ZZ\|_{W^{5\alpha -\frac 92 - \eps, \infty}}^4,  \\
&\int_{\T^3} \Dr^4 dx \leq 2 \int_{\T^3} \Ups^4 dx + C \|\ZZ\|_{W^{5\alpha -\frac 92 - \eps, \infty}}^4. 
\end{align*}

\noi
for
$\Dr = \Ups - \ZZ$.
\end{lemma}

\begin{proof}
By H\"older's inequalities and Young's ineqaulities, we have
\begin{align*}
\| \Ups \ZZ^3 \|_{L^1} &\leq \| \Ups \|_{L^4} \|\ZZ\|_{L^4}^3 \leq \dl \| \Ups \|_{L^4}^4 + C(\dl) \|\ZZ\|_{W^{5\alpha - \frac 92 - \eps, \infty}}^4, \\
\| \Ups^2 \ZZ^2 \|_{L^1} &\leq \| \Ups \|_{L^4}^2 \|\ZZ\|_{L^4}^2 \leq \dl \| \Ups \|_{L^4}^4 + C(\dl) \|\ZZ\|_{W^{5\alpha - \frac 92 - \eps, \infty}}^4, \\ 
\| \Ups^3 \ZZ \|_{L^1} &\leq \| \Ups \|_{L^4}^3\|\ZZ\|_{L^4} \leq \dl \| \Ups \|_{L^4}^4 + C(\dl) \|\ZZ\|_{W^{5\alpha - \frac 92 - \eps, \infty}}^4
\end{align*}

\noi
for some $\dl > 0$ arbitrarily small and a constant $C(\dl) > 0$.
This proves the two estimates.
\end{proof}

\begin{lemma}
\label{LEM:Dr2}
Let $1 < \al \leq \frac 98$ and $\eps > 0$ be small.
Then, for any $\dl > 0$, there exist constants  $C(\dl, \eps), c > 0$ such that
\begin{align}
\begin{split}
\bigg| \int_{\T^3}  Y \Dr^3  dx \bigg| \leq \dl \big( \| \Ups \|_{L^4}^4 +  \| \Ups \|_{H^\alpha}^2 \big) + C(\dl) \Big(1 +  \| Y \|_{W^{\alpha -\frac 32 - \eps, \infty}}^{c} + \|\ZZ\|_{W^{5\alpha - \frac 92 - \eps, \infty}}^c \Big),
\end{split}
\label{YY1y}
\end{align}
\begin{align}
\begin{split}
\bigg| &\int_{\T^3} :\! Y^2 \!: \Dr^2 dx  - \int_{\T^3} :\! Y^2 \!: \ZZ^2 dx \bigg| \\
&\leq \dl \big( \|\Ups\|_{L^4}^4+\|\Ups\|_{H^\al}^2 \big) + C(\dl, \eps) \Big( 1 + \| :\! Y^2 \!: \|_{W^{2 \al - 3 - \eps, \infty}}^{c}
+ \| :\! Y^2 \!: \ZZ \|_{W^{2 \al - 3 - \eps, \infty}} \Big)
\end{split}
\label{YY2y}
\end{align}

\noi
for $\Dr = \Ups - \ZZ$.
\end{lemma}

\begin{proof}
We first prove \eqref{YY1y}. Note that
\begin{align*}
\bigg| \int_{\T^3}  Y \Dr^3  dx\bigg| = \bigg| \int_{\T^3}  Y \Ups^3  dx - 3 \int_{\T^3}  Y \Ups^2 \ZZ  dx + 3 \int_{\T^3}  Y \Ups \ZZ^2  dx - \int_{\T^3}  Y \ZZ^3  dx\bigg|.
\end{align*}

\noi
By duality, the product estimate  (Lemma~\ref{LEM:gko}~(i)), Young's inequality, interpolation (Lemma~\ref{LEM:interp}), and Young's inequality again, we have
\begin{align}
\begin{split}
\bigg| \int_{\T^3} Y \Ups^3  dx \bigg| &\leq \| Y \|_{W^{\al - \frac 32 - \eps, \infty}} \| \Ups^3 \|_{W^{- \al + \frac 32 + \eps, 1}} \\
&\leq C \| Y \|_{W^{\al - \frac 32 - \eps, \infty}} \| \Ups \|_{H^{- \al + \frac 32 + \eps}} \| \Ups \|_{L^4}^2 \\
&\leq C(\dl') \| Y \|_{W^{\al - \frac 32 - \eps, \infty}}^c + \dl' \| \Ups \|_{H^{\frac{\al}{2}}}^{\frac 83} + \dl' \| \Ups \|_{L^4}^4 \\
&\leq C(\dl') \| Y \|_{W^{\al - \frac 32 - \eps, \infty}}^c + C \dl' \| \Ups \|_{H^\al}^{\frac 43} \| \Ups \|_{L^2}^{\frac 43} +  \dl' \| \Ups \|_{L^4}^4 \\
&\leq C(\dl') \| Y \|_{W^{\al - \frac 32 - \eps, \infty}}^c + C \dl' \| \Ups \|_{H^\al}^2 + 2 \dl' \| \Ups \|_{L^4}^4
\end{split}
\label{YY1y-1}
\end{align}

\noi
for any $0 < \dl' \ll 1$, where we used $0 < - \al + \frac 32 + \eps \leq \frac{\al}{2}$ given $\al > 1$ and $\eps > 0$ sufficiently small. By duality, the product estimate (Lemma~\ref{LEM:gko}~(i)), interpolation (Lemma~\ref{LEM:interp}), and Young's inequalities, we have
\begin{align}
\begin{split}
\bigg| \int_{\T^3} &Y \Ups^2 \ZZ dx \bigg| \leq \| Y \|_{W^{\al - \frac 32 - \eps, \infty}} \| \Ups^2 \ZZ \|_{W^{-\al + \frac 32 + \eps, 1}} \\ 
&\leq C \| Y \|_{W^{\al - \frac 32 - \eps, \infty}} \| \Ups \|_{W^{\frac{\al}{2}, 2 + \eps}}^2 \| \ZZ \|_{W^{5 \al - \frac 92 - \eps, \infty}} \\
&\leq C(\dl') \Big( \| Y \|_{W^{\al - \frac 32 - \eps, \infty}}^{c} + \| \ZZ \|_{W^{5 \al - \frac 92 - \eps, \infty}}^c \Big) + \dl' \| \Ups \|_{W^{\frac{\al}{2}, 2 + \eps}}^{\frac 83} \\
&\leq C(\dl') \Big( \| Y \|_{W^{\al - \frac 32 - \eps, \infty}}^{c} + \| \ZZ \|_{W^{5 \al - \frac 92 - \eps, \infty}}^c \Big) + C \dl' \| \Ups \|_{H^\al}^2 + 2 \dl' \| \Ups \|_{L^4}^4
\end{split}
\label{YY1y-2}
\end{align}

\noi
for any $0 < \dl' \ll 1$, where we used $0 < - \al + \frac 32 + \eps < \frac{\al}{2}$ and $- \al + \frac 32 + \eps < 5 \al - \frac 92 - \eps$ given $\al > 1$ and $\eps > 0$ sufficiently small. By duality, the product estimate (Lemma~\ref{LEM:gko}~(i)), and Young's inequality, we have
\begin{align}
\begin{split}
\bigg| \int_{\T^3} Y \Ups \ZZ^2 dx \bigg| &\leq \| Y \|_{W^{\al - \frac 32 - \eps, \infty}} \| \Ups \ZZ^2 \|_{W^{-\al + \frac 32 + \eps, 1}} \\
&\leq C \| Y \|_{W^{\al - \frac 32 - \eps, \infty}} \| \Ups \|_{H^\al} \| \ZZ \|_{W^{5 \al - \frac 92 - \eps, \infty}}^2 \\
&\leq C(\dl') \Big( \| Y \|_{W^{\al - \frac 32 - \eps, \infty}}^{c} + \| \ZZ \|_{W^{5 \al - \frac 92 - \eps, \infty}}^c \Big) + \dl' \| \Ups \|_{H^{\al}}^2,
\end{split}
\label{YY1y-3}
\end{align}

\noi
for any $0 < \dl' \ll 1$, where we used $0 < - \al + \frac 32 + \eps < \al$ and $- \al + \frac 32 + \eps < 5 \al - \frac 92 - \eps$ given $\al > 1$ and $\eps > 0$ sufficiently small. By duality, the product estimate (Lemma~\ref{LEM:gko}~(i)), and Young's inequality, we have
\begin{align}
\begin{split}
\bigg| \int_{\T^3} Y \ZZ^3 dx \bigg| &\leq \| Y \|_{W^{\al - \frac 32 - \eps, \infty}} \| \ZZ^3 \|_{W^{-\al + \frac 32 + \eps, \infty}} \\
&\les \| Y \|_{W^{\al - \frac 32 - \eps, \infty}}^c + \| \ZZ \|_{W^{5 \al - \frac 92 - \eps, \infty}}^c,
\end{split}
\label{YY1y-4}
\end{align}

\noi
where we used $0 < -\al + \frac 32 + \eps < 5 \al - \frac 92 - \eps$ given $\al > 1$ and $\eps > 0$ sufficiently small. Combining \eqref{YY1y-1}, \eqref{YY1y-2}, \eqref{YY1y-3}, and \eqref{YY1y-4}, we obtain \eqref{YY1y}.

We now prove \eqref{YY2y}. Note that
\begin{align}
\bigg| \int_{\T^3} :\! Y^2 \!: \Dr^2 dx \bigg| = \bigg|  \int_{\T^3} :\! Y^2 \!: \Ups^2 dx - 2  \int_{\T^3} :\! Y^2 \!: \Ups \ZZ dx +  \int_{\T^3} :\! Y^2 \!: \ZZ^2 dx \bigg|.
\label{YY2y-0}
\end{align}

\noi
By duality, the product estimate (Lemma~\ref{LEM:gko}~(i)), interpolation (Lemma~\ref{LEM:interp}), and Young's inequality, we have
\begin{align}
\begin{split}
\bigg| \int_{\T^3} :\! Y^2 \!: \Ups^2 dx \bigg| &\leq C \| :\! Y^2 \!: \|_{W^{2\al - 3 - \eps, \infty}} \| \Ups \|_{H^{-2 \al + 3 + \eps}}^2 \\
&\leq C(\dl', \eps) \| :\! Y^2 \!: \|_{W^{2 \al - 3 - \eps, \infty}}^c + \dl' \| \Ups \|_{H^{\frac{2 - \eps}{2 + \eps} \al}}^{2 + \eps} \\
&\leq C(\dl', \eps) \| :\! Y^2 \!: \|_{W^{2 \al - 3 - \eps, \infty}}^c + C \dl' \| \Ups \|_{H^\al}^{2 - \eps} \| \Ups \|_{L^2}^{2 \eps} \\
&\leq C(\dl', \eps) \| :\! Y^2 \!: \|_{W^{2 \al - 3 - \eps, \infty}}^c + C\dl' \| \Ups \|_{H^\al}^2 + C\dl' \| \Ups \|_{L^4}^4
\end{split}
\label{YY2y-1}
\end{align}

\noi
for any $0 < \dl' \ll 1$, where we used $0 < -2\al + 3 + \eps < \frac{2 - \eps}{2 + \eps} \al$ given $\al > 1$ and $\eps > 0$ sufficiently small. By duality and Young's inequality, we have
\begin{align}
\begin{split}
\bigg| \int_{\T^3} :\! Y^2 \!: \Ups \ZZ dx \bigg| &\leq \| :\! Y^2 \!: \ZZ \|_{W^{- \al, \infty}}  \| \Ups \|_{W^{\al, 1}} \\
&\leq C \| :\! Y^2 \!: \ZZ \|_{W^{2 \al - 3 - \eps, \infty}} \| \Ups \|_{H^\al} \\
&\leq C(\dl') \| :\! Y^2 \!: \ZZ \|_{W^{2 \al - 3 - \eps, \infty}}^2 + \dl' \| \Ups \|_{H^\al}^2
\end{split}
\label{YY2y-2}
\end{align} 

\noi
for any $0 < \dl' \ll 1$, where we used $- \al + \eps < 2\al - 3 - \eps$ given $\al > 1$ and $\eps > 0$ sufficiently small. 
Combining \eqref{YY2y-0}, \eqref{YY2y-1}, and \eqref{YY2y-2}, we obtain \eqref{YY2y}.
\end{proof}

\begin{lemma}
\label{LEM:regYZ}
Let $1 < \al \leq \frac 98$ and $\eps > 0$. Let $Y_N = Y_N (1) = \pi_N Y (1)$ with $Y$ being defined in \eqref{defY} and $\ZZ_N = \ZZ_N (1) = \pi_N \ZZ^N (1)$ with $\ZZ^N$ being defined in \eqref{defZZ}. Then, 

\smallskip \noi
\textup{(i)} $:\! Y_N^2 \!: \ZZ_N$ is uniformly bounded \textup{(}in $N \in \N$\textup{)} in $L^p ( \O; W^{2 \al - 3 - \eps, \infty} (\T^3) )$ for any $p \geq 1$.

\smallskip \noi
\textup{(ii)} $\E [ \int :\! Y_N^2 \!: \ZZ_N^2 dx ]$ is uniformly bounded in $N \in \N$.
\end{lemma}
\begin{proof}
(i) We recall from \eqref{defZZ} that $\dot{\ZZ}^N = \jb{\nb}^{- 2\al} :\! Y_N^3 \!:$. Below, we ignore the time variable for simplicity. 
In view of the notation \eqref{gn_mul} and \eqref{HkIk}, we can treat $:\! Y_N^2 \!:$ and $\ZZ_N$ as multiple stochastic integrals. For convenience, we denote $n_1, n_2$ be frequencies of the nodes of $:\! Y_N^2 \!:$ and denote $n_3, n_4, n_5$ be frequencies the nodes of $\ZZ_N$. By the product formula in \eqref{prod}, we need to discuss the pairings occurring between $:\! Y_N^2 \!:$ and $\ZZ_N$.

For the case when there is no pairings between $:\! Y_N^2 \!:$ and $\ZZ_N$, we use the Ito isometry \eqref{ito}, \eqref{jensen}, and the discrete convolution inequality in Lemma~\ref{LEM:SUM}~(i) repetitively to obtain that for any $n \in \Z^3$, the contribution for $\E [ | (:\! Y_N^2 \!: \ZZ_N)^\wedge (n) |^2 ]$ is given by
\begin{align}
\begin{split}
&\sum_{\substack{n_1, \dots, n_5 \in \Z^3 \\ n_1 + \cdots + n_5 = n}} \frac{1}{\jb{n_1}^{2 \al} \jb{n_2}^{2 \al} \jb{n_3 + n_4 + n_5}^{4 \al} \jb{n_3}^{2 \al} \jb{n_4}^{2 \al} \jb{n_5}^{2 \al}} \\
&\quad \les \sum_{\substack{n', n'' \in \Z^3 \\ n' + n'' = n}} \frac{1}{\jb{n'}^{4\al - 3 - \eps} \jb{n''}^{10 \al - 6 - \eps}} \les \jb{n}^{3 - 4 \al + \eps}.
\end{split}
\label{YYZ1}
\end{align}

\noi
For the case when there is one pairing between $:\! Y_N^2 \!:$ and $\ZZ_N$, say $(n_1, n_3)$ is a pairing without loss of generality, by the Ito isometry \eqref{ito} and the discrete convolution inequality in Lemma~\ref{LEM:conv}~(i) repetitively, we obtain the following contribution for $\E [ | (:\! Y_N^2 \!: \ZZ_N)^\wedge (n) |^2 ]$:
\begin{align}
\begin{split}
&\sum_{\substack{n_2, n_4, n_5 \in \Z^3 \\ n_2 + n_4 + n_5 = n}} \frac{1}{\jb{n_2}^{2 \al} \jb{n_4}^{2 \al} \jb{n_5}^{2 \al}} \Big( \sum_{n_1 \in \Z^3} \frac{1}{\jb{n_1}^{2 \al} \jb{n_1 + n_4 + n_5}^{2 \al}} \Big)^2 \\
&\quad \les \sum_{\substack{n_2, n_4, n_5 \in \Z^3 \\ n_2 + n_4 + n_5 = n}} \frac{1}{\jb{n_2}^{2 \al} \jb{n_4}^{2 \al} \jb{n_5}^{2 \al} \jb{n_4 + n_5}^{8 \al - 6}}  \les \jb{n}^{- 2 \al}.
\end{split}
\label{YYZ2}
\end{align}

\noi
For the case where there are two pairings between $:\! Y_N^2 \!:$ and $\ZZ_N$, say $(n_1, n_3)$ and $(n_2, n_4)$ are pairings without loss of generality, by the Ito isometry \eqref{ito} and the discrete convolution inequality in Lemma~\ref{LEM:conv}~(i) repetitively, we obtain the following contribution for $\E [ | (:\! Y_N^2 \!: \ZZ_N)^\wedge (n) |^2 ]$:
\begin{align}
\frac{1}{\jb{n}^{2 \al}} \Big( \sum_{n_1, n_2 \in \Z^3} \frac{1}{\jb{n_1}^{2 \al} \jb{n_2}^{2 \al} \jb{n_1 + n_2 + n}^{2 \al}} \Big)^2 \les \jb{n}^{12 - 14 \al}.
\label{YYZ3}
\end{align}

\noi
Combining \eqref{YYZ1}, \eqref{YYZ2}, and \eqref{YYZ3}, we obtain the desired regularity estimate in view of \cite[Lemma~2.6]{GKO2}.

\smallskip \noi
(ii) We denote by $n_1, n_2$ the frequencies of the nodes of $:\! Y_N^2 \!:$, $n_3, n_4, n_5$ the frequencies of the nodes of the first $\ZZ_N$, and $n_6, n_7, n_8$ the frequencies of the nodes of the second $\ZZ_N$. Due to the appearance of the expectation, in order to obtain a non-zero contribution, all frequencies must be paired. Note that the only possibility (up to permutation) of pairings is $\{(n_1, n_3), (n_2, n_6), (n_4, n_7), (n_5, n_8)\}$. Then, by the discrete convolution inequality in Lemma~\ref{LEM:conv}~(i) repetitively, we obtain
\begin{align*}
\bigg| &\E \bigg[ \int_{\T^3} :\! Y_N^2 \!: \ZZ_N^2 dx \bigg] \bigg| \\
&\les \sum_{n_1, n_2, n_4, n_5 \in \Z^3} \frac{1}{\jb{n_1}^{2 \al} \jb{n_2}^{2 \al} \jb{n_4}^{2 \al} \jb{n_5}^{2 \al} \jb{n_1 + n_4 + n_5}^{2 \al} \jb{n_2 + n_4 + n_5}^{2 \al}} \\
&\les \sum_{n_4, n_5 \in \Z^3} \frac{1}{\jb{n_4}^{2 \al} \jb{n_5}^{2 \al} \jb{n_4 + n_5}^{8 \al - 6}} \les 1.
\end{align*}

\noi
Therefore, we have finished the proof.
\end{proof}

We now have all the necessary tools to prove Theorem~\ref{THM:Gibbs}~(ii).  
\begin{proof}[Proof of Theorem~\ref{THM:Gibbs}~(ii)]
Let us mainly show the following uniform $L^1$-bound:
\begin{align}
\sup_{N \in \N} \big\| e^{- R_N^\dia (u)} \big\|_{L^1 (\mu_\al)} < \infty.
\label{uni_bdd}
\end{align}

\noi
Recalling $\wt R^{\dia}_N$ in \eqref{RNt}, by using Lemma~\ref{LEM:Dr1}, Lemma~\ref{LEM:Dr2}, Lemma~\ref{LEM:regYZ}, and Lemma~\ref{LEM:YNk}, we obtain
\begin{align}
\begin{split}
\E &\bigg[  - \wt R^{\dia}_{N}(Y + \Ups^{N} - \ZZ_{N}) - \frac12 \int_0^1 \| \dot{\Ups}^{N} (t) \|_{H_x^\al}^2 dt \bigg] \\
&\leq \E \bigg[ C(\dl, \eps) \Big( 1 + \| Y_N \|_{W^{\al - \frac 32 - \eps, \infty}}^c + \| :\! Y_N^2 \!: \|_{W^{2\alpha - 3 - \eps, \infty}}^{c}
+ \| \ZZ_N \|_{W^{5\al - \frac 92 - \eps, \infty}}^c  \\
&\qquad +  \| :\! Y_N^2 \!: \ZZ_N \|_{W^{2 \al - 3 - \eps, \infty}}^2 \Big) - \frac 32 \int_{\T^3} :\! Y_N^2 \!: \ZZ_N^2 dx \bigg]  \\
&\quad + \E \bigg[ \dl \| \Ups_N \|_{L^4}^4 + \dl \| \Ups_N \|_{H^\al}^2 - \frac 18 \| \Ups_N \|_{L^4}^4 - \frac12 \int_0^1 \| \dot{\Ups}^{N} (t) \|_{H_x^\al}^2 dt  \bigg]  \\
&\leq C(\eps),
\end{split}
\label{logZN}
\end{align}

\noi
where we took $\dl > 0$ sufficiently small with constants $c, C(\dl, \eps), C(\eps) > 0$ and we used
\begin{align}
\| \Ups_N (1) \|_{H^\al}^2 \leq \int_0^1 \| \dot{\Ups}^N (t) \|_{H_x^\al}^2 dt
\label{UNbdd}
\end{align}

\noi
due to Minkowski's integral inequality, the Cauchy-Schwarz inequality, and the fact that $\Ups^N (0) = 0$.
The uniform bound \eqref{uni_bdd} then follows directly from \eqref{BDRNd} and \eqref{logZN}.

The rest of the argument has been performed several times in other literature. Briefly speaking, we first show tightness of the frequency truncated Gibbs measures $\{ \rho_N \}_{N \in \N}$ in \eqref{GibbsN2}.
Then, from Prokhorov's theorem, there exists a weakly convergent subsequence of $\{ \rho_N \}_{N \in \N}$. Finally, we show uniqueness of the limit of any subsequence of $\{Z_N\}_{N \in \N}$ and the uniqueness of the limiting Gibbs measure, the latter of which implies the weak convergence of the whole sequence $\{ \rho_N \}_{N \in \N}$ and so finishes the proof of Theorem~\ref{THM:Gibbs}~(ii). 
Note that the tightness part requires us to show the boundedness of the normalizing constants $\{ Z_N \}_{N \in \N}$. The upper bound of $\{ Z_N \}_{N \in \N}$ follows from \eqref{uni_bdd} and to show a lower bound, we simply take the drift function $\dot{\Ups}^N$ to be zero and use \eqref{mean0X}, Lemma~\ref{LEM:Dr1}, Lemma~\ref{LEM:Dr2}, Lemma~\ref{LEM:regYZ}, and Lemma~\ref{LEM:YNk} to obtain
\begin{align}
\begin{split}
- \log Z_N &\leq \E \big[ \wt R_N^\dia (Y + 0 - \ZZ_N) \big] \\
&\leq 1 + \E \big[ \| Y_N \|_{W^{\al - \frac 32 - \eps, \infty}}^c + \| \ZZ_N \|_{W^{5 \al - \frac 92 - \eps, \infty}}^c + \| :\! Y_N^2 \!: \ZZ_N \|_{W^{2 \al - 3 - \eps, \infty}}  \big] \\
&\quad + \bigg| \E \bigg[ \int_{\T^3} :\! Y_N^2 \!: \ZZ_N^2 dx \bigg] \bigg| \\
&\les 1 .
\end{split}
\label{ZNlow}
\end{align}

\noi
All the other required estimates are already established in Lemma~\ref{LEM:Dr1}, Lemma~\ref{LEM:Dr2}, and Lemma~\ref{LEM:regYZ}, and so we omit details and refer the readers to \cite[Subsection~6.2 and Proposition~6.6]{OOTol1} and \cite[Subsection~3.2 and Proposition~3.8]{OOT2}.
\end{proof}

\begin{remark} \rm
\label{RMK:ZZY}
From \eqref{defY}, \eqref{defZZ}, and Lemma~\ref{LEM:YNk}, and \eqref{ZZdiv}, we see that $Y$ has regularity $\al - \frac 32 -$ and $\ZZ_N$ has (uniform-in-$N$) regularity $5 \al - \frac 92 -$. Then, $\ZZ_N$ is more regular than $Y$ if $5 \al - \frac 92 > \al - \frac 32$, which is equivalent to $\al > \frac 34$. This provides further support for Remark~\ref{RMK:quad}, where we conjectured that the fractional $\Phi^4_3$-measure can be constructed in the range $\al > \frac 34$. Indeed, the heuristic here is that $\ZZ_N$ is considered as a perturbation of $Y$, so that things are expected to fail if $\ZZ_N$ is rougher than $Y$. 
\end{remark}

\subsection{Singularity of the fractional $\Phi^4_3$-measure for $1 < \alpha \leq \frac 98$}
\label{SUBSEC:sing}

In this subsection, we prove that when $1 < \al \leq \frac 98$, the fractional $\Phi^4_3$-measure $\rho$ constructed in Theorem~\ref{THM:Gibbs}~(ii) is singular with respect to the base Gaussian measure $\mu_\al$. We mainly follow the approach in \cite[Subsection~6.4]{OOTol1} and \cite[Subsection~3.4]{OOT2}, where the original inspiration comes from \cite[Section~4]{BG21}.

Given $N \in \N$, we define the constants $a_N$ and $b_N$ by
\begin{align}
a_N \deff \sum_{|n|_\infty \leq N} \jb{n}^{-8 \al + 6} \sim
\begin{cases}
\log N, &\text{if } \al = \frac 98, \\
N^{- 8 \al + 9}, &\text{if } 1 < \al < \frac 98,
\end{cases}
\label{defAN}
\end{align}
and
\begin{align}
b_N \deff (\log N)^{-\frac 18} a_N^{- \frac 12} \sim
\begin{cases}
(\log N)^{- \frac 58}, &\text{if } \al = \frac 98, \\
(\log N)^{- \frac 18} N^{4 \al - \frac 92}, &\text{if } 1 < \al < \frac 98.
\end{cases} 
\label{defBN}
\end{align}


\begin{proposition}[Singularity of the fractional $\Phi^4_3$-measure]
\label{PROP:sing}
Let $1 < \al \leq \frac 98$ and $\eps > 0$. Let $R_N$ be as in \eqref{RNu}. Then, there exists an increasing sequence $\{ N_k \}_{k \in \N} \subset \N$ such that the set
\begin{align*}
S \deff \big\{ u \in H^{\al - \frac 32 - \eps} (\T^3): \lim_{k \to \infty} b_{N_k} R_{N_k} (u) = 0 \big\}
\end{align*} 

\noi
satisfies $\mu_\al (S) = 1$ and $\rho(S) = 0$.
Consequently, fractional $\Phi^4_3$-measure $\rho$ and the base Gaussian measure $\mu_\al$ are mutually singular in the case $1 < \al \leq \frac 98$.
\end{proposition}

\begin{proof}
Using \eqref{RNu}, the Ito isometry \eqref{itoX}, the discrete convolution inequality in Lemma~\ref{LEM:SUM}~(i) repetitively, and \eqref{defAN}, we have
\begin{align}
\| R_N (u) \|_{L^2 (\mu_\al)}^2 \sim \sum_{\substack{n_1 + \cdots + n_4 = 0 \\ |n_1|_\infty, \dots, |n_4|_\infty \leq N}} \frac{1}{\jb{n_1}^{2 \al} \cdots \jb{n_4}^{2 \al}} \les \sum_{|n_1|_\infty \leq N} \frac{1}{\jb{n_1}^{- 6 + 8 \al}} = a_{N}.
\label{RNL2}
\end{align}

\noi
By \eqref{defBN} and \eqref{RNL2}, we have
\begin{align*}
\lim_{N \to \infty} b_N \| R_N (u) \|_{L^2 (\mu)} \les \lim_{N \to \infty} (\log N)^{- \frac 18} = 0.
\end{align*}

\noi
Thus, there exists a subsequence $\{N_k\}_{k \in \N}$ such that
\begin{align*}
\lim_{k \to \infty} b_{N_k} R_{N_k} (u) = 0
\end{align*}
$\mu_\al$-almost surely. This shows that $\mu_\al (S) = 1$.

We now show that $e^{b_{N_k} R_{N_k} (u)}$ tends to 0 in $L^1 (\rho)$, which will imply that up to a further subsequence,
\begin{align*}
\lim_{k \to \infty} b_{N_k} R_{N_k} (u) = - \infty
\end{align*}

\noi
$\rho$-almost surely, so that $\rho (S) = 0$ which finishes our proof. Let $\phi$ be a smooth bump function on $\R$ such that $\phi \equiv 1$ on $[-1, 1]$ and $\phi \equiv 0$ outside of $[-2, 2]$. By Fatou's lemma and the weak convergence of the truncated fractional $\Phi^4_3$-measure $\rho_M$ to $\rho$ as stated in Theorem~\ref{THM:Gibbs}~(ii), we have
\begin{align}
\begin{split}
\int e^{b_{N_k} R_{N_k} (u)} d \rho (u) &\leq \liminf_{K \to \infty} \int \phi \Big( \frac{b_{N_k} R_{N_k} (u)}{K} \Big) e^{b_{N_k} R_{N_k} (u)} d \rho (u) \\
&= \liminf_{K \to \infty} \lim_{M \to \infty} \int \phi \Big( \frac{b_{N_k} R_{N_k} (u)}{K} \Big) e^{b_{N_k} R_{N_k} (u)} d \rho_M (u) \\
&\leq \lim_{M \to \infty} \int e^{b_{N_k} R_{N_k} (u)} d \rho_M (u) \\
&= Z^{-1} \lim_{M \to \infty} \int e^{b_{N_k} R_{N_k} (u) -R_M^\dia(u)} d \mu_\al (u),
\end{split}
\label{eGk1}
\end{align}

\noi
where $Z = \lim_{M \to \infty} Z_M$ makes sense as mentioned in the proof of Theorem~\ref{THM:Gibbs}~(ii) and we recall that $R_M^\dia = R_M + \al_M$ with $\al_M$ being defined in \eqref{defalN}.
By the Bou\'e-Dupuis variational formula (Lemma~\ref{LEM:BD0}) and \eqref{BDRNd}, we have
\begin{align*}
\log &\int e^{b_{N_k} R_{N_k} (u) -R_M^\dia(u)} d \mu_\al (u) \\
&\leq 
\sup_{\dot{\Ups}^M \in \mathbb{H}_a^\al} \E \bigg[ b_{N_k} R_{N_k} (Y + \Ups^M - \ZZ_M) - \wt{R}_M^\dia (Y + \Ups^M - \ZZ_M) - \frac 12 \int_0^1 \| \dot{\Ups}^M (t) \|_{H_x^\al}^2 dt \bigg] ,
\end{align*}

\noi
where $Y = Y(1)$ is defined in \eqref{defY}, $\ZZ_M = \ZZ_M (1) = \pi_M \ZZ^M (1)$ is defined in \eqref{defZZ}, and $\wt{R}_M^\dia$ is defined as in \eqref{RNt}. By \eqref{logZN} and \eqref{UNbdd}, we have
\begin{align}
\begin{split}
\log &\int e^{b_{N_k} R_{N_k} (u) -R_M^\dia(u)} d \mu_\al (u) \\
&\leq 
\sup_{\dot{\Ups}^M \in \mathbb{H}_a^\al} \E \bigg[ b_{N_k} R_{N_k} (Y + \Ups^M - \ZZ_M) - \frac{1}{16} \| \Ups_M \|_{L^4}^4 - \frac 14  \int_0^1 \| \dot{\Ups}^{M} (t) \|_{H_x^\al}^2 dt  \bigg] + C 
\end{split}
\label{eGk2}
\end{align}

\noi
for some constant $C > 0$, where $\Ups_M = \pi_M \Ups^M$.

From \eqref{RNu}, we know that for $1 \leq N_k \leq M$,
\begin{align}
\begin{split}
R_{N_k} &(Y + \Ups^M - \ZZ_M) = \frac 14 \int_{\T^3} :\!Y_{N_k}^4\!: dx + \int_{\T^3} :\!Y_{N_k}^3\!: \pi_{N_k} \Dr dx \\
&+ \frac 32 \int_{\T^3} :\!Y_{N_k}^2\!: (\pi_{N_k} \Dr)^2 dx + \int_{\T^3} Y_{N_k} (\pi_{N_k} \Dr)^3 dx + \frac 14 \int_{\T^3} (\pi_{N_k} \Dr)^4 dx,
\end{split}
\label{RNk}
\end{align}

\noi
where $Y_{N_k} = \pi_{N_k} Y$ and $\Dr = \Ups^{M} - \ZZ_{M}$. Using Lemma~\ref{LEM:Dr1}, Lemma~\ref{LEM:YNk}, and the $L^4$-boundedness of the smooth frequency cutoff $\pi_{N_k}$ along with $N_k \leq M$, we have the bound
\begin{align}
\begin{split}
\E \bigg[ \frac 14 \int_{\T^3} (\pi_{N_k} \Dr)^4 dx \bigg] &\leq \frac 12 \E \big[ \| \pi_{N_k} \Ups^M \|_{L^4}^4 \big] + C \E \big[ \| \pi_{N_k} \ZZ_M \|_{W^{5\al - \frac 92 - \eps, \infty}}^4 \big] \\
&\les \E \big[ \| \Ups_M \|_{L^4}^4 \big] + C
\end{split}
\label{RNk5}
\end{align}

\noi
for some constant $C > 0$. Similarly, using Lemma~\ref{LEM:Dr2}, Lemma~\ref{LEM:regYZ} (which still works with two different frequency cutoffs), and Lemma~\ref{LEM:YNk}, we bound the third and the fourth term on the right-hand-side of \eqref{RNk} by
\begin{align}
\E \bigg[ \frac 32 \int_{\T^3} :\!Y_{N_k}^2\!: (\pi_{N_k} \Dr)^2 dx \bigg] \les \E \big[ \| \Ups_M \|_{L^4}^4 \big] + \E \big[ \| \Ups_M \|_{H^\al}^2 \big] + C,
\label{RNk3}
\end{align}
\begin{align}
\E \bigg[ \int_{\T^3} Y_{N_k} (\pi_{N_k} \Dr)^3 dx \bigg] \les \E \big[ \| \Ups_M \|_{L^4}^4 \big] + \E \big[ \| \Ups_M \|_{H^\al}^2 \big] + C.
\label{RNk4}
\end{align}

\noi
Thus, by \eqref{eGk2}, \eqref{RNk}, \eqref{mean0X}, \eqref{RNk5}, \eqref{RNk3}, \eqref{RNk4}, \eqref{UNbdd}, and \eqref{defBN}, we can take $M$ and $k$ large enough so that
\begin{align}
\begin{split}
\log &\int e^{b_{N_k} R_{N_k} (u) - R_M^\dia(u)} d \mu_\al (u) \\
&\leq 
\sup_{\dot{\Ups}^M \in \mathbb{H}_a^\al} \E \bigg[ b_{N_k} \int_{\T^3} :\! Y_{N_k}^3 \!: \pi_{N_k} \Dr dx - \frac{1}{32} \| \Ups_M \|_{L^4}^4 - \frac 18 \int_0^1 \| \dot{\Ups}^{N} (t) \|_{H_x^\al}^2 dt \bigg] + C .
\end{split}
\label{eGk3}
\end{align}

\noi
for some constant $C > 0$.

We now show that the contribution from the first term on the right-hand-side of \eqref{eGk3} tends to $- \infty$ as $M \to \infty$ and $k \to \infty$. 
By the Ito product formula, \eqref{defUps}, Lemma~\ref{LEM:EZZNM} below, and the Cauchy-Schwarz inequality, we can compute that for $M \gg N_k \geq 1$,
\begin{align}
\begin{split}
\E &\bigg[ \int_{\T^3} :\! Y_{N_k}^3 \!: \pi_{N_k} \Dr dx \bigg] \\
&= \E \bigg[ \int_0^1 \int_{\T^3} :\! Y_{N_k}^3 (t) \!: \pi_{N_k} \dot{\Dr} (t) dx dt \bigg] \\
&= \E \bigg[ \int_0^1 \int_{\T^3} \jb{\nb}^{2 \al} \dot{\ZZ}_{N_k} (t) \big( \dot{\ZZ}_{N_k} (t) - \dot{\ZZ}_M (t) \big) dx dt \bigg] \\
&\quad - \E \bigg[ \int_0^1 \| \dot{\ZZ}_{N_k} (t) \|_{H_x^{\al}}^2 dt \bigg] + \E \bigg[ \int_0^1 \int_{\T^3} \jb{\nb}^{2 \al} \dot{\ZZ}_{N_k} (t) \pi_{N_k} \dot{\Ups}^M (t) dx dt \bigg] \\
&\leq - \frac 12 \E \bigg[ \int_0^1 \| \dot{\ZZ}_{N_k} (t) \|_{H_x^{\al}}^2 dt \bigg] + \frac 12 \E \bigg[ \int_0^1 \| \dot{\Ups}^M (t) \|_{H_x^\al}^2 dt \bigg] .
\end{split}
\label{EY3}
\end{align}

\noi
Note that by \eqref{defZZ}, the Ito isometry \eqref{itoX}, \eqref{convsum1-1} in Lemma~\ref{LEM:SUM}~(i) twice, and \eqref{defAN}, we have for any $0 \leq t \leq 1$ that
\begin{align}
\begin{split}
\E \big[ \| \dot{\ZZ}_{N_k} (t) \|_{H_x^\al}^2 \big] &\sim \sum_{|n|_\infty \leq N_k} \jb{n}^{-2 \al} \sum_{\substack{n_1 + n_2 + n_3 = n \\ |n_1|_\infty, |n_2|_\infty, |n_3|_\infty \leq N_k}} \frac{t^3}{\jb{n_1}^{2 \al} \jb{n_2}^{2 \al} \jb{n_3}^{2 \al}} \\
&\sim t^3 \sum_{|n|_\infty \leq N_k} \jb{n}^{-8 \al + 6} = t^3 a_{N_k}
\end{split}
\label{eGk4}
\end{align}

\noi
Combining \eqref{eGk3}, \eqref{EY3}, and \eqref{eGk4}, we can take $M$ and $k$ large enough to obtain
\begin{align*}
\log \int e^{b_{N_k} R_{N_k} (u) - R_M^\dia(u)} d \mu_\al (u) \leq - C b_{N_k} a_{N_k} + C_1 \to - \infty
\end{align*} 
as $M \to \infty$ and $k \to \infty$ for $1 < \al \leq \frac 98$ due to \eqref{defAN} and \eqref{defBN}, where $C, C_1 > 0$ are constants. Thus, by \eqref{eGk1}, we obtain
\begin{align*}
\lim_{k \to \infty} \int e^{b_{N_k} R_{N_k} (u)} d \rho (u) = 0,
\end{align*}
which completes the proof of the proposition.
\end{proof}

\begin{lemma}
\label{LEM:EZZNM}
Let $\al > 1$. Then, we have
\begin{align*}
\E \bigg[ \int_0^1 \int_{\T^3} \jb{\nb}^{2 \al} \dot{\ZZ}_{N_k} (t) \big( \dot{\ZZ}_{N_k} (t) - \dot{\ZZ}_M (t) \big) dx dt \bigg] = 0
\end{align*}

\noi
for $M \geq N \geq 1$.
\end{lemma}

\begin{proof}
From \eqref{defY}, we can write
\begin{align*}
Y (t) = \frac{1}{(2 \pi)^{\frac 32}} \sum_{n \in \Z^3} \frac{1}{\jb{n}^\al} \int_0^t 1 d B_n (t') e^{in \cdot x} ,
\end{align*}

\noi
so that by \eqref{HkIk}, we have
\begin{align*}
\ft{\dot{\ZZ}}_N (n, t) = \sum_{\substack{n_1 + n_2 + n_3 = n \\ |n_1|_\infty, |n_2|_\infty, |n_3|_\infty \leq N}} \frac{\ind_{\{ |n|_\infty \leq N \}}}{\jb{n}^{2 \al} \jb{n_1}^\al \jb{n_2}^\al \jb{n_3}^\al} \int_0^t \int_0^t \int_0^t 1 \, d B_{n_1} (t) d B_{n_2} (t) d B_{n_3} (t).
\end{align*}

\noi
Note that for $|n|_\infty \leq N$, we have
\begin{align*}
&\ft{\dot{\ZZ}}_M (n, t) - \ft{\dot{\ZZ}}_N (n, t) \\
&= \sum_{\substack{n_1 + n_2 + n_3 = n \\ N < \max ( |n_1|_\infty, |n_2|_\infty, |n_3|_\infty ) \leq M}} \frac{1}{\jb{n}^{2 \al} \jb{n_1}^\al \jb{n_2}^\al \jb{n_3}^\al} \int_0^t \int_0^t \int_0^t 1 \, d B_{n_1} (t) d B_{n_2} (t) d B_{n_3} (t).
\end{align*}

\noi
Note that the sets
\begin{align*}
&\big\{ (n_1, n_2, n_3) \in \Z^3 : n_1 + n_2 + n_3 = n, \, |n_j|_\infty \leq N \text{ for } j = 1, 2, 3 \big\}, \\
&\big\{ (n_1, n_2, n_3) \in \Z^3 : n_1 + n_2 + n_3 = n, \, N < \max (|n_1|_\infty, |n_2|_\infty, |n_3|_\infty) \leq M \big\}
\end{align*}

\noi
are disjoint. Thus, by the Ito isometry \eqref{ito}, for $|n|_\infty \leq N$ we have
\begin{align*}
\E \Big[ \ft{\dot{\ZZ}}_N (n, t) \big( \ft{\dot{\ZZ}}_N (n, t) - \ft{\dot{\ZZ}}_M (n, t) \big) \Big] = 0,
\end{align*}

\noi
so that we have
\begin{align*}
\E &\bigg[ \int_0^1 \int_{\T^3} \jb{\nb}^{2 \al} \dot{\ZZ}_{N_k} (t) \big( \dot{\ZZ}_{N_k} (t) - \dot{\ZZ}_M (t) \big) dx dt \bigg] \\
&= \int_0^1 \sum_{|n|_\infty \leq N} \jb{n}^{2 \al} \E \Big[ \ft{\dot{\ZZ}}_N (n, t) \big( \ft{\dot{\ZZ}}_N (n, t) - \ft{\dot{\ZZ}}_M (n, t) \big) \Big] dt = 0.
\end{align*}
\end{proof}

\begin{remark} \rm
\label{RMK:shift}
By adapting the argument in \cite[Appendix~C]{OOTol1} and \cite[Appendix~A]{OOT2} (see also \cite[Section~3]{BG21}), we can show that the fractional $\Phi^4_3$-measure $\rho$ is absolutely continuous with respect to the shifted measure $\Law (Y - \ZZ)$, where $Y$ is defined in \eqref{defY} and $\ZZ$ is essentially the limit of $\ZZ^N$ in \eqref{defZZ} (see \cite{OOTol1, OOT2} for a more precise definition). 

From \eqref{ZZdiv}, we know that $\ZZ$ has regularity $5\al - \frac 92 -$, which is smaller than $\al$ given $1 < \al \leq \frac 98$. This means that $\ZZ$ almost surely does not lie in the Cameron-Martin space $H^\al (\T^3)$ when $1 < \al \leq \frac 98$.
In view of this phenomenon, the threshold for the phase transition to singularity appears to be closely related to the regularity of the Cameron-Martin space, reflecting the role of this space in Girsanov's theorem. 
It would be of interest to make this heuristic rigorous by proving directly that the random shift $Y \mapsto Y - \ZZ(Y)$ transports the base Gaussian measure to a measure that is singular with respect to it.
\end{remark}

\section{Local well-posedness with Gaussian initial data}
\label{SEC:LWP}

We now study the dynamical problem for the renormalized fractional hyperbolic $\Phi^4_3$-model. Let us recall the frequency truncated fractional hyperbolic $\Phi^4_3$-model \eqref{SNLW2} for readers' convenience:
\begin{align}
\dt^2 u_N + \dt u_N + (1 - \Dl)^\al u_N + \pi_N :\! (\pi_N u_N)^3 \!: \, = \sqrt{2} \xi
\label{SNLW4}
\end{align}

\noi
with
\begin{align*}
:\! (\pi_N u_N)^3 \!: \, = (\pi_N u_N)^3 - 3 \s_N \pi_N u_N,
\end{align*}

\noi
where $\s_N$ is as defined in \eqref{sigmaN}.
In this section, we consider local-in-time convergence of the solution $u_N$ for the equation \eqref{SNLW4} with Gaussian initial data distributed according to the base Gaussian measure $\mu_\al \otimes \mu_0$ in \eqref{gauss2}. Our goal is to prove the following statement. 

\begin{proposition}[Local well-posedness for the fractional hyperbolic $\Phi^4_3$-model]
\label{PROP:LWP}
Let $1 < \al \leq \frac 32$. Then, for any $N \in \N$, there exists an almost surely positive stopping time $T = T(\o)$ independent of $N$ such that the solution $\vec u_N$ to \eqref{SNLW4} with Gaussian initial data distributed by $\mu_\al \otimes \mu_0$ in \eqref{gauss2} exists in the class $C([0, T]; \vec H^{\al - \frac 32 - \eps} (\T^3))$. Moreover, the sequence of solutions $\{ \vec u_N \}_{N \in \N}$ converges in $C([0, T]; \vec H^{\al - \frac 32 - \eps} (\T^3))$ as $N \to \infty$.
\end{proposition}

Let us now present our solution ansatz for \eqref{SNLW4}. We denote $\vec{\<1>} = (\<1>, \dt \<1>)$ as the stochastic convolution
\begin{align}
\vec{\<1>} (t) = \vec S(t) (\phi_0, \phi_1) + \sqrt{2} \int_0^t \vec \D (t - t') d W(t'),
\label{loli}
\end{align}

\noi
where $\vec S(t)$ is defined in \eqref{vecSI} and \eqref{St0}, $\vec \D (t)$ is defined in \eqref{vecSI} and \eqref{W3}, $(\phi_0, \phi_1) = (\phi_0^\o, \phi_1^\o)$ are Gaussian random distributions with $\Law (\phi_0^\o, \phi_1^\o) =  \mu_\al \otimes \mu_0$ in \eqref{gauss2}, and $W$ denotes a cylindrical Wiener process on $L^2 (\T^3)$ defined by
\begin{align}
W (t) = \frac{1}{(2 \pi)^{\frac 32}} \sum_{n \in \Z^3} B_n (t) e^{in \cdot x}
\label{defW}
\end{align}

\noi
with $B_n (t) = (2 \pi)^{- 3 / 2} \langle \xi, \ind_{[0, t]} \cdot e^{in \cdot x} \rangle_{x, t}$ satisfying the three properties listed in Subsection~\ref{SUBSEC:mul}. Note that $\<1>$ is the solution to the following linear stochastic damped wave equation:
\begin{align*}
\begin{cases}
\dt^2 \<1> + \dt \<1> + (1 - \Dl)^\al \<1> = \sqrt{2} \xi \\
(\<1>, \dt \<1>) |_{t = 0} = (\phi_0, \phi_1).
\end{cases}
\end{align*}

\noi
Given $N \in \N$,
we define the truncated stochastic terms  $\<1>_N$, $\<2>_N$ and $\<3>_N$ by
\begin{align}
\<1>_N \deff \pi_N \<1>, \quad
\<2>_N \deff \, :\! (\<1>_N)^2 \!: \, = \<1>_N^2-\s_N, \quad
\<3>_N \deff \, :\! (\<1>_N)^3 \!: \, \<1>_N^3 - 3\s_N \<1>_N,
\label{so4a}
\end{align}

\noi
and the second order process
$\vec{\<30>}_N = (\<30>_N, \dt \<30>_N)$ by
\begin{align}
\vec{\<30>}_N \deff \pi_N \vec \I(\<3>_N)=\int_0^t \pi_N \vec \D(t-t')\<3>_N(t')dt',
\label{so4b}
\end{align}

\noi
where a direct computation yields that
\begin{align*}
\s_N = \E \big[ \<1>_N^2(x, t)\big] = \frac{1}{(2 \pi)^3} \sum_{\substack{n \in \Z^3 \\ |n|_\infty \leq N }} \frac{1}{ \jb{n}^{2\al}}.
\end{align*}

\noi
Note that this $\s_N$ agrees with the definition in \eqref{sigmaN}.

Inspired by \cite{OPTz, Bring2, OWZ}, we proceed with the second order expansion
\begin{align}
u_N = \<1> - \<30>_N + v_N,
\label{exp3}
\end{align}

\noi
and rewrite \eqref{SNLW4} as
\begin{align}
\begin{split}
(\dt^2 + \dt + (1-\Dl)^\al) v_N
&=-\pi_N (\pi_N v_N)^3+3\pi_N \big( (\<30>_N-\<1>_N)(\pi_N v_N)^2 \big) \\
&\quad - 3\pi_N \big( (\<30>_N^2 - 2 \<30>_N \<1>_N)\pi_N v_N \big) - 3\pi_N (\<2>_N \pi_N v_N)\\
&\quad + \pi_N (\<30>_N^3)-3\pi_N (\<30>_N^2 \<1>_N) + 3\pi_N (\<30>_N\<2>_N).
\end{split}
\label{SfNLW9}
\end{align}

\noi
Taking $N\to\infty$, we formally obtain the following limiting equation:
\begin{align*}
\begin{split}
(\dt^2+\dt+(1-\Dl)^\al)v 
&= - v^3 + 3\big( (\<30>-\<1>)v^2 \big) - 3\big( (\<30>^2 - 2 \<30>\<1>) v \big) - 3\<2> \, v \\
&\quad + \<30>^3 - 3 \<30>^2 \<1> + 3\<30> \, \<2>,
\end{split}
\end{align*}

\noi
where the tree-structured stochastic objects appearing above are corresponding limits of the stochastic objects in \eqref{so4a} and \eqref{so4b} as $N \to \infty$ (if the limits exist).

Let us introduce more tree-structured stochastic objects for later uses:
\begin{align}
\vec{\<320>}_N &\overset{\text{def}}{=} \pi_N \vec \I(\<30>_N \<2>_N), \label{sto6} \\
\vec{\<70>}_N &\overset{\text{def}}{=} \pi_N  \vec \I(\<30>_N^2 \<1>_N), \label{sto6a} \\
\vec \If^{\,\<2>_N}(v) &\overset{\text{def}}{=} \pi_N  \vec \I(\<2>_N \cdot v), \label{ran1} \\
\vec \If^{\,\<31>_N} (v) &\overset{\text{def}}{=} \pi_N \vec \I \big( ( \<30>_N \<1>_N ) \cdot  v \big),   \label{sto7a}
\end{align}

\noi 
where $\I$ is the Duhamel operator defined in \eqref{lin1}. Putting everything together, by writing \eqref{SfNLW9} with zero initial data in the Duhamel formulation, we have
\begin{align}
\begin{split}
\vec v_N
&=\pi_N \vec \I\big(-(\pi_N v_N)^3+3(\<30>_N-\<1>_N)(\pi_N v_N)^2 -3\<30>_N^2 (\pi_N v_N)\big) \\
&\quad + 6 \vec \If^{\,\<31>_N} ( \pi_N v_N ) - 3 \vec \If^{\,\<2>_N} (\pi_N v_N) + \pi_N \vec \I\big( \<30>_N^3 \big) - 3 \vec{\<70>}_N + 3 \vec{\<320>}_N.
\end{split}
\label{SfNLW10}
\end{align}

\noi
We can then take $N\to\infty$ in \eqref{SfNLW10} and obtain the following (formal) limiting equation for $v=u-\<1>+\<30>$:
\begin{align}
\begin{split}
\vec v &= \vec \I\big(-v^3 + 3(\<30>-\<1>)v^2-3\<30>^2v\big)
+ 6 \vec \If^{\,\<31>} (v) \\
&\quad - 3 \vec \If^{\,\<2>}(v) +\I\big(\<30>^3\big) - 3\vec{\<70>} +3\vec{\<320>},
\end{split}
\label{SfNLW11}
\end{align}

\noi
where the tree-structured stochastic objects appearing in \eqref{SfNLW11} are corresponding limits of the stochastic objects in \eqref{so4a}, \eqref{so4b}, \eqref{sto6}, \eqref{sto6a}, \eqref{ran1}, and \eqref{sto7a} as $N \to \infty$ (if the limits exist).

\subsection{Stochastic terms and random operators}
\label{SUBSEC:para}

In this subsection, we establish regularity properties of stochastic terms and
the random operators.
We mainly follow the presentation in \cite{OWZ}.

Before stating the main results, we need to define the notation of an operator norm. 
For $s_1, s_2, b_1, b_2 \in \R$ and $T > 0$, we set
\begin{align*}
 \L^{s_1, b_1, s_2, b_2}_{T}
= \bigcap_{0 < T' \leq T} \L\big(X^{s_1, b_1}_{T'}; X^{s_2, b_2}_{T'}\big)
\end{align*}

\noi
endowed
with the norm given by
\begin{align}
\| S\|_{\L^{s_1, b_1, s_2, b_2}_T}
= \sup_{0 < T' \leq T} (T')^{- \ta} \|S\|_{\L(X^{s_1, b_1}_{T'}; X^{s_2, b_2}_{T'})}
\label{Op2}
\end{align}

\noi
for some small $\ta = \ta (b_1, b_2) > 0$. 
We also write $\vec \L^{\,s_1, b_1, s_2, b_2}_{T} = \L^{s_1, b_1, s_2, b_2}_{T} \times \L^{s_1, b_1, s_2 - \al, b_2}_{T}$.

Let us first show the following lemma regarding the regularity and convergence of the stochastic objects defined at the beginning of Section \ref{SEC:LWP}.
\begin{lemma}
\label{LEM:sto_reg}
Let $1 < \al \leq \frac 32$ and $T > 0$. Let $\eps = \eps(\al) > 0$ and $\dl = \dl (\al) > 0$ be sufficiently small. We also let $s \in \R$ be such that
$\max (- \al + \frac 32, \al - 1) < s < 2 \al - \frac 32$. We define general stochastic objects $\Xi_N$ and $\Xi$, the function space $X$, an integer $k$ to be any one of the following cases:
\begin{itemize}
\item[(i)] $\Xi_N = \vec{\<1>}_N = \pi_N \vec{\<1>}$ as in \eqref{loli}, $\Xi = \<1>$, $X = C ([0, T]; \vec W^{\al - \frac 32 - \eps, \infty} (\T^3))$, $k = 1$;

\smallskip 
\item[(ii)] $\Xi_N = \<2>_N$ as in \eqref{so4a}, $\Xi = \<2>$, $X = C ([0, T]; W^{2\al - 3 - \eps, \infty} (\T^3))$, $k = 2$;

\smallskip 
\item[(iii)] $\Xi_N = \vec{\<30>}_N$ as in \eqref{so4b}, $\Xi = \vec{\<30>}$, $X = C ([0, T]; \vec W^{ 3 \al - 3 - \eps, \infty} (\T^3))$ or $\vec X^{3 \al - 3 - \eps, \frac 12 + \dl} ([0, T])$, $k = 3$;

\smallskip 
\item[(iv)] $\Xi_N = \vec{\<320>}_N$ as in \eqref{sto6}, $\Xi = \vec{\<320>}$, $X = \vec X^{\al - \frac 12, \frac 12 + \dl} ([0, T])$, $k = 5$;

\smallskip
\item[(v)] $\Xi_N = \vec{\<70>}_N$ as in \eqref{sto6a}, $\Xi = \vec{\<70>}$, $X = \vec X^{\al - \frac 12, \frac 12 + \dl} ([0, T])$, $k = 7$;

\smallskip
\item[(vi)] $\Xi_N = \vec \If^{\,\<31>_N}$ as in \eqref{sto7a}, $\Xi = \vec \If^{\,\<31>}$, $X = \vec \L_T^{\,\al - \frac 12 - \eps, \frac 12 + \dl, \al - \frac 12 + \eps, \frac 12 + \dl}$, $k = 4$;

\smallskip
\item[(vii)] $\Xi_N = \vec \If^{\,\<2>_N}$ as in \eqref{ran1}, $\Xi = \vec \If^{\,\<2>}$, $X = \vec \L_T^{\,s, \frac 12 + \dl, s + \eps, \frac 12 + \dl}$, $k = 2$.
\end{itemize}
Then, $\{ \Xi_N \}_{N \in \N}$ is almost surely a Cauchy sequence in the space $X$, and we denote the limit by $\Xi$. Also, we have the tail estimate
\begin{align}
\PP \big( \| \Xi_N \|_X > \ld \big) \leq C(T) \exp \big( - c (T) \ld^{\frac 2k} \big)
\label{tail}
\end{align}

\noi
for any $\ld > 0$ and some constants $C(T), c(T) > 0$ depending on $T$, uniformly in $N \in \N \cup \{\infty\}$ with the understanding that $\Xi_\infty = \Xi$. Moreover, there exists small $\g > 0$ such that
\begin{align}
\PP \big( N_2^\g \| \Xi_{N_1} - \Xi_{N_2} \|_X > \ld \big) \leq C(T) \exp \big( - c (T) \ld^{\frac 2k} \big)
\label{tail_diff}
\end{align}

\noi
for any $\ld > 0$ and some constants $C(T), c(T) > 0$ depending on $T$, uniformly in $N_1 \geq N_2 \geq 1$.
\end{lemma}

\begin{remark} \rm
\label{RMK:gbn}
In view of \eqref{gn_mul}, for convenience of discussion of stochastic objects below, we write
\begin{align*}
g_n = \int_0^1 1 \, d B_n^0 (t) \quad \text{and} \quad h_n = \int_0^1 1 \, d B_n^1 (t),
\end{align*}

\noi
where $\{ B_n^0 \}_{n \in \Z^3}$ and $\{ B_n^1 \}_{n \in \Z^3}$ are two families of Gaussian processes satisfying the three properties listed in Subsection \ref{SUBSEC:mul}. Also, recalling that $W$ defined in \eqref{defW} is a cylindrical Wiener process on $L^2 (\T^3)$ independent with $\{g_n\}_{n \in \Z^3}$ and $\{h_n\}_{n \in \Z^3}$, we can assume that the three families $\{ B_n \}_{n \in \Z^3}$, $\{ B_n^0 \}_{n \in \Z^3}$, and $\{ B_n^1 \}_{n \in \Z^3}$ are independent. In this way, we can treat the random data term $S(t) (\phi_0, \phi_1)$ and the stochastic convolution term $\int_0^t \D (t - t') d W(t')$ in \eqref{loli} in a uniform manner.
\end{remark}

\begin{proof}[Proof of Lemma \ref{LEM:sto_reg}]
We only present some details for (iii), (vi), and (vii). For the other stochastic objects, the proof follows from a similar manner from some existing literature. For these objects, we only point out the main difference of the proof in our case.

\smallskip \noi
(i) The regularity and convergence of $\vec{\<1>}_N$ follow from the same way as in \cite[Lemma 3.1 (i)]{GKO2}. The tail bounds \eqref{tail} and \eqref{tail_diff} follow from the same way as in \cite[Lemma 2.3]{GKOT}.

\smallskip \noi
(ii) The regularity and convergence of $\<2>_N$ follow from the same way as in \cite[Lemma 3.1 (ii)]{GKO2}. The tail bounds \eqref{tail} and \eqref{tail_diff} follow from the same way as in \cite[Lemma 2.3]{GKOT}.

\smallskip \noi
(iii) For $\vec{\<30>}_N$, we follow the approach in \cite{Bring2} (see also \cite[Lemma~3.1~(ii)]{OWZ}) and first show 
\begin{align*}
\Big\| \big\| \vec{\<30>}_N \big\|_{\vec X_T^{3\al - 3 - \eps, \frac 12 + \dl}} \Big\|_{L^p (\O)} \les p^{\frac 32} e^{\frac 32 T}
\end{align*}

\noi
for any $p \geq 1$ and sufficiently small $\dl > 0$, uniform in $N \in \N$. By Lemma \ref{LEM:nhomo}, we only need to show
\begin{align}
\big\| \| \<3>_N \|_{X_T^{2\al - 3 - \eps, -\frac 12 + \dl}} \big\|_{L^p (\O)} \les p^{\frac 32} e^{\frac 32 T}.
\label{cub_goal}
\end{align}

Note that $\<1>_N$ (recall \eqref{loli} and Remark \ref{RMK:gbn}) is of the form $I_1^0 [g] + I_1^1 [h] + I_1 [f]$ for some $g, h, f \in L^2 (\Z^3 \times \R_+)$, where $I_k^0$ and $I_k^1$ denote the multiple stochastic integrals with respect to $\{ B_n^0 \}_{n \in \Z^3}$ and $\{ B_n^1 \}_{n \in \Z^3}$, respectively, for any $k \in \N$. We denote $\s_g = \| g \|_{L^2 (\R_+ \times \Z^3)}^2$, $\s_h = \| h \|_{L^2 (\R_+ \times \Z^3)}^2$, and $\s_f = \| f \|_{L^2 (\R_+ \times \Z^3)}^2$. By \eqref{so4a}, \eqref{HkIk}, and \eqref{herm_decomp}, we have

\begin{align}
\begin{split}
\<3>_N &= H_3 (I_1^0 [g] + I_1^1 [h] + I_1 [f]; \s_g + \s_h + \s_f) \\
&=  I_3^0 [g \otimes g \otimes g] + I_3^1 [h \otimes h \otimes h] + I_3 [f \otimes f \otimes f] + 3 I_2^0 [g \otimes g] \cdot I_1^1 [h] \\
&\quad + 3 I_2^0 [g \otimes g] \cdot I_1 [f] + 3 I_1^0 [g] \cdot I_2^1 [h \otimes h] + 3 I_2^1 [h \otimes h] \cdot I_1 [f] \\
&\quad + 3 I_1^0 [g] \cdot I_2 [f \otimes f] + 3 I_1^1 [h] \cdot I_2 [f \otimes f] + 6 I_1^0 [g] \cdot I_1^1 [h] \cdot I_1 [f].
\end{split}
\label{sto3_expand}
\end{align}

\noi
By independence, we can abuse notation by denoting
\begin{align}
\begin{split}
I_2^0 [g \otimes g] \cdot I_1^1 [h] &= I_3 [g \otimes g \otimes h], \\
I_2^0 [g \otimes g] \cdot I_1 [f] &= I_3 [g \otimes g \otimes f], \\
I_1^0 [g] \cdot I_2^1 [h \otimes h] &= I_3 [g \otimes h \otimes h], \\
I_2^1 [h \otimes h] \cdot I_1 [f] &= I_3 [h \otimes h \otimes f], \\
I_1^0 [g] \cdot I_2 [f \otimes f] &= I_3 [g \otimes f \otimes f], \\
I_1^1 [h] \cdot I_2 [f \otimes f] &= I_3 [h \otimes f \otimes f], \\
I_1^0 [g] \cdot I_1^1 [h] \cdot I_1 [f] &= I_3 [g \otimes h \otimes f],
\end{split}
\label{sto3_simp}
\end{align}

\noi
where the $I_3$'s on the right-hand-side satisfies the same properties for multiple stochastic integrals as shown in Subsection \ref{SUBSEC:mul}.

In the following, we only consider the $I_3 [f \otimes f \otimes f]$ term, and the estimates for other terms follow from a similar (and easier) treatment. Note that the estimate for the $I_3 [f \otimes f \otimes f]$ term follows essentially from the proof of \cite[Lemma~3.1~(ii)]{OWZ}. For readers' convenience and for later purpose for the proof of our weak universality result, we describe some details below.

By taking the spatial Fourier transform, we have
\begin{align*}
\ft{\<3>}_N (n, t) = I_3 [f_{n, t}].
\end{align*}

\noi
where $f_{n, t}$ is defined by
\begin{align*}
f_{n, t} (n_1, t_1, n_2, t_2, n_3, t_3) = \ind_{\{n = n_{123}\}} \cdot \bigg( \prod_{j = 1}^3 e^{- \frac{t - t_j}{2}} \frac{\sin ((t - t_j) \jbb{n_j})}{\jbb{n_j}} \cdot \ind_{\{|n_j| \leq N\}} \cdot \ind_{[0, t]} (t_j) \bigg)
\end{align*}
By inserting a sharp time cut-off $\ind_{[0, T]}$,  we compute the space-time Fourier transform and use a stochastic Fubini theorem (see \cite[Lemma~B.2]{OWZ} or \cite[Theorem~4.33]{DZ}) to obtain
\begin{align}
\ft{\ind_{[0, T]} \<3>}_N (n, \tau) = I_3 \big[ \F_t (\ind_{[0, T]} f_{n, \cdot}) (\tau) \big],
\label{cub_main}
\end{align}

\noi
where $\F_t$ denotes the Fourier transform in time.

Given dyadic $N_j \in 2^\N$ for $j = 1, 2, 3$, we denote by $\<3>_N^{N_1, N_2, N_3}$ the contribution to $\<3>_N$ from $|n_j| \sim N_j$ for $j = 1, 2, 3$, and we denote
\begin{align*}
f_{n, t}^{\bar{N}} = f_{n, t} \cdot \prod_{j = 1}^3 \ind_{|n_j|_\infty \sim N_j}. 
\end{align*}

\noi
We also recall the definition of $\kappa (\bar n) = \kappa_2 (n_1, n_2, n_3)$ in \eqref{defk2}. Thus, by \eqref{equi_norm}, the Ito isometry \eqref{ito}, and expanding the sine functions into complex exponentials, we compute that
\begin{align}
\begin{split}
&\Big\| \big\| \<3>_N^{N_1, N_2, N_3} \big\|_{X_T^{2\al - 3 - \eps, -\frac 12 - \dl}} \Big\|_{L^2 (\O)}^2 \\
&\les \sum_{\eps_{123} \in {\{\pm 1\}}} \sum_{n \in \Z^3} \int_\R \jb{n}^{4\al - 6 - 2\eps} \jb{\tau}^{-1 - 2 \dl} \big\| \F_t \big( \ind_{[0, T]} f_{n, \cdot}^{\bar{N}} \big) (\tau - \eps_{123} \jbb{n}) \big\|_{\l_{n_1, n_2, n_3}^2 L_{t_1, t_2, t_3}^2}^2 d \tau \\
&\les e^{3T} \sum_{\eps_{123}, \eps_1, \eps_2, \eps_3 \in \{\pm 1\}} \sum_{n \in \Z^3} \int_\R \jb{n}^{4\al - 6 - 2\eps} \jb{\tau}^{-1 - 2 \dl} \bigg( \sum_{\substack{n = n_1 + n_2 + n_3 \\ |n_j| \sim N_j}} \prod_{j = 1}^3 \jbb{n_j}^{-2} \\
&\quad \times \int_{[0, T]^3} \bigg| \int_{\max (t_1, t_2, t_3)}^T e^{- i t (\tau - \kappa (\bar n)) - \frac 32 t} dt \bigg|^2 d t_3 d t_2 d t_1 \bigg) d\tau \\
&\les e^{3T} \sum_{\eps_{123}, \eps_1, \eps_2, \eps_3 \in \{\pm 1\}} \sum_{n \in \Z^3} \sum_{\substack{n = n_1 + n_2 + n_3 \\ |n_j| \sim N_j}} \frac{\jb{n}^{4\al - 6 - 2\eps}}{\prod_{j = 1}^3 \jbb{n_j}^2}   \int_{\R} \frac{1}{\jb{\tau}^{1 + 2 \dl} \jb{\tau - \kappa (\bar n)}^2} d\tau.
\end{split}
\label{cub_step1}
\end{align}

\noi
By \eqref{cub_step1}, the convolution inequality in Lemma \ref{LEM:conv}, and a level set decomposition, we obtain
\begin{align}
\begin{split}
&\Big\| \big\| \<3>_N^{N_1, N_2, N_3} \big\|_{X_T^{2\al - 3 - \eps, -\frac 12 - \dl}} \Big\|_{L^2 (\O)}^2 \\
&\quad \les e^{3T} \sum_{\eps_{123}, \eps_1, \eps_2, \eps_3 \in \{\pm 1\}} \sum_{n \in \Z^3} \sum_{\substack{n = n_1 + n_2 + n_3 \\ |n_j| \sim N_j}} \frac{\jb{n}^{4\al - 6 - 2\eps}}{\prod_{j = 1}^3 \jbb{n_j}^2}  \jb{\kappa (\bar n)}^{-1 - 2 \dl} \\
&\quad \les e^{3T} \sum_{\eps_{123}, \eps_1, \eps_2, \eps_3 \in \{\pm 1\}} \sup_{m \in \Z} \sum_{n \in \Z^3} \sum_{\substack{n = n_1 + n_2 + n_3 \\ |n_j| \sim N_j}} \frac{\jb{n}^{4\al - 6 - 2\eps}}{\prod_{j = 1}^3 \jbb{n_j}^2} \cdot \ind_{\{ |\kappa (\bar n) - m| \leq 1 \}}.
\end{split}
\label{cub_step2}
\end{align}

\noi
By the summation counting estimate in Lemma~\ref{LEM:CS} with $s = 3\al - 3 - \eps$ and $k = 3$, we obtain
\begin{align}
\Big\| \big\| \<3>_N^{N_1, N_2, N_3} \big\|_{X_T^{2\al - 3 - \eps, -\frac 12 - \dl}} \Big\|_{L^2 (\O)}^2 \les e^{3T} N_{\text{max}}^{-\ta}
\label{cub1}
\end{align}

\noi
for some $\ta > 0$, where $N_{\text{max}} = \max (N_1, N_2, N_3)$. 

On the other hand, using \eqref{cub_main}, we have
\begin{align}
\begin{split}
\Big\| \big\| \<3>_N^{N_1, N_2, N_3} \big\|_{X_T^{0, 0}} \Big\|_{L^2 (\O)}^2 &= \Big\| \big\| \<3>_N^{N_1, N_2, N_3} \big\|_{L_T^2 L_x^2} \Big\|_{L^2 (\O)}^2 \\
&\les T^\eta \sum_{\substack{n_1, n_2, n_3 \in \Z^3 \\ |n_j| \sim N_j}} \prod_{j = 1}^3 \jbb{n_j}^{-2} \\
&\les T^\eta N_{\text{max}}^{9 - 6 \al}
\end{split}
\label{cub2}
\end{align}

\noi
for some $\eta > 0$. Thus, by interpolating \eqref{cub1} (with $\dl > 0$ sufficiently small) and \eqref{cub2} and using the Gaussian hypercontractivity, we have
\begin{align*}
\Big\| \big\| \<3>_N^{N_1, N_2, N_3} \big\|_{X_T^{2 \al - 3 - \eps, -\frac 12 + \dl}} \Big\|_{L^p (\O)}^2 \les p^{3} e^{3T} N_{\text{max}}^{-\ta'}
\end{align*}

\noi
for some $\ta' > 0$. By summing up dyadic $N_1, N_2, N_3 \geq 1$, we obtain the bound \eqref{cub_goal}. Then, using Chebyshev's inequality, we can obtain the following tail estimate:
\begin{align}
\PP\Big( \|\<30>_N\|_{X_T^{3\al - 3 - \eps, \frac 12 + \dl} }> \ld\Big)
\leq C(T)\exp\big(-c (T) \ld^{\frac 23} \big).
\label{cub_tail}
\end{align}

For the convergence of $\<3>_N$ to $\<3>$ in $X_T^{2 \al - 3 - \eps, -\frac 12 + \dl}$, we use similar steps as above to estimate the difference $\ind_{[0,T]}\<3>_M - \ind_{[0,T]} \<3>_N$ for $M \geq N \geq 1$. The essential modification is that, in \eqref{cub_step1} and \eqref{cub_step2}, we need to insert the condition $N \leq \max (|n_1|, |n_2|, |n_3|) \leq M$, which allows us to gain a small negative power of $N$. Thus, we are able to obtain
\begin{align*}
\Big\| \big\| \<3>_M^{N_1, N_2, N_3} - \<3>_N^{N_1, N_2, N_3} \big\|_{X_T^{2 \al - 3 - \eps, -\frac 12 + \dl}} \Big\|_{L^p (\O)}^2 \les p^{3} e^{3T} N^{-\ta_1} N_{\text{max}}^{-\ta_2}
\end{align*}

\noi
for any $p \geq 1$ and some small $\ta_1, \ta_2 > 0$, and so we can sum up dyadic $N_1, N_2, N_3 \geq 1$ to obtain
\begin{align*}
\Big\| \big\| \<3>_M - \<3>_N \big\|_{X_T^{2 \al - 3 - \eps, -\frac 12 + \dl}} \Big\|_{L^p (\O)}^2 \les p^{3} e^{3 T} N^{-\ta_1}.
\end{align*}

\noi
By Chebyshev's inequality and the Borel-Cantelli lemma, we can obtain the almost sure convergence of $\<3>_N$ in $X_T^{2 \al - 3 - \eps, - \frac 12 + \dl}$, which in turn implies the almost sure convergence of $\<30>_N$ in $X_T^{3\al - 3 - \eps, - \frac 12 + \dl}$.

The bounds and convergence of $\<30>_N$ in the $C_T W_x^{s, \infty}$-norm follows from exactly the same way as in the proof of \cite[Lemma~3.1~(ii)]{OWZ}, and so we omit details. Also, the tail bounds \eqref{tail} and \eqref{tail_diff} follow from a similar treatment along with \eqref{cub_tail} (and also a corresponding difference tail estimate).

\smallskip \noi
(iv) In view of the expansion in \eqref{sto3_expand} and \eqref{sto3_simp} and also noting that the smoothing property for our equation is slightly stronger than that of \cite{OWZ}, we can use similar steps as in \cite[Lemma 3.4 (i)]{OWZ} to obtain
\begin{align*}
\Big\| \big\| \<30>_N \<2>_N \big\|_{X_T^{-\frac 12, - \frac 12 + \dl}} \Big\|_{L^p (\O)} \les p^{\frac 52} e^{\frac 52 T},
\end{align*}

\noi
so that the regularity and convergence of $\<320>_N$ hold due to Lemma~\ref{LEM:nhomo}.
The essential difference of the proof is the use of counting estimates. In view of the product formula \eqref{prod}, the product $\<30>_N \<2>_N$ is decomposed into three parts: a non-resonant term (no pairing), a single-resonance term (one pairing), and a double-resonance term (two pairings). For the non-resonance term, the use of Lemma~A.4 in \cite{OWZ} needs to be replaced by our Lemma~\ref{LEM:QC}. For the single-resonance term, the use of Lemma~A.3 in \cite{OWZ} needs to be replaced by our Lemma~\ref{LEM:BR}. For the double-resonance term, the use of Lemma~A.5 in \cite{OWZ} needs to be replaced by our Lemma~\ref{LEM:DRC}. The tail estimates \eqref{tail} and \eqref{tail_diff} follow from the same way as \eqref{cub_tail}.

\smallskip \noi
(v) As in (iv), we can use similar steps as in \cite[Lemma 3.4 (ii)]{OWZ} to obtain
\begin{align*}
\Big\| \big\| \<30>_N^2 \<1>_N \big\|_{X_T^{-\frac 12, - \frac 12 + \dl}} \Big\|_{L^p (\O)} \les p^{\frac 72} e^{\frac 72 T},
\end{align*}

\noi
so that the regularity and convergence of $\<70>_N$ hold due to Lemma~\ref{LEM:nhomo}. The essential difference of the proof is again the use of counting estimates. In view of the product formula \eqref{prod}, the product $\<30>_N^2 \<1>_N$ is decomposed into two parts: a non-resonance septic term (no pairing) and resonance septic terms (at least one pairing).
For the non-resonance septic term, the use of Lemma~A.1 in \cite{OWZ} needs to be replaced by our Lemma~\ref{LEM:CS} with $s = \al - 1$ and $k = 3$. For the resonance septic terms, the use of Lemma~A.6 in \cite{OWZ} needs to be replaced by our Lemma~\ref{LEM:SepC}. The tail estimates \eqref{tail} and \eqref{tail_diff} follow from the same way as \eqref{cub_tail}.

\smallskip \noi
(vi) We first use similar steps as in the proof of \cite[Lemma~3.3]{OWZ} (with the use of Lemma~A.3 in \cite{OWZ} replaced by our Lemma~\ref{LEM:BR}) to obtain
\begin{align*}
\Big\| \big\| \<30>_N \<1>_N \big\|_{L_T^q W_x^{\al - \frac 32 - \eps, \infty}} \Big\|_{L^p (\O)} \les_T p^2
\end{align*}

\noi
and
\begin{align}
\PP \Big( \big\| \<30>_N \<1>_N \big\|_{L_T^q W_x^{\al - \frac 32 - \eps, \infty}} > \ld \Big) \leq C(T) \exp \big( -c(T) \ld^{\frac 12} \big).
\label{cublin_tail}
\end{align}

\noi
for any $1 < p, q < \infty$. For the operator $\vec \If^{\,\<31>_N}$, by \eqref{Op2}, the inhomogeneous linear estimate in Lemma~\ref{LEM:nhomo}, the time localization estimate \eqref{time2} in Lemma~\ref{LEM:time}, the product estimate in Lemma~\ref{LEM:gko}~(ii), Sobolev's embeddings, and the fact that $\jb{\tau}^{a} \les \jb{|\tau| - \jbb{n}}^{a} \jbb{n}^{a}$ for any $a > 0$, we have
\begin{align*}
\big\| &\vec \If^{\,\<31>_N} \big\|_{\vec \L_T^{\,\al - \frac 12 - \eps, \frac 12 + \dl, \al - \frac 12 + \eps, \frac 12 + \dl}} \\
&\leq \sup_{0 < T' \leq T} (T')^{- \ta}  \sup_{\| v \|_{X_{T'}^{\al - \frac 12 - \eps, \frac 12 + \dl} } \leq 1} \big\| \pi_N \vec \I \big( (\<30>_N \<1>_N) v \big) \big\|_{\vec X_{T'}^{\al - \frac 12 + \eps, \frac 12 + \dl}} \\
&\les \sup_{0 < T' \leq T} \sup_{\| v \|_{X_{T'}^{\al - \frac 12 - \eps, \frac 12 + \dl} } \leq 1} \big\| (\<30>_N \<1>_N) v \big\|_{L_{T'}^2 H_x^{- \frac 12 + \eps}} \\
&\les \sup_{0 < T' \leq T} \sup_{\| v \|_{X_{T'}^{\al - \frac 12 - \eps, \frac 12 + \dl} } \leq 1} \big\| \<30>_N \<1>_N \big\|_{L_{T'}^{q} W_x^{-\frac 12 + \eps, 6}} \| v \|_{L_{T'}^{2 + \dl_1} W_x^{\frac 12 - \eps, \frac{6}{3 - 2 \eps}}} \\
&\les \sup_{0 < T' \leq T} \sup_{\| v \|_{X_{T'}^{\al - \frac 12 - \eps, \frac 12 + \dl} } \leq 1} \big\| \<30>_N \<1>_N \big\|_{L_{T'}^{q} W_x^{\al - \frac 32 - \eps, 6}} \big\| \jb{\nb}_t^{\dl_2} v \big\|_{L_{T'}^2 H_x^{\frac 12}} ,
\end{align*}

\noi
so that we get
\begin{align}
\big\| \vec \If^{\,\<31>_N} \big\|_{\vec \L_T^{\,\al - \frac 12 - \eps, \frac 12 + \dl, \al - \frac 12 + \eps, \frac 12 + \dl}}
\les \big\| \<30>_N \<1>_N \big\|_{L_T^{q} W_x^{\al - \frac 32 - \eps, 6}} ,
\label{cublin_op}
\end{align}

\noi
where $\dl_1, \dl_2, \eps > 0$ are sufficiently small, $q \gg 1$ satisfies $\frac 12 = \frac{1}{q} + \frac{1}{2+\dl_1}$, and we also used $\al > 1$. This in turn implies the regularity and convergence of $\If^{\<31>_N}$. The tail estimates \eqref{tail} and \eqref{tail_diff} then follow easily from \eqref{cublin_tail} (and also a corresponding difference tail estimate).

\smallskip \noi
(vii) This part is similar to  \cite[Lemma~3.5]{OWZ}, and so we will be brief.
As in (iii), we can assume without loss of generality that
\begin{align*}
\ft{\<2>}_N (n, t) = I_2 [g_{n, t}],
\end{align*}

\noi
where $g_{n, t}$ is defined by
\begin{align*}
g_{n, t} (n_1, t_1, n_2, t_2) = \ind_{\{ n = n_{12} \}} \cdot \bigg( \prod_{j = 1}^2 e^{- \frac{t - t_j}{2}} \frac{\sin ((t - t_j) \jbb{n_j})}{\jbb{n_j}} \cdot \ind_{\{|n_j|_\infty \leq N\}} \cdot \ind_{[0, t]} (t_j) \bigg) ,
\end{align*}

\noi
and so we can write
\begin{align*}
\F_x (\<2>_N v) (n, t) = \sum_{n_3 \in \Z^3} \ft v (n_3, t) I_2 [ g_{n - n_3, t} ].
\end{align*}

\noi
Given dyadic $N_j \geq 1$ for $j = 1, 2, 3$, we define $\<2>_N^{N_1, N_2}$ be $\<2>_N$ with frequency localization $|n_1| \sim N_1$ and $|n_2| \sim N_2$, and we also define $v_{N_3}$ be $v$ with frequency localization $|n_3| \sim N_3$. Thus, by the inhomogeneous linear estimate in Lemma~\ref{LEM:nhomo}, it suffices to show that 
\begin{align}
\Big\|  \sup_{\| v \|_{X^{s, \frac 12 + \dl} } \leq 1} \big\| \ind_{[0, T]} (t) \cdot \<2>_N^{N_1, N_2}  v_{N_3} \big\|_{X^{s - \al + \eps, - \frac 12 - \dl}} \Big\|_{L^p (\O)} \les_T p N_{\text{max}}^{- \ta}
\label{quad_goal}
\end{align}

\noi
for some $\ta > 0$, any $p \geq 1$, $\dl > 0$ sufficiently small, and $\max(-\al + \frac 32, \al - 1) < s < 2\al - \frac 32$, where $N_{\text{max}} = \max (N_1, N_2, N_3)$. Once the bound \eqref{quad_goal} is established, the remaining argument (i.e.~interpolation with an $X_{T'}^{0, 0}$-norm, the tail bounds \eqref{tail} and \eqref{tail_diff}, and the convergence) is similar to that in (iii) along with an additional application of the time localization estimate \eqref{time2} in Lemma~\ref{LEM:time} so that a factor $(T')^\ta$ can be created.

By writing 
\begin{align*}
\ft {v}_{N_3, 1} (n, \tau) = \ind_{[0, \infty)} (\tau) \ft{v}_{N_3} (n, \tau) \quad \text{and} \quad \ft {v}_{N_3, -1} (n, \tau) = \ind_{(- \infty, 0)} \ft{v}_{N_3} (n, \tau),
\end{align*}

\noi
we have
\begin{align}
\| v_{N_3} \|_{X^{s, b}}^2 = \sum_{\eps_3 \in \{\pm 1\}} \big\| \jb{n}^s \jb{\tau}^b \ft{v}_{N_3, \eps_3} (n, \tau + \eps_3 \jbb{n}) \big\|_{\l_n^2 L_\tau^2}^2.
\label{vN3}
\end{align}

\noi
We also denote
\begin{align*}
H_{N_1, N_2, N_3} (n, n_3, \tau, \tau_3) = \jb{n}^{s - \al + \eps} \jb{n_3}^{-s} I_2 [h_{n, n_3, \tau, \tau_3}^{N_1, N_2, N_3, I}],
\end{align*}

\noi
where
\begin{align*}
h_{n, n_3, \tau, \tau_3}^{N_1, N_2, N_3} & (n_1, t_1, n_2, t_2) = \ind_{\{n = n_{123}\}} \int_0^T e^{-it (\tau - \tau_3 - \eps_0 \jbb{n} - \eps_3 \jbb{n_3})} \\
&\quad \times \bigg( \prod_{j = 1}^2 e^{- \frac{t - t_j}{2}} \frac{\sin ((t - t_j) \jbb{n_j})}{\jbb{n_j}} \cdot \ind_{\{|n_j|_\infty \leq N\}} \cdot \ind_{\{ |n_j| \sim N_j \}} \cdot \ind_{[0, t]} (t_j) \bigg)
\end{align*}

\noi
with $\eps_0 \in \{\pm 1\}$. Thus, by \eqref{equi_norm}, a stochastic Fubini theorem (see \cite[Lemma~B.2]{OWZ} or \cite[Theorem~4.33]{DZ}), the Cauchy-Schwarz inequality in $\tau_3$, \eqref{vN3}, and Minkowski's integral inequality (with $p \geq 2$), we have
\begin{align}
\begin{split}
\bigg\| &\sup_{\| v \|_{X^{s, \frac 12 + \dl} } \leq 1} \| \ind_{[0, T]} (t) \cdot \<2>_N^{N_1, N_2}  v_{N_3} \|_{X^{s - \al + \eps, - \frac 12 - \dl}} \bigg\|_{L^p (\O)} \\
&\les \sum_{\eps_0, \eps_3 \in \{\pm 1\}} \bigg\| \sup_{\| v \|_{X^{s, \frac 12 + \dl} } \leq 1} \Big\| \jb{\tau}^{- \frac 12 - \dl} \sum_{|n_3| \leq N} \jb{n_3}^{s} \\
&\quad \times \Big| \int_\R \ft{v}_{N_3, \eps_3} (n_3, \tau_3 + \eps_3 \jbb{n_3}) H_{N_1, N_2, N_3} (n, n_3, \tau, \tau_3) d \tau_3 \Big| \Big\|_{\l_n^2} \bigg\|_{L^p (\O; L_\tau^2)}  \\
&\les \sup_{\eps_0, \eps_3 \in \{\pm 1\}} \Big\|  \jb{\tau}^{- \frac 12 - \dl} \jb{\tau_3}^{- \frac 12 - \dl}  \| H_{N_1, N_2, N_3} (n, n_3, \tau, \tau_3) \|_{\l_{n_3}^2 \to \l_n^2} \Big\|_{L^p (\O ; L_{\tau, \tau_3}^2)} \\
&\les \sup_{\eps_0, \eps_3 \in \{\pm 1\}} \sup_{\tau, \tau_3 \in \R} \Big\| \| H_{N_1, N_2, N_3} (n, n_3, \tau, \tau_3) \|_{\l_{n_3}^2 \to \l_n^2} \Big\|_{L^p (\O)}.
\end{split}
\label{quad1}
\end{align}

We now define
\begin{align*}
\kappa (\cj{n}) = \eps_0 \jbb{n} + \eps_1 \jbb{n_1} + \eps_2 \jbb{n_2} + \eps_3 \jbb{n_3}
\end{align*}
for $\eps_0, \eps_1, \eps_2, \eps_3 \in \{\pm 1\}$. For $m \in \Z$, we define the tensor
\begin{align*}
\mathfrak{h}^m = \mathfrak{h}^{m}_{n n_1 n_2 n_3} = \ind_{\{n = n_{123} \}} \bigg( \prod_{j = 1}^3 \ind_{\{|n_j| \sim N_j\}} \prod_{j = 1}^2 \ind_{\{|n_j|_\infty \leq N\}} \bigg) \cdot \ind_{\{|\kappa (\cj{n}) - m| \leq 1\}} \frac{\jb{n}^{s - \al + \eps}}{\jbb{n_1} \jbb{n_2} \jb{n_3}^{s}}.
\end{align*}

\noi
Thus, we can write
\begin{align*}
H_{N_1, N_2, N_3} (n, n_3, \tau, \tau_3) = \sum_{\eps_1, \eps_2 \in \{\pm 1\}} \sum_{|m| \les N_{\text{max}}} I_2 [ \mathfrak{h}^{m}_{n n_1 n_2 n_3} \mathfrak{H}_{n_3, \tau, \tau_3}^{m} ],
\end{align*}

\noi
where
\begin{align*}
\mathfrak{H}_{n_3, \tau, \tau_3}^{m} (n_1, t_1, n_2, t_2) = \int_0^T \ind_{\{ |\kappa (\cj{n}) - m| \leq 1 \}} e^{- it (\tau - \tau_3 - \kappa (\cj{n})) - t} \bigg( \prod_{j = 1}^2 e^{\frac{t_j}{2} - i t_j \eps_j \jb{n_j}} \cdot \ind_{[0, t]} (t_j) \bigg) dt.
\end{align*}

\noi
By the random tensor estimate in \cite[Lemma~C.3]{OWZ}, the deterministic tensor estimate (Lemma~\ref{LEM:ten1} with $k = 3$), and performing the integration in $t$ in $\mathfrak{H}^m$, we obtain
\begin{align}
\begin{split}
\Big\| &\| H_{N_1, N_2, N_3} (n, n_3, \tau, \tau_3) \|_{\l_{n_3}^2 \to \l_n^2} \Big\|_{L^p (\O)} \\
&\les p e^T N_{\text{max}}^{\frac{\dl}{2}} \sum_{|m| \les N_{\text{max}}} \max \Big( \| \mathfrak{h}^m \|_{n_1 n_2 n_3 \to n}, \| \mathfrak{h}^m \|_{n_3 \to n_1 n_2 n}, \| \mathfrak{h}^m \|_{n_1 n_3 \to n n_2}, \\
&\quad \| \mathfrak{h}^m \|_{n_2 n_3 \to n n_1} \Big) \| \mathfrak{H}_{n_3, \tau, \tau_3}^{m} (n_1, t_1, n_2, t_2) \|_{\l_{n_1, n_2}^\infty L_{t_1, t_2}^2 ([0, T]^2)} \\
&\les p e^T N_{\text{max}}^{- \frac{\dl}{2}} \sum_{|m| \les N_{\text{max}}^\al} \jb{\tau - \tau_3 - m}^{-1} \\
&\les p e^T N_{\text{max}}^{- \frac{\dl}{4}}
\end{split}
\label{quad2}
\end{align}

\noi
for some $\dl > 0$. Combining \eqref{quad1} and \eqref{quad2}, we obtain the desired estimate \eqref{quad_goal}.
\end{proof}

\subsection{Proof of local well-posedness}
\label{SUBSEC:lwp}
In this subsection, we prove local well-posedness statement in Proposition~\ref{PROP:LWP}. In fact, since the same argument will be used later again in Section~\ref{SEC:GWP}, we show a more general local well-posedness result. 

We recall that if $X_1$ and $X_2$ are two Banach spaces, then the space $X_1 + X_2$ induced by the norm
\begin{align*}
\| u \|_{X_1 + X_2} \overset{\text{def}}{=} \inf \{ \| u_1 \|_{X_1} + \| u_2 \|_{X_2}: u = u_1 + u_2 \},
\end{align*}

\noi
is also a Banach space.

\begin{proposition}[General local well-posedness for the fractional hyperbolic $\Phi^4_3$-model]
\label{THM:1}
Let $1 < \al \leq \frac 32$, $s \geq \al - \frac 12$, and $T > 0$.
Suppose that there exist $\eps=\eps(\al,s)$ and $\dl=\dl(\al,s)>0$ small enough
such that
\begin{itemize}
\item   $\Xi_1$ is a distribution-valued function belonging to
 $C_T W^{\al-\frac 32 - \eps, \infty}$,


\smallskip
\item
$\Xi_3$ is a distribution-valued function belonging to
$ C_T W^{3\al - 3 - \eps, \infty} + X_T^{\al - \frac 12, \frac 12 + \dl}$,

\smallskip
\item
$\vec \Xi_5$ is a function belonging to $\vec X_T^{\al-\frac 12,\frac 12+\dl}$ with $\vec \Xi_5 |_{t = 0} = (0, 0)$,

\smallskip
\item
$\vec \Xi_7$ is a function belonging to $\vec X_T^{\al-\frac 12,\frac 12+\dl}$ with $\vec \Xi_7 |_{t = 0} = (0, 0)$,

\smallskip
\item
the operator $\vec \If_{31}$ belongs to the class $\vec \L^{\,\al - \frac 12, \frac 12 + \dl, \al - \frac 12, \frac 12 + \dl}_{T}$ such that $\vec \If_{31} (u) |_{t = 0} = (0, 0)$ for any input $u$,

\smallskip
\item
the operator $\vec \If_2$ belongs to the class $\vec \L^{\,\al - \frac 12, \frac 12 + \dl, \al - \frac 12,  \frac 12 + \dl}_{T}$ such that $\vec \If_{2} (u) |_{t = 0} = (0, 0)$ for any input $u$.

\smallskip

\end{itemize}

\noi
Then, given any $(v_0, v_1) \in \vec H^{s}(\T^3)$, there exist $0 < T_0 \leq T$ and  a unique solution $\vec v$ to the equation
\begin{align*}
\vec v &= \vec S (\cdot) (v_0, v_1) + \vec \I \big(-v^3 + 3(\Xi_3 - \Xi_1) v^2 - 3\Xi_3^2 v\big)
		   + 6 \vec \If_{31} (v) \\
		 &\quad - 3 \vec \If_2 (v) +\vec \I \big(\Xi_3^3\big) - 3 \vec \Xi_7 + 3 \vec \Xi_5,
\end{align*}

\noi
with $\vec S(t)$ and $\vec \I$ being defined in \eqref{vecSI}, \eqref{St0}, and \eqref{lin1}, in the class
\begin{align*}
 \vec X^{\al - \frac 12,\frac 12 + \dl}([0,T_0])
 \subset C([0,T_0]; \vec H^{\al - \frac 12}(\T^3)).
\end{align*}

\noi
Furthermore, the solution $v$
depends continuously
on the enhanced data set\textup{:}
\begin{align}
\pmb{\Xi} (v_0, v_1) =
\big(v_0, v_1,
\Xi_1, \,  \Xi_3,  \, \vec \Xi_5, \, \vec \Xi_7, \, \vec \If_{31}, \, \vec \If_2 \big)
\label{data1}
\end{align}

\noi
in the class\textup{:}
\begin{align*}
\mathcal{X}^{s, \al, \eps, \dl}_{T_0}
& = \vec H^{s}(\T^3)
\times
C([0,T_0]; W^{\al-\frac 32 - \eps, \infty}(\T^3)) \\
& \hphantom{X}
\times
\big( C([0,T_0]; W^{3\al - 3 - \eps, \infty}(\T^3)) + X^{\al - \frac 12,\frac 12+\dl}([0,T_0]) \big) \\
& \hphantom{X}
\times
\vec X^{\al-\frac 12,\frac 12+\dl}([0,T_0])
\times
\vec X^{\al-\frac 12,\frac 12+\dl}([0,T_0])\\
& \hphantom{X}
\times
\vec \L^{\,\al - \frac 12, \frac 12 + \dl, \al - \frac 12, \frac 12 + \dl}_{T_0}
\times  \vec \L^{\,\al - \frac 12, \frac 12 + \dl, \al - \frac 12, \frac 12 + \dl}_{T_0}.
\end{align*}

\end{proposition}

Note that with minor modifications, Proposition~\ref{THM:1}, combined with the regularities of the stochastic objects in Lemma~\ref{LEM:sto_reg}, implies Proposition~\ref{PROP:LWP}. One needs some extra steps to justify the convergence of solutions as the frequency cutoff is removed, but this part is not difficult and so we omit details.

\begin{remark} \rm
In the statement of Proposition \ref{THM:1}, $\Xi_3$ corresponds to the cubic stochastic object $\<30>$. We choose to put $\Xi_3$ in the larger space $C_T W_x^{3\al - 3 - \eps, \infty} + X_T^{\al - \frac 12, \frac 12 + \dl}$ instead of the simpler $C_T W_x^{3\al - 3 - \eps, \infty}$-space, because the use of this larger space is a crucial component in Proposition~\ref{PROP:exptail2} below, where the stochastic objects are constructed from initial data distributed by the Gibbs measure. 
Another difference from the local well-posedness argument in \cite{OWZ} is that, instead of considering the stochastic object $\<30> \<1>$ directly, we choose to consider the operator $\vec \If_{31} = \vec \If^{\,\<31>}$, which is also important for our globalization argument in the next section.
\end{remark}

\begin{proof}
Given the enhanced data set \eqref{data1}, we set
\begin{align*}
\pmb{\Xi} = \big(
\Xi_1, \,  \Xi_3,  \, \vec \Xi_5, \, \vec \Xi_7, \, \vec \If_{31}, \, \vec \If_2\big).
\end{align*}
For any $\al > 1$ and $\eps, \dl > 0$, we define the norm
\begin{align}
\begin{split}
\| &\pmb{\Xi} \|_{\mathcal{X}_{T}^{\al, \eps, \dl}} = \| \Xi_1 \|_{C_{T} W_x^{\al - \frac 32 - \eps, \infty}} + \| \Xi_3 \|_{C_{T} W_x^{3\al - 3 - \eps, \infty} + X_T^{\al - \frac 12, \frac 12 + \dl}} + \| \vec \Xi_5 \|_{\vec X_{T}^{\al - \frac 12, \frac 12 + \dl}} \\
&\quad + \| \vec \Xi_7 \|_{\vec X_{T}^{\al - \frac 12, \frac 12 + \dl}} + \| \vec \If_{31} \|_{\vec \L_{T}^{\,\al - \frac 12, \frac 12 + \dl, \al - \frac 12, \frac 12 + \dl}} + \| \vec \If_2 \|_{\vec \L^{\,\al - \frac 12, \frac 12 + \dl, \al - \frac 12,  \frac 12 + \dl}_{T}},
\end{split}
\label{Xae1}
\end{align}

\noi
where the $\L$-norms are as defined in \eqref{Op2}. Note that we can write $\Xi_3 = \Xi_{3}^{(1)} + \Xi_3^{(2)}$ with $\Xi_3^{(1)} \in C_T W_x^{3\al - 3 - \eps, \infty}$ and $\Xi_3^{(2)} \in X_T^{\al - \frac 12, \frac 12 + \dl}$ satisfying
\begin{align*}
\big\| \Xi_3^{(1)} \big\|_{C_T W_x^{3\al - 3 - \eps, \infty}} + \big\| \Xi_3^{(2)} \big\|_{X_T^{\al - \frac 12, \frac 12 + \dl}} \leq \| \Xi_3 \|_{C_{T} W_x^{\al - 1 - \eps, \infty} + X_T^{\al - \frac 12, \frac 12 + \dl}} + 1.
\end{align*}

\noi
Due to the assumptions of the proposition, we can assume that
\begin{align}
\| (v_0, v_1) \|_{\vec H^s} + \| \pmb{\Xi} \|_{\mathcal{X}_{T}^{\al, \eps, \dl}} + \big\| \Xi_3^{(1)} \big\|_{C_T W_x^{3\al - 3 - \eps, \infty}} + \big\| \Xi_3^{(2)} \big\|_{X_T^{\al - \frac 12, \frac 12 + \dl}} \leq K
\label{data_norm}
\end{align}

\noi
for some $K \geq 1$.

We first show the following four general estimates which are able to cover most of the required estimates below. Let $u_1$, $u_2$, $u_3$ be any functions and $0 < T_0 \leq T$. By the inhomogeneous linear estimate in Lemma~\ref{LEM:nhomo}, the time localization estimate \eqref{time2} in Lemma~\ref{LEM:time}, and Lemma~\ref{LEM:str_var} with $v \equiv 1$, we obtain
\begin{align}
\begin{split}
\| \vec \I (u_1 u_2 u_3) \|_{\vec X_{T_0}^{\al - \frac 12, \frac 12 + \dl}} &\les T_0^{3 \dl} \| u_1 u_2 u_3 \|_{X_{T_0}^{- \frac 12, - \frac 12 + 4 \dl}} \\
&\les T_0^{3 \dl} \| u_1 \|_{X_{T_0}^{\al - \frac 12, \frac 12 + \dl}} \| u_2 \|_{X_{T_0}^{\al - \frac 12, \frac 12 + \dl}} \| u_3 \|_{X_{T_0}^{\al - \frac 12, \frac 12 + \dl}}.
\end{split}
\label{lwp_est1}
\end{align}

\noi
By Lemma~\ref{LEM:nhomo}, \eqref{time2} in Lemma~\ref{LEM:time}, H\"older's inequality, and the Strichartz estimate \eqref{Lp_str} with $p = 4$, we obtain
\begin{align}
\begin{split}
\| \vec \I (u_1 u_2 u_3) \|_{\vec X_{T_0}^{\al - \frac 12, \frac 12 + \dl}} 
&\les T_0^{\ta} \| u_1 u_2 u_3 \|_{L_{T_0}^2 L_x^2} \\
&\les T_0^{\ta} \| u_1 \|_{L_{T_0}^\infty L_x^\infty} \| u_2 \|_{L_{T_0}^4 L_x^4} \| u_3 \|_{L_{T_0}^4 L_x^4} \\
&\les T_0^{\ta} \| u_1 \|_{L_{T_0}^\infty W_x^{3\al - 3 - \eps, \infty}} \| u_2 \|_{X_{T_0}^{\al - \frac 12, \frac 12 + \dl}} \| u_3 \|_{X_{T_0}^{\al - \frac 12, \frac 12 + \dl}}
\end{split}
\label{lwp_est2}
\end{align}

\noi
for some $\ta > 0$, where we used $3\al - 3 - \eps > 0$ for $\eps > 0$ sufficiently small. By Lemma~\ref{LEM:nhomo}, \eqref{time2} in Lemma~\ref{LEM:time}, and H\"older's inequality, we obtain
\begin{align}
\begin{split}
\| \vec \I (u_1 u_2 u_3) \|_{\vec X_{T_0}^{\al - \frac 12, \frac 12 + \dl}} 
&\les T_0^{\ta} \| u_1 u_2 u_3 \|_{L_{T_0}^2 L_x^2} \\
&\les T_0^{\ta} \| u_1 \|_{L_{T_0}^\infty L_x^\infty} \| u_2 \|_{L_{T_0}^\infty L_x^\infty} \| u_3 \|_{L_{T_0}^2 L_x^2} \\
&\les T_0^{\ta} \| u_1 \|_{L_{T_0}^\infty W_x^{3\al - 3 - \eps, \infty}} \| u_2 \|_{L_{T_0}^\infty W_x^{3\al - 3 - \eps, \infty}} \| u_3 \|_{X_{T_0}^{\al - \frac 12, \frac 12 + \dl}}
\end{split}
\label{lwp_est3}
\end{align}

\noi
for some $\ta > 0$. Lastly, by Lemma~\ref{LEM:nhomo}, \eqref{time2} in Lemma~\ref{LEM:time}, and H\"older's inequality, we obtain
\begin{align}
\begin{split}
\| \vec \I (u_1 u_2 u_3) \|_{\vec X_{T_0}^{\al - \frac 12, \frac 12 + \dl}} 
&\les T_0^{\ta} \| u_1 u_2 u_3 \|_{L_{T_0}^2 L_x^2} \\
&\les T_0^{\ta} \| u_1 \|_{L_{T_0}^\infty L_x^\infty} \| u_2 \|_{L_{T_0}^\infty L_x^\infty} \| u_3 \|_{L_{T_0}^\infty L_x^\infty} \\
&\les T_0^{\ta} \| u_1 \|_{L_{T_0}^\infty W_x^{3\al - 3 - \eps, \infty}} \| u_2 \|_{L_{T_0}^\infty W_x^{3\al - 3 - \eps, \infty}} \| u_3 \|_{L_{T_0}^\infty W_x^{3\al - 3 - \eps, \infty}}
\end{split}
\label{lwp_est4}
\end{align}

\noi
for some $\ta > 0$.

We now define the map $\G_{\pmb{\Xi}}$ by
\begin{align*}
\G_{\pmb{\Xi}} (\vec v) &= \vec S(t) (v_0, v_1) + \vec \I \big( - v^3 + 3 (\Xi_3 - \Xi_1) v^2 - 3 \Xi_3^2 v \big) + 6 \vec \If_{31} (v) \\
&\quad - 3 \vec \If_2 (v) + \vec \I \big( \Xi_3^3 \big) - 3 \vec \Xi_7 + 3 \vec \Xi_5.
\end{align*}

\noi
We let $0 < T_0 \leq T$. By \eqref{lwp_est1}, \eqref{lwp_est2}, \eqref{lwp_est3}, \eqref{lwp_est4}, and \eqref{data_norm}, we obtain
\begin{align}
\begin{split}
\| \vec \I (v^3) &\|_{\vec X_{T_0}^{\al - \frac 12, \frac 12 + \dl}} + \| \vec \I (\Xi_3 v^2) \|_{\vec X_{T_0}^{\al - \frac 12, \frac 12 + \dl}} + \| \vec \I (\Xi_3^2 v) \|_{\vec X_{T_0}^{\al - \frac 12, \frac 12 + \dl}} + \| \vec \I (\Xi_3^3) \|_{\vec X_{T_0}^{\al - \frac 12, \frac 12 + \dl}} \\
&\les T_0^\ta \Big( \| v \|_{X_{T_0}^{\al - \frac 12, \frac 12 + \dl}}^3 + K^3 \Big)
\end{split}
\label{lwp5-0}
\end{align}

\noi
for some $\ta > 0$. By \eqref{tri3} in Lemma~\ref{LEM:str4}, the time localization estimate \eqref{time2} in Lemma~\ref{LEM:time},  and \eqref{data_norm}, we have
\begin{align}
\begin{split}
\| \vec \I ( \Xi_1 v^2 ) \|_{\vec X_{T_0}^{\al - \frac 12, \frac 12 + \dl}} &\les T_0^{3 \dl} \| \Xi_1 v^2 \|_{X_{T_0}^{- \frac 12, - \frac 12 + 4 \dl}} \\
&\les T_0^{3 \dl} \| \Xi_1 \|_{L_{T_0}^\infty W_x^{\al - \frac 32 - \eps, \infty}} \| v \|_{X_{T_0}^{\al - \frac 12, \frac 12 + \dl}}^2 \\
&\leq T_0^{3 \dl} K \| v \|_{X_{T_0}^{\al - \frac 12, \frac 12 + \dl}}^2.
\end{split}
\label{lwp5-1}
\end{align}

\noi
By \eqref{data_norm} and \eqref{Op2}, we have
\begin{align}
\begin{split}
\| \vec \If_{31} (v) \|_{\vec X_{T_0}^{\al - \frac 12, \frac 12 + \dl}} 
&\les T_0^{\ta} \| \vec \If_{31} \|_{\vec \L_{T_0}^{\,\al - \frac 12, \frac 12 + \dl, \al - \frac 12, \frac 12 + \dl}} \| v \|_{X_{T_0}^{\al - \frac 12, \frac 12 + \dl}} \\
&\leq T_0^{\ta} K \| v \|_{X_{T_0}^{\al - \frac 12, \frac 12 + \dl}}
\end{split}
\label{lwp5-2}
\end{align}

\noi
for some $\ta > 0$. Similarly,
\begin{align}
\| \vec \If_2 (v) \|_{\vec X_{T_0}^{\al - \frac 12, \frac 12 + \dl}} \les T_0^{\ta} K \| v \|_{X_{T_0}^{\al - \frac 12, \frac 12 + \dl}}.
\label{lwp5-3}
\end{align}

Thus, by the homogeneous linear estimate Lemma~\ref{LEM:homo} along with $s \geq \al - \frac 12$, \eqref{lwp5-0}, \eqref{lwp5-1}, \eqref{lwp5-2}, \eqref{lwp5-3}, and \eqref{data_norm}, we obtain
\begin{align*}
\| \G_{\pmb{\Xi}} (\vec v) \|_{\vec X_{T_0}^{\al - \frac 12, \frac 12 + \dl}} &\leq C \| (v_0, v_1) \|_{\vec H^s} + C T_0^\ta \Big( \| v \|_{X_{T_0}^{\al - \frac 12, \frac 12 + \dl}}^3 + K^3 \Big) + C K \\
&\leq C T_0^\ta \Big( \| \vec v \|_{\vec X_{T_0}^{\al - \frac 12, \frac 12 + \dl}}^3 + K^3 \Big) + C K
\end{align*}

\noi
for some constant $C \geq 1$. Using similar estimates, we can also obtain the following difference estimate:
\begin{align*}
\| \G_{\pmb{\Xi}} (\vec v_1) - \G_{\pmb{\Xi}} (\vec v_2) \|_{\vec X_{T_0}^{\al - \frac 12, \frac 12 + \dl}} \leq C T_0^\ta \Big( \| \vec v_1 \|_{\vec X_{T_0}^{\al - \frac 12, \frac 12 + \dl}}^2 + \| \vec v_2 \|_{\vec X_{T_0}^{\al - \frac 12, \frac 12 + \dl}}^2 + K^2 \Big) \| \vec v_1 - \vec v_2 \|_{\vec X_{T_0}^{\al - \frac 12, \frac 12 + \dl}}.
\end{align*}
Therefore, by choosing $T_0 > 0$ small enough such that
\begin{align*}
(8 C^3 + 1) K^3 C T_0^\ta \leq \frac 12 \quad \Longleftrightarrow \quad T_0^\ta \leq \frac{1}{2 C (8 C^3 + 1) K^3},
\end{align*}

\noi
we obtain that $\G_{\pmb{\Xi}}$ is a contraction on a ball in the space $\vec X_{T_0}^{\al - \frac 12, \frac 12 + \dl}$ of radius $2CK$. This implies the desired local well-posedness result.
\end{proof}

\begin{remark} \rm
\label{RMK:LWP}
Let us assume the same conditions as in Proposition~\ref{THM:1} but with zero initial data. Then, from the proof of Proposition~\ref{THM:1}, we see that if we have the bound for the enhanced data set
\begin{align*}
\| \pmb{\Xi} \|_{\mathcal{X}_{T}^{\al, \eps, \dl}} \leq K,
\end{align*}

\noi
then by choosing $T_0 > 0$ satisfying
\begin{align*}
T_0^\ta \leq \frac{1}{2 C (8 C^3 + 1) K^3}
\end{align*}

\noi
for some absolute constant $C \geq 1$, the local solution $\vec v$ on $[0, T_0]$ constructed in the proof satisfies
\begin{align*}
\| \vec v \|_{\vec X_{T_0}^{\al - \frac 12, \frac 12 + \dl}} \leq 2 CK.
\end{align*}

\noi
This a priori bound for $\vec v$ will be used later in the globalization argument (in Proposition~\ref{PROP:ns}).
\end{remark}

\section{Global well-posedness and invariance of Gibbs measure}
\label{SEC:GWP}

In this section, we present the proof of Theorem \ref{THM:GWP}, almost sure global well-posedness of the fractional hyperbolic $\Phi_3^4$-model and the invariance of the fractional $\Phi^4_3$-measure $\rho$.

Before proceeding, let us introduce some notations. Given $N \in \N$, we let $\vec u_0 = (u_0, u_1)$ be the initial data independent with the space-time white noise $\xi$ such that $\Law (\vec u_0) = \vec \rho_N = \rho_N \otimes \mu_0$, where $\rho_N$ is the truncated fractional $\Phi^4_3$-measure defined in \eqref{GibbsN2} and $\mu_0$ is the white noise measure defined in \eqref{gauss2}. Let us then write the probability space $\O$ as $\O = \O_1 \times \O_2$
and the probability measure $\PP$ as $\PP = \PP_1 \otimes \PP_2$,
so that $\vec u_0$ depends only on $\o_1 \in \O_1$, $\xi$ depends only on $\o_2 \in \O_2$, and $\PP_j$ is the marginal probability measure on $\O_j$ for $j = 1, 2$.

Similar to Section~\ref{SEC:LWP}, let us now define the stochastic convolution
\begin{align}
\vec{\<1>} (\vec u_0, \o_2) (t) = \vec S(t) \vec u_0 + \sqrt 2 \int_0^t \vec \D(t-t') dW(t', \o_2),
\label{enh0a}
\end{align}

\noi
where $\vec S (t)$ is defined in \eqref{vecSI} and \eqref{St0}, $\vec \D (t)$ is defined in \eqref{vecSI} and \eqref{W3}, and $W$ is a cylindrical Wiener process given by \eqref{defW}. Given $t_0 \geq 0$, we recall that $\vec \I_{t_0}$ is the Duhamel operator defined in \eqref{vecSI} and \eqref{lin1} and $\vec \I = \vec \I_0$. Given $N \in \N$, we set 
\begin{align}
\begin{split}
\vec{\<1>}_N (\vec u_0, \o_2) &\deff \pi_N \vec{\<1>} (\vec u_0, \o_2), \\
\<2>_N (\vec u_0, \o_2) &\deff \<1>_N^2 (\vec u_0, \o_2) - \s_N, \\
\<3>_N (\vec u_0, \o_2)
  &\deff \<1>_N^3 (\vec u_0, \o_2) - 3\s_N\<1>_N (\vec u_0, \o_2), \\
\vec{\<30>}_N (\vec u_0, \o_2) &\deff \pi_N \vec \I (\<3>_N (\vec u_0,\o_2)), \\
\vec{\<320>}_{t_0, N} (\vec u_0, \o_2) &\deff \pi_N \vec \I_{t_0} (\<30>_N (\vec u_0, \o_2) \<2>_N (\vec u_0, \o_2)),  \\
\vec{\<70>}_{t_0, N} (\vec u_0, \o_2) &\deff \pi_N \vec \I_{t_0} (\<30>_N^2 (\vec u_0, \o_2) \<1>_N (\vec u_0, \o_2)), \\
\vec \If_{t_0}^{\,\<31>_N} (\vec u_0, \o_2) (v) &\deff \pi_N \vec \I_{t_0} \big( \<30>_N (\vec u_0, \o_2) \<1>_N (\vec u_0, \o_2) \cdot v \big), \\
\vec \If_{t_0}^{\,\<2>_N} (\vec u_0, \o_2) (v) &\deff \pi_N \vec \I_{t_0} (\<2>_N (\vec u_0, \o_2) \cdot v).
\end{split}
\label{enh0}
\end{align}

\noi
where $\sigma_N$ is as defined in \eqref{sigmaN}. We then gather all stochastic terms and define the truncated enhanced data set $\pmb{\Xi}_{t_0, N} = \pmb{\Xi}_{t_0, N} (\vec u_0, \o_2)$ as
\begin{align}
\pmb{\Xi}_{t_0, N} 
  = \big(   \vec{\<1>}_N , \,  \vec{\<30>}_N,  \, \vec{\<320>}_{t_0, N}, \, \vec{\<70>}_{t_0, N}, \, \vec \If_{t_0}^{\,\<31>_N}, \, \vec \If_{t_0}^{\,\<2>_N}\big).
\label{data3x}
\end{align}

\noi
For conciseness in later steps, we simply write $\pmb{\Xi}_N = \pmb{\Xi}_{\cdot, N}$ and
\begin{align*}
\vec{\<320>}_N = \vec{\<320>}_{\cdot, N}, \, \vec{\<70>}_N = \vec{\<70>}_{\cdot, N}, \, \vec \If^{\,\<31>_N} = \vec \If_{\cdot}^{\,\<31>_N}, \, \vec \If^{\,\<2>_N} = \vec \If_{\cdot}^{\,\<2>_N} .
\end{align*}

\subsection{On the truncated dynamics}
\label{SUBSEC:GWP1}

In this subsection, we go over global well-posedness of the frequency truncated fractional $\Phi^4_3$-model:
\begin{align}
\begin{split}
\dt^2 & u_N + \dt u_N  + (1 -  \Dl)^\al  u_N
+  \pi_N \big(   :\! (\pi_N u_N)^3 \!: \big)
= \sqrt{2} \xi
\end{split}
\label{SNLW3a}
\end{align}

\noi
and its invariance with respect to the truncated fractional $\Phi^4_3$-measure $\vec \rho_N = \rho_N \otimes \mu_0$. We will be brief since the procedure has already been performed several times in other literature.

We first claim that the equation \eqref{SNLW3a} is globally well-posed by using a classical Gronwall's argument as in \cite[Proposition~5.1]{ORTz}. As for the invariance of $\vec \rho_N$ under the flow of \eqref{SNLW3a}, we write the truncated dynamics \eqref{SNLW3a} as a superposition of the deterministic fractional NLW:
\begin{align}
\dt^2 u_N + (1 - \Dl)^\al u_N + \pi_N \big( :\! (\pi_N u_N)^3 \!: \big) = 0
\label{fNLW}
\end{align}

\noi
and the Ornstein-Uhlenbeck process for $\dt u_N$:
\begin{align}
\dt (\dt u_N) = - \dt u_N + \sqrt{2} \xi.
\label{OUd}
\end{align}

\noi
Since $\vec \rho_N$ is invariant under the fractional NLW dynamics \eqref{fNLW} and the white noise measure $\mu_0$ on $\dt u_N$ is invariant under the Ornstein-Uhlenbeck dynamics \eqref{OUd}, we deduce the invariance of $\vec \rho_N$ under the dynamics of \eqref{SNLW3a}. To summarize, we have the following result.

\begin{lemma}
\label{LEM:GWP4}
Let $\al > 1$, $\eps > 0$, and $M, N \in \N$ with $M \geq N$. Then, there exists a $\vec \rho_M \otimes \PP_2$-full measure set $\Si_M \subset \vec H^{\al - \frac 32 - \eps} \times \O_2$ such that for any $(\vec u_0, \o_2) \in \Si_M$, the solution $\vec u_N = \vec u_N (\vec u_0, \o_2) $ to the frequency truncated fractional hyperbolic $\Phi^4_3$-model \eqref{SNLW3a} exists globally in time. Moreover, if $M = N$, then $\vec \rho_N$ is invariant under the resulting dynamics in the sense that if $\Law (\vec u_0) = \rhoo_N$ then $\Law (\vec u_N(t)) = \rhoo_N$ for any $t \in \R_+$.
\end{lemma}

We refer the readers to \cite[Proposition~5.1]{ORTz} for a more detailed proof of Lemma~\ref{LEM:GWP4}.

\subsection{Uniform exponential integrability of the truncated enhanced data set}
\label{SUBSEC:exp}

In this subsection, we show uniform exponential integrability of the truncated enhanced data set $\pmb{\Xi}_{t_0, N} (\vec u_0, \o_2)$ as defined in \eqref{data3x} with respect to the truncated fractional $\Phi^4_3$-measure $\vec{\rho}_N \otimes \PP_2$. Also, we establish uniform integrability for the difference of the truncated enhanced data sets. For this purpose, for any $T > 0$, $1 < \al \leq \frac 32$, and $\eps, \dl > 0$ sufficiently small, we define the following norm similar to \eqref{Xae1}:
\begin{align}
\begin{split}
\| \pmb{\Xi}_N \|_{\mathcal{Y}_{T}^{\al, \eps, \dl}} &\deff \| \vec{\<1>}_N \|_{C_{T} \vec W_x^{\al - \frac 32 - \eps, \infty}} + \| \vec{\<30>}_N \|_{C_{T} \vec W_x^{3\al - 3 - \eps, \infty} + \vec Y_T^{\al - \frac 12, \frac 12 + \dl}} \\ 
&\quad + \big\| \vec{\<320>}_N \big\|_{\vec Y_{T}^{\al - \frac 12, \frac 12 + \dl}}   + \big\| \vec{\<70>}_N \big\|_{\vec Y_{T}^{\al - \frac 12, \frac 12 + \dl}} \\
&\quad + \big\| \vec \If^{\,\<31>_N} \big\|_{\vec \J^{\,\al - \frac 12 - \eps, \frac 12 + \dl, \al - \frac 12 + \eps,  \frac 12 + \dl}_{T}}  + \big\| \vec \If^{\,\<2>_N} \big\|_{\vec \J^{\,\al - \frac 12 - \eps, \frac 12 + \dl, \al - \frac 12 + \eps,  \frac 12 + \dl}_{T}} ,
\end{split}
\label{Xae}
\end{align}

\noi
where, given $u_{t_0} \in X_{[t_0, t_1]}^{s, b}$ and $S_{t_0} \in \L (X_{[t_0, t_1]}^{s_1, b_1}; X_{[t_0, t_1]}^{s_2, b_2})$ for any $0 \leq t_0 < t_1 \leq T$ with $s, s_1, s_2, b, b_1, b_2 \in \R$, we define
\begin{align}
\| u_\cdot \|_{Y_T^{s, b}} \deff \sup_{0 \leq t_0 < t_1 \leq T} \| u_{t_0} \|_{X_{[t_0, t_1]}^{s, b}}
\label{Ysb}
\end{align}

\noi
and
\begin{align}
\| S_\cdot \|_{\J_{T}^{s_1, b_1, s_2, b_2}} \deff \sup_{0 \leq t_0 < t_1 \leq T} (t_1 - t_0)^{- \ta} \| S_{t_0} \|_{\L (X_{[t_0, t_1]}^{s_1, b_1}; X_{[t_0, t_1]}^{s_2, b_2})}
\label{Jop}
\end{align}

\noi
for some small $\ta = \ta (b_1, b_2) > 0$ and we write $\vec Y_T^{s, b} = Y_T^{s, b} \times Y_T^{s - \al, b}$ and $\vec \J_{T}^{\,s_1, b_1, s_2, b_2} = \J_{T}^{s_1, b_1, s_2, b_2} \times \J_{T}^{s_1, b_1, s_2 - \al, b_2}$.
As in \cite[Proposition~6.5]{OOT2}, our proof of the regularity estimates for the enhanced data set $\pmb{\Xi}$ is based on the Bou\'e-Dupuis variational formula. However, as mentioned in Subsection~\ref{SUBSEC:1.3}, we have more involved stochastic objects compared to those in \cite{OOT2}. 

Our main novelty here is the use of the sum space $C_{T} \vec W_x^{3\al - 3 - \eps, \infty} + \vec X_T^{\al - \frac 12, \frac 12 + \dl}$ for the cubic stochastic object $\vec{\<30>}_N$. On the one hand, the drift terms in $\vec{\<30>}_N$ contain the Duhamel operator $\vec \I$ in \eqref{lin1} which does not have good boundedness properties in $L^p$ for $2 < p \leq \infty$, so that the Sobolev space $\vec W^{3\al - 3 - \eps, \infty}$ alone is not suitable for $\vec{\<30>}_N$. On the other hand, the smoothing on time regularity for $\I$ as in Lemma~\ref{LEM:nhomo} inspires us to use the $\vec X_T^{\al - \frac 12, \frac 12 + \dl}$ to bound the drift terms.

Let us now state and prove the following proposition.

\begin{proposition}[Uniform exponential integrability for the enhanced data set] \label{PROP:exptail2}
Let $T>0$ and $1 < \al \leq \frac 32$. Let $\eps, \dl > 0$ be sufficiently small.
Then, for $0 < \be < \frac{2}{21}$,
we have
\begin{align}
\int \E_{\PP_2} \Big[ \exp \Big( \| \pmb{\Xi}_N (\vec u_0, \o_2) \|_{\mathcal{Y}^{\al, \eps, \dl}_T}^\be \Big)
\Big]d \rhoo_N (\vec u_0)
\le C(T,  \be) < \infty
\label{exp01}
\end{align}
uniformly in $N \in \N$,
where the $\mathcal{Y}^{\al, \eps, \dl}_T$-norm and
the truncated enhanced data set $\pmb{\Xi}_N (\vec u_0, \o_2)$ are as in \eqref{Xae} and \eqref{data3x}, respectively.
Here,
$\E_{\PP_2}$ denotes an expectation
with respect to
the  probability measure $\PP_2$ on $\o_2 \in \O_2$.

Moreover, for $0< \be < \frac{2}{21}$, there exists small $\gamma > 0$ such that
\begin{align}
\int \E_{\PP_2} \Big[ \exp \Big( N_2^\gamma \| \pmb{\Xi}_{N_1} (\vec u_0, \o_2)
- \pmb{\Xi}_{N_2} (\vec u_0, \o_2)
 \|_{\mathcal{Y}^{\al, \eps, \dl}_T}^\be \Big) \Big]d \rhoo_N (\vec u_0) \le C(T, \be) < \infty
\label{exp02}
\end{align}
uniformly in $N, N_1, N_2 \in \N$ with $N \ge N_1 \ge N_2$.

\end{proposition}

\begin{proof}
We consider the six stochastic objects separately.

\smallskip \noi
\textbf{Case 1:} The stochastic convolution $\vec{\<1>}$.

We first consider \eqref{exp01}.
For fixed $u_1$ and $\o_2$, by the Bou\'e-Dupuis variational formula (Lemma~\ref{LEM:BD}) and \eqref{BDRNd}, we have
\begin{align*}
\log &\int \exp \Big( \big\| \vec{\<1>}_N (u_0, u_1, \o_2) \big \|_{ C_T \vec W_x^{\al - \frac 32 - \eps, \infty} }^\be \Big) d \rho_N(u_0) \\
&= \sup_{\dot{\Ups}^N \in \mathbb{H}_a^\al} 
\E \bigg[ \big\| \vec{\<1>}_N (Y + \Dr, u_1, \o_2) \big \|_{ C_T \vec W_x^{\al - \frac 32 - \eps, \infty} }^\be
- \wt R_N^\dia (Y + \Ups^N - \ZZ_N) \\
&\quad - \frac 12 \int_0^1 \| \dot{\Ups}^N (t) \|_{H_x^\al}^2 dt \bigg]
- \log Z_N,
\end{align*}

\noi
where $Y = Y(1)$ is defined in \eqref{defY}, $\ZZ_N = \ZZ_N (1) = \pi_N \ZZ^N (1)$ is defined in \eqref{defZZ}, $\Dr = \Ups^N - \ZZ_N$, and $\wt R_N^\dia$ is defined as in \eqref{RNt}.
By \eqref{logZN} and \eqref{ZNlow}, there exists $\eps_0, C_0 > 0$ such that
\begin{align}
\begin{split}
\log &\int \exp \Big( \big\| \vec{\<1>}_N (u_0, u_1, \o_2) \big \|_{ C_T \vec W_x^{\al - \frac 32 - \eps, \infty} }^\be \Big) d \rho_N(u_0) \\
& \leq \sup_{\dot{\Ups}^N \in \mathbb{H}_a^\al} 
\E \bigg[ \big\| \vec{\<1>}_N (Y + \Dr, u_1, \o_2) \big \|_{ C_T \vec W_x^{\al - \frac 32 - \eps, \infty} }^\be
 - \eps_0 \| \Ups_N \|_{H^\al}^2  \bigg] + C_0
\end{split}
\label{et2a}
\end{align}

\noi
with $\Ups_N = \Ups_N (1) = \pi_N \Ups^N (1)$, uniformly in $u_1$ and $\o_2$.

In view of \eqref{enh0a}, we can write
\begin{align}
\vec{\<1>}_N (Y + \Dr, u_1, \o_2) = \vec{\<1>}_N (Y, u_1, \o_2) + \pi_N \vec S(t) (\Dr, 0).
\label{elin1}
\end{align}
We recall from \eqref{defY} that $\Law (Y) = \mu_\al$ with $\mu_\al$ being defined in \eqref{gauss2} and so $\Law (Y, u_1) = \vec \mu_\al = \mu_\al \otimes \mu_0$. Thus, from the tail estimate in Lemma~\ref{LEM:sto_reg} for $\vec{\<1>}_N$ and Lemma~\ref{LEM:YNk} along with \eqref{vecSI} and \eqref{St0}, we have
\begin{align}
\big\| \vec{\<1>}_N &(Y, u_1, \o_2) \big\|_{C_T \vec W_x^{\al - \frac 32 - \eps, \infty}}^2 + \big\| \pi_N \vec S(t) (\ZZ_N, 0) \big\|_{C_T \vec W_x^{5 \al - \frac 92 - \eps, \infty}}^{\frac 23} \leq K(Y, u_1, \o_2),
\label{elin2}
\end{align}
where $K(Y, u_1, \o_2) \geq 1$ satisfies
\begin{align}
\E_{\vec \mu_\al \otimes \PP_2} \big[ \exp (\dl K (Y, u_1, \o_2)) \big] < \infty
\label{expK}
\end{align}
for $\dl > 0$ sufficiently small. Also, by Sobolev's embedding, \eqref{vecSI}, and \eqref{St0}, we have
\begin{align}
\big\| \pi_N \vec S(t) (\Ups^N, 0) \big\|_{C_T \vec W_x^{\al - \frac 32 - \eps, \infty}} \les \big\|  \vec S(t) (\Ups_N, 0) \big\|_{C_T \vec H_x^{\al}} \les \| \Ups_N \|_{H^{\al}} .
\label{elin3}
\end{align}

\noi
Thus, by \eqref{elin1}, the fact that $\Dr = \Ups^N - \ZZ_N$, \eqref{elin2}, and \eqref{elin3}, we obtain
\begin{align}
\begin{split}
\big\| &\vec{\<1>}_N (Y + \Dr, u_1, \o_2) \big\|_{ C_T \vec W_x^{\al - \frac 32 - \eps, \infty} } \\ 
&\les \big\| \vec{\<1>}_N (Y, u_1, \o_2) \big\|_{ C_T \vec W_x^{\al - \frac 32 - \eps, \infty} } +  \big\| \pi_N \vec S(t) (\Dr, 0) \big\|_{C_T \vec W_x^{\al - \frac 32 - \eps, \infty}} \\
&\les K(Y, u_1, \o_2)^{\frac 32} + \| \Ups_N \|_{H^\al}.
\end{split}
\label{elin4}
\end{align}

By \eqref{elin4} and Young's inequality (with $0 < \be < 1$), we have
\begin{align}
\begin{split}
&\sup_{\dot{\Ups}^N \in \mathbb{H}_a^\al}  \E \bigg[ \big\| \vec{\<1>}_N (Y + \Dr, u_1, \o_2) \big \|_{ C_T \vec W_x^{\al - \frac 32 - \eps, \infty} }^\be
- \eps_0 \| \Ups_N \|_{H^\al}^2  \bigg] \\
&\leq C \E \Big[ K(Y, u_1, \o_2)^{\frac 32 \be} \Big] + \sup_{\dot{\Ups}^N \in \mathbb{H}_a^\al} \E \bigg[ C \| \Ups_N \|_{H^\al}^{\beta} - \eps_0 \| \Ups_N \|_{H^\al}^2 \bigg] \\
&\les \E \Big[ K(Y, u_1, \o_2)^{\frac 32 \be} \Big] + 1
\end{split}
\label{elin5}
\end{align}

\noi
for some constant $C > 0$. Therefore, by \eqref{et2a}, \eqref{elin5}, Young's inequality, and Jensen's inequality, we have
\begin{align}
\begin{split}
\int \exp \Big( \big\| \vec{\<1>}_N (u_0, u_1, \o_2) \big \|_{ C_T \vec W_x^{\al - \frac 32 - \eps, \infty} }^\be \Big) d \rho_N(u_0) &\les \exp \Big( C \E \big[ K(Y, u_1, \o_2)^{\frac 32 \be} \big] \Big) \\
&\les \exp \Big( \dl \E \big[ K(Y, u_1, \o_2) \big]  \Big) \\
&\leq \int \exp \big( \dl K(Y, u_1, \o_2) \big) d \mu_\al (Y)
\end{split}
\label{elin6}
\end{align}
for $0 < \be < \frac 23$. We obtain the desired bound \eqref{exp01} by integrating in $(u_1, \o_2)$ with respect to $\mu_0 \otimes \PP_2$ and using \eqref{expK}.

We now briefly discuss how to obtain the difference estimate \eqref{exp02} for $\vec{\<1>}_N (\vec u_0, \o_2)$ with $0 < \be < \frac 23$. By using the Bou\'e-Dupuis variational formula as above, we obtain the following bound which is similar to \eqref{et2a}:
\begin{align*}
\log &\int \exp \Big( N_2^\g \big\| \vec{\<1>}_{N_1} (u_0, u_1, \o_2) - \vec{\<1>}_{N_2} (u_0, u_1, \o_2) \big\|_{C_T \vec W_x^{\al - \frac 32 - \eps, \infty}}^\be \Big) d \rho_N (u_0) \\
& \leq 
\sup_{\dot{\Ups}^N \in \mathbb{H}_a^\al} \E \bigg[ N_2^\g \big\| \vec{\<1>}_{N_1} (Y + \Dr, u_1, \o_2) - \vec{\<1>}_{N_2} (Y + \Dr, u_1, \o_2) \big \|_{ C_T \vec W_x^{\al - \frac 32 - \eps, \infty} }^\be \\
&\quad - \eps_0 \| \Ups_N \|_{H^\al}^2  \bigg] + C_0
\end{align*}

\noi
for some constant $C_0 > 0$. We also use the decomposition \eqref{elin1}. By the difference tail estimates in Lemma \ref{LEM:sto_reg} for $\vec{\<1>}_N$ and Lemma~\ref{LEM:YNk} along with \eqref{vecSI} and \eqref{St0}, we have
\begin{align*}
N_2^{\g'} &\big\| \vec{\<1>}_{N_1} (Y, u_1, \o_2) - \vec{\<1>}_{N_2} (Y, u_1, \o_2) \big\|_{C_T \vec W_x^{\al - \frac 32 - \eps, \infty}}^2 \\
&+ N_2^{\g'} \big\| \pi_{N_1} \vec S(t) (\ZZ_N, 0) - \pi_{N_2} \vec S(t) (\ZZ_N, 0) \big\|_{C_T \vec W_x^{5 \al - \frac 92 - \eps, \infty}}^{\frac 23} \leq K(Y, u_1, \o_2),
\end{align*}

\noi
where $\g' > 0$ is sufficiently small and $K (Y, u_1, \o_2) \geq 1$ satisfies \eqref{expK}. Then, by using similar steps as in \eqref{elin3}, \eqref{elin4}, \eqref{elin5}, and \eqref{elin6}, we obtain
\begin{align*}
\int \exp &\Big( N_2^{\g} \big\| \vec{\<1>}_{N_1} (Y, u_1, \o_2) - \vec{\<1>}_{N_2} (Y, u_1, \o_2) \big\|_{C_T \vec W_x^{\al - \frac 32 - \eps, \infty}}^\be \Big) d \rho_N (u_0) \\
&\leq \int \exp \big( \dl K (Y, u_1, \o_2) \big) d \mu_\al (Y),
\end{align*}

\noi
given $\g > 0$ sufficiently small. Thus, the desired difference estimate follows by integrating in $(u_1, \o_2)$ with respect to $\mu_0 \otimes \PP_2$. From then on, we only focus on proving the uniform exponential integrability \eqref{exp01}.

\smallskip \noi
\textbf{Case 2:} The cubic stochastic object $\vec{\<30>}$.

For simplicity, from now on we drop the dependence on $u_1$ and $\o_2$ also drop the ``0'' in the second argument of $\vec S(t)$. Also, for inequalities, we suppress the dependence on $T$.

We proceed as in Case 1 using the Bou\'e-Dupuis variational formula to obtain a similar bound as \eqref{et2a}:
\begin{align*}
\begin{split}
\log &\int \exp \Big( \big\| \vec{\<30>}_N (u_0) \big \|_{ C_T \vec W_x^{3\al - 3 - \eps, \infty} + \vec X_T^{\al - \frac 12, \frac 12 + \dl} }^\be \Big) d \rho_N(u_0) \\
& \leq 
\sup_{\dot{\Ups}^N \in \mathbb{H}_a^\al}  \E \bigg[ \big\| \vec{\<30>}_N (Y + \Dr) \big \|_{ C_T \vec W_x^{3\al - 3 - \eps, \infty} + \vec X_T^{\al - \frac 12, \frac 12 + \dl} }^\be
- \eps_0 \| \Ups_N \|_{H^\al}^2  \bigg] + C_0
\end{split}
\end{align*}

\noi
for some constant $C_0 > 0$, where $\Law (Y) = \mu_\al$. With $\Dr = \Ups^N - \ZZ_N$, the following bound will be useful and it follows from Lemma~\ref{LEM:homo} and Lemma~\ref{LEM:YNk}:
\begin{align}
\begin{split}
\| \pi_N \vec S(t) \Dr \|_{\vec X_T^{\al - \frac 12, \frac 12 + \dl}} &\les \| \pi_N \Dr \|_{H^{\al - \frac 12}} \\
&\les \| \Ups_N \|_{H^\al} + \| \ZZ_N \|_{W^{5\al - \frac 92 - \eps, \infty}} \\
&\les \| \Ups_N \|_{H^\al} + K(Y)^{\frac 32},
\end{split}
\label{helper}
\end{align}

\noi
where $K(Y) = K(Y, u_1, \o_2) \geq 1$ satisfies \eqref{expK}.

By \eqref{enh0} and \eqref{elin1}, we have
\begin{align}
\begin{split}
\vec{\<30>}_N (Y + \Dr) &= \vec{\<30>}_N (Y) + 3 \pi_N \vec \I \big( \<2>_N (Y) \pi_N S(t) \Dr \big)  \\
&\quad + 3 \pi_N \vec \I \big( \<1>_N (Y) (\pi_N S(t) \Dr)^2 \big) + \pi_N \vec \I \big( (\pi_N S(t) \Dr)^3 \big).
\end{split}
\label{ecub}
\end{align}

\noi
By Lemma \ref{LEM:sto_reg} for $\vec{\<30>}_N$, we have
\begin{align}
\big\| \vec{\<30>}_N (Y) \big\|_{L_T^\infty \vec W_x^{3\al - 3 - \eps, \infty}}^{\frac 23} \leq K (Y).
\label{ecub1}
\end{align}

\noi
By Lemma \ref{LEM:sto_reg} for $\vec \If^{\,\<2>_N}$ and \eqref{helper}, we have 
\begin{align}
\begin{split}
\big\| \pi_N \vec \I \big( \<2>_N (Y) \pi_N S(t) \Dr \big) \big\|_{\vec X_T^{\al - \frac 12, \frac 12 + \dl}} &= \big\| \vec \If^{\,\<2>_N} (Y) \big( \pi_N S(t) \Dr \big) \big\|_{\vec X_T^{\al - \frac 12, \frac 12 + \dl}} \\
&\les K(Y) \| \pi_N S(t) \Dr \|_{X_T^{\al - \frac 12, \frac 12 + \dl}} \\
&\les K(Y) \big( \| \Ups_N \|_{H^\al} + K(Y)^{\frac 32} \big).
\end{split}
\label{ecub2}
\end{align}

\noi
By the inhomogeneous linear estimate in Lemma \ref{LEM:nhomo}, \eqref{tri3} in Lemma \ref{LEM:str4}, \eqref{elin2}, and \eqref{helper}, we have
\begin{align}
\begin{split}
\big\| \pi_N \vec \I \big( \<1>_N (Y) (\pi_N S(t) \Dr)^2 \big) \big\|_{\vec X_T^{\al - \frac 12, \frac 12 + \dl}} &\les K(Y)^{\frac 12} \| \pi_N S(t) \Dr \|_{X_T^{\al - \frac 12, \frac 12 + \dl}}^2 \\
&\les K(Y)^{\frac 12} \| \Ups_N \|_{H^\al}^2 + K(Y)^{\frac 72}.
\end{split}
\label{ecub3}
\end{align}

\noi
By Lemma \ref{LEM:nhomo}, H\"older's inequality, Sobolev's embedding, and Lemma~\ref{LEM:YNk}, we have
\begin{align}
\begin{split}
\big\| \pi_N \vec \I \big( (\pi_N S(t) \Dr)^3 \big) \big\|_{\vec X_T^{\al - \frac 12, \frac 12 + \dl}} &\les \| \pi_N S(t) \Dr \|_{L_T^\infty L_x^6}^3 \\
&\les \| \Ups_N \|_{H^\al}^3 + \| \ZZ_N \|_{W^{5 \al - \frac 92 - \eps, \infty}}^3 \\
&\les \| \Ups_N \|_{H^\al}^3 + K(Y)^{\frac 92}.
\end{split}
\label{ecub4}
\end{align}

Combining \eqref{ecub}, \eqref{ecub1}, \eqref{ecub2}, \eqref{ecub3}, and \eqref{ecub4} and proceeding as in Case 1, we can show that \eqref{exp01} and \eqref{exp02} hold for $\vec{\<30>}_N$ with $0 < \be < \frac 29$.

\smallskip \noi
\textbf{Case 3:} The quintic stochastic object $\<320>$.

We proceed as in Case 1 using the Bou\'e-Dupuis variational formula to obtain a similar bound as \eqref{et2a}:
\begin{align*}
\begin{split}
\log &\int \exp \Big( \big\| \vec{\<320>}_N (u_0) \big \|_{\vec Y_T^{\al - \frac 12, \frac 12 + \dl} }^\be \Big) d \rho_N(u_0) \\
& \leq 
\sup_{\dot{\Ups}^N \in \mathbb{H}_a^\al}  \E \bigg[ \big\| \vec{\<320>}_N (Y + \Dr) \big \|_{\vec Y_T^{\al - \frac 12, \frac 12 + \dl} }^\be - \eps_0 \| \Ups_N \|_{H^\al}^2 \bigg] + C_0
\end{split}
\end{align*}

\noi
for some constant $C_0 > 0$, where $\Law (Y) = \mu_\al$ and the $\vec Y_T^{\al - \frac 12, \frac 12 + \dl}$-norm is defined in \eqref{Ysb}. For convenience, we define
\begin{align}
\pmb{\Ld} \overset{\textup{def}}{=} \big\{ 3 \pi_N \I \big( \<2>_N (Y) \pi_N S(t) \Dr \big), 3 \pi_N \I \big( \<1>_N (Y) (\pi_N S(t) \Dr)^2 \big), \pi_N \I \big( (\pi_N S(t) \Dr)^3 \big) \big\},
\label{defLd}
\end{align}

\noi
so that by \eqref{ecub2}, \eqref{ecub3}, and \eqref{ecub4}, we have
\begin{align}
\sum_{\Ld \in \pmb{\Ld}} \| \Ld \|_{X_T^{\al - \frac 12, \frac 12 + \dl}} \les \| \Ups_N \|_{H^\al}^3 + K(Y)^{\frac 92},
\label{helper2-4}
\end{align}

\noi
where $K(Y) = K(Y, u_1, \o_2) \geq 1$ satisfies \eqref{expK}.

By \eqref{enh0} and \eqref{elin1}, we have
\begin{align}
\begin{split}
&\vec{\<320>}_{t_0, N} (Y + \Dr) \\
&= \vec{\<320>}_{t_0, N} (Y) + 2 \pi_N \vec \I_{t_0} \big( \<30>_N (Y) \<1>_N (Y) \pi_N S(t) \Dr \big)  + \pi_N \vec \I_{t_0} \big( \<30>_N (Y) (\pi_N S(t) \Dr)^2 \big) \\
&\quad + \sum_{\Ld \in \pmb{\Ld}} \Big( \vec \If_{t_0}^{\,\<2>_N} (Y) (\Ld) + 2 \pi_N \vec \I_{t_0} \big( \Ld \cdot \<1>_N (Y) \pi_N S(t) \Dr \big) + \pi_N \vec \I_{t_0} \big( \Ld \cdot (\pi_N S(t) \Dr)^2 \big) \Big).
\end{split}
\label{equin}
\end{align}

\noi
Let $t_1$ be arbitrary such that $t_0 < t_1 \leq T$. By a slight variant of Lemma \ref{LEM:sto_reg} for $\vec{\<320>}_N$, we have
\begin{align}
\big\| \vec{\<320>}_{t_0, N} (Y) \big\|_{\vec X_{[t_0, t_1]}^{\al - \frac 12, \frac 12 + \dl}} \leq K(Y)^{\frac 52} ,
\label{equin1}
\end{align}

\noi
where $K(Y) = K(Y, u_1, \o_2) \geq 1$ satisfies \eqref{expK}.
By a slight variant of Lemma \ref{LEM:sto_reg} for $\vec \If^{\,\<31>_N}$ and \eqref{helper}, we have
\begin{align}
\begin{split}
\big\| &\pi_N \vec \I_{t_0} \big( \<30>_N (Y) \<1>_N (Y) \pi_N S(t) \Dr \big) \big\|_{\vec X_{[t_0, t_1]}^{\al - \frac 12, \frac 12 + \dl}} \\
&\les \big\| \vec \If_{t_0}^{\,\<31>_N} (Y) \big\|_{\L (X_{[t_0, t_1]}^{\al - \frac 12, \frac 12 + \dl}; \vec X_{[t_0, t_1]}^{\al - \frac 12, \frac 12 + \dl})} \| \pi_N S(t) \Dr \|_{X_T^{\al - \frac 12, \frac 12 + \dl}} \\
&\les K(Y)^2 \| \Ups_N \|_{H^\al} + K(Y)^{\frac 72} .
\end{split}
\label{equin2}
\end{align}

\noi
By the inhomogeneous linear estimate in Lemma \ref{LEM:nhomo}, H\"older's inequality, \eqref{ecub1}, the Strichartz estimate \eqref{Lp_str} with $p = 4$, and \eqref{helper}, we have
\begin{align}
\begin{split}
\big\| \pi_N \vec \I_{t_0} \big( \<30>_N (Y) (\pi_N S(t) \Dr)^2 \big) \big\|_{\vec X_{[t_0, t_1]}^{\al - \frac 12, \frac 12 + \dl}} &\les \big\| \<30>_N (Y) (\pi_N S(t) \Dr)^2 \big\|_{L_T^2 L_x^2} \\
&\les \big\| \<30>_N (Y) \big\|_{L_T^\infty L_x^\infty} \| \pi_N S(t) \Dr \|_{L_T^4 L_x^4}^4 \\
&\les K(Y)^{\frac 32} \| \pi_N S(t) \Dr \|_{X_T^{\al - \frac 12, \frac 12 + \dl}}^2 \\
&\les K(Y)^{\frac 32} \big( \| \Ups_N \|_{H^\al}^2 + K(Y)^3 \big).
\end{split}
\label{equin3}
\end{align}

\noi
By a slight variant of Lemma \ref{LEM:sto_reg} for $\vec \If^{\,\<2>_N}$ and \eqref{helper2-4}, we have
\begin{align}
\begin{split}
\bigg\| \sum_{\Ld \in \pmb{\Ld}} \vec \If_{t_0}^{\,\<2>_N} (Y) (\Ld) \bigg\|_{\vec X_{[t_0, t_1]}^{\al - \frac 12, \frac 12 + \dl}} &\les \sum_{\Ld \in \pmb{\Ld}} \big\| \vec \If_{t_0}^{\,\<2>_N} (Y) \big\|_{\L (X_{[t_0, t_1]}^{\al - \frac 12, \frac 12 + \dl}; \vec X_{[t_0, t_1]}^{\al - \frac 12, \frac 12 + \dl})} \| \Ld \|_{X_T^{\al - \frac 12, \frac 12 + \dl}} \\
&\les K(Y) \big( \| \Ups_N \|_{H^\al}^3 + K(Y)^{\frac 92} \big).
\end{split}
\label{equin4}
\end{align}

\noi
By Lemma \ref{LEM:nhomo}, \eqref{tri3} in Lemma \ref{LEM:str4}, \eqref{elin2}, \eqref{helper}, and \eqref{helper2-4}, we have
\begin{align}
\begin{split}
\bigg\| &\sum_{\Ld \in \pmb{\Ld}} \pi_N \vec \I_{t_0} \big( \Ld \cdot \<1>_N (Y) \pi_N S(t) \Dr \big) \bigg\|_{\vec X_{[t_0, t_1]}^{\al - \frac 12, \frac 12 + \dl}} \\
&\les \sum_{\Ld \in \pmb{\Ld}} \| \<1>_N (Y) \|_{L_T^\infty W_x^{\al - \frac 32 - \eps, \infty}} \| \Ld \|_{X_T^{\al - \frac 12, \frac 12 + \dl}} \| \pi_N S(t) \Dr \|_{X_T^{\al - \frac 12, \frac 12 + \dl}} \\
&\les K(Y)^{\frac 12} \big( \| \Ups_N \|_{H^\al}^4 + K(Y)^{6} \big).
\end{split}
\label{equin5}
\end{align}

\noi
By Lemma \ref{LEM:nhomo}, Lemma \ref{LEM:str_var} with $v \equiv 1$, \eqref{helper}, and \eqref{helper2-4}, we have
\begin{align}
\begin{split}
\bigg\| &\sum_{\Ld \in \pmb{\Ld}} \pi_N \vec \I_{t_0} \big( \Ld \cdot (\pi_N S(t) \Dr)^2 \big) \bigg\|_{\vec X_{[t_0, t_1]}^{\al - \frac 12, \frac 12 + \dl}} \\
&\les \sum_{\Ld \in \pmb{\Ld}} \| \Ld \|_{X_T^{\al - \frac 12, \frac 12 + \dl}} \| \pi_N S(t) \Dr \|_{X_T^{\al - \frac 12, \frac 12 + \dl}}^2 \\
&\les \| \Ups_N \|_{H^\al}^5 + K(Y)^{\frac{15}{2}}.
\end{split}
\label{equin6}
\end{align}

Combining \eqref{equin}, \eqref{equin1}, \eqref{equin2}, \eqref{equin3}, \eqref{equin4}, \eqref{equin5}, and \eqref{equin6} and proceeding as in Case 1, we can show that \eqref{exp01} and \eqref{exp02} hold for $\vec{\<320>}_N$ with $0 < \be < \frac{2}{15}$.

\smallskip \noi
\textbf{Case 4:} The septic stochastic object $\vec{\<70>}$.

We proceed as in Case 1 using the Bou\'e-Dupuis variational formula to obtain a similar bound as \eqref{et2a}:
\begin{align*}
\begin{split}
\log &\int \exp \Big( \big\| \vec{\<70>}_N (u_0) \big \|_{\vec Y_T^{\al - \frac 12, \frac 12 + \dl} }^\be \Big) d \rho_N(u_0) \\
& \leq 
\sup_{\dot{\Ups}^N \in \mathbb{H}_a^\al}  \E \bigg[ \big\| \vec{\<70>}_N (Y + \Dr) \big \|_{\vec Y_T^{\al - \frac 12, \frac 12 + \dl} }^\be - \eps_0 \| \Ups_N \|_{H^\al}^2 \bigg] + C_0
\end{split}
\end{align*}

\noi
for some constant $C_0 > 0$, where $\Law (Y) = \mu_\al$ and the $\vec Y_T^{\al - \frac 12, \frac 12 + \dl}$-norm is defined in \eqref{Ysb}. By \eqref{enh0}, \eqref{elin1}, and \eqref{defLd}, we have
\begin{align}
\begin{split}
&\vec{\<70>}_{t_0, N} (Y + \Dr) \\
&= \vec{\<70>}_{t_0, N} (Y) + \pi_N \vec \I_{t_0} \big( (\<30>_N (Y))^2 \pi_N S(t) \Dr \big) \\
&\quad + 2 \sum_{\Ld \in \pmb{\Ld}} \pi_N \vec \I_{t_0} \big( \Ld \cdot \<30>_N (Y) \<1>_N (Y) \big)
 + 2 \sum_{\Ld \in \pmb{\Ld}} \pi_N \vec \I_{t_0} \big( \Ld \cdot \<30>_N (Y) \pi_N S(t) \Dr \big) \\
&\quad + \sum_{\Ld_1, \Ld_2 \in \pmb{\Ld}} \pi_N \vec \I_{t_0} \big( \Ld_1 \cdot \Ld_2 \cdot \<1>_N (Y)  \big) + \sum_{\Ld_1, \Ld_2 \in \pmb{\Ld}} \pi_N \vec \I_{t_0} \big( \Ld_1 \cdot \Ld_2 \cdot \pi_N S(t) \Dr  \big).
\end{split}
\label{esep}
\end{align}

Let $t_1$ be arbitrary such that $t_0 < t_1 \leq T$. By a slight variant of Lemma \ref{LEM:sto_reg} for $\vec{\<70>}_N$, we have
\begin{align}
\big\| \vec{\<70>}_{t_0, N} (Y) \big\|_{\vec X_{[t_0, t_1]}^{\al - \frac 12, \frac 12 + \dl}} \les K(Y)^{\frac 72},
\label{esep1}
\end{align}

\noi
where $K(Y) = K(Y, u_1, \o_2) \geq 1$ satisfies \eqref{expK}. 
By thee inhomogeneous linear estimate in Lemma~\ref{LEM:nhomo}, H\"older's inequality, \eqref{ecub1}, and \eqref{helper}, we have
\begin{align}
\begin{split}
\big\| \pi_N \vec \I_{t_0} \big( (\<30>_N (Y))^2 \pi_N S(t) \Dr \big) \big\|_{\vec X_{[t_0, t_1]}^{\al - \frac 12, \frac 12 + \dl}} &\les \big\| (\<30>_N (Y))^2 \pi_N S(t) \Dr \big\|_{L_T^2 L_x^2} \\
&\les \| \<30>_N (Y) \|_{L_T^\infty L_x^\infty}^2 \| \pi_N S (t) \Dr \|_{L_T^2 L_x^2} \\
&\les K (Y)^3 \| \pi_N S (t) \Dr \|_{X_T^{\al - \frac 12, \frac 12 + \dl}} \\
&\les K (Y)^3 \big( \| \Ups_N \|_{H^\al} + K (Y)^{\frac 32} \big) .
\end{split}
\label{esep1-1}
\end{align}

\noi
By a slight variant of Lemma \ref{LEM:sto_reg} for $\vec \If^{\,\<31>_N}$ and \eqref{helper2-4}, we have
\begin{align}
\begin{split}
\bigg\| &\sum_{\Ld \in \pmb{\Ld}} \pi_N \vec \I_{t_0} \big( \Ld \cdot \<30>_N (Y) \<1>_N (Y) \big) \bigg\|_{\vec X_{[t_0, t_1]}^{\al - \frac 12, \frac 12 + \dl}} \\
&\les \big\| \vec \If_{t_0}^{\,\<31>_N} (Y) \big\|_{\L (X_{[t_0, t_1]}^{\al - \frac 12, \frac 12 + \dl}; \vec X_{[t_0, t_1]}^{\al - \frac 12, \frac 12 + \dl})} \sum_{\Ld \in \pmb{\Ld}} \| \Ld \|_{X_T^{\al - \frac 12, \frac 12 + \dl}} \\
&\les K(Y)^2 \| \Ups_N \|_{H^\al}^3 + K(Y)^{\frac{13}{2}} .
\end{split}
\label{esep2}
\end{align}

\noi
By the inhomogeneous linear estimate in Lemma \ref{LEM:nhomo}, H\"older's inequality, the Strichartz estimate \eqref{Lp_str} with $p = 4$, \eqref{ecub1}, \eqref{helper2-4}, and \eqref{helper}, we have
\begin{align}
\begin{split}
\bigg\| &\sum_{\Ld \in \pmb{\Ld}} \pi_N \vec \I_{t_0} \big( \Ld \cdot \<30>_N (Y) \pi_N S(t) \Dr \big) \bigg\|_{\vec X_{[t_0, t_1]}^{\al - \frac 12, \frac 12 + \dl}} \\
&\les \sum_{\Ld \in \pmb{\Ld}} \big\| \Ld \cdot \<30>_N (Y) \pi_N S(t) \Dr \big) \big\|_{L_T^2 L_x^2} \\
&\les \sum_{\Ld \in \pmb{\Ld}} \| \Ld \|_{L_T^4 L_x^4} \big\| \<30>_N (Y) \big\|_{L_T^\infty L_x^\infty} \| \pi_N S(t) \Dr \|_{L_T^4 L_x^4} \\
&\les K(Y)^{\frac 32} \sum_{\Ld \in \pmb{\Ld}} \| \Ld \|_{X_T^{\al - \frac 12, \frac 12 + \dl}} \| \pi_N S(t) \Dr \|_{X_T^{\al - \frac 12, \frac 12 + \dl}} \\
&\les K(Y)^{\frac 32} \| \Ups_N \|_{H^\al}^4 + K(Y)^{\frac{15}{2}}.
\end{split}
\label{esep3}
\end{align}

\noi
By Lemma \ref{LEM:nhomo}, \eqref{tri3} in Lemma \ref{LEM:str4}, \eqref{elin2}, and \eqref{helper2-4}, we have
\begin{align}
\begin{split}
\bigg\| &\sum_{\Ld_1, \Ld_2 \in \pmb{\Ld}} \pi_N \vec \I_{t_0} \big( \Ld_1 \cdot \Ld_2 \cdot \<1>_N (Y)  \big) \bigg\|_{\vec X_{[t_0, t_1]}^{\al - \frac 12, \frac 12 + \dl}} \\
&\les \sum_{\Ld_1, \Ld_2 \in \pmb{\Ld}} \| \<1>_N \|_{L_T^\infty W_x^{\al - \frac 32 - \eps, \infty}} \| \Ld_1 \|_{X_T^{\al - \frac 12, \frac 12 + \dl}} \| \Ld_2 \|_{X_T^{\al - \frac 12, \frac 12 + \dl}} \\
&\les K(Y)^{\frac 12} \| \Ups_N \|_{H^\al}^6 + K(Y)^{\frac{19}{2}}.
\end{split}
\label{esep4}
\end{align}

\noi
By Lemma \ref{LEM:nhomo}, Lemma \ref{LEM:str_var} with $v \equiv 1$, \eqref{helper2-4}, and \eqref{helper}, we have
\begin{align}
\begin{split}
\bigg\| &\sum_{\Ld_1, \Ld_2 \in \pmb{\Ld}} \pi_N \vec \I_{t_0} \big( \Ld_1 \cdot \Ld_2 \cdot \pi_N S(t) \Dr  \big) \bigg\|_{\vec X_{[t_0, t_1]}^{\al - \frac 12, \frac 12 + \dl}} \\
&\les \sum_{\Ld_1, \Ld_2 \in \pmb{\Ld}} \| \Ld_1 \|_{X_T^{\al - \frac 12, \frac 12 + \dl}} \| \Ld_2 \|_{X_T^{\al - \frac 12, \frac 12 + \dl}} \| \pi_N S(t) \Dr \|_{X_T^{\al - \frac 12, \frac 12 + \dl}} \\
&\les \| \Ups_N \|_{H^\al}^7 + K(Y)^{\frac{21}{2}}.
\end{split}
\label{esep5}
\end{align}

Combining \eqref{esep}, \eqref{esep1}, \eqref{esep1-1}, \eqref{esep2}, \eqref{esep3}, \eqref{esep4}, and \eqref{esep5} and proceeding as in Case 1, we can show that \eqref{exp01} and \eqref{exp02} hold for $\vec{\<70>}_N$ with $0 < \be < \frac{2}{21}$.

\smallskip \noi
\textbf{Case 5:} The random operator with the cubic-linear stochastic object $\vec \If^{\,\<31>}$.

We proceed as in Case 1 using the Bou\'e-Dupuis variational formula along with \eqref{Jop} to obtain a similar bound as \eqref{et2a}:
\begin{align*}
\begin{split}
&\log \int \exp \Big( \big\| \vec \If^{\,\<31>_N} (u_0) \big \|_{\vec \J_T^{\,\al - \frac 12 - \eps, \frac 12 + \dl, \al - \frac 12 + \eps, \frac 12 + \dl } }^\be \Big) d \rho_N(u_0) \\
& \leq 
\sup_{\dot{\Ups}^N \in \mathbb{H}_a^\al}  \E \bigg[  \sup_{0 \leq t_0 < t_1 \leq T} (t_1 - t_0)^{-\ta} \sup_{\| v \|_{X_{[t_0, t_1]}^{\al - \frac 12 - \eps, \frac 12 + \dl} } \leq 1 } \big\| \pi_N \vec \I_{t_0} \big( \<30>_N \<1>_N (Y + \Dr) \cdot v \big) \big\|_{\vec X_{[t_0, t_1]}^{\al - \frac 12 + \eps, \frac 12 + \dl} }^\be \\
&\quad - \eps_0 \| \Ups_N \|_{H^\al}^2 \bigg] + C_0
\end{split}
\end{align*}

\noi
for some constant $C_0 > 0$, where $\Law (Y) = \mu_\al$ and $\ta = \ta (\dl) > 0$ is sufficiently small. By \eqref{enh0}, \eqref{elin1}, and \eqref{defLd}, we have
\begin{align}
\begin{split}
\<30>_N \<1>_N (Y + \Dr) &= \<30>_N (Y) \<1>_N (Y) + \<30>_N (Y) \pi_N S(t) \Dr \\
&\quad + \sum_{\Ld \in \pmb{\Ld}} \Ld \cdot \<1>_N (Y) + \sum_{\Ld \in \pmb{\Ld}} \Ld \cdot \pi_N S(t) \Dr.
\end{split}
\label{ecl}
\end{align}

By the proof of Lemma \ref{LEM:sto_reg} for $\vec \If^{\,\<31>_N}$, we have
\begin{align}
\big\| \<30>_N (Y) \<1>_N (Y) \big\|_{L_T^q W_x^{\al - \frac 32 - \eps, 6}}^{\frac 12} \leq K (Y),
\label{ecl1}
\end{align}

\noi
where $q > 1$ and $K(Y) = K(Y, u_1, \o_2) \geq 1$ satisfies \eqref{expK}. By H\"older's inequality, the change of variable $\Dr = \Ups^N - \ZZ_N$, Sobolev's embedding, and \eqref{ecub1}, we have 
\begin{align}
\begin{split}
\big\| \<30>_N (Y) \pi_N S(t) \Dr \big\|_{L_T^\infty W_x^{\al - \frac 32 - \eps, 6}} &\les \big\| \<30>_N (Y) \big\|_{L_T^\infty L_x^\infty} \| \pi_N S(t) \Dr \|_{L_T^\infty L_x^6} \\
&\les K(Y)^{\frac 32} \big( \| \Ups_N \|_{H^\al} + \| \ZZ_N \|_{W^{5 \al - \frac 92 - \eps, \infty}} \big) \\
&\les K(Y)^{\frac 32} \| \Ups_N \|_{H^\al} + K(Y)^3 .
\end{split}
\label{ecl2}
\end{align}

\noi
Thus, using similar steps as in \eqref{cublin_op} with $q \gg 1$ along with \eqref{ecl1} and \eqref{ecl2}, we obtain
\begin{align}
\begin{split}
\big\| &\pi_N \vec \I_{t_0} \big( \<30>_N (Y) \<1>_N (Y) \cdot v \big) \big\|_{\vec X_{[t_0, t_1]}^{\al - \frac 12 + \eps, \frac 12 +  \dl} } \\
&\les (t_1 - t_0)^{\ta} \big\| \<30>_N (Y) \<1>_N (Y) \big\|_{L_T^q W_x^{\al - \frac 32 - \eps, 6}} \| v \|_{X_{[t_0, t_1]}^{\al - \frac 12 - \eps, \frac 12 + \dl}} \\
&\les (t_1 - t_0)^\ta K(Y)^2 \| v \|_{X_{[t_0, t_1]}^{\al - \frac 12 - \eps, \frac 12 + \dl}} 
\end{split}
\label{ecl3}
\end{align}

\noi
and
\begin{align}
\begin{split}
\big\| &\pi_N \vec \I_{t_0} \big( \<30>_N (Y) \pi_N S(t) \Dr \cdot v \big) \big\|_{\vec X_{[t_0, t_1]}^{\al - \frac 12 + \eps, \frac 12 + \dl} } \\
&\les (t_1 - t_0)^\ta \big\| \<30>_N (Y) \pi_N S(t) \Dr \big\|_{L_T^\infty W_x^{\al - \frac 32 - \eps, 6}} \| v \|_{X_{[t_0, t_1]}^{\al - \frac 12 - \eps, \frac 12 + \dl}} \\
&\les (t_1 - t_0)^\ta \big( K(Y)^{\frac 32} \| \Ups_N \|_{H^\al} + K(Y)^3 \big) \| v \|_{X_{[t_0, t_1]}^{\al - \frac 12 - \eps, \frac 12 + \dl}}.
\end{split}
\label{ecl4}
\end{align}

\noi
By the inhomogeneous linear estimate in Lemma \ref{LEM:nhomo}, the time localization estimate \eqref{time2} in Lemma \ref{LEM:time}, \eqref{tri3} in Lemma \ref{LEM:str4}, \eqref{elin2}, and \eqref{helper2-4}, we have
\begin{align}
\begin{split}
\bigg\| &\sum_{\Ld \in \pmb{\Ld}} \pi_N \vec \I_{t_0} \big( \Ld \cdot \<1>_N (Y) \cdot v \big) \bigg\|_{\vec X_{[t_0, t_1]}^{\al - \frac 12 + \eps, \frac 12 + \dl}} \\
&\les (t_1 - t_0)^\ta \sum_{\Ld \in \pmb{\Ld}} \| \<1>_N (Y) \|_{L_T^\infty W_x^{\al - \frac 32 - \eps, \infty}} \| \Ld \|_{X_T^{\al - \frac 12, \frac 12 + \dl}} \| v \|_{X_{[t_0, t_1]}^{\al - \frac 12 - \eps, \frac 12 + \dl}} \\
&\les (t_1 - t_0)^\ta \big( K(Y)^{\frac 12} \| \Ups_N \|_{H^\al}^3 + K(Y)^5 \big) \| v \|_{X_{[t_0, t_1]}^{\al - \frac 12 - \eps, \frac 12 + \dl}}.
\end{split}
\label{ecl5}
\end{align}

\noi
By Lemma \ref{LEM:nhomo}, \eqref{time2} in Lemma \ref{LEM:time}, Lemma \ref{LEM:str_var} with $v \equiv 1$, \eqref{helper2-4}, and \eqref{helper}, we have
\begin{align}
\begin{split}
\bigg\| &\sum_{\Ld \in \pmb{\Ld}} \pi_N \vec \I_{t_0} \big( \Ld \cdot \pi_N S(t) \Dr \cdot v \big) \bigg\|_{\vec X_{[t_0, t_1]}^{\al - \frac 12 + \eps, \frac 12 + \dl}} \\
&\les (t_1 - t_0)^\ta \sum_{\Ld \in \pmb{\Ld}}  \| \Ld \|_{X_T^{\al - \frac 12, \frac 12 + \dl}} \| \pi_N S(t) \Dr \|_{X_T^{\al - \frac 12, \frac 12 + \dl}} \| v \|_{X_{[t_0, t_1]}^{\al - \frac 12, \frac 12 + \dl}} \\
&\les (t_1 - t_0)^\ta \big( \| \Ups_N \|_{H^\al}^4 + K(Y)^6 \big) \| v \|_{X_{[t_0, t_1]}^{\al - \frac 12 - \eps, \frac 12 + \dl}}. 
\end{split}
\label{ecl6}
\end{align}

Combining \eqref{ecl}, \eqref{ecl3}, \eqref{ecl4},  \eqref{ecl5}, and \eqref{ecl6} and proceeding as in Case 1, we can show that \eqref{exp01} and \eqref{exp02} hold for $\vec \If^{\,\<31>_N}$ with $0 < \be < \frac{1}{6}$.

\smallskip \noi
\textbf{Case 6:} The random operator with the quadratic stochastic object $\vec \If^{\,\<2>}$.

We proceed as in Case 1 using the Bou\'e-Dupuis variational formula along with \eqref{Jop} to obtain a similar bound as \eqref{et2a}:
\begin{align*}
\begin{split}
&\log \int \exp \Big( \big\| \vec \If^{\,\<2>_N} (u_0) \big \|_{\vec \J_T^{\,\al - \frac 12 - \eps, \frac 12 + \dl, \al - \frac 12 + \eps, \frac 12 + \dl } }^\be \Big) d \rho_N(u_0) \\
& \leq 
\sup_{\dot{\Ups}^N \in \mathbb{H}_a^\al}  \E \bigg[  \sup_{0 \leq t_0 < t_1 \leq T} (t_1 - t_0)^{- \ta} \sup_{\| v \|_{X_{[t_0, t_1]}^{\al - \frac 12 - \eps, \frac 12 + \dl} } \leq 1 } \big\| \pi_N \vec \I_{t_0} \big( \<2>_N (Y + \Dr) \cdot v \big) \big\|_{ \vec X_{[t_0, t_1]}^{\al - \frac 12 + \eps, \frac 12 + \dl} }^\be \\
&\quad - \eps_0 \| \Ups_N \|_{H^\al}^2 \bigg] + C_0
\end{split}
\end{align*}

\noi
for some constant $C_0 > 0$, where $\Law (Y) = \mu_\al$. By \eqref{enh0} and \eqref{elin1}, we have
\begin{align}
\<2>_N (Y + \Dr) = \<2>_N (Y) + 2 \<1>_N (Y) \pi_N S(t) \Dr + (\pi_N S(t) \Dr)^2.
\label{equad}
\end{align}

By a slight variant of Lemma \ref{LEM:sto_reg} for $\vec \If^{\,\<2>_N}$, we have
\begin{align}
\big\| \pi_N \vec \I_{t_0} (\<2>_N (Y) \cdot v) \big\|_{\vec X_{[t_0, t_1]}^{\al - \frac 12 + \eps, \frac 12 + \dl}} \les (t_1 - t_0)^\ta K(Y) \| v \|_{X_{[t_0, t_1]}^{\al - \frac 12 - \eps, \frac 12 + \dl}}.
\label{equad1}
\end{align}

\noi
By Lemma \ref{LEM:nhomo}, \eqref{time2} in Lemma \ref{LEM:time}, \eqref{tri3} in Lemma \ref{LEM:str4}, \eqref{elin2}, and \eqref{helper}, we have
\begin{align}
\begin{split}
\big\| &\pi_N \vec \I_{t_0} (\<1>_N (Y) \pi_N S(t) \Dr \cdot v) \big\|_{\vec X_{[t_0, t_1]}^{\al - \frac 12 + \eps, \frac 12 + \dl}} \\
&\les (t_1 - t_0)^\ta \| \<1>_N (Y) \|_{L_T^\infty W_x^{\al - \frac 32 - \eps, \infty}} \| \pi_N S(t) \Dr \|_{X_T^{\al - \frac 12, \frac 12 + \dl}} \| v \|_{X_{[t_0, t_1]}^{\al - \frac 12 - \eps, \frac 12 + \dl}} \\
&\les (t_1 - t_0)^\ta \big( K(Y)^{\frac 12} \| \Ups_N \|_{H^\al} + K(Y)^2 \big) \| v \|_{X_{[t_0, t_1]}^{\al - \frac 12 - \eps, \frac 12 + \dl}}.
\end{split}
\label{equad2}
\end{align}

\noi
By Lemma \ref{LEM:nhomo}, \eqref{time2} in Lemma \ref{LEM:time}, Lemma \ref{LEM:str_var} with $v \equiv 1$, and \eqref{helper}, we have 
\begin{align}
\begin{split}
\big\| &\pi_N \vec \I_{t_0} \big( ( \pi_N S(t) \Dr )^2 \cdot v \big) \big\|_{\vec X_{[t_0, t_1]}^{\al - \frac 12 + \eps, \frac 12 + \dl}} \\
&\les (t_1 - t_0)^\ta \| \pi_N S(t) \Dr \|_{X_T^{\al - \frac 12, \frac 12 + \dl}}^2 \| v \|_{X_{[t_0, t_1]}^{\al - \frac 12 - \eps, \frac 12 + \dl}} \\
&\les (t_1 - t_0)^\ta \big( \| \Ups_N \|_{H^\al}^2 + K(Y)^3 \big) \| v \|_{X_{[t_0, t_1]}^{\al - \frac 12 - \eps, \frac 12 + \dl}}.
\end{split}
\label{equad3}
\end{align}

Combining \eqref{equad}, \eqref{equad1}, \eqref{equad2}, and \eqref{equad3} and proceeding as in Case 1, we can show that \eqref{exp01} and \eqref{exp02} hold for $\vec \If^{\,\<2>_N}$ with $0 < \be < \frac{1}{3}$.
\end{proof}

\subsection{Stability estimate}
\label{SUBSEC:stb}

In this subsection, we establish a stability result. Inspired by \cite[Proposition~6.8]{OOT2}, we would like to introduce an exponential decaying time weight and modify the local well-posedness argument in the proof of Proposition~\ref{THM:1}. However, our local well-posedness is based on the Fourier restriction norm method, which uses Fourier transform in time and so the time weight is also involved in the Fourier transform. Though the situation becomes more complicated (than that in \cite{OOT2}), we observe that the desired decay effect from the time weight can be obtained from slightly losing some time regularity. This is our new ingredient for this part.

Let us define our space with the time weight. Given an interval $I \subset \R$, $s, b \in \R$, and $\ld \geq 1$, we define the space $X_{I}^{s, b, \ld}$ by the norm
\begin{align}
\| u \|_{X_I^{s, b, \ld}} = \| e^{-\ld |t|} u \|_{X_I^{s, b}}.
\label{exp_norm}
\end{align}

\noi
If $I = [0, T]$ for some $T > 0$, we write $X_T^{s, b, \ld} = X_{[0, T]}^{s, b}$. We also write $\vec X_T^{s, b, \ld} = X_T^{s, b, \ld} \times X_T^{s - \al, b, \ld}$. Note that $(e^{-|t|})^\wedge (\tau) \sim \frac{1}{1 + \tau^2} \les \jb{\tau}^{-2}$. Thus, by \eqref{time1} in Lemma \ref{LEM:time}, for any $s \in \R$ and $- \frac 12 < b < \frac 12$, we have
\begin{align}
\| u \|_{X_I^{s, b, \ld}} \les  \| u \|_{X_I^{s, b}}.
\label{exp_key}
\end{align}

\noi
Also, for any $s, b \in \R$, $T > 0$, and $\ld \geq 1$, by the triangle inequality and Young's convolution inequality, we have
\begin{align}
\begin{split}
\| u \|_{X^{s, b}_T} &\leq \| \eta_{\ld, T} (t) e^{- \ld |t|} u \|_{X^{s, b}} \\
&= \bigg\| \jb{n}^2 \jb{|\tau| - \jbb{n}}^b \int_\R \ft{\eta_{\ld, T}} (\tau_1)  \ft{e^{-\ld |\cdot|} u} (n, \tau - \tau_1) d\tau_1 \bigg\|_{\l_n^2 L_\tau^2} \\
&\les \Big\| \jb{\tau}^{|b|} \ft{ \eta_{\ld, T} } (\tau) \Big\|_{L_\tau^1} \| u \|_{X_T^{s, b, \ld}},
\end{split}
\label{exp_crude}
\end{align}

\noi
where $\eta_{\ld, T}$ is a Schwartz function that satisfies $\eta_{\ld, T} = e^{\ld t}$ when $t \in [0, T]$.

We now exploit the decay effect of the exponential decaying time weight.
\begin{lemma}
\label{LEM:nhomo2}
Let $\al \geq 0$, $T > 0$, $\ld > 0$, and $s \in \R$. Let $\dl > 0$ be sufficiently small. Then, we have
\begin{align*}
\| \vec \I (F) \|_{\vec X_T^{s, \frac 12 + \dl, \ld}} \les \ld^{- \dl} \| F \|_{X_T^{s - \al, -\frac 12 + 4 \dl, \ld}},
\end{align*}

\noi
where the operator $\vec \I$ is as defined in \eqref{vecSI} and \eqref{lin1}.
\end{lemma}

\begin{proof}
By the definition \eqref{exp_norm} and a similar reduction as in Lemma \ref{LEM:nhomo}, it suffices to show
\begin{align*}
\big\| \eta (t) \I_\pm^\ld (F) \big\|_{X^{0, \frac 12 + \dl}} \les \ld^{-\dl} \| F \|_{X^{0, - \frac 12 + 4 \dl}}
\end{align*}

\noi
for some $\dl > 0$, where $\eta: \R \to [0,1]$ is a smooth cut-off function such that $\eta \equiv 1$ on $[0, T]$ and
\begin{align*}
\I_\pm^\ld (F) (t) = \sum_{n \in \Z^3} e^{in \cdot x} \int_\R \frac{e^{it \mu} - e^{- (\frac 12 + \ld) |t| \pm i t \jbb{n}}}{\frac 12 + \ld + i \mu \mp i \jbb{n}} \ft F (n, \mu) d\mu.
\end{align*}

By using similar steps that lead to \eqref{goal}, it suffices to show
\begin{align}
\bigg\| \int_\R \mathcal{K}^\ld (n, \tau, \mu) G(\mu) d\mu \bigg\|_{L_\tau^2} \les \| G (\mu) \|_{L_\mu^2}
\label{goal2}
\end{align}

\noi
uniformly in $n \in \Z^3$, where
\begin{align*}
\mathcal{K}^\ld (n, \tau, \mu) = \jb{|\tau| - \jbb{n}}^{\frac 12 + \dl} \jb{|\mu| - \jbb{n}}^{\frac 12 - 4 \dl} \int_\R \eta(t) \frac{e^{it (\mu - \tau)} - e^{- (\frac 12 + \ld) |t| - i \tau t \pm i t \jbb{n}}}{\frac 12 + \ld + i \mu \mp i \jbb{n}} dt.
\end{align*}

When $|\tau - \mu| \leq 1$, we have
\begin{align*}
|\mathcal{K}^\ld (n, \tau, \mu)| \ind_{\{|\tau - \mu| \leq 1\}} \les \frac{\jb{|\tau| - \jbb{n}}^{\frac 12 + \dl} \jb{|\mu| - \jbb{n}}^{\frac 12 - 4 \dl}}{\ld^\dl \jb{\mu + \jbb{n}}^{1 - \dl}} \ind_{\{|\tau - \mu| \leq 1\}} \les \ld^{- \dl}.
\end{align*}

\noi
Thus, by Schur's test, we have
\begin{align*}
\big\| \mathcal{K}^\ld (n, \tau, \mu) \ind_{\{|\tau - \mu| \leq 1\}} \big\|_{L_\mu^2 \to L_\tau^2}^2 &\leq \sup_{\tau \in \R} \int_{|\tau - \mu| \leq 1} |\mathcal{K}^\ld (n, \tau, \mu)| d\mu \cdot \sup_{\mu \in \R} \int_{|\tau - \mu| \leq 1} |\mathcal{K} (n, \tau, \mu)| d\tau \\
&\les \ld^{-2\dl},
\end{align*}

\noi
so that \eqref{goal2} holds when $|\tau - \mu| \leq 1$.

When $|\tau - \mu| > 1$, by using integration by parts, we obtain
\begin{align*}
\mathcal{K}^\ld (n, \tau, \mu) \ind_{\{|\tau - \mu| > 1\}} &= \ind_{\{|\tau - \mu| > 1\}} \jb{|\tau| - \jbb{n}}^{\frac 12 + \dl} \jb{|\mu| - \jbb{n}}^{\frac 12 - 4 \dl} \\
&\quad \times \bigg( \int_\R \eta'' (t) \frac{e^{i t (\mu - \tau)}}{-(\mu - \tau)^2 (\frac 12 + \ld  + i \mu \mp i \jbb{n})} dt \\
&\qquad -\frac{\eta(0) (1 + 2\ld)}{\big( (\frac 12 + \ld)^2 + (\tau \mp \jbb{n})^2 \big) (\frac 12 + \ld + i \mu \mp i \jbb{n})}  \\
&\qquad - \int_0^\infty \eta' (t) \frac{e^{-( \frac{1}{2} + \ld ) t - i\tau t \pm i t \jbb{n}}}{(\frac 12 + \ld + i \tau \mp i \jbb{n}) (\frac 12 + \ld + i \mu \mp i \jbb{n})} dt \\
&\qquad + \int_{-\infty}^0 \eta' (t) \frac{e^{(\frac{1}{2} + \ld) t - i \tau t \pm i t \jbb{n}}}{(\frac 12 + \ld - i \tau \pm i \jbb{n}) (\frac 12 + \ld + i \mu \mp i \jbb{n})} dt \bigg) .
\end{align*}

\noi
Using integration by parts again and triangle inequalities, we can deduce that
\begin{align*}
|\mathcal{K}^\ld (&n, \tau, \mu)| \ind_{\{|\tau - \mu| > 1\}}  \\
&\les \frac{\jb{|\tau| - \jbb{n}}^{\frac 12 + \dl} \jb{|\mu| - \jbb{n}}^{\frac 12 - 4 \dl}}{\jb{\mu - \tau}^2 \big( \ld + |\mu \mp \jbb{n}| \big)} + \frac{\jb{|\tau| - \jbb{n}}^{\frac 12 + \dl} \jb{|\mu| - \jbb{n}}^{\frac 12 - 4 \dl} \ld}{\big( \ld^2 + (\tau \mp \jbb{n})^2 \big) \big( \ld + |\mu \mp \jbb{n}| \big)} \\
&\les \ld^{-\dl} \frac{\jb{|\tau| - \jbb{n}}^{\frac 12 + \dl}}{\jb{\mu - \tau}^2 \jb{\mu \mp \jbb{n}}^{\frac 12 + 3\dl}} + \ld^{-\dl} \frac{1}{\jb{\tau \mp \jbb{n}}^{\frac 12 + \dl} \jb{\mu \mp \jbb{n}}^{\frac 12 + \dl}}.
\end{align*}

\noi
The rest of the steps follows from almost the same way as in Lemma \ref{LEM:nhomo}, and so we omit details.
\end{proof}

We are now ready to state and prove our stability estimate. Let us consider the truncated fractional $\Phi^4_3$-model \eqref{SNLW3a} for $u_N$. We use a second order expansion as in \eqref{exp3} to write
\begin{align}
u_N = \<1> (\vec u_0, \o_2) - \<30>_N (\vec u_0, \o_2) + v_N ,
\label{second}
\end{align}

\noi
where the remainder term $v_N$ satisfies
\begin{align}
\begin{split}
\vec v_N  = \G_N [\vec v_N]   &\overset{\textup{def}}{=} -\pi_N \vec \I \big( (\pi_N v_N)^3 \big) + 3\pi_N \vec \I \big( (\<30>_N-\<1>_N)(\pi_N v_N)^2 \big) \\
		&\quad - 3 \pi_N \vec \I \big( \<30>_N^2 \pi_N v_N \big) + 6 \vec \If^{\,\<31>_N} (\pi_N v_N) -3 \vec \If^{\,\<2>_N} (\pi_N v_N) \\
		&\quad 
		  + \pi_N \vec \I \big( \<30>_N^3 \big)-3 \vec{\<70>}_N + 3 \vec{\<320>}_N 
\end{split}
\label{SdNLW_vN}
\end{align}

\noi
Here, all the stochastic objects are defined as in \eqref{enh0} and are included in the enhanced data set $\pmb{\Xi}_N = \pmb{\Xi}_N (\vec u_0, \o_2)$ as in \eqref{data3x}.

\begin{proposition}[Stability estimate]
\label{PROP:stb}
Let $T > 0$, $1 < \al \leq \frac 32$, $K \gg 1$, and $C_0 \gg 1$. Let $\eps, \dl > 0$ be sufficiently small. Then, there exist $N_0 = N_0 (T, K, C_0) \in \N$ and small $\kappa_0 = \kappa_0 (T, K, C_0) > 0$ such that the following statement holds. Let $N_1 \geq N_2 \geq N_0$. Suppose that
\begin{align}
\| \pmb{\Xi}_{N_1} \|_{\mathcal{Y}_T^{\al, \eps, \dl}} \leq K
\label{stb1}
\end{align}

\noi
and
\begin{align}
\| \vec v_{N_1} \|_{\vec X_T^{\al - \frac 12, \frac 12 + \dl}} \leq C_0,
\label{stb2}
\end{align}

\noi
where $\vec v_{N_1}$ is the solution to the truncated equation \eqref{SdNLW_vN} with $N = N_1$ on $[0,T]$ with the truncated enhanced data set $\pmb{\Xi}_{N_1}$ and the $\mathcal{Y}_T^{\al, \eps, \dl}$-norm is as defined in \eqref{Xae}. Furthermore, suppose that
\begin{align}
\| \pmb{\Xi}_{N_1} - \pmb{\Xi}_{N_2} \|_{\mathcal{Y}_T^{\al, \eps, \dl}} \leq \kappa
\label{stb_diff}
\end{align}

\noi
for some $0 < \kappa \leq \kappa_0$. Then, there exists a solution $\vec v_{N_2}$ to the truncated equation \eqref{SdNLW_vN} with $N = N_2$ on $[0, T]$ with the enhanced data set $\pmb{\Xi}_{N_2}$ satisfying
\begin{align}
\| \vec v_{N_1} - \vec v_{N_2} \|_{\vec X_T^{\al - \frac 12, \frac 12 + \dl}} \leq C(T, K, C_0) (\kappa + N_2^{- \ta})
\label{stb_vdiff}
\end{align}

\noi
for some constant $C(T, K, C_0) > 0$ and some $\ta > 0$.
\end{proposition}

\begin{proof}
We can assume without loss of generality that $\kappa_0 \leq 1$, so that $\kappa \leq \kappa_0 \leq 1$. By \eqref{stb1} and \eqref{stb_diff}, we have
\begin{align}
\| \pmb{\Xi}_{N_2} \|_{\mathcal{Y}_T^{\al, \eps, \dl}} \leq K + 1.
\label{stbp1}
\end{align}

By setting 
\begin{align*}
\vec \dl v_{N_1, N_2} = \vec v_{N_2} - \vec v_{N_1},
\end{align*}

\noi
we have
\begin{align*}
\vec v_{N_2} = \vec \dl v_{N_1, N_2} + \vec v_{N_1}.
\end{align*}

\noi
Thus, in view of \eqref{SdNLW_vN}, $\vec \dl v_{N_1, N_2}$ is the unknown for the equation
\begin{align}
\vec \dl v_{N_1, N_2} = \G' [\vec \dl v_{N_1, N_2}] \overset{\textup{def}}{=} \G_{N_2} [\vec \dl v_{N_1, N_2} + \vec v_{N_1}] - \G_{N_1} [\vec v_{N_1}].
\label{stbp2}
\end{align}

\noi
Our goal is to show that $\G'$ is a contraction on a small ball in $\vec X_T^{\al - \frac 12, \frac 12 + \dl, \ld}$ for some $\ld \geq 1$ to be chosen at a later point. To achieve this, we first establish a bound on $\G'$ for $\vec \dl v_{N_1, N_2} \in B_1$, where $B_1 \subset \vec X_T^{\al - \frac 12, \frac 12 + \dl}$ is the closed ball of radius 1 with respect to the $\vec X_T^{\al - \frac 12, \frac 12 + \dl}$-norm centered at the origin.

In view of \eqref{stbp2}, we can write
\begin{align}
e^{-\ld t} \G' [\vec \dl v_{N_1, N_2}] (t) = e^{-\ld t} \textup{I}_1 (t) + e^{-\ld t} \textup{I}_2 (t) + e^{-\ld t} \textup{I}_3 (t),
\label{stb_I}
\end{align}

\noi
where each term in $\textup{I}_1$ contains the difference of one of the elements in the enhanced data sets $\pmb{\Xi}_{N_1}$ and $\pmb{\Xi}_{N_2}$, each term in $\textup{I}_2$ contains the difference $\pi_{N_1} - \pi_{N_2}$ of frequency projections, and each term in $\textup{I}_3$ contains $\dl v_{N_1, N_2}$.

We now modify the local well-posedness argument (Proposition \ref{THM:1}) to deal with $\textup{I}_1$, $\textup{I}_2$, and $\textup{I}_3$. We replace the use of Lemma~\ref{LEM:time} and Lemma~\ref{LEM:nhomo} in the local well-posedness argument by Lemma~\ref{LEM:nhomo2}. For $\textup{I}_1$, by Lemma~\ref{LEM:nhomo2}, \eqref{exp_key}, \eqref{stb_diff},  \eqref{stb1}, \eqref{stbp1}, \eqref{stb2}, and the fact that $\vec \dl v_{N_1, N_2} \in B_1$, we obtain
\begin{align}
\| e^{- \ld t} \textup{I}_1 \|_{\vec X_T^{\al - \frac 12, \frac 12 + \dl}} \leq C(T) (C_0^3 + K^3) \kappa.
\label{stb_I1}
\end{align}

\noi
For $\textup{I}_2$, we note that in the proof of local well-posedness (Proposition~\ref{THM:1}), for each term there is room for losing some small amount of spatial regularity at the beginning or the end of each estimate (see also Lemma~\ref{LEM:str2} and Lemma~\ref{LEM:str4}). Thus, the contribution from $\pi_{N_1} - \pi_{N_2}$ along with this slight loss of regularity allows us to gain a negative power of $N_2$ for each term. As a result, by Lemma~\ref{LEM:nhomo2}, \eqref{exp_key}, \eqref{stb1}, \eqref{stbp1}, \eqref{stb2}, and the fact that $\vec \dl v_{N_1, N_2} \in B_1$, we obtain
\begin{align}
\| e^{- \ld t} \textup{I}_2 \|_{\vec X_T^{\al - \frac 12, \frac 12 + \dl}} \leq C(T) (C_0^3 + K^3) N_2^{-\ta}
\label{stb_I2}
\end{align}

\noi
for some $\ta > 0$. For $\textup{I}_3$, we put the weight $e^{-\ld t}$ on the $\dl v_{N_1, N_2}$ term after the use of Lemma \ref{LEM:nhomo2}. Thus, by \eqref{stb1}, \eqref{stbp1}, \eqref{stb2}, and the fact that $\vec \dl v_{N_1, N_2} \in B_1$, we obtain
\begin{align}
\| e^{- \ld t} \textup{I}_3 \|_{\vec X_T^{\al - \frac 12, \frac 12 + \dl}} \leq C(T) (C_0^2 + K^2) \ld^{- \dl} \| \dl v_{N_1, N_2} \|_{X_T^{\al - \frac 12, \frac 12 + \dl, \ld}}.
\label{stb_I3}
\end{align}

Combining \eqref{stb_I}, \eqref{stb_I1}, \eqref{stb_I2}, and \eqref{stb_I3}, we obtain
\begin{align}
\begin{split}
\big\| &\G' [\vec \dl v_{N_1, N_2}] \big\|_{\vec X_T^{\al - \frac 12, \frac 12 + \dl, \ld}} \\
&\leq C(T, K, C_0) (\kappa + N_2^{-\ta}) + C(T, K, C_0) \ld^{-\dl} \| \vec \dl v_{N_1, N_2} \|_{\vec X_T^{\al - \frac 12, \frac 12 + \dl, \ld}}
\end{split}
\label{stb_contr1}
\end{align}

\noi
for any $\vec \dl v_{N_1, N_2} \in B_1$. Using a similar computation, we also obtain the difference estimate
\begin{align}
\begin{split}
\big\| &\G' [ (\vec \dl v_{N_1, N_2})^{(1)}] - \G' [ (\vec \dl v_{N_1, N_2})^{(2)}] \big\|_{\vec X_T^{\al - \frac 12, \frac 12 + \dl, \ld}} \\
&\leq C(T, K, C_0) \ld^{-\dl} \big\| (\vec \dl v_{N_1, N_2})^{(1)} - (\vec \dl v_{N_1, N_2})^{(2)} \big\|_{\vec X_T^{\al - \frac 12, \frac 12 + \dl, \ld}}
\end{split}
\label{stb_contr2}
\end{align}

\noi
for any $(\vec \dl v_{N_1, N_2})^{(1)}, (\vec \dl v_{N_1, N_2})^{(2)} \in B_1$. Let us choose $\ld = \ld(T, K, C_0) \gg 1$ to be large enough such that 
\begin{align*}
C(T, K, C_0) \ld^{-\dl} \leq \frac 12.
\end{align*}

\noi
We also note that by \eqref{exp_crude}, we can choose $r = r(T, \ld) > 0$ be sufficiently small such that
\begin{align*}
\| \vec \dl v_{N_1, N_2} \|_{\vec X^{\al - \frac 12, \frac 12 + \dl}_T} \leq C(\ld, T) \| \vec \dl v_{N_1, N_2} \|_{\vec X_T^{\al - \frac 12, \frac 12 + \dl, \ld}} \leq C(\ld, T) r \leq 1
\end{align*}

\noi
for any $\vec \dl v_{N_1, N_2} \in B_r^\ld$, where $B_r^\ld \subset \vec X_T^{\al - \frac 12, \frac 12 + \dl, \ld}$ is the closed ball of radius $r$ with respect to the $\vec X_T^{\al - \frac 12, \frac 12 + \dl, \ld}$-norm centered at the origin. This shows that \eqref{stb_contr1} and \eqref{stb_contr2} hold on $B_r^\ld$. Thus, by choosing $\kappa = \kappa(T, K, C_0)$ sufficiently small and $N_0 = N_0(T, K, C_0)$ sufficiently large in such a way that
\begin{align*}
C (T, K, C_0) ( \kappa + N_2^{- \ta} ) \leq \frac{r}{2},
\end{align*}

\noi
we obtain that $\G'$ is a contraction on $B_r^\ld$ for any $N_1 \geq N_2 \geq N_0$, so that there exists a unique $\vec \dl v_{N_1, N_2} \in B_r^\ld$ that satisfies \eqref{stbp2}. Therefore, by setting $\vec v_{N_2} = \vec \dl v_{N_1, N_2} + \vec v_{N_1}$ and noting that
\begin{align*}
\vec v_{N_2} &= \vec \dl v_{N_1, N_2} + \vec v_{N_1} \\
&= \G_{N_2} [\vec \dl v_{N_1, N_2} + \vec v_{N_1}] - \G_{N_1} [\vec v_{N_1}] + \vec v_{N_1} \\
&= \G_{N_2} [\vec \dl v_{N_1, N_2} + \vec v_{N_1}] \\
&= \G_{N_2} [\vec v_{N_2}],
\end{align*}

\noi
we conclude that $\vec v_{N_2}$ satisfies the equation \eqref{SdNLW_vN} with $N = N_2$ on $[0, T]$. 
Furthermore, by \eqref{stb_contr1} with $\ld = \ld(T, K, C_0) \gg 1$ large enough and \eqref{exp_crude}, we have
\begin{align*}
\| \vec v_{N_1} - \vec v_{N_2} \|_{\vec X_T^{\al - \frac 12, \frac 12 + \dl}} &= \| \vec \dl v_{N_1, N_2} \|_{\vec X_T^{\al - \frac 12, \frac 12 + \dl}} \\
&\leq C(\ld, T) \| \vec \dl v_{N_1, N_2} \|_{\vec X_T^{\al - \frac 12, \frac 12 + \dl, \ld}} \\
&\leq C(T, K, C_0) (\kappa + N_2^{-\ta}),
\end{align*}

\noi
which verifies \eqref{stb_vdiff}.
\end{proof}

\subsection{Nonlinear smoothing and uniform bound with large probability}
\label{SUBSEC:ns}
In this subsection, we prove that the solution $v_N$ to the truncated equation \eqref{SdNLW_vN} has a uniform bound with a large probability. To achieve this, we need the following nonlinear smoothing estimate, which is similar to  \cite[Proposition~3.3 and Lemma~4.3]{BDNY}.
\begin{proposition}[Nonlinear smoothing]
\label{PROP:ns}
Let $T \geq 1$, $1 < \al \leq \frac 32$, and $\eps, \dl > 0$ be sufficiently small. Let $u_N$ be the solution to the truncated equation \eqref{SNLW3a}. Then, given any $\eta > 0$, there exists $C_0 = C_0 (T, \eta) \gg 1$ such that
\begin{align}
\vec \rho_N \otimes \PP_2 \Big( \big\| \pi_N \I \big( :\! (\pi_N u_N)^3 \! : \big) \big\|_{L_T^\infty W_x^{3\al - 3 - \eps, \infty} + X_T^{\al - \frac 12, \frac 12 + \dl}} > C_0 \Big) < \eta,
\label{ns_main}
\end{align}

\noi
uniformly in $N \in \N$.
\end{proposition}

\begin{proof}
We first show that for any $\eta' > 0$, there exists $C_1 = C_1 (\eta') \gg 1$ and $0 < T_0 = T_0 (\eta') \ll 1$ such that
\begin{align}
\vec \rho_N \otimes \PP_2 \Big( \big\| \pi_N \I \big( \ind_{[0, T_0]} (t') :\! (\pi_N u_N (t') )^3 \! : \big) \big\|_{L_T^\infty W_x^{3\al - 3 - \eps, \infty} + X_T^{\al - \frac 12, \frac 12 + \dl}} > C_1 \Big) < \eta'.
\label{ns_loc}
\end{align}



\noi
From Proposition~\ref{PROP:exptail2}, there exists $C_1' = C_1' (T, \eta') \geq 1$ such that
\begin{align}
\vec \rho_N &\otimes \PP_2 \Big( \| \pmb{\Xi}_N (\vec u_0, \o_2) \|_{\mathcal{Y}_T^{\al, \eps, \dl}} > C_1' \Big) < e^{- c(T) (C_1')^\be} < \eta'.
\label{ns_data}
\end{align}

\noi
for some constant $c(T) > 0$ and $0 < \be < \frac{2}{21}$. In the following, we work on the event
\begin{align}
\| \pmb{\Xi}_N (\vec u_0, \o_2) \|_{\mathcal{Y}_T^{\al, \eps, \dl}} \leq C_1'.
\label{ns_data2}
\end{align}

By the uniqueness statement of local well-posedness in Proposition \ref{THM:1} and Remark \ref{RMK:LWP}, we know that if we choose $T_0 = T_0 (\eta') > 0$ small enough such that
\begin{align}
T_0^\ta \leq \frac{1}{2C (8C^3 + 1) (C_1')^3} < \frac{c(T)^{\frac{3}{\be}}}{2C (8 C^3 + 1) \log (\frac{1}{\eta'})^{\frac{3}{\be}}}
\label{T0_bdd}
\end{align}

\noi
for some absolute constant $C \geq 1$, we have the a priori bound
\begin{align}
\| v_N \|_{X_{T_0}^{\al - \frac 12, \frac 12 + \dl}} \leq 2 C C_1'.
\label{ns_vN}
\end{align}

From the second order expansion \eqref{second}, we have
\begin{align}
\begin{split}
\pi_N \I \big( \ind_{[0, T_0]} :\! (\pi_N u_N )^3 \! : \big) &= \pi_N \I \big( \ind_{[0, T_0]} \<3>_N \big) + \pi_N \I \big( \ind_{[0, T_0]} (\pi_N v_N)^3  \big) \\
&\quad - 3 \pi_N \I \big( \ind_{[0, T_0]} \<30>_N  (\pi_N v_N)^2 \big) + 3 \pi_N \I \big( \ind_{[0, T_0]} \<1>_N  (\pi_N v_N)^2 \big) \\
&\quad + 3 \pi_N \I \big( \ind_{[0, T_0]} \<30>_N^2 \pi_N v_N \big) - 6 \pi_N \I \big( \ind_{[0, T_0]} \<30>_N \<1>_N \pi_N v_N \big) \\
&\quad + 3 \pi_N \I \big( \ind_{[0, T_0]} \<2>_N \pi_N v_N \big) - \pi_N \I \big(  \ind_{[0, T_0]} \<30>_N^3 \big) \\
&\quad + 3 \pi_N \I \big( \ind_{[0, T_0]} \<30>_N^2 \<1>_N \big) - 3 \pi_N \I \big( \ind_{[0, T_0]} \<30>_N \<2>_N \big).
\end{split}
\label{u_cub}
\end{align}

\noi
For all terms on the right-hand-side of \eqref{u_cub} involving $v_N$, we can put $\ind_{[0, T_0]}$ in front of $v_N$ and repeat the estimates in the local well-posedness argument in Proposition \ref{THM:1}. Also, by using \eqref{lwp_est1}, \eqref{lwp_est2}, \eqref{lwp_est3}, \eqref{lwp_est4}, and \eqref{ns_data2}, we have
\begin{align*}
\big\| \pi_N \I \big(  \ind_{[0, T_0]} \<30>_N^3 \big) \big\|_{X_T^{\al - \frac 12, \frac 12 + \dl}} &\les \big\|  \<30>_N^3 \big\|_{X_T^{- \frac 12, \frac 12 + \dl}} \\
&\les \| \<30>_N \|_{L_T^\infty W_x^{3\al - 3 - \eps, \infty} + X_T^{\al - \frac 12, \frac 12 + \dl}}^3 \\
&\leq (C_1')^3.
\end{align*}

\noi
For all other terms on the right-hand-side of \eqref{u_cub}, the appearance of $\ind_{[0, T_0]}$ does not affect the regularity properties of these stochastic objects (recall that in Lemma \ref{LEM:sto_reg} and Proposition~\ref{PROP:exptail2}, we always use Lemma \ref{LEM:nhomo} to deal with the operator $\I$ when estimating the stochastic objects, so that the cutoff $\ind_{[0, T_0]}$ can be dealt with easily). Thus, the local well-posedness argument in Proposition \ref{THM:1} along with \eqref{ns_data2} and \eqref{ns_vN} yield
\begin{align*}
\big\| \pi_N &\I \big( \ind_{[0, T_0]} :\! (\pi_N u_N )^3 \! : \big) \big\|_{L_T^\infty W_x^{3\al - 3 - \eps, \infty} + X_T^{\al - \frac 12, \frac 12 + \dl}} \\
&\leq C' \Big( \| v_N \|_{X_{T_0}^{\al - \frac 12, \frac 12 + \dl}}^3 + (C_1')^3 \Big) + C' C_1' \\
&\leq C' (C_1')^3
\end{align*}

\noi
for some constant $C' > 0$. Therefore, \eqref{ns_loc} then follows from \eqref{ns_data} by taking $C_1 = C' (C_1')^3$.

We now show \eqref{ns_main} by using the invariance of the truncated Gibbs measure. We let $J \in \N$ to be determined later and let $T_0 = \frac{T}{J}$. We also let $\eta' = \frac{\eta}{J}$ and let $C_1' \geq 1$ be such that $\eqref{ns_data}$ holds. Thus, in order for \eqref{ns_loc} to hold with $C_1 = C' (C_1')^3$, we require \eqref{T0_bdd}, namely,
\begin{align*}
T_0^\ta < \frac{c(T)^{\frac{3}{\be}}}{2C (8 C^3 + 1) \log (\frac{1}{\eta'})^{\frac{3}{\be}}} = \frac{c(T)^{\frac{3}{\be}}}{2C (8 C^3 + 1) \log (\frac{T}{\eta T_0})^{\frac{3}{\be}}}.
\end{align*}

\noi
This is possible by choosing $T_0 = T_0 (T, \eta) > 0$ sufficiently small.
We now choose $J = \frac{T}{T_0}$ (we can further shrink $T_0$ so that $\frac{T}{T_0}$ is an integer). Since $J$ depends only on $T$ and $\eta$, the above choices of $C_1$ and $C_1'$ depend only on $T$ and $\eta$. By using a change of variable in time, the invariance statement in Lemma \ref{LEM:GWP4}, and \eqref{ns_loc}, we have
\begin{align*}
&\vec \rho_N \otimes \PP_2 \Big( \big\| \pi_N \I \big( :\! (\pi_N u_N)^3 \! : \big) \big\|_{L_T^\infty W_x^{3\al - 3 - \eps, \infty} + X_T^{\al - \frac 12, \frac 12 + \dl}} > J  C_1 \Big) \\
&\leq \vec \rho_N \otimes \PP_2 \Big( \sum_{j = 0}^{J - 1} \big\| \pi_N \I_{j T_0} \big( \ind_{[j T_0, (j + 1) T_0]} (t') :\! (\pi_N u_N (t'))^3 \! : \big) \big\|_{L_T^\infty W_x^{3\al - 3 - \eps, \infty} + X_T^{\al - \frac 12, \frac 12 + \dl}} > J C_1 \Big) \\
&\leq \sum_{j = 0}^{J - 1} \vec \rho_N \otimes \PP_2 \Big( \big\| \pi_N \I \big( \ind_{[0, T_0]} (t') :\! (\pi_N u_N (t' + j T_0))^3 \! : \big) \big\|_{L_T^\infty W_x^{3\al - 3 - \eps, \infty} + X_T^{\al - \frac 12, \frac 12 + \dl}} > C_1 \Big) \\
&= J \cdot \rho_N \otimes \PP_2 \Big( \big\| \pi_N \I \big( \ind_{[0, T_0]} (t') :\! (\pi_N u_N (t'))^3 \! : \big) \big\|_{L_T^\infty W_x^{3\al - 3 - \eps, \infty} + X_T^{\al - \frac 12, \frac 12 + \dl}} > C_1 \Big) \\
&< J \cdot \eta' = \eta.
\end{align*}

\noi
We can conclude by choosing $C_0 = J C_1$, which only depends on $T$ and $\eta$.
\end{proof}

We now show a uniform bound for the solution $v_N$ to \eqref{SdNLW_vN} in an $X^{s,b}$ space (solution space) with large probability. To achieve this, we need to ensure that our estimates for nonlinear terms on the right-hand-side of \eqref{SdNLW_vN} are linear in $\| v_N \|_{X^{s, b}}$. Most of these terms are easily tractable by modifying the local well-posedness argument in Proposition~\ref{THM:1}, except for the term $\I (\<1>_N (\pi_N v_N)^2)$ due to the roughness of $\<1>_N$. However, we observe that one can distribute the derivative loss to only one of the two $\pi_N v_N$'s thanks to symmetry. This is our new ingredient here and is achieved via \eqref{tri2} in Lemma~\ref{LEM:str4}.

\begin{proposition}[Uniform bound with large probability]
\label{PROP:bdd}
Let $T \geq 1$, $1 < \al \leq \frac 32$, and $\eps, \dl > 0$ be sufficiently small. Let $\vec v_N$ be the solution to the truncated equation \eqref{SdNLW_vN} with the truncated enhanced data set $\pmb{\Xi}_N (\vec u_0, \o_2)$ in \eqref{data3x}. Then, given any $\eta > 0$, there exists $C_0 = C_0 (T, \eta) \gg 1$ such that
\begin{align}
\vec \rho_N \otimes \PP_2 \Big( \| \vec v_N \|_{\vec X_T^{\al - \frac 12, \frac 12 + \dl}} > C_0 \Big) < \eta,
\label{bdd_main}
\end{align}

\noi
uniformly in $N \in \N$.
\end{proposition}

\begin{proof}
We define 
\begin{align*}
w_N = u_N - \<1> .
\end{align*}

\noi
Note that $w_N$ satisfies the equation
\begin{align*}
(\dt^2 + \dt + (1 - \Dl)^\al) w_N = - \pi_N \big( :\! (\pi_N u_N)^3 \!: \big)
\end{align*}

\noi
with zero initial data, which can be written as
\begin{align*}
w_N (t) = - \pi_N \I \big( :\! (\pi_N u_N)^3 \!: \big) (t)
\end{align*}

\noi
for $t \geq 0$, where $\I$ is the Duhamel operator as defined in \eqref{lin1}. Let $\al > 1$ and $\eps = \eps(\al) > 0$ be sufficiently small. By Proposition \ref{PROP:ns}, there exists $C_1 = C_1 (T, \eta) \gg 1$ such that
\begin{align}
\begin{split}
&\vec \rho_N \otimes \PP_2 \Big( \| w_N \|_{L_T^\infty W_x^{3\al - 3 - \eps, \infty} + X_T^{\al - \frac 12, \frac 12 + \dl}} > \frac{C_1}{2} \Big) \\
&= \vec \rho_N \otimes \PP_2 \Big( \big\| \pi_N \I \big( :\! (\pi_N u_N)^3 \!: \big) \|_{L_T^\infty W_x^{3\al - 3 - \eps, \infty} + X_T^{\al - \frac 12, \frac 12 + \dl}} > \frac{C_1}{2} \Big) < \frac{\eta}{4}
\end{split}
\label{wN_bdd}
\end{align}

\noi
uniformly in $N \in \N$. Note that from \eqref{second}, we have
\begin{align*}
v_N = u_N - \<1> + \<30>_N = w_N + \<30>_N .
\end{align*}

\noi
Thus, by \eqref{wN_bdd} and Proposition \ref{PROP:exptail2}, we have
\begin{align}
\begin{split}
\vec \rho_N &\otimes \PP_2 \Big( \| v_N \|_{L_T^\infty W_x^{3\al - 3 - \eps, \infty} + X_T^{\al - \frac 12, \frac 12 + \dl}} > C_1 \Big) \\
&\leq \vec \rho_N \otimes \PP_2 \Big( \| w_N \|_{L_T^\infty W_x^{3\al - 3 - \eps, \infty} + X_T^{\al - \frac 12, \frac 12 + \dl}} > \frac{C_1}{2} \Big) \\
&\quad + \vec \rho_N \otimes \PP_2 \Big( \| \<30>_N \|_{L_T^\infty W_x^{3\al - 3 - \eps, \infty} + X_T^{\al - \frac 12, \frac 12 + \dl}} > \frac{C_1}{2} \Big) < \frac{\eta}{2}.
\end{split}
\label{vN_bdd}
\end{align}

\noi
Furthermore, from Proposition \ref{PROP:exptail2}, there exists $C_2 = C_2 (T,  \eta) \geq 1$ such that
\begin{align}
\vec \rho_N &\otimes \PP_2 \Big( \| \pmb{\Xi}_N (\vec u_0, \o_2) \|_{\mathcal{Y}_T^{\al, \eps, \dl}} > C_2 \Big) < \frac{\eta}{2}, 
\label{data_bdd}
\end{align}

\noi
where the $\mathcal{Y}_T^{\al, \eps, \dl}$-norm is defined in \eqref{Xae}. In the following, we work on the event
\begin{align}
\| v_N \|_{L_T^\infty W_x^{3\al - 3 - \eps, \infty} + X_T^{\al - \frac 12, \frac 12 + \dl}} \leq C_1 \quad \text{and} \quad \| \pmb{\Xi}_N (\vec u_0, \o_2) \|_{\mathcal{Y}_T^{\al, \eps, \dl}} \leq C_2.
\label{bdd_set}
\end{align}

We now consider the truncated equation \eqref{SdNLW_vN} for $v_N$.
Let $I = [t_0, t_0 + T_0] \subset [0,T]$ be a closed interval with $0 < T_0 \leq 1$. Let us write the equation for $\vec v_N$ starting at time $t_0$: 
\begin{align}
\begin{split}
\vec v_N &= \vec S(\cdot - t_0) ( \vec v_N (t_0) ) - \pi_N \vec \I_{t_0} \big( (\pi_N v_N)^3 \big) + 3 \pi_N \vec \I_{t_0} \big( \<30>_N (\pi_N v_N)^2 \big) \\
&\quad - 3 \pi_N \vec \I_{t_0} \big( \<1>_N (\pi_N v_N)^2 \big) - 3 \pi_N \vec \I_{t_0} \big( \<30>_N^2 \pi_N v_N \big) + 6 \vec \If_{t_0}^{\,\<31>_N} (\pi_N v_N) \\
&\quad - 3 \vec \If_{t_0}^{\,\<2>_N} (\pi_N v_N) + \pi_N \vec \I_{t_0} \big( \<30>_N^3 \big) - 3 \vec{\<70>}_{t_0, N} + 3 \vec{\<320>}_{t_0, N}.
\end{split}
\label{Duh_vN}
\end{align}

\noi
Using the estimates in \eqref{lwp_est1}, \eqref{lwp_est2}, \eqref{lwp_est3}, \eqref{lwp_est4}, and \eqref{bdd_set}, obtain
\begin{align}
\begin{split}
\big\| &\pi_N \vec \I_{t_0} \big( (\pi_N v_N)^3 \big) \big\|_{\vec X_I^{\al - \frac 12, \frac 12 + \dl}} + \big\| \pi_N \vec \I_{t_0} \big( \<30>_N (\pi_N v_N)^2 \big) \big\|_{\vec X_I^{\al - \frac 12, \frac 12 + \dl}} \\
&\quad + \big\| \pi_N \vec \I_{t_0} \big( \<30>_N \pi_N v_N \big) \big\|_{\vec X_I^{\al - \frac 12, \frac 12 + \dl}} + \big\| \pi_N \vec \I_{t_0} \big( \<30>_N^3 \big) \big\|_{\vec X_I^{\al - \frac 12, \frac 12 + \dl}} \\
&\leq C T_0^{\ta} \Big( \| v_N \|_{L_T^\infty W_x^{3\al - 3 - \eps, \infty} + X_T^{\al - \frac 12, \frac 12 + \dl}}^2 \| v_N \|_{X_I^{\al - \frac 12, \frac 12 + \dl}} \\
&\quad + \| \<30>_N \|_{L_T^\infty W_x^{3\al - 3 - \eps, \infty} + X_T^{\al - \frac 12, \frac 12 + \dl}} \| v_N \|_{L_T^\infty W_x^{3\al - 3 - \eps, \infty} + X_T^{\al - \frac 12, \frac 12 + \dl}} \| v_N \|_{X_I^{\al - \frac 12, \frac 12 + \dl}} \\
&\quad + \| \<30>_N \|_{L_T^\infty W_x^{3\al - 3 - \eps, \infty} + X_T^{\al - \frac 12, \frac 12 + \dl}}^2 \| v_N \|_{X_I^{\al - \frac 12, \frac 12 + \dl}} \\
&\quad + \| \<30>_N \|_{L_T^\infty W_x^{3\al - 3 - \eps, \infty} + X_T^{\al - \frac 12, \frac 12 + \dl}}^3 \Big) \\
&\leq C C_1^2 C_2^2 T_0^\ta \| v_N \|_{X_I^{\al - \frac 12, \frac 12 + \dl}} + C C_2^3.
\end{split}
\label{bdd1}
\end{align}

\noi
for some $\ta > 0$ and some constant $C > 0$ that may vary from line to line below. By the inhomogeneous linear estimate in Lemma~\ref{LEM:nhomo}, the time regularization estimate \eqref{time2} in Lemma~\ref{LEM:time}, \eqref{tri2} in Lemma~\ref{LEM:str4} (with $r \geq \frac{10}{3}$ satisfying $\al - \frac 12 > \frac 32 - \frac{3 + \al}{r}$), \eqref{bdd_set}, and the Strichartz estimate \eqref{Lp_str} with $p = r$, we have
\begin{align}
\begin{split}
\big\| &\pi_N \vec \I_{t_0} \big( \<1>_N (\pi_N v_N)^2 \big) \big\|_{\vec X_I^{\al - \frac 12, \frac 12 + \dl}} \\
&\leq C T_0^{3 \dl} \big\| \<1>_N (\pi_N v_N)^2 \big\|_{X_I^{- \frac 12, - \frac 12 + 4 \dl}} \\
&\leq C T_0^{3 \dl} \| \<1>_N \|_{L_T^\infty W_x^{\al - \frac 32 - \eps, \infty}} \| v_N \|_{L_T^r L_x^r} \| v_N \|_{X_I^{\al - \frac 12, \frac 12 + \dl}} \\
&\leq C C_2 T_0^{3 \dl} \| v_N \|_{L_T^\infty W_x^{\al - 1 - \eps, \infty} + X_T^{\al - \frac 12, \frac 12 + \dl}} \| v_N \|_{X_I^{\al - \frac 12, \frac 12 + \dl}} \\
&\leq C C_1 C_2 T_0^{3 \dl} \| v_N \|_{X_I^{\al - \frac 12, \frac 12 + \dl}},
\end{split}
\label{bdd4}
\end{align}

\noi
where we used $3\al - 3 - \eps > 0$ for $\eps > 0$ sufficiently small. Also, by \eqref{bdd_set} along with the definition of the norms in \eqref{Xae} and \eqref{Jop}, we get
\begin{align}
\begin{split}
\big\| \vec \If_{t_0}^{\,\<31>_N} ( \pi_N v_N ) \big\|_{\vec X_I^{\al - \frac 12, \frac 12 + \dl}} &\leq C_2 T_0^\ta \| v_N \|_{X_I^{\al - \frac 12, \frac 12 + \dl}}, \\
\big\| \vec \If_{t_0}^{\,\<2>_N} (\pi_N v_N) \big\|_{\vec X_I^{\al - \frac 12, \frac 12 + \dl}} &\leq C_2 T_0^\ta \| v_N \|_{X_I^{\al - \frac 12, \frac 12 + \dl}}.
\end{split}
\label{bdd5}
\end{align}

\noi
Thus, from \eqref{Duh_vN}, the homogeneous linear estimate in Lemma~\ref{LEM:homo}, \eqref{bdd1}, \eqref{bdd4}, \eqref{bdd5}, and \eqref{bdd_set}, we obtain
\begin{align*}
\| \vec v_N \|_{\vec X_I^{\al - \frac 12, \frac 12 + \dl}} \leq C \| \vec v_N (t_0) \|_{\vec H^{\al - \frac 12}} + C C_2^3 + C C_1^2 C_2^2 T_0^\ta \| \vec v_N \|_{\vec X_I^{\al - \frac 12, \frac 12 + \dl}} .
\end{align*}

\noi
We can then choose $T_0 = T_0 (T, \eta) > 0$ small enough such that $C C_1^2 C_2^2 T_0^\ta \ll 1$, so that
\begin{align}
\| \vec v_N \|_{\vec X_I^{\al - \frac 12, \frac 12 + \dl}} \les \| \vec v_N (t_0) \|_{\vec H^{\al - \frac 12}} +  C_2^3.
\label{bdd6}
\end{align}

\noi
We can further shrink $T_0$ so that $\frac{T}{T_0}$ is an integer. Thus, we can write
\begin{align*}
[0, T] = \bigcup_{j = 0}^{\frac{T}{T_0} - 2} [j T_0, (j + 2) T_0],
\end{align*}

\noi
so that from \eqref{bdd6} and Lemma~\ref{LEM:HsXsb}, we get
\begin{align*}
\| \vec v_N \|_{\vec X_{[0, 2T_0]}^{\al - \frac 12, \frac 12 + \dl}} \les C_2^3
\end{align*}

\noi
and for any $j = 0, 1, \dots, \frac{T}{T_0} - 3$, we get
\begin{align*}
\| \vec v_N \|_{\vec X_{[(j + 1)T_0, (j + 3)T_0]}^{\al - \frac 12, \frac 12 + \dl}} \les \| \vec v_N \|_{\vec X_{[j T_0, (j + 2) T_0]}^{\al - \frac 12, \frac 12 + \dl}} + C_2^3 .
\end{align*}

\noi
Thus, using a simple iteration and the gluing lemma (Lemma~\ref{LEM:glue}), we see that there exists a constant $C_0 = C_0 (T, \eta) > 0$ such that
\begin{align*}
\| \vec v_N \|_{\vec X_T^{\al - \frac 12, \frac 12 + \dl}} \leq C_0,
\end{align*}

\noi
and so the desired tail bound \eqref{bdd_main} follows from \eqref{vN_bdd} and \eqref{data_bdd}.
\end{proof}

\subsection{Proof of Theorem \ref{THM:GWP}}
\label{SUBSEC:gwp}
In this subsection, we prove Theorem~\ref{THM:GWP}. The argument in this subsection is based on \cite{BDNY}.

\begin{proof}[Proof of Theorem~\ref{THM:GWP}]
Let us first consider the almost sure global well-posedness of the fractional hyperbolic $\Phi_3^4$-model. Let $\vec u_N$ be the solution to the truncated equation \eqref{SNLW3a} guaranteed by Lemma~\ref{LEM:GWP4}. Our goal is to show that for any $T \geq 1$, there exists some $\ta > 0$ such that
\begin{align*}
\lim_{N_0 \to \infty} \vec \rho \otimes \mathbb{P}_2 \Big( \| \vec u_{N_1} - \vec u_{N_2} \|_{C_T \vec H_x^{\al - \frac 32 - \eps}} \leq N_2^{- \ta} \text{ for all } N_1 \geq N_2 \geq N_0 \Big) = 1.
\end{align*}

\noi
Note that it suffices to show
\begin{align*}
\lim_{N_0 \to \infty} \lim_{M \to \infty} \vec \rho \otimes \mathbb{P}_2 \Big( \| \vec u_{N_1} - \vec u_{N_2} \|_{C_T \vec H_x^{\al - \frac 32 - \eps}} \leq N_2^{- \ta} \text{ for all } M \geq N_1 \geq N_2 \geq N_0 \Big) = 1.
\end{align*}

\noi
By the weak convergence of the truncated Gibbs measure $\vec \rho_N$ to the Gibbs measure $\vec \rho$ from Theorem~\ref{THM:Gibbs}, it suffices to show
\begin{align}
\begin{split}
\lim_{N_0 \to \infty} \lim_{M \to \infty} \limsup_{N \to \infty} \vec \rho_N \otimes \mathbb{P}_2 \Big( &\| \vec u_{N_1} - \vec u_{N_2} \|_{C_T \vec H_x^{\al - \frac 32 - \eps}} \leq N_2^{- \ta} \\
&\text{ for all } M \geq N_1 \geq N_2 \geq N_0 \Big) = 1.
\end{split}
\label{GWP_goal2}
\end{align}

\noi
By using the second order expansion \eqref{second}, Proposition~\ref{PROP:exptail2} for the $\<30>$ term, and the embedding of the $X^{s, b}$-spaces in Lemma~\ref{LEM:HsXsb}, we only need to show
\begin{align}
\begin{split}
\lim_{N_0 \to \infty} \lim_{M \to \infty} \limsup_{N \to \infty} \vec \rho_N \otimes \mathbb{P}_2 \Big( &\| \vec v_{N_1} - \vec v_{N_2} \|_{\vec X_T^{\al - \frac 12, \frac 12 + \dl}} \leq N_2^{- \ta} \\
&\text{ for all } M \geq N_1 \geq N_2 \geq N_0 \Big) = 1,
\end{split}
\label{GWP_goal}
\end{align}

\noi
where $\vec v_N$ is the solution to the truncated equation \eqref{SdNLW_vN}. Here, we used the uniqueness of the solution $\vec u_N$ to \eqref{SNLW3a}; see \cite[Remark~6.3~(ii)]{OOT2}.

Let $\eta > 0$ be arbitrary. From Proposition~\ref{PROP:exptail2} and Chebyshev's inequality, there exist $K = K(T, \eta) \gg 1$, $N_0 = N_0 (T, \eta) \gg 1$ sufficiently large, and $\ta > 0$ sufficiently small such that
\begin{align}
\vec \rho_N \otimes \mathbb{P}_2 \Big( \| \pmb{\Xi}_N \|_{\mathcal{Y}_T^{\al, \eps, \dl}} > K \Big) < \frac{\eta}{6}
\label{gwp1}
\end{align}

\noi
and
\begin{align}
\vec \rho_N \otimes \mathbb{P}_2 \Big( \| \pmb{\Xi}_{N} - \pmb{\Xi}_{N'} \|_{\mathcal{Y}_T^{\al, \eps, \dl}} > (N')^{- 2 \ta} \text{ for some } N \geq N' \geq N_0 \Big) < \frac{\eta}{6}
\label{gwp2}
\end{align}

\noi 
for any $N \geq N_0$, where the $\mathcal{Y}_T^{\al, \eps, \dl}$-norm is as defined in \eqref{Xae}. By Proposition~\ref{PROP:bdd}, there exists $C_0 = C_0 (T, \eta) \gg 1$ such that
\begin{align}
\vec \rho_N \otimes \mathbb{P}_2 \Big( \| \vec v_N \|_{\vec X_T^{\al - \frac 12, \frac 12 + \dl}} > C_0 \Big) < \frac{\eta}{6}
\label{gwp3}
\end{align}

\noi
for any $N \geq N_0$. Thus, by Proposition~\ref{PROP:stb} along with \eqref{gwp1}, \eqref{gwp2}, and \eqref{gwp3}, we can take $N_0 = N_0 (T, \eta) \gg 1$ to be sufficiently large such that
\begin{align}
&\vec \rho_N \otimes \mathbb{P}_2 \Big( \| \vec v_N - \vec v_{N_1} \|_{\vec X_T^{\al - \frac 12, \frac 12 + \dl}} > \tfrac{1}{2} N_1^{- \ta} \text{ for some } N \geq N_1 \geq N_0 \Big) < \frac{\eta}{2}, \label{vNN} \\
&\vec \rho_N \otimes \mathbb{P}_2 \Big( \| \vec v_N - \vec v_{N_2} \|_{\vec X_T^{\al - \frac 12, \frac 12 + \dl}} > \tfrac{1}{2} N_2^{- \ta} \text{ for some } N \geq N_2 \geq N_0 \Big) < \frac{\eta}{2}. \nonumber
\end{align}

\noi
This shows that for any $N \geq N_0$,
\begin{align*}
\vec \rho_N \otimes \mathbb{P}_2 \Big( \| \vec v_{N_1} - \vec v_{N_2} \|_{\vec X_T^{\al - \frac 12, \frac 12 + \dl}} > N_2^{- \ta} \text{ for some } N \geq N_1 \geq N_2 \geq N_0 \Big) < \eta,
\end{align*}

\noi
which implies the desired \eqref{GWP_goal}.

It remains to prove the invariance part of Theorem~\ref{THM:GWP}. Let $t \in \R_+$ and $F : \mathcal{C}^{-100} (\T^3) \times \mathcal{C}^{-100} (\T^3) \to \R$ be bounded and Lipschitz. From the above steps, we can let $\vec u$ be the $\vec \rho \otimes \mathbb{P}_2$-a.s. limit of the truncated solution $\vec u_N$. Our goal is to show that
\begin{align}
\int F (\vec u (t)) d (\vec \rho \otimes \mathbb{P}_2) = \int F (\vec u_0) d \vec \rho (\vec u_0).
\label{inv_goal}
\end{align}

By the dominated convergence theorem, the weak convergence of the truncated Gibbs measure $\vec \rho_N$ to the Gibbs measure $\vec \rho$ from Theorem~\ref{THM:Gibbs}, and \eqref{GWP_goal2}, we have
\begin{align}
\begin{split}
\int F (\vec u (t)) d (\vec \rho \otimes \mathbb{P}_2) &= \lim_{M \to \infty} \int F (\vec u_M (t)) d (\vec \rho \otimes \mathbb{P}_2) \\
&= \lim_{M \to \infty} \lim_{N \to \infty} \int F (\vec u_M (t)) d (\vec \rho_N \otimes \mathbb{P}_2) \\
&= \lim_{N \to \infty} \int F (\vec u_N (t)) d (\vec \rho_N \otimes \mathbb{P}_2) ,
\end{split}
\label{inv1}
\end{align}

\noi
where the second equality follows from \eqref{gwp2} and \eqref{vNN}.
By Lemma \ref{LEM:GWP4}, we have
\begin{align}
\int F (\vec u_N (t)) d (\vec \rho_N \otimes \mathbb{P}_2) = \int F (\vec u_0) d \vec \rho_N (\vec u_0).
\label{inv2}
\end{align}

\noi
Thus, by \eqref{inv1}, \eqref{inv2}, and the weak convergence of $\vec \rho_N$ to  $\vec \rho$, we obtain the desired \eqref{inv_goal}. Thus, we have finished the proof of Theorem~\ref{THM:GWP}.
\end{proof}

\section{Dynamical weak universality}
\label{SEC:wu}
In this section, we prove our weak universality result in Theorem~\ref{THM:wu}.

As mentioned in Subsection~\ref{SUBSEC:wu}, in the case $\frac 98 < \al < \frac 32$, the cubic stochastic object $\vec{\<30>}$ in Lemma~\ref{LEM:sto_reg}~(iii) has higher regularity than the scaling subcritical regularity of the cubic fractional NLW. Thus, we shall proceed as in \cite{STzX} with the first order expansion $u_N = \<1> + w_N$ with $\<1>$ being defined in \eqref{loli}. By plugging this ansatz into the nonlinear terms in $V_N (\pi_N u_N)$ in \eqref{defVN}, we see many Wick orderings of $\<1>_N = \pi_N \<1>$, whose regularities are to be determined in various cases.

\subsection{Regularities of stochastic objects}
In this subsection, we show preliminary results on regularity properties of general Wick-ordered stochastic objects mentioned above. We perform more delicate analysis (in Lemma~\ref{LEM:sto_k2} and Lemma~\ref{LEM:sto_k3}) than that in \cite{STzX} in order to cover the full range $\al > \frac 98$ in which we have strong convergence of Gibbs measures in Proposition~\ref{PROP:Gibbs2} defined in \eqref{HkX}.

We start with the following regularity estimate for Wick orderings of $\<1>_N$.
\begin{lemma}
\label{LEM:sto_k}
Let $1 < \al < \frac 32$, $k \geq 1$ be an integer, and $\eps > 0$. Let $\eps' = 0$ if $k = 1$ and $0 < \eps' < \eps$ if $k \geq 2$. Then, for any $1 \leq p < \infty$ and $T > 0$, we have the tail estimate
\begin{align*}
\PP \Big( N^{(k - 1) (\al - \frac 32) + \eps' } \| :\! \<1>_N^{k} \!: \|_{L_T^p W_x^{\al - \frac 32 - \eps, \infty}} > \ld \Big) \leq C(T) \exp \big(- c (T) \ld^{\frac{2}{k}} \big)
\end{align*}

\noi
for any $\ld > 0$ and some constants $C(T), c(T) > 0$, uniformly in $N \in \N$.
\end{lemma}
\begin{proof}
The tail estimate follows directly from Lemma~\ref{LEM:YNk} and Chebyshev's inequality, since $\<1>$ and $Y$ in \eqref{defY} have similar structures and have the same spatial regularity for a fixed $t > 0$.
\end{proof}

By exploiting the smoothing from the dispersion via the $X^{s, b}$-space, we gain extra spatial regularity as follows. We recall the operator $\vec \I_{t_0}$ in \eqref{vecSI} and \eqref{lin1}. 
\begin{lemma}
\label{LEM:sto_k2}
Let $1 < \al < \frac 32$, $k \geq 3$ be an integer, and $\dl > 0$ be sufficiently small. Let $\eps = 0$ if $k = 3$ and $\eps > 0$ be sufficiently small if $k \geq 4$. Let $s \in \R$ be such that $\al - 1 \leq s < 3 \al - 3$. Then, for any $T > 0$, we have the tail estimate
\begin{align*}
\PP \Big( N^{(k - 3) (\al - \frac 32) + \eps } \sup_{0 \leq t_0 < t_1 \leq T} \| \vec \I_{t_0} (:\! \<1>_N^{k} \!:) \|_{\vec X_{[t_0, t_1]}^{s, \frac 12 + \dl}} > \ld \Big) \leq C(T) \exp (- c (T) \ld^{\frac{2}{k}} )
\end{align*}

\noi
for any $\ld > 0$ and some constants $C(T), c(T) > 0$, uniformly in $N \in \N$.
\end{lemma}

\begin{proof}
The tail estimate follows from the proof of Lemma \ref{LEM:sto_reg} for $\vec{\<30>}$ with minor modifications. The main difference is that we now need to apply the summation counting estimate Lemma~\ref{LEM:CS} for a general $k \geq 3$.
\end{proof}

Furthermore, we have an operator norm estimate, which is also new compared to \cite{STzX}. For $k \in \N$, we define the operator
\begin{align*}
\vec \If_{t_0, N}^k (u) = \vec \I_{t_0} ( :\! \<1>_N^{k} \!: u ).
\end{align*}

\noi
For simplicity, we write $\vec \If_{N}^k = \vec \If_{\cdot, N}^k$. We recall the $\J_T^{s_1, b_1, s_2, b_2}$-norm defined in \eqref{Jop}.
\begin{lemma}
\label{LEM:sto_k3}
Let $1 < \al < \frac 32$, $k \geq 3$ be an integer, and $\dl > 0$ be sufficiently small. Let $\eps = 0$ if $k = 3$ and $\eps > 0$ be sufficiently small if $k \geq 4$. Let $s \in \R$ be such that $\max ( -\al + \frac 32, \al - 1 ) < s < 2 \al - \frac 32$. Then, for any $T > 0$, we have the tail estimate
\begin{align*}
\PP \Big( N^{(k - 3) (\al - \frac 32) + \eps } \| \vec \If_N^{k - 1} \|_{\vec \J_T^{\,s, \frac 12 + \dl, s, \frac 12 + \dl}} > \ld \Big) \leq C(T) \exp \big( - c (T) \ld^{\frac{2}{k - 1}} \big)
\end{align*}

\noi
for any $\ld > 0$ and some constants $C(T), c(T) > 0$, uniformly in $N \in \N$.
\end{lemma}

\begin{proof}
The tail estimate follows from the proof of Lemma \ref{LEM:sto_reg} for $\vec \If^{\,\<2>}$ with minor modifications. The main difference is that we now need to apply the deterministic tensor estimate Lemma~\ref{LEM:ten1} for a general $k \geq 3$.
\end{proof}

\subsection{Invariance of the Gibbs measures and a priori bounds}
\label{SUBSEC:inv}
From now on, we restrict our attention to the range $\frac 98 < \al < \frac 32$. In this range, we recall from Proposition~\ref{PROP:Gibbs2} that the Gibbs measure $\vec \nu = \nu \otimes \mu_0$ is equivalent to (i.e.~mutually absolutely continuous with respect to) the base Gaussian measure $\vec \mu_\al = \mu_\al \otimes \mu_0$ in \eqref{gauss2}.

As in \cite{STzX}, our proof of dynamical convergence relies on the invariance of both the truncated Gibbs measure and the limiting Gibbs measure under their corresponding flows. We first state the result on almost sure global well-posedness of the equation \eqref{fNLWw} and invariance of the truncated Gibbs measure $\vec \nu_N = \nu_N \otimes \mu_0$ (see \eqref{Gibbs_nuN} and \eqref{gauss2}) under the flow of \eqref{fNLWw}, which is in the same spirit as Lemma~\ref{LEM:GWP4}. For a proof, we refer the readers to  \cite[Proposition~5.1]{ORTz} and \cite[Lemma~9.3]{OOTol1}. 
 
\begin{lemma}
\label{LEM:inv_N}
Let $\frac 98 < \al < \frac 32$ and $\eps > 0$. Fix $N \in \N$. Then, the macroscopic model for the fractional stochastic nonlinear damped wave equation \eqref{fNLWw} is almost surely globally well-posed with random initial data distributed by the truncated Gibbs measure $\vec \nu_N$, and $\vec \nu_N$ is invariant under the resulting dynamics. More precisely, there exists $\Sigma_N \subset \vec H^{\al - \frac 32 - \eps}$ with $\vec \nu_N (\Sigma_N) = 1$ such that for all $\vec u_0 \in \Sigma_N$, the solution $\vec u_N = \vec u_N (\vec u_0)$ to \eqref{fNLWw} exists globally-in-time, and if $\Law(\vec u_0) = \vec \nu_N$ then $\Law (\vec u_N (t)) = \vec \nu_N$ for any $t \in \R_+$.
\end{lemma}

We now use the above invariance result to prove the following uniform bound of $u_N$ with large probability. Below, we denote $u_N (\vec \phi) (t)$ as the flow of the equation \eqref{fNLWw} at time $t \in \R_+$ starting from the initial data $\vec \phi = (\phi, \dt \phi)$. We also recall from the beginning of Section~\ref{SEC:GWP} that the randomness coming from the space-time white noise $\xi$ depends only on $\o_2 \in \O_2$ equipped with probability measure $\PP_2$.

\begin{proposition}[$L^q$-based a priori bound]
\label{PROP:uN_bdd}
Let $T > 0$, $\frac 98 < \al < \frac 32$, $2 \leq p, q < \infty$, and $\eps, \g > 0$. Let $u_N$ be the solution to the macroscopic model \eqref{fNLWw}. Then, given any $\eta > 0$, there exists $N_0 = N_0 (T, \eta, \g) \gg 1$ such that
\begin{align*}
\vec \mu_\al \otimes \PP_2 \Big( \| u_N \|_{L_T^p W_x^{\al - \frac 32 - \eps, q}} > N^{\g} \textup{ for some } N \geq N_0 \Big) < \eta.
\end{align*}
\end{proposition} 

\begin{proof}
By the equivalence of the Gibbs measure $\vec \nu$ and the base Gaussian measure $\vec \mu_\al$ in Proposition~\ref{PROP:Gibbs2}, it suffices to show that there exists $N_0 = N_0 (T, \eta') \gg 1$ such that
\begin{align*}
\vec \nu \otimes \PP_2 \Big( \| u_N \|_{L_T^p W_x^{\al - \frac 32 - \eps, q}} > N^{\g} \textup{ for some } N \geq N_0 \Big) < \eta',
\end{align*}

\noi
for some $\eta' = \eta' (\eta) > 0$.
We note that it is enough to show
\begin{align}
\vec \nu \otimes \PP_2 \Big( \| u_N \|_{L_T^p W_x^{\al - \frac 32 - \eps, q}} > N^{\g} \Big) < \frac{C (T, \g)}{N^{1 +}}
\label{nuN_goal0}
\end{align}

\noi
for any $N \geq N_0$ with $N_0 = N_0 (T, \g) \gg 1$ sufficiently large, where $1 +$ is any number larger than 1 and $C(T, \g) > 0$ is some constant independent of $N$, because by the union bound we have
\begin{align*}
\vec \nu &\otimes \PP_2 \Big( \| u_N \|_{L_T^p W_x^{\al - \frac 32 - \eps, q}} > N^{\g} \textup{ for some } N \geq N_0 \Big) \\
&\leq \sum_{N \geq N_0} \vec \nu \otimes \PP_2 \Big( \| u_N \|_{L_T^p W_x^{\al - \frac 32 - \eps, q}} > N^{\g} \Big) \\
&< \sum_{N \geq N_0} \frac{C (T, \g)}{N^{1 +}} < \eta'.
\end{align*}

In order to use the invariance of the Gibbs measure, we need to change the Gibbs measure $\vec \nu$ into the frequency truncated Gibbs measure $\vec \nu_N$. 
For this purpose, we define $D_N = \{ |\mathcal{R}_N - \mathcal{R}| > 1 \}$, where $\mathcal{R}_N$ is as defined in \eqref{defRN2} and $\mathcal{R}$ is as defined in \eqref{defR2}. 
From \eqref{Gibbs_nu}, the Cauchy-Schwarz inequality, the $L^2 (\mu_\al)$-boundedness of $e^{- \mathcal{R} (u)}$ due to \eqref{Lp_conv2}, Chebyshev's inequality, and Lemma~\ref{LEM:RNconv}, we have
\begin{align}
\begin{split}
\nu (D_N) &= \mathcal{Z}^{-1} \int_{D_N} e^{- \mathcal{R} (u)} d \mu_\al (u) \\
&\leq \mathcal{Z}^{-1} \bigg( \int e^{- 2 \mathcal{R} (u)} d \mu_\al (u) \bigg)^{1/2} \mu_\al (D_N)^{\frac 12} \\
&\les \bigg( \int |\mathcal{R}_N (u) - \mathcal{R} (u)|^{2r_1} d \mu_\al (u) \bigg)^{1/2} \\
&\les N^{- r_1 \ta}
\end{split}
\label{nuDN1}
\end{align}

\noi
for any $r_1 \gg 1$ and some $\ta > 0$. Also, for any event $A_N$, by the fact that $D_N^c = \{ |\mathcal{R}_N - \mathcal{R}| \leq 1 \}$, we have
\begin{align}
\begin{split}
\nu (D_N^c \cap A_N) &= \mathcal{Z}^{-1} \int_{D_N^c \cap A_N} e^{- \mathcal{R} (u)} d \mu_\al (u) \\
&= \mathcal{Z}^{-1} \mathcal{Z}_N \mathcal{Z}_N^{-1} \int_{D_N^c \cap A_N} e^{- \mathcal{R} (u) + \mathcal{R}_N (u)} e^{- \mathcal{R}_N (u)} d \mu_\al (u) \\
&\les \mathcal{Z}_N^{-1} \int_{D_N^c \cap A_N} e^{- \mathcal{R}_N (u)} d \mu_\al (u) \\
&\leq \nu_N (A_N) ,
\end{split}
\label{nuDN2}
\end{align}

\noi
where we used the fact that $\mathcal{Z}^{-1} \mathcal{Z}_N \les 1$ for sufficiently large $N \gg 1$ in view of the convergence of $\mathcal{Z}_N$ to $\mathcal{Z}$ due to \eqref{Lp_conv2} (with $p = 1$). 
Combining \eqref{nuDN1} and \eqref{nuDN2}, we know that for any event $A_N$, we have
\begin{align*}
\nu (A_N) = \nu (D_N \cap A_N) + \nu (D_N^c \cap A_N) \les N^{- r_1 \ta} + \nu_N (A_N)
\end{align*}

\noi
for any $r_1 \gg 1$ and some $\ta > 0$. Thus, by taking $r_1$ sufficiently large such that $r_1 \ta > 1$, we know that in order to show \eqref{nuN_goal0}, it suffices to show \eqref{nuN_goal0} with $\nu$ replaced by $\nu_N$.

Let $r_2 > \max (p, q, \frac{1}{\g})$. By Chebyshev's inequality and Minkowski's integral inequality, we obtain
\begin{align}
\begin{split}
\vec \nu_N &\otimes \PP_2 \Big( \| u_N \|_{L_T^p W_x^{\al - \frac 32 - \eps, q}} > N^\g \Big) \\
&\leq \frac{1}{N^{r_2 \g}} \int \| u_N (\vec \phi) (t) \|_{L_T^p W_x^{\al - \frac 32 - \eps, q}}^{r_2} d (\vec \nu_N \otimes \PP_2) (\vec \phi, \o_2) \\
&\leq \frac{1}{N^{r_2 \g}} \bigg\| \bigg( \int \big| \jb{\nabla}^{\al - \frac 32 - \eps} u_N (\vec \phi) (t) \big|^{r_2} d (\vec \nu_N \otimes \PP_2) (\vec \phi, \o_2) \bigg)^{1/r_2} \bigg\|_{L_T^p L_x^{q}}^{r_2}.
\end{split}
\label{uN_bdd1}
\end{align}

\noi
We recall the definition of $\mathcal{Z}_N$ from \eqref{ZN} and note that by Jensen's inequality,
\begin{align}
-\log \mathcal{Z}_N = -\log \int e^{- \mathcal{R}_N (u)} d \mu (u) \leq \int \mathcal{R}_N (u) d \mu(u) = 0.
\label{ZN_bdd}
\end{align}

\noi
Then, by the invariance part in Lemma~\ref{LEM:inv_N}, the Cauchy-Schwarz inequality, \eqref{ZN_bdd}, the uniform $L^p$-bound \eqref{Lp_bdd2}, and the Gaussian hypercontractivity, we have
\begin{align}
\begin{split}
\int \big| &\jb{\nabla}^{\al - \frac 32 - \eps} u_N (\vec \phi) (t) \big|^{r_2} d (\vec \nu_N \otimes \PP_2) (\vec \phi, \o_2) \\
&= \int \big| \jb{\nabla}^{\al - \frac 32 - \eps} \phi \big|^{r_2} d \nu_N (\phi) \\
&\leq \frac{1}{\mathcal{Z}_N} \bigg( \int e^{- 2 \mathcal{R}_N (\phi)} d \mu (\phi) \bigg)^{1/2} \bigg( \int \big| \jb{\nabla}^{\al - \frac 32 - \eps} \phi \big|^{2 r_2} d \mu (\phi) \bigg)^{1/2} \\
&\les r_2^{\frac{r_2}{2}} \bigg( \int \big| \jb{\nabla}^{\al - \frac 32 - \eps} \phi \big|^{2} d \mu (\phi) \bigg)^{r_2 / 2} \les r_2^{\frac{r_2}{2}}.
\end{split}
\label{uN_bdd2}
\end{align}

\noi
Combining \eqref{uN_bdd1} and \eqref{uN_bdd2}, we obtain
\begin{align*}
\vec \nu_N \otimes \PP_2 \Big( \| u_N \|_{L_T^p W_x^{\al - \frac 32 - \eps, q}} > N^{\g} \Big) \les \frac{r_2^{\frac{r_2}{2}} T^{\frac{r_2}{p}}}{N^{r_2 \g}},
\end{align*}

\noi
and so we obtain the desired estimate given $r_2 \g > 1$.
\end{proof}

We now exploit invariance of Gibbs measures from another perspective. Let us consider the solution $u_N^\dagger$ to the frequency truncated cubic SdfNLW \eqref{fNLWw_v}. Similar to $\rho_N$ in \eqref{GibbsN1}, we define the frequency truncated Gibbs measure $\nu_N^\dagger$ as
\begin{align*}
d \nu_N^\dagger (u) \deff (\mathcal{Z}_N^\dagger)^{-1} \exp \bigg( - \frac{\kappa}{2} \int_{\T^3} :\! (\pi_N u)^2 \!: dx - \cj{a}_2 \int_{\T^3} :\! (\pi_N u)^4 \!: dx \bigg) d \mu_\al (u),
\end{align*}

\noi
where $\mathcal{Z}_N^\dagger$ is a normalizing factor. Similar to Lemma~\ref{LEM:GWP4} and Lemma~\ref{LEM:inv_N}, we know that the equation \eqref{fNLWw_v} is globally well-posed and that $\nu_N^\dagger$ is invariant under the dynamics of \eqref{fNLWw_v}. Furthermore, similar to Theorem~\ref{THM:Gibbs}~(i), we know that $\nu_N^\dagger$ converges to $\nu$ in \eqref{Gibbs_nu} in total variation as $N \to \infty$.

As alluded at the beginning of the section, we proceed with the following first order expansion
\begin{align*}
u_N^\dagger = \<1> + w_N^\dagger,
\end{align*}

\noi
where $\<1>$ is defined in \eqref{loli} and $\vec w_N^\dagger = (w_N^\dagger, \dt w_N^\dagger)$ satisfies 
\begin{align}
\begin{split}
\vec w_N^\dagger &= - \kappa \vec \I (\<1>_N + \pi_N w_N^\dagger) \\
&\quad - 4 \cj{a}_2 \Big( \pi_N \vec{\<30>}_N + 3 \vec \If^{\,\<2>_N} (\pi_N w_N^\dagger) + 3 \pi_N \vec \I \big( \<1>_N (\pi_N w_N^\dagger)^2 \big) + \pi_N \vec \I \big( (\pi_N w_N^\dagger)^3 \big) \Big).
\end{split}
\label{eq_wN}
\end{align}

\noi
Here, $\vec \I$ is the Duhamel operator as defined in \eqref{vecSI} and \eqref{lin1}, $\vec{\<30>}_N$ is as defined in \eqref{so4b}, and $\vec \If^{\,\<2>_N}$ is as defined in \eqref{ran1}.

We now show the following proposition regarding the convergence of $\vec w_N^\dagger$ as $N \to \infty$ and probabilistic a priori bounds.

\begin{proposition}[$X^{s,b}$-based a priori bound]
\label{PROP:w_bdd}
Let $T > 0$, $\frac 98 < \al < \frac 32$, and $\dl > 0$ be small. Let $s \in \R$ satisfying
\begin{align}
\max \Big( - \al + \frac 32, \al - 1, - \frac{3}{10} \al + \frac 35 \Big) < s < \min \Big( \frac{13}{10} \al - \frac 35, 2 \al - \frac 32, 3 \al - 3 \Big),
\label{s_bdd}
\end{align}

\noi
which is possible given $\al > \frac 98$. Given $N \in \N$, let $w_N^\dagger$ be the solution of the equation \eqref{eq_wN}. Then, the following statements hold true.

\smallskip \noi
\textup{(i)} For any $N \in \N$ and $\eta > 0$, there exists $C_0 = C_0 (T, \eta) \geq 1$ independent of $N$ such that
\begin{align}
\vec \nu_N^\dagger \otimes \PP_2 \Big( \| \vec w_N^\dagger \|_{\vec X_T^{s, \frac 12 + \dl}} > C_0 \Big) < \eta. 
\label{wN_nuN}
\end{align} 

\smallskip \noi
\textup{(ii)} The sequence $\{ \vec w_N^\dagger \}_{N \in \N}$ converges $(\vec \nu \otimes \PP_2)$-almost surely to a limiting process $\vec w^\dagger$ in $\vec X_T^{s, \frac 12 + \dl}$. Moreover, there exists $\ta > 0$ such that
\begin{align}
\lim_{N \to \infty} \vec \nu \otimes \PP_2 \Big( \| \vec w_N^\dagger - \vec w^\dagger \|_{\vec X_T^{s, \frac 12 + \dl}} \leq N^{- \ta} \Big) = 1.
\label{wN_lim}
\end{align}

\smallskip \noi
Consequently, for any $\eta > 0$, there exists $C_0 = C_0 (T, \eta) \geq 1$ such that
\begin{align}
\vec \mu_\al \otimes \PP_2 \Big( \| \vec w^\dagger \|_{\vec X_T^{s, \frac 12 + \dl}} > C_0 \Big) < \eta.
\label{w_bdd}
\end{align}
\end{proposition}

\begin{proof}
The proof for (i) and (ii) follows similarly from the steps in Section~\ref{SEC:LWP} and Section~\ref{SEC:GWP}, and so we will be brief.

We first present local well-posedness of \eqref{fNLWw_v} with Gaussian initial data distributed by $\vec \mu_\al = \mu_\al \otimes \mu_0$ as in \eqref{gauss2}, which is the counterpart of Proposition~\ref{PROP:LWP}. 
By the inhomogeneous linear estimate in Lemma \ref{LEM:nhomo} and the time localization estimate \eqref{time2} in Lemma \ref{LEM:time}, we have
\begin{align}
\| \vec \I (\<1>_N) \|_{\vec X_{T_0}^{s, \frac 12 + \dl}} \les T_0^{3 \dl} \| \<1>_N \|_{X_{T_0}^{s - \al, -\frac 12 + 4 \dl}} \les T_0^{3 \dl} \| \<1>_N \|_{L_{T_0}^2 H_x^{\al - \frac 32 - \eps}},
\label{u_bdd1}
\end{align}

\noi
where $\dl > 0$ is small and we used $s < 2 \al - \frac 32 - \eps$ for $\eps > 0$ sufficiently small. By Lemma \ref{LEM:nhomo} and \eqref{time2} in Lemma~\ref{LEM:time}, we have
\begin{align}
\| \vec \I (\pi_N w_N^\dagger) \|_{\vec X_{T_0}^{s, \frac 12 + \dl}} \les T_0^{3 \dl} \| w_N^\dagger \|_{X_{T_0}^{s - \al, - \frac 12 + 4 \dl}} \les T_0^{3 \dl} \| w_N^\dagger \|_{X_{T_0}^{s, \frac 12 + \dl}}.
\label{u_bdd2}
\end{align}

\noi
By \eqref{Op2}, we have
\begin{align}
\| \vec \If^{\,\<2>_N} (\pi_N w_N^\dagger) \|_{\vec X_{T_0}^{s, \frac 12 + \dl}} \les T_0^{\ta} \| \vec \If^{\,\<2>_N} \|_{\vec \L_{T_0}^{\,s, \frac 12 + \dl, s, \frac 12 + \dl}} \| w_N^\dagger \|_{X_{T_0}^{s, \frac 12 + \dl}}
\label{u_bdd4}
\end{align}

\noi
for some $\ta > 0$, where the $\L_{T_0}^{s, \frac 12 + \dl, s, \frac 12 + \dl}$-norm is as defined in \eqref{Op2}. By Lemma \ref{LEM:nhomo}, \eqref{time2} in Lemma \ref{LEM:time}, and \eqref{tri3} in Lemma \ref{LEM:str4}, we have
\begin{align}
\big\| \vec \I \big( \<1>_N (\pi_N w_N^\dagger)^2 \big) \big\|_{\vec X_{T_0}^{s, \frac 12 + \dl}} \les T_0^{3 \dl} \| \<1>_N \|_{L_{T_0} W_x^{\al - \frac 32 - \eps, \infty}} \| w_N^\dagger \|_{X_{T_0}^{s, \frac 12 + \dl}}^2 ,
\label{u_bdd5}
\end{align}

\noi
where we require $\max (- \al + \frac 32, - \frac{3}{10} \al + \frac 35) < s < \min (\frac{13}{10} \al - \frac 35, 2 \al - \frac 32)$. Lastly, by Lemma \ref{LEM:nhomo}, \eqref{time2} in Lemma \ref{LEM:time}, and Lemma \ref{LEM:str_var} with $v \equiv 1$, we have
\begin{align}
\big\| \vec \I \big( (\pi_N w_N^\dagger)^3 \big) \big\|_{\vec X_{T_0}^{s, \frac 12 + \dl}} \les T_0^{\ta} \| w_N^\dagger \|_{X_{T_0}^{s, \frac 12 + \dl}}^3,
\label{u_bdd6}
\end{align}

\noi
where we require $\max (- \al + \frac 32, - \frac{3}{10} \al + \frac 35) < s < \frac{13}{10} \al - \frac 35$.
Combining \eqref{u_bdd1}--\eqref{u_bdd6}, and Lemma~\ref{LEM:sto_reg} for $\<1>_N$, $\vec \If^{\,\<2>_N}$ (where we need $\max (- \al + \frac 32, \al - 1) < s < 2 \al - \frac 32$), and $\vec{\<30>}_N$, we easily obtain local well-posedness of \eqref{eq_wN} via a contraction argument. 

We now briefly describe the globalization procedure. The counterpart of the uniform exponential integrability in Proposition~\ref{PROP:exptail2} for $\<1>_N$, $\vec \If^{\,\<2>_N}$, and $\vec{\<30>}_N$ follows in the same way. The counterpart of the stability estimate in Proposition~\ref{PROP:stb} follows from a similar procedure. The counterparts of the nonlinear smoothing in Proposition~\ref{PROP:ns} and the uniform bound with large probability in Proposition~\ref{PROP:bdd} also follow similarly. Note that the uniform bound with large probability leads to \eqref{wN_nuN} in part (i).
We can then follow the same steps in Subsection~\ref{SUBSEC:gwp} to obtain
\begin{align*}
\lim_{N_0 \to \infty} \vec \nu \otimes \PP_2 \Big( \| \vec w_{N_1}^\dagger - \vec w_{N_2}^\dagger \|_{\vec X_T^{s, \frac 12 + \dl}} \leq N_2^{- \ta} \text{ for all } N_1 \geq N_2 \geq N_0 \Big) = 1
\end{align*}

\noi
for some $\ta > 0$, and so $\vec w_N^\dagger$ converges $(\vec \nu \otimes \PP_2)$-almost surely to a process $\vec w^\dagger$ in $\vec X_T^{s, \frac 12 + \dl}$. This leads to part (ii) and also the identity \eqref{wN_lim}.

It remains to prove \eqref{w_bdd}. By the equivalence of the Gibbs measure $\vec \nu$ and the base Gaussian measure $\vec \mu_\al$ as stated in Proposition~\ref{PROP:Gibbs2}, we only need to show that for all $\eta' > 0$, there exists $C_0 = C_0 (T, \eta') \geq 1$ such that
\begin{align}
\vec \nu \otimes \PP_2 \Big( \| \vec w^\dagger \|_{\vec X_T^{s, \frac 12 + \dl}} > C_0 \Big) < \eta'.
\label{w_bdd2}
\end{align}

\noi
From \eqref{wN_lim} in part (ii), we know that there exists $N_0 \in \N$ such that
\begin{align}
\vec \nu \otimes \PP_2 \Big( \| \vec w_N^\dagger - \vec w^\dagger \|_{\vec X_T^{s, \frac 12 + \dl}} > N^{- \ta} \Big) < \frac{\eta'}{3}.
\label{w_step1}
\end{align}

\noi
From \eqref{wN_nuN} in part (i), we know that for any $N \in \N$, there exists $C_1 = C_1 (T, \eta) \geq 1$ such that
\begin{align}
\vec \nu_N^\dagger \otimes \PP_2 \Big( \| \vec w_N^\dagger \|_{\vec X_T^{s, \frac 12 + \dl}} > C_1 \Big) < \frac{\eta'}{3}.
\label{w_step2}
\end{align}

\noi
By enlarging $N_0$ if necessary, we deduce from the convergence of $\nu_N^\dagger$ to $\nu$ in total variation that
\begin{align}
\Big| (\vec \nu_N^\dagger \otimes \PP_2 - \vec \nu \otimes \PP_2) \Big( \| \vec w_N^\dagger \|_{\vec X_T^{s, \frac 12 + \dl}} > C_1 \Big) \Big| < \frac{\eta'}{3}.
\label{w_step3}
\end{align}

\noi
Thus, by taking $N \geq N_0$, we combine \eqref{w_step1}, \eqref{w_step2}, and \eqref{w_step3} to obtain
\begin{align*}
\vec \nu &\otimes \PP_2 \Big( \| \vec w^\dagger \|_{\vec X_T^{s, \frac 12 + \dl}} > C_1 + 1 \Big) \\
&\leq \vec \nu \otimes \PP_2 \Big( \| \vec w_N^\dagger \|_{\vec X_T^{s, \frac 12 + \dl}} > C_1 \Big) 
+ \vec \nu \otimes \PP_2 \Big( \| \vec w_N^\dagger - \vec w^\dagger \|_{\vec X_T^{s, \frac 12 + \dl}} > N^{- \ta} \Big) \\
&< \frac{\eta'}{3} + \vec \nu_N^\dagger \otimes \PP_2 \Big( \| \vec w_N^\dagger \|_{\vec X_T^{s, \frac 12 + \dl}} > C_1 \Big) + \frac{\eta'}{3} < \eta',
\end{align*}

\noi
which proves \eqref{w_bdd2}.
\end{proof}

\subsection{Proof of Theorem \ref{THM:wu}}
In this subsection, we present the proof of Theorem~\ref{THM:wu}, the dynamical weak universality result. For this purpose, we proceed with the first order expansion
\begin{align*}
u_N = \<1> + w_N,
\end{align*}

\noi
where $N \in \N$, $u_N$ is the solution to \eqref{fNLWw}, $\<1>$ is as defined in \eqref{loli}, and $w_N$ is the remainder term satisfying the following Duhamel formulation (with the use of \eqref{HkX} and \eqref{herm_decomp1}):
\begin{align}
\begin{split}
&\vec w_N = - 2 \cj{a}_{1, N} N^{2 (\frac 32 - \al)} \pi_N \vec \I (\<1>_N + w_N) - 4 \cj{a}_{2, N} \vec{\<30>}_N - 12 \cj{a}_{2, N} \vec \If^{\,\<2>_N} (\pi_N w_N)  \\
&\quad - 12 \cj{a}_{2, N} \pi_N \vec \I \big( \<1>_N (\pi_N w_N)^2 \big) - 4 \cj{a}_{2, N} \pi_N \vec \I \big( (\pi_N w_N)^3 \big) - \sum_{\l = 0}^{2m - 1} \pi_N \vec \I \big( \mathcal{G}_{N, \l} \cdot (\pi_N w_N)^\l \big),
\end{split}
\label{wN_eq}
\end{align}

\noi
where
\begin{align*}
\mathcal{G}_{N, \l} = \sum_{j = 3 \vee \lceil \frac{\l + 1}{2} \rceil}^m 2j \binom{2j - 1}{\l} \cj{a}_{j, N} N^{- (2j - 4) (\frac 32 - \al)} :\! \<1>_N^{2j - \l - 1} \!:.
\end{align*}

\noi
Here, $\vec \I$ is the Duhamel operator as defined in \eqref{vecSI} and \eqref{lin1}, $\<1>_N = \pi_N \<1>$, $\vec{\<30>}_N$ is as defined in \eqref{so4b}, and $\vec \If^{\,\<2>_N}$ is as defined in \eqref{ran1}.

We also recall the solution $\vec w_N^\dagger$ to the equation \eqref{eq_wN}. From Proposition~\ref{PROP:w_bdd}, we see that $\vec w_N^\dagger$ converges as $N \to \infty$ to a process $\vec w^\dagger$ which satisfies
\begin{align}
\vec w^\dagger = - \kappa \vec \I (\<1> + w^\dagger) - 4 \cj{a}_2 \Big( \vec{\<30>} + 3 \vec \If^{\,\<2>} (w^\dagger) + 3 \vec \I \big( \<1> (w^\dagger)^2 \big) + \vec \I \big( (w^\dagger)^3 \big) \Big),
\label{w_eq}
\end{align}

\noi
where $\vec{\<30>}$ and $\vec \If^{\,\<2>}$ are limiting stochastic objects of $\vec{\<30>}_N$ and $\vec \If^{\,\<2>_N}$, respectively, as $N \to \infty$ given by Lemma~\ref{LEM:sto_reg}. For later convenience, we extend the definitions of $\vec{\<30>}_N$, $\vec{\<30>}$, $\vec \If^{\,\<2>_N}$, and $\vec \If^{\,\<2>}$ by setting $\vec{\<30>}_{t_0, N}$, $\vec{\<30>}_{t_0}$, $\vec \If_{t_0}^{\,\<2>_N}$, and $\vec \If_{t_0}^{\,\<2>}$ to be the same objects with the Duhamel integral operator $\vec \I$ replaced by $\vec \I_{t_0}$ given $t_0 \geq 0$.

Our goal is to prove the following more precise statement of Theorem \ref{THM:wu}. We recall that $\vec \mu_\al = \mu_\al \otimes \mu_0$ is the Gaussian measure defined in \eqref{gauss2} and that $\PP_2$ is the underlying probability measure for the space-time white noise $\xi$.
\begin{proposition}
\label{PROP:wu}
Let $\frac 98 < \al < \frac 32$, $T \geq 1$, and $\dl > 0$ be sufficiently small. Let $s \in \R$ satisfying \eqref{s_bdd}.
Let $\vec w_N$ be the solution to the equation \eqref{wN_eq} and let $\vec w^\dagger$ be the solution to the equation \eqref{w_eq}. Then, $\vec w_N$ converges $(\vec \mu_\al \otimes \PP_2)$-almost surely to $\vec w^\dagger$ in the space $\vec X_T^{s, \frac 12 + \dl}$.
\end{proposition}

Note that Proposition \ref{PROP:wu}, combined with the convergence of $\vec{\<1>}_N$ to $\vec{\<1>}$ as in Lemma \ref{LEM:sto_reg}, implies Theorem \ref{THM:wu}.

\medskip
We now show the convergence of $w_N$ to $w$ by assuming some a priori bounds established in previous two subsections. We recall the $Y_T^{s, b}$-norm and the $\J_T^{s_1, b_1, s_2, b_2}$-norm in \eqref{Ysb} and \eqref{Jop}, respectively.
\begin{proposition}
\label{PROP:conv}
Let $\frac 98 < \al < \frac 32$, $T \geq 1$, and $\eps, \g, \dl > 0$ be sufficiently small. Let $s \in \R$ satisfying \eqref{s_bdd}. Assume that there exists a constant $C_0 \gg 1$ such that
\begin{align}
\begin{split}
&\| w_N \|_{L^{\frac{1}{\eps}}_T L_x^{\infty}} \leq N^{- \al + \frac 32 + \eps + \g}, \qquad \| \vec w^\dagger \|_{\vec X_T^{s, \frac 12 + \dl}} \leq C_0, \\
&\| \<1> \|_{L^{2m}_T W_x^{\al - \frac 32 - \eps, \infty}} \leq C_0, \qquad \| \<1>_N - \<1>  \|_{L^{2m}_T W_x^{\al - \frac 32 - \eps, \infty}} \leq C_0 N^{- \frac{\eps}{2}}, \\
&\| \vec{\<30>} \|_{\vec Y_T^{s, \frac 12 + \dl}} \leq C_0, \qquad \| \vec{\<30>}_N - \vec{\<30>} \|_{\vec Y_T^{s, \frac 12 + \dl}} \leq C_0 N^{- \eps}, \\
&\| \vec \If^{\,\<2>} \|_{\vec \J_T^{\,s, \frac 12 + \dl, s, \frac 12 + \dl}} \leq C_0, \qquad \| \vec \If^{\,\<2>_N} - \vec \If^{\,\<2>} \|_{\vec \J_T^{\,s, \frac 12 + \dl, s, \frac 12 + \dl}} \leq C_0 N^{- \eps}, \\
&\| :\! \<1>_N^k \!: \|_{L_T^{\frac{1}{\eps}} W_x^{\al - \frac 32 - \eps, \infty}} \leq C_0 N^{(k - 1) (\frac 32 - \al) - \frac{\eps}{4}}, \\
&\| \vec \I ( :\! \<1>_N^{k'} \!: ) \|_{\vec Y_T^{s, \frac 12 + 4 \dl}} \leq C_0 N^{(k' - 3) (\frac 32 - \al) - \frac{\eps}{2}}, \\
&\| \vec \If_N^{k' - 1} \|_{\vec \J_T^{\,s, \frac 12 + \dl, s, \frac 12 + \dl}} \leq C_0 N^{(k' - 3) (\frac 32 - \al) - \frac{\eps}{2}},
\end{split}
\label{apriori}
\end{align}

\noi
where $2 \leq k \leq 2m - 1$ and $4 \leq k' \leq 2m - 1$.
Then, there exist $\ta > 0$ small and a constant $C_1 > 0$ independent of $N \in \N$ such that
\begin{align*}
\| \vec w_N - \vec w^\dagger \|_{\vec X_T^{s, \frac 12 + \dl}} < C_1 N^{-\ta}
\end{align*}

\noi
for $N$ sufficiently large.
\end{proposition}

\begin{proof}
Let $I = [t_0, t_0 + T_0] \subset [0, T]$ be a closed interval with $0 < T_0 \leq 1$. Let us write the equations for $\vec w_N$ and $\vec w^\dagger$ starting at time $t_0$: 
\begin{align*}
\vec w_N &= \vec S(\cdot - t_0) ( \vec w_N (t_0) ) - 2 \cj{a}_{1, N} N^{2 (\frac 32 - \al)} \pi_N \vec \I_{t_0} (\<1>_N + w_N) - 4 \cj{a}_{2, N} \vec{\<30>}_{t_0, N} \\
&\quad - 12 \cj{a}_{2, N} \vec \If_{t_0}^{\,\<2>_N} (\pi_N w_N) - 12 \cj{a}_{2, N} \pi_N \vec \I_{t_0} \big( \<1>_N (\pi_N w_N)^2 \big) \\
&\quad - 4 \cj{a}_{2, N} \pi_N \vec \I_{t_0} \big( (\pi_N w_N)^3 \big) - \sum_{\l = 0}^{2m - 1} \pi_N \vec \I_{t_0} \big( \mathcal{G}_{N, \l} \cdot (\pi_N w_N)^\l \big)
\end{align*}

\noi
and
\begin{align*}
\vec w^\dagger &= \vec S(\cdot - t_0) \big( \vec w^\dagger (t_0) \big) - \kappa \vec \I_{t_0} (\<1> + w^\dagger) \\
&\quad - 4 \cj{a}_2 \Big( \vec{\<30>}_{t_0} + 3 \vec \If_{t_0}^{\,\<2>} (w^\dagger) + 3 \vec \I_{t_0} \big( \<1> (w^\dagger)^2 \big) + \vec \I_{t_0} \big( (w^\dagger)^3 \big) \Big) .
\end{align*}

\noi
Thus, by taking the difference of the above two equations and using the homogeneous linear estimate in Lemma~\ref{LEM:homo}, we obtain
\begin{align}
\begin{split}
\| \vec w_N - \vec w^\dagger \|_{\vec X_I^{s, \frac 12 + \dl}} &\les \| \vec w_N (t_0) - \vec w^\dagger (t_0) \|_{\vec H^s} \\
&\quad + A_N (I) + \sum_{\l = 4}^{2m - 1} B_{N, \l} (I) + \sum_{\l = 0}^3 \sum_{j = 3}^m C_{N, j, \l} (I),
\end{split}
\label{wN_diff}
\end{align}

\noi
where 
\begin{align*}
A_N (I) &\overset{\text{def}}{=} \big\| 2 \cj{a}_{1, N} N^{2 (\frac 32 - \al)} \pi_N \vec \I_{t_0} (\<1>_N + w_N) - \kappa \vec \I_{t_0} (\<1> + w^\dagger) \big\|_{\vec X_I^{s, \frac 12 + \dl}} \\
&\quad +  4 \big\| \cj{a}_{2, N} \vec{\<30>}_{t_0, N} - \cj{a}_2 \vec{\<30>}_{t_0} \big\|_{\vec X_I^{s, \frac 12 + \dl}} + 12 \big\| \cj{a}_{2, N} \vec \If_{t_0}^{\,\<2>_N} (\pi_N w_N) - \cj{a}_2 \vec \If_{t_0}^{\,\<2>} (w^\dagger) \big\|_{\vec X_I^{s, \frac 12 + \dl}} \\
&\quad + 12 \big\| \cj{a}_{2, N} \pi_N \vec \I_{t_0} \big( \<1>_N (\pi_N w_N)^2 \big) - \cj{a}_2 \vec \I_{t_0} \big( \<1> (w^\dagger)^2 \big) \big\|_{\vec X_I^{s, \frac 12 + \dl}} \\
&\quad + 4 \big\| \cj{a}_{2, N} \pi_N \vec \I_{t_0} \big( (\pi_N w_N)^3 \big) - \cj{a}_2 \vec \I_{t_0} \big( (w^\dagger)^3 \big) \big\|_{\vec X_I^{s, \frac 12 + \dl}}, \\
B_{N, \l} (I) &\overset{\text{def}}{=} \big\| \pi_N \vec \I_{t_0} \big( \mathcal{G}_{N, \l} \cdot (\pi_N w_N)^\l \big) \big\|_{\vec X_I^{s, \frac 12 + \dl}}, \\
C_{N, j, \l} (I) &\overset{\text{def}}{=} N^{- (2j - 4) (\frac 32 - \al)} \big\| \pi_N \vec \I_{t_0} ( :\! \<1>_N^{2 j - \l - 1} \!: (\pi_N w_N)^\l ) \big\|_{\vec X_I^{s, \frac 12 + \dl}}.
\end{align*}

For $A_N (I)$, by using slight modifications of \eqref{u_bdd1}--\eqref{u_bdd6} along with the bounds in \eqref{apriori}, we obtain that for $N \in \N$ sufficiently large,
\begin{align}
\begin{split}
A_N (I) &\les  |2 \cj{a}_{1, N} N^{2 (\frac 32 - \al)} - \kappa| C_0 + \kappa C_0 N^{- \ta} + T_0^{\ta'} \kappa  \| w_N - w^\dagger \|_{X_I^{s, \frac 12 + \dl}} \\
&\quad + |\cj{a}_{2, N} - \cj{a}_2| C_0^3 + C_0^3 N^{- \ta} + T_0^{\ta'} \big( C_0^2 + \| w_N \|_{X_T^{s, \frac 12 + \dl}}^2 \big) \| w_N - w^\dagger \|_{X_I^{s, \frac 12 + \dl}}
\end{split}
\label{AN}
\end{align}

\noi
for some $\ta, \ta' > 0$ sufficiently small, where we used the fact that $\cj{a}_{2, N} \to \cj{a}_2$ as $N \to \infty$. Here, to deal with $\Id - \pi_N = \pi_N^\perp$, we can always create a negative power of $N$ be slightly losing some spatial regularity (see Lemma \ref{LEM:str_var} and Lemma \ref{LEM:str4}, where we can always create some room for slight regularity loss).

For $B_{N, \l} (I)$, we have $4 \leq \l \leq 2m - 1$. By the inhomogeneous linear estimate in Lemma~\ref{LEM:nhomo}, Lemma \ref{LEM:str_ext}, Bernstein's inequality, and the bounds in \eqref{apriori}, we have
\begin{align}
\begin{split}
B_{N, \l} (I) &\les \sum_{j = 3 \vee \lceil \frac{\l + 1}{2} \rceil}^m  | \cj{a}_{j, N} | N^{- (2j - 4) (\frac 32 - \al)} \| :\! \<1>_N^{2j - \l - 1} \!: (\pi_N w_N)^\l \|_{X_I^{s - \al, - \frac 12 + \dl}} \\
&\les \sum_{j = 3 \vee \lceil \frac{\l + 1}{2} \rceil}^m  | \cj{a}_{j, N} | N^{- \eps_j} \| w_N \|_{X_I^{s, \frac 12 + \dl}}^{3 + \eps'} N^{(\l - 3 - \eps') (\al - \frac 32 - \eps - \g)} \| w_N \|_{L_I^{\frac{1}{\eps'}} L_x^{\frac{1}{\eps'}}}^{\l - 3 - \eps'} \\
&\quad \times N^{(2j - \l - 1) (\al - \frac 32 - \eps)} \| :\! \<1>_N^{2j - \l - 1} \!: \|_{L_I^{\frac{1}{\eps'}} L_x^{\frac{1}{\eps'}}} \\
&\leq C_0^{\l - 2 - \eps'} \| w_N \|_{X_T^{s, \frac 12 + \dl}}^{3 + \eps'} \sum_{j = 3 \vee \lceil \frac{\l + 1}{2} \rceil}^m  | \cj{a}_{j, N} | N^{- \eps_j},
\end{split}
\label{BN}
\end{align}

\noi
where $\eps' > 0$ is small and
\begin{align*}
\eps_j = \Big( \frac 32 - \al \Big) \eps' - (2j - 4 - \eps') \eps - (\l - 3 - \eps') \g > 0
\end{align*}

\noi
for all $j$ if we require $0 < \eps, \g \ll \eps'$.

For $C_{N, j, \l} (I)$, we apply \eqref{tri3} in Lemma \ref{LEM:str4} (for $\l = 2$), Lemma \ref{LEM:str_var} (for $\l = 3$, with Bernstein's inequality applied to the Wick power), and the bounds in \eqref{apriori} to obtain
\begin{align}
C_{N, j, \l} (I) \les C_0 N^{- \ta} \| w_N \|_{X_T^{s, \frac 12 + \dl}}^\l
\label{CN}
\end{align}

\noi
for some $\ta > 0$ sufficiently small. Combining \eqref{wN_diff}, \eqref{AN}, \eqref{BN}, \eqref{CN} and using the fact that $\cj{a}_{j, N} \to \cj{a}_{j}$ for all $2 \leq j \leq m$, and $2 \cj{a}_{1, N} N^{2 (\frac 32 - \al)} \to \kappa$ as $N \to \infty$ (see Subsection \ref{SUBSEC:wu}), we obtain
\begin{align}
\begin{split}
\| \vec w_N - \vec w^\dagger \|_{\vec X_I^{s, \frac 12 + \dl}} &\les \| \vec w_N (t_0) - \vec w^\dagger (t_0) \|_{\vec H^s} \\
&\quad + (\kappa C_0 + C_0^3) N^{- \ta} + C_0 N^{- \ta} \big( 1 + \| w_N \|_{X_T^{s, \frac 12 + \dl}}^3 \big) \\
&\quad + T_0^{\ta'} (\kappa + C_0^2 + \| w_N \|_{X_T^{s, \frac 12 + \dl}}^2) \| w_N - w^\dagger \|_{X_I^{s, \frac 12 + \dl}} \\
&\quad + C_0^{\l - 2 - \eps'} N^{-\ta} \| w_N \|_{X_T^{s, \frac 12 + \dl}}^{3 + \eps'} \sum_{j = 3 \vee \lceil \frac{\l + 1}{2} \rceil}^m  |\cj{a}_{j}|.
\end{split}
\label{wN_diffest}
\end{align}

We now use a bootstrap argument as in \cite[Section~7.2]{STz21} and \cite[Lemma~4.5]{STzX}. Suppose that we have the bound
\begin{align}
\| \vec w_N \|_{\vec X_T^{s, \frac 12 + \dl}} \leq 2 C_0^{10 m}
\label{wN_bdd_goal}
\end{align}

\noi
uniformly in $N \geq N_0$ for some $N_0 = N_0 (C_0, T) \in \N$ sufficiently large. Then, from \eqref{wN_diffest} and \eqref{wN_bdd_goal}, there exists a constant $C > 0$ such that
\begin{align*}
\| \vec w_N - \vec w^\dagger \|_{\vec X_I^{s, \frac 12 + \dl}} &\leq \| \vec w_N (t_0) - \vec w^\dagger (t_0) \|_{\vec H^s}\\
&\quad + C C_0^{40 m} N^{- \ta} + C C_0^{20 m} T_0^{\ta'} \| \vec w_N - \vec w^\dagger \|_{\vec X_I^{s, \frac 12 + \dl}}.
\end{align*}

\noi
We choose $T_0 > 0$ small enough such that $C C_0^{20} T_0^{\ta'} \ll 1$, so that
\begin{align}
\| \vec w_N - \vec w^\dagger \|_{\vec X_I^{s, \frac 12 + \dl}} \leq 2 \| \vec w_N (t_0) - \vec w^\dagger (t_0) \|_{\vec H^s} + 2 C C_0^{40 m} N^{- \ta}.
\label{wudiff}
\end{align}

\noi
We can further shrink $T_0$ so that $\frac{T}{T_0}$ is an integer. Thus, we can write
\begin{align*}
[0, T] = \bigcup_{j = 0}^{\frac{T}{T_0} - 2} [j T_0, (j + 2) T_0],
\end{align*}

\noi
so that from \eqref{wudiff} and Lemma~\ref{LEM:HsXsb}, we get
\begin{align*}
\| \vec w_N - \vec w^\dagger \|_{\vec X_{[0, 2 T_0]}^{s, \frac 12 + \dl}} \les C_0^{40 m} N^{- \ta}
\end{align*}

\noi
and for any $j = 0, 1, \dots, \frac{T}{T_0} - 2$, we get
\begin{align*}
\| \vec w_N - \vec w^\dagger \|_{\vec X_{[(j + 1)T_0, (j + 3) T_0]}^{s, \frac 12 + \dl}} \les \| \vec w_N - \vec w^\dagger \|_{\vec X_{[j T_0, (j + 2) T_0]}^{s, \frac 12 + \dl}} + C_0^{40 m} N^{- \ta} .
\end{align*}

\noi
Thus, using a simple iteration and the gluing lemma (Lemma~\ref{LEM:glue}), we see that there exists a constant $C_1 = C_1 (C_0, T) > 0$ such that 
\begin{align}
\| \vec w_N - \vec w^\dagger \|_{\vec X_T^{s, \frac 12 + \dl}} \leq C_1 N^{- \ta}.
\label{wN_desired}
\end{align}

\noi
This is the desired estimate.

It remains to show that 
\begin{align*}
\| \vec w_N \|_{\vec X_T^{s, \frac 12 + \dl}} \leq C_0^{10 m} < 2 C_0^{10 m}
\end{align*}

\noi
uniformly in $N \geq N_0$ for some $N_0 = N_0 (C_0, T) \in \N$ sufficiently large. Assume that $T_* < T$ is the largest number such that
\begin{align*}
\| \vec w_N \|_{\vec X_{T_*}^{s, \frac 12 + \dl}} \leq C_0^{10 m}.
\end{align*}

\noi
By continuity of the function\footnote{See, for example, \cite[Lemma 2.7]{BLLZ}.}
\begin{align*}
T \mapsto \| \vec w_N \|_{\vec X_T^{s, \frac 12 + \dl}},
\end{align*}

\noi
there exists $a > 0$ such that $T_* + a \leq T$ and
\begin{align*}
\| \vec w_N \|_{\vec X_{T_* + a}^{s, \frac 12 + \dl}} \leq 2C_0^{10 m}.
\end{align*}

\noi
By running the same argument above, we then obtain \eqref{wN_desired} with $T$ replaced by $T_* + a$. Then, by the bounds in \eqref{apriori}, we have
\begin{align*}
\| \vec w_N \|_{\vec X_{T_* + a}^{s, \frac 12 + \dl}} &\leq \| \vec w_N - \vec w^\dagger \|_{\vec X_{T_* + a}^{s, \frac 12 + \dl}} + \| \vec w^\dagger \|_{\vec X_{T_* + a}^{s, \frac 12 + \dl}} \\
&\leq C_1 (C_0, T) N^{- \ta} + C_0 \\
&\leq C_0^{10 m},
\end{align*}

\noi
as long as $N = N(C_0, T) \in \N$ is large enough. This contradicts the definition of $T_*$, so that we must have $T_* = T$. Thus, we finish the proof of the proposition. 
\end{proof}

We are now ready to prove our main proposition. We recall that $\O_2$ is the underlying probability space equipped with the probability measure $\PP_2$ for the space-time white noise $\xi$.
\begin{proof}[Proof of Proposition \ref{PROP:wu}]
By using the tail bounds in Lemma \ref{LEM:sto_reg}, Lemma \ref{LEM:sto_k}, Lemma \ref{LEM:sto_k2}, Lemma \ref{LEM:sto_k3}, Proposition~\ref{PROP:uN_bdd}, and Proposition~\ref{PROP:w_bdd}, we know that for any $T \geq 1$ and $\eta > 0$, there exists $\Sigma_{T, \eta} \subset \vec H^{\al - \frac 32 - \eps} \times \O_2$ such that 
$$\vec \mu_\al \otimes \PP_2 (\Sigma_{T, \eta}^c) < \eta$$ and for each $(\vec \phi, \omega_2) \in \Sigma_{T, \eta}$, the bounds in \eqref{apriori} in Proposition \ref{PROP:conv} hold. By Proposition~\ref{PROP:conv}, the bounds in \eqref{apriori} imply that $\vec w_N$ converges to $\vec w^\dagger$ in $\vec X_T^{s, \frac 12 + \dl}$ as $N \to \infty$. It remains to define
\begin{align*}
\Sigma = \bigcup_{k = 1}^\infty \bigcap_{j = 1}^\infty \Sigma_{2^j, 2^{-j} k^{-1}},
\end{align*}

\noi
since we have $\vec \mu_\al \otimes \PP_2 (\Sigma) = 1$ and for each $\vec \phi \in \Sigma$, $\vec w_N$ converges to $\vec w^\dagger$ in $\vec X_T^{s, \frac 12 + \dl}$ as $N \to \infty$.
\end{proof}

\begin{ackno}\rm
The authors would like to thank Prof.~Tadahiro Oh for proposing the problem and for helpful suggestions. The authors also thank Chenmin Sun, Tomoyuki Tanaka, and Leonardo Tolomeo for helpful discussions. R.L. was supported by the European Research Council (grant no. 864138 ``SingStochDispDyn''). R.L. also acknowledge support from the Deutsche Forschungsgemeinschaft (DFG, German Research Foundation) under Germany's Excellence Strategy -- EXC-2047/1 -- 390685813. R.L. was also funded by the Deutsche Forschungsgemeinschaft (DFG, German Research Foundation) -- Project-ID 211504053 - SFB 1060. R.L. was also funded by the Deutsche Forschungsgemeinschaft (DFG, German Research Foundation) -- Project-ID 539309657 -- SFB 1720. 
N.T. was partially supported by the ANR project Smooth ANR-22-CE40-0017 and by the European Research Council (ERC) under the European Union Horizon 2020 research and innovation programme (Grant agreement 101097172 - GEOEDP). 
Y.W. was supported by the EPSRC New Investigator Award (grant no. EP/V003178/1) and by the EPSRC Mathematical Sciences Small Grant (grant no. UKRI1116). The authors would like to warmly thank the anonymous reviewers for the helpful comments.
\end{ackno}

\noi
{\bf AI disclosure.} No AI tool was used to elaborate any part of the mathematical content of this paper. The authors acknowledge the use of AI tools for improvement of the presentation and for correction of typographical errors.


\begin{thebibliography}{99}







\bibitem{AK}
S.~Albeverio, S.~Kusuoka,
{\it The invariant measure and the flow associated to the $\Phi^4_3$-quantum field model},
Ann. Sc. Norm. Super. Pisa Cl. Sci.  20 (2020), no. 4, 1359--1427.





%
%
%


\bibitem{BG}
N.~Barashkov, M.~Gubinelli,
{\it  A variational method for $\Phi^4_3$},
Duke Math. J. 169 (2020), no. 17, 3339--3415.

\bibitem{BG21}
N.~Barashkov, M.~Gubinelli,
{\it The $\Phi^4_3$ measure via Girsanov's theorem},
Electron. J. Probab. 26 (2021), Paper No. 81, 29 pp.


\bibitem{BOP2}
\'A.~B\'enyi, T.~Oh, O.~Pocovnicu,
{\it On the probabilistic Cauchy theory of the cubic nonlinear Schr\"odinger equation on
$\R^3$, $d\geq 3$}, Trans. Amer. Math. Soc. Ser. B  2 (2015), 1--50.
%
%

\bibitem{BOZ}
\'A.~B\'enyi, T.~Oh, T.~Zhao,
{\it Fractional Leibniz rule on the torus},
Proc. Amer. Math. Soc. 153 (2025), no. 1, 207--221.

%
%
%
%
%
%
%
%

\bibitem{BD}
M.~Bou\'e, P.~Dupuis,
{\it A variational representation for certain functionals of Brownian motion},
Ann. Probab. 26 (1998), no. 4, 1641--1659.


\bibitem{Bour93}
J.~Bourgain,
{\it Fourier transform restriction phenomena for certain lattice subsets and applications to nonlinear evolution equations. I. Schr\"odinger equations},
Geom. Funct. Anal. 3 (1993), no. 2, 107--156.


\bibitem{BO94}
J.~Bourgain,
{\it Periodic nonlinear Schr\"odinger equation and invariant measures},
Comm. Math. Phys. 166 (1994), no. 1, 1--26.


\bibitem{BO96}
J.~Bourgain,
{\it Invariant measures for the 2D-defocusing nonlinear Schr\"odinger equation},
Comm. Math. Phys. 176 (1996), no. 2, 421--445.
%


\bibitem{BB12}
J.~Bourgain, A.~Bulut,
{\it Gibbs measure evolution in radial nonlinear wave and Schr\"odinger equations on the ball},
C. R. Math. Acad. Sci. Paris 350 (2012), no.11-12, 571--575.


\bibitem{BB14}
J.~Bourgain, A.~Bulut,
{\it Invariant Gibbs measure evolution for the radial nonlinear wave equation on the 3d ball},
J. Funct. Anal. 266 (2014), no.4, 2319--2340.

\bibitem{BD15}
J.~Bourgain, C.~Demeter,
{\it The proof of the $\l^2$ decoupling conjecture},
Ann. of Math. 182 (2015), no. 1, 351--389.


\bibitem{Bring1}
B.~Bringmann,
{\it Invariant Gibbs measures for the three-dimensional wave equation with a Hartree nonlinearity I: measures},
Stoch. Partial Differ. Equ. Anal. Comput. 10 (2022), no. 1, 1--89.



\bibitem{Bring2}
B.~Bringmann,
{\it Invariant Gibbs measures for the three-dimensional wave equation
with a Hartree nonlinearity II: dynamics},
J. Eur. Math. Soc. (JEMS) 26 (2024), no. 6, 1933--2089.


\bibitem{BDNY}
B.~Bringmann, Y.~Deng, A.~Nahmod, H.~Yue,
{\it Invariant Gibbs measures for the three dimensional cubic nonlinear wave equation},
Invent. Math. 236 (2024), no. 3, 1133--1411.



\bibitem{BLLZ}
E.~Brun, G.~Li, R.~Liu, Y.~Zine,
{\it Global well-posedness of one-dimensional fractional cubic nonlinear Schr\"odinger equations in negative Sobolev spaces},
arXiv:2311.13370 [math.AP].


\bibitem{BMS}
Z.~Brze\'zniak, B.~Maslowski, J.~Seidler, 
{\it Stochastic nonlinear beam equations}, 
Probab. Theory Related Fields 132 (2005), no. 1, 119--149.


\bibitem{BOS}
Z.~Brze\'zniak, M.~Ondrej\'at, J.~Seidler, 
{\it Invariant measures for stochastic nonlinear beam and wave equations},
J. Differential Equations 260 (2016), no. 5, 4157--4179.


\bibitem{BTz07}
N.~Burq, N.~Tzvetkov,
{\it Invariant measure for a three dimensional nonlinear wave equation},
Int. Math. Res. Not. IMRN (2007), no. 22, Art. ID rnm108, 26 pp.


\bibitem{BTz08-1}
N.~Burq, N.~Tzvetkov,
{\it Random data Cauchy theory for supercritical wave equations. I. Local theory},
Invent. Math. 173 (2008), no. 3, 449--475.


\bibitem{BTz08}
N.~Burq, N.~Tzvetkov,
{\it Random data Cauchy theory for supercritical wave equations. II. A global existence result}, 
Invent. Math. 173 (2008), no. 3, 477--496.



%


\bibitem{CLL}
A.~Chapouto, G.~Li, R.~Liu,
{\it Global dynamics for the stochastic nonlinear beam equations on the four-dimensional torus}, 
Proc. Roy. Soc. Edinburgh Sect. A 156
(2026), no. 4, 1353-1391.





\bibitem{DPD2}
G.~Da Prato, A.~Debussche,
{\it Strong solutions to the stochastic quantization equations,} Ann. Probab. 31 (2003), no. 4, 1900--1916.






\bibitem{DZ}
G.~Da Prato, J.~Zabczyk,
{\it Stochastic equations in infinite dimensions,} Second edition. Encyclopedia of Mathematics and its Applications, 152. Cambridge University Press, Cambridge, 2014. xviii+493 pp.
%



\bibitem{DNY3}
Y.~Deng, A.~Nahmod, H.~Yue, 
{\it Random tensors, propagation of randomness, and nonlinear dispersive equations}, 
Invent. Math. 228 (2022), no. 2, 539-686.



\bibitem{DNY2}
Y.~Deng, A.~Nahmod, H.~Yue,
{\it Invariant Gibbs measures and global strong solutions for nonlinear Schr\"odinger equations in dimension two},
Ann. of Math. 200 (2024), no. 2, 399--486.




\bibitem{DNYprob}
Y.~Deng, A.~Nahmod, H.~Yue,
{\it The probabilistic scaling paradigm},
Vietnam J. Math. 52 (2024), no. 4, 1001--1015.

\bibitem{Deya1}
A.~Deya,
{\it A nonlinear wave equation with fractional perturbation},
Ann. Probab. 47 (2019), no. 3, 1775--1810.

\bibitem{Deya2}
A.~Deya,
{\it On a non-linear 2D fractional wave equation},
Ann. Inst. Henri Poincar\'e Probab. Stat. 56 (2020), no. 1, 477--501.


\bibitem{Duch21}
P.~Duch,
{\it Flow equation approach to singular stochastic PDEs},
Probab. Math. Phys. 6 (2025), no. 2, 327--437.


\bibitem{Duch22}
P.~Duch,
{\it Renormalization of singular elliptic stochastic PDEs using flow equation},
Probab. Math. Phys. 6 (2025), no. 1, 111--138.


\bibitem{DGR}
P.~Duch, M.~Gubinelli, P.~Rinaldi,
{\it Parabolic stochastic quantisation of the fractional $\Phi^4_3$ model in the full subcritical regime},
J. Eur. Math. Soc. (2026), DOI:10.4171/JEMS/1816.





\bibitem{Edw}
D.~A.~Edwards,
{\it On the Kantorovich-Rubinstein theorem},
Expo. Math. 29 (2011), no. 4, 387--398.


\bibitem{ErTz}
M.~B.~Erdo\v{g}an, N.~Tzirakis, 
{\it Dispersive partial differential equations. Wellposedness and applications}, London Math. Soc. Stud. Texts, 86 Cambridge University Press, Cambridge, 2016. xvi+186 pp.


\bibitem{EX}
D.~Erhard, W.~Xu,
{\it Weak universality of dynamical $\Phi^4_3$: polynomial potential and general smoothing mechanism},
Electron. J. Probab. 27 (2022), Paper No. 112, 43 pp.


\bibitem{EW}
S.~Esquivel, H.~Weber,
{\it A priori bounds for the dynamic fractional $\Phi^4$ model on $\T^3$ in the full subcritical regime},
arXiv:2411.16536 [math.AP].



\bibitem{FP}
L.~Forcella, O.~Pocovnicu,
{\it Invariant Gibbs dynamics for two-dimensional fractional wave equations in negative Sobolev spaces},
arXiv:2306.07857 [math.AP].




\bibitem{FT}
J.~Forlano, L.~Tolomeo,
{\it Quasi-invariance of Gaussian measures of negative regularity for fractional nonlinear Schr\"odinger equations},
J. Eur. Math. Soc. (2025), DOI:10.4171/JEMS/1643.



\bibitem{FG}
M.~Furlan, M.~Gubinelli,
{\it Weak universality for a class of 3d stochastic reaction-diffusion models,}
Probab. Theory Related Fields 173 (2019), no.3-4, 1099--1164.




\bibitem{GTV}
J.~Ginibre, Y.~Tsutsumi, G.~Velo,
{\it On the Cauchy problem for the Zakharov system}, J. Funct. Anal. 151 (1997), no. 2, 384--436.

\bibitem{GJ}
J.~Glimm, A.~Jaffe,
{\it Positivity of the $\phi^4_3$ Hamiltonian},
Fortschr. Physik 21 (1973), 327--376.

\bibitem{Graf}
L.~Grafakos,
{\it Modern Fourier analysis,}
 Third edition. Graduate Texts in Mathematics, 250. Springer, New York, 2014. xvi+624 pp.


%




\bibitem{GH18b}
M.~Gubinelli, M.~Hofmanov\'a,
{\it
A PDE construction of the Euclidean $\Phi^4_3$ quantum field theory},
 Comm. Math. Phys. 384 (2021), no. 1, 1--75.


\bibitem{GIP}
M.~Gubinelli, P.~Imkeller, N.~Perkowski,
{\it Paracontrolled distributions and singular PDEs},
Forum Math. Pi 3 (2015), e6, 75 pp.


\bibitem{GKO}
M.~Gubinelli, H.~Koch, T.~Oh,
{\it  Renormalization of the two-dimensional stochastic nonlinear wave equations,}
 Trans. Amer. Math. Soc.
 370 (2018), no 10, 7335--7359.

\bibitem{GKO2}
M.~Gubinelli, H.~Koch, T.~Oh,
{\it Paracontrolled approach to the three-dimensional stochastic nonlinear wave equation with quadratic nonlinearity},
J. Eur. Math. Soc. (JEMS) 26 (2024), no. 3, 817--874.

%

\bibitem{GKOT}
M.~Gubinelli, H.~Koch, T.~Oh, L.~Tolomeo,
{\it Global dynamics for  the two-dimensional stochastic nonlinear wave equations,}
Int. Math. Res. Not. IMRN (2022), no. 21, 16954--16999.


\bibitem{GP16}
M.~Gubinelli, N.~Perkowski,
{\it The Hairer-Quastel universality result at stationarity},
Stochastic analysis on large scale interacting systems, RIMS K\^{o}ky\^{u}roku Bessatsu, B59, Res. Inst. Math. Sci. (RIMS), Kyoto, 2016, pp. 101--115.



\bibitem{GP17}
M.~Gubinelli, N.~Perkowski,
{\it KPZ reloaded},
Comm. Math. Phys. 349 (2017), no. 1, 165--269.








\bibitem{HKN}
M.~Hairer, S.~Kusuoka, H.~Nagoji,
{\it Singularity of solutions to singular SPDEs},
arXiv:2409.10037 [math.PR].



\bibitem{HQ}
M.~Hairer, J.~Quastel,
{\it A class of growth models rescaling to KPZ},
Forum Math. Pi 6 (2018), e3, 112 pp.

\bibitem{HX18}
M.~Hairer, W.~Xu,
{\it Large-scale behavior of three-dimensional continuous phase coexistence models},
Comm. Pure Appl. Math. 71 (2018), no. 4, 688--746.

\bibitem{HX19}
M.~Hairer, W.~Xu,
{\it Large scale limit of interface fluctuation models},
Ann. Probab. 47 (2019), no.6, 3478--3550.



\bibitem{Kell}
H.~G.~Kellerer,
{\it Duality theorems and probability metrics.}
Proceedings of the seventh conference on probability theory (Bra\c{s}ov, 1982), 211--220.
VNU Science Press, b.v., Utrecht, 1985.




\bibitem{KM93}
S.~Klainerman, M.~Machedon,
{\it Space-time estimates for null forms and the local existence theorem},
Comm. Pure Appl. Math. 46 (1993), no. 9, 1221--1268.

\bibitem{KWX}
F.~Kong, H.~Wang, W.~Xu,
{\it Hairer-Quastel universality for KPZ -- polynomial smoothing mechanisms, general nonlinearities and Poisson noise},
arXiv:2403.06191 [math.PR].







\bibitem{LW1}
R.~Liang, Y.~Wang,
{\it Gibbs measure for the focusing fractional NLS on the torus},
SIAM J. Math. Anal. 54 (2022), no. 6, 6096--6118.


\bibitem{LW2}
R.~Liang, Y.~Wang,
{\it Gibbs dynamics for fractional nonlinear Schr\"odinger equations with weak dispersion},
Comm. Math. Phys. 405 (2024), no. 10, Paper No. 250, 69 pp.




\bibitem{Liu}
R.~Liu,
{\it Weak universality of global dynamics for the fractional hyperbolic $\Phi^4_3$-model},
arXiv:2610.15135 [math.PR].


\bibitem{McK}
H.P.~McKean,
{\it Statistical mechanics of nonlinear wave equations. IV. Cubic Schr\"odinger}, Comm. Math.
Phys., 168 (1995), no. 3, 479--491. 
{\it Erratum: Statistical mechanics of nonlinear wave equations. IV.
Cubic Schr\"odinger}, Comm. Math. Phys., 173 (1995), no. 3, 675.


\bibitem{MR}
L.~Molinet and F.~Ribaud,
{\it On the low regularity of the Korteweg-de Vries-Burgers equation},
Int. Math. Res. Not. (2002), no. 37, 1979--2005.




%
%
%
%

\bibitem{MWX}
J.-C.~Mourrat, H.~Weber, W.~Xu,
{\it Construction of $\Phi^4_3$ diagrams for pedestrians,}
 From particle systems to partial differential equations, 1--46, Springer Proc. Math. Stat., 209, Springer, Cham, 2017.


%



\bibitem{Nua}
D.~Nualart,
{\it The Malliavin calculus and related topics,}
 Second edition. Probability and its Applications (New York). Springer-Verlag, Berlin, 2006. xiv+382 pp.



\bibitem{OOcomp}
T.~Oh, M.~Okamoto,
{\it Comparing the stochastic nonlinear wave and heat equations: a case study},
Electron. J. Probab. 26 (2021), paper no. 9, 44 pp.



\bibitem{OOR}
T.~Oh, M.~Okamoto, T.~Robert,
{\it A remark on triviality for the two-dimensional stochastic nonlinear wave equation},
Stochastic Process. Appl. 130 (2020), no. 9, 5838--5864.

\bibitem{OOTol1}
T.~Oh, M.~Okamoto, L.~Tolomeo,
{\it Focusing $\Phi^4_3$-model with a Hartree-type nonlinearity}, 
Mem. Amer. Math. Soc. 304 (2024), no. 1529.


\bibitem{OOT2}
T.~Oh, M.~Okamoto, L.~Tolomeo,
{\it Stochastic quantization of the $\Phi_3^3$-model}, 
Mem. Eur. Math. Soc., 16
EMS Press, Berlin, 2025. viii+145 pp.



\bibitem{OPTz}
T.~Oh, O.~Pocovnicu, N.~Tzvetkov,
{\it Probabilistic local well-posedness of the cubic nonlinear wave equation in negative Sobolev spaces},
Ann. Inst. Fourier (Grenoble) 72 (2022), no. 2, 771--830.









\bibitem{ORTz}
T.~Oh, T.~Robert, N.~Tzvetkov,
{\it Stochastic nonlinear wave dynamics on compact surfaces},
Ann. H. Lebesgue 6 (2023), 161--223












\bibitem{OTh}
T.~Oh, L.~Thomann,
{\it A pedestrian approach to the invariant Gibbs measure for the 2-$d$ defocusing nonlinear Schr\"odinger equations},  Stoch. Partial Differ. Equ. Anal. Comput. 6 (2018), 397--445.
%
%
%
%
\bibitem{OTh2}
T.~Oh, L.~Thomann,
{\it Invariant Gibbs measure for
the 2-$d$ defocusing nonlinear wave  equations},
 Ann. Fac. Sci. Toulouse Math.
 29 (2020), no. 1, 1--26


\bibitem{OTWZ}
T.~Oh, L.~Tolomeo, Y.~Wang, G.~Zheng,
{\it Hyperbolic $P (\Phi)_2$-model on the plane},
Comm. Math. Phys. 407 (2026), no. 2, Paper No. 34, 84 pp.


\bibitem{OWZ}
T.~Oh, Y.~Wang, Y.~Zine,
{\it Three-dimensional stochastic cubic nonlinear wave equation with almost space-time white noise},
Stoch. Partial Differ. Equ. Anal. Comput. 10 (2022), no. 3, 898--963.


\bibitem{PW}
G.~Parisi, Y.S.~Wu,
{\it Perturbation theory without gauge fixing,}
Sci. Sinica 24 (1981), no. 4, 483--496.


%


\bibitem{RSS}
S.~Ryang, T.~Saito, K.~Shigemoto,
{\it Canonical stochastic quantization}, Progr. Theoret. Phys. 73 (1985),
no. 5, 1295--1298.

\bibitem{Sch}
R.~Schippa,
{\it On Strichartz estimates from $\l^2$-decoupling and applications}, In: D\"orfler, W., et al. Mathematics of Wave Phenomena. Trends in Mathematics. Birkh\"auser, Cham, 2020.

\bibitem{SX}
H.~Shen, W.~Xu,
{\it Weak universality of dynamical $\Phi^4_3$: non-Gaussian noise},
Stoch. Partial Differ. Equ. Anal. Comput. 6 (2018), no.2, 211--254.

\bibitem{Shige}
I.~Shigekawa,
{\it Stochastic analysis,}
Translated from the 1998 Japanese original by the author. Translations of Mathematical Monographs, 224. Iwanami Series in Modern Mathematics. American Mathematical Society, Providence, RI, 2004. xii+182 pp.
%
\bibitem{Simon}
B.~Simon,
{\it  The $P(\varphi)_2$ Euclidean (quantum) field theory,} Princeton Series in Physics. Princeton University Press, Princeton, N.J., 1974. xx+392 pp.
%

\bibitem{STz20}
C.~Sun, N.~Tzvetkov,
{\it Refined probabilistic global well-posedness for the weakly dispersive NLS},
Nonlinear Anal. 213 (2021), Paper No. 112530, 91 pp.


\bibitem{STz21}
C.~Sun, N.~Tzvetkov,
{\it Gibbs measure dynamics for the fractional NLS},
SIAM J. Math. Anal. 52 (2020), no. 5, 4638--4704.


\bibitem{STzX}
C.~Sun, N.~Tzvetkov, W.~Xu,
{\it Weak universality results for a class of nonlinear wave equations}, Ann. Inst. Fourier (Grenoble) 76 (2026), no. 2, 525--576.

\bibitem{Tao}
T.~Tao, 
{\it Nonlinear dispersive equations. Local and global analysis,}
CBMS Regional Conference Series in Mathematics, 106. Published for the Conference Board of the Mathematical Sciences, Washington, DC; by the American Mathematical Society, Providence, RI, 2006. xvi+373 pp.


%
%





\bibitem{Tolomeo2}
L.~Tolomeo,
{\it Global well-posedness of the two-dimensional stochastic nonlinear wave equation on an unbounded domain},
 Ann. Probab. 49 (2021), no. 3,
 1402--1426.







\bibitem{Tz08}
N.~Tzvetkov,
{\it Invariant measures for the defocusing nonlinear Schr\"odinger equation},
Ann. Inst. Fourier (Grenoble) 58 (2008), no.7, 2543--2604.



\bibitem{Ust}
A.~ \"Ust\"unel,
{\it Variational calculation of Laplace transforms via entropy on Wiener space and applications},
J. Funct. Anal. 267 (2014), no. 8, 3058--3083.



\bibitem{Villani}
C.~Villani,
{\it Topics in optimal transportation,} Graduate Studies in Mathematics, 58. American Mathematical Society, Providence, RI, 2003. xvi+370 pp.




\end{thebibliography}
\end{document}